\documentclass[reqno]{amsart}
\usepackage{geometry}
\usepackage{enumitem}

\usepackage{amsmath,amssymb,amsthm,amsfonts}
\usepackage{physics}
\allowdisplaybreaks[4]
\usepackage{hyperref}
\usepackage{mathrsfs}
\usepackage{appendix}
\usepackage{graphicx}
\usepackage{setspace}
\usepackage{mathtools}
\usepackage{latexsym}
\usepackage{mleftright}
\usepackage{cite}
\usepackage{color}

\newtheorem{remark}{Remark}[section]

\numberwithin{equation}{section}
\newtheorem{Theorem}{Theorem}[section]
\newtheorem{Lemma}{Lemma}[section]
\newtheorem{Proposition}{Proposition}[section]
\newtheorem{Remark}{Remark}[section]

\allowdisplaybreaks

\begin{document}
\title[Navier-Stokes/Allen-Cahn system with moving contact line]
{well-posedness of the Navier-Stokes/Allen-Cahn system with moving contact lines}
\thanks{$^*$Corresponding author}
\author[Yinghua Li]{Yinghua Li}
\address[Y. Li]{School of Mathematical Sciences, South China Normal University,
Guangzhou, 510631, China}
\email{yinghua@scnu.edu.cn}
\author[Yuanxiang Yan]{Yuanxiang Yan}
\address[Y. Yan]{School of Mathematical Sciences, South China Normal University,
Guangzhou, 510631, China}
\email{yuanxiangyan@m.scnu.edu.cn}
\author[Xijun Yin]{Xijun Yin$^*$}
\address[X. Yin]{School of Mathematical Sciences, South China Normal University,
Guangzhou, 510631, China}
\email{yiner\_61@163.com}

\date{\today}

\begin{abstract}
In this paper, we study a diffuse interface model for two-phase immiscible flows coupled by Navier-Stokes equations
and mass-conserving Allen-Cahn equations.
The contact line (the intersection of the fluid-fluid interface with the solid wall) moves along the wall when one fluid replaces the other, such as in liquid spreading or oil-water displacement.
The system is equipped with the generalized Navier boundary conditions (GNBC)
for the fluid velocity ${\boldsymbol u}$, and dynamic boundary condition
or relaxation boundary condition  for the phase field variable $\phi$.
We first obtain the local-in-time existence of unique strong solutions
to the 2D and 3D Navier-Stokes/Allen-Cahn (NSAC) system with generalized Navier boundary conditions  and dynamic boundary condition.
For the 2D case in channels, we further show these solutions can be extended to any large time $T$.
Additionally, we prove the local-in-time strong solutions for systems with generalized Navier boundary conditions and relaxation boundary condition in 3D channels. Finally, we establish a global unique strong solution accompany with some exponential decay estimates when the fluids are near phase separation states and the contact angle closes to 90 degrees or the fluid-fluid interface tension constant is small.
\end{abstract}

\maketitle
\tableofcontents
\vspace{-5mm}

\section{Introduction}
Diffuse interface models for two-phase flows have attracted increasing attention due to both theoretical advancements and their successful application in computational simulations. In these models, the interface between phases is treated as a diffuse region, where sharp interfaces are replaced by narrow transition layers. The distribution of the phases and their components is characterized by a phase field variable $\phi$, governed by either the Cahn-Hilliard equations or the Allen-Cahn equations. The fluid dynamics of the system are described by the Navier-Stokes equations. This coupled framework results in either the Navier-Stokes/Cahn-Hilliard (NSCH) system or the Navier-Stokes/Allen-Cahn (NSAC) system.
This paper is concerned with the following incompressible NSAC system
\begin{equation}
\label{1}
\begin{cases}
{\bf u}_t+({\bf u}\cdot\nabla){\bf u}+\nabla p-{\rm div}{\mathbb S} ({\bf u})=-\rm div(\nabla\phi\otimes\nabla\phi) ,&\rm in~Q,
\\
{\rm div}{\bf u}=0,&\rm in~Q,
\\
\phi_t+{\bf u}\cdot\nabla\phi=\bar\mu-\mu,&\rm in~Q,
\\
\mu=-\Delta\phi+f(\phi),&\rm in~Q,
\end{cases}
\end{equation}
where $Q=\Omega\times(0, T)$, $\Omega\subset\mathbb R^N (N=2,3)$ denotes a bounded domain with smooth boundary $\Gamma$
and $T>0$ is a given time. $\bf u$ is the average velocity of the mixed fluids, the order parameter $\phi$ represents the relative concentration difference of the two components, $p$ is the pressure and $\bar\mu$ is the spatial average defined by
$\bar\mu=\frac{1}{|\Omega|}\int_\Omega\mu {\rm d}x$.
The Newtonian viscous stress tensor ${\mathbb S}({\bf u})=\eta(\nabla{\bf u}+\nabla{\bf u}^T)$ with viscosity coefficient $\eta>0$.
$f(s)=F^{\prime}(s)$ and we choose the double-well potential $F$ as Landau type, i.e. $F(s)=\frac14(s^2-1)^2$ for $s\in\mathbb R$.
The coupled nonlinear system (\ref{1}) is supplemented with the initial conditions
\begin{align} \label{I}
{\bf u}\big|_{t=0}={\bf u}_0, \quad \phi\big|_{t=0}=\phi_0, \qquad {\rm in}~ \Omega.
\end{align}

When modeling the dynamics of two immiscible fluids flowing over a solid surface, it was demonstrated in \cite{D-D-74} that the classical no-slip boundary condition results in nonphysical contact-line singularities or infinite viscous dissipation near moving contact lines. To address this issue, Qian-Wang-Shen \cite{25,26,27} proposed the following generalized Navier boundary conditions and relaxation boundary condition
\begin{equation} \label{2*}
\begin{cases}
{\bf u}\cdot{\bf n}=0,\quad &\rm on~\Gamma\times(0,T),
\\
\beta{\bf u}_{\boldsymbol {\tau}}+(\mathbb S({\bf u})\cdot{\bf n})_{\boldsymbol{\tau}}=L(\phi)\nabla_{\boldsymbol\tau}\phi,
\quad &\rm on~\Gamma\times(0,T),
\\
\phi_t+{\bf u}_{\boldsymbol{\tau}}\cdot\nabla_{\boldsymbol{\tau}}\phi=- L(\phi),&\rm on~\Gamma\times(0,T),
\end{cases}
\end{equation}
where $\beta>0$ is the slip coefficient, ${\boldsymbol {\tau}}$ and ${\bf n}$ are the unit tangential vector and the unit normal vector on the boundary $\Gamma$, respectively.
The subscript ${\boldsymbol {\tau}}$ denotes the tangential component of a vector $\boldsymbol w$, i.e.
${\boldsymbol w}_{\boldsymbol\tau}={\boldsymbol w}-({\boldsymbol w}\cdot {\bf n}){\bf n}$,
$\nabla_{\boldsymbol\tau}=\nabla-({\bf n}\cdot\nabla){\bf n}$ is the tangential gradient on $\Gamma$.
The term $L(\phi)\nabla_{\boldsymbol\tau}\phi$ on the right hand of (\ref{2*})$_2$ represents the uncompensated Young stress, and
\begin{align} \label{L-phi}
L(\phi):=\partial_{\bf n}\phi+\gamma^\prime_{fs}(\phi),
\end{align}
in which $\gamma_{fs}(\phi)$ is the fluid-solid interfacial free energy per unit area.
(\ref{2*})$_1$ and (\ref{2*})$_2$ are the {\it generalized Navier boundary conditions} (GNBC) for velocity $\bf u$,
and (\ref{2*})$_3$ with (\ref{L-phi}) is called the {\it relaxation boundary condition} for phase field variable $\phi$ (cf. \cite{C-W-X}).

\vskip1mm
The boundary conditions (\ref{2*}) with (\ref{L-phi}) have been successfully used to simulate moving contact lines,
see \cite{M-C-Y-Z, C-Y-Z, S-Y-Y, X-D-Y, Y-Y} for example.
However, due to the lack of diffusion effects for $\phi$ on the boundary, theoretical analysis of this case remains challenging.
To our knowledge, the existing analytical results for (\ref{2*})--(\ref{L-phi}) are \cite{D-L-L-Y} and \cite{C-H-H-S}, which established the existence of local strong solutions for the incompressible NSCH system and compressible NSAC system, respectively.
When surface diffusion is taken into account, $L(\phi)$ is adjusted as follows
\begin{align}\label{L-phi-D}
\mathcal L(\phi):=-\gamma\Delta_{\boldsymbol\tau}\phi+\partial_{\bf n}\phi+G^\prime(\phi),
\end{align}
for some $\gamma>0$, where $G(\cdot)$ is a prescribed function,
and $\Delta_{\boldsymbol\tau}$ denotes the Laplace-Beltrami operator on $\Gamma$.
In this case, (\ref{2*})$_3$ with (\ref{L-phi-D}) is referred to as the {\it dynamic boundary condition} for $\phi$.
In this paper, we choose $G(\phi)=\gamma_{fs}(\phi)$.

\vskip1mm
Assuming sufficient regularity of the solutions to the NSAC system with GNBC and dynamic boundary condition (i.e. the problem (\ref{1})--(\ref{2*}), (\ref{L-phi-D})), we can derive the total energy balance as follows
\begin{align} \label{energy-0}
\frac{\rm d}{{\rm d}t}E(t)
+\int_\Omega\frac12|\mathbb S({\bf u})|^2+\int_\Omega|\mu-\bar\mu|^2{\rm d}x
+\int_\Gamma\beta|{\bf u}_{\boldsymbol\tau}|^2{\rm d}S+\int_\Gamma|\mathcal L(\psi)|^2{\rm d}S=0,
\end{align}
where the {\it total energy} $E(t)$ is defined as
\begin{align*}
E(t)=E_{kin}({\bf u}(t))+E_{free}(\phi(t), \psi(t)),
\end{align*}
with
$$
E_{free}(\phi, \psi)=E_{bulk}(\phi)+E_{surf}(\psi).
$$
The individual energy components are
\begin{align*}
E_{kin}({\bf u})=&\int_\Omega\frac12|{\bf u}|^2{\rm d}x, \qquad
E_{bulk}(\phi)=\int_\Omega\left(\frac12|\nabla\phi|^2+F(\phi)\right){\rm d}x, \quad
\\
&E_{surf}(\psi)=\int_\Gamma\left(\frac\gamma2|\nabla_\tau\psi|^2+\gamma_{fs}(\psi)\right){\rm d}S,
\end{align*}
representing the {\it kinetic energy}, the {\it bulk free energy} and the {\it surface free energy}, respectively.
Here $\psi=\phi\big|_\Gamma$ is the trace of the phase field variable $\phi$ on the boundary, describing the material distribution on the surface,
the constant $\gamma\ge0$ acts as a weight for surface diffusion effects.
Evidently, when $\gamma=0$, (\ref{energy-0}) reduces to the energy balance of the NSAC system with GNBC and relaxation boundary condition (i.e. the problem (\ref{1})--(\ref{2*}), (\ref{L-phi})).

\vskip1mm
It is worth noting that the presence of the nonlocal term $\bar\mu$,
combined with the boundary condition (\ref{2*})$_1$ and the divergence free condition for $\bf u$,
ensures the Allen-Cahn equations (\ref{1})$_{3}$ and (\ref{1})$_{4}$ have the property of mass conservation, i.e.
$$
\frac{\rm d}{{\rm d}t}\int_\Omega\phi(x,t){\rm d}x=0, \qquad t\ge0,
$$
which is differ from the classical Allen-Cahn equations.
Unlike the classical case, the mass-conserving Allen-Cahn equations lack a comparison principle.
For further discussion of their distinct properties, see \cite{B-S, C-H-L, G-97} and the references therein.
Additionally, both the mass-conserving Allen-Cahn and Cahn-Hilliard equations satisfy the mass conservation.
Their dynamics maintain the integrity of the interface: the
mixing-demixing mechanism (which also translates into $\mu$) establishes a balance which avoids uncontrolled
expansion or shrinkage of the interface layer (cf. \cite{D-F-20, F-L-S-Y-05}).
Furthermore, in \cite{G-G-W}, the authors investigated the mass-conserving NSAC system with unmatched densities and established the existence and uniqueness of both global weak and strong solutions under classical boundary conditions.

For single fluid (i.e. $\phi=\pm1$), the boundary conditions (\ref{2*}) reduce to the following {\it Navier boundary conditions} (NBC) with friction,
which were proposed in \cite{Navier}
\begin{equation}
\label{NBC}
\begin{cases}
{\bf u}\cdot{\bf n}=0,\quad &\rm on~\Gamma\times(0,T),
\\
\beta{\bf u}_{\boldsymbol {\tau}}+(\mathbb S({\bf u})\cdot{\bf n})_{\boldsymbol{\tau}}=0,&\rm on~\Gamma\times(0,T).
\end{cases}
\end{equation}
Formally, if the slip coefficient $\beta=\infty$, $(\ref{NBC})$ are reduced to the {\it no-slip} condition,
and when $\beta=0$, $(\ref{NBC})$ is referred as the {\it Navier slip conditions without friction}.
For detailed discussion of Navier-Stokes equations with NBC, see \cite{2, 6, 16, 40, 41, 45, 48}.
Actually, a number of molecular dynamics studies have shown that for single-phase flows or away from the contact line region, the NBC is valid in describing the fluid slipping at solid surface, see \cite{C-K-B} for instance. It can account for a small amount of relative slip between the fluid and the solid surface is detected at a high flow rate.
However, in the molecular-scale vicinity of the contact line the NBC failed totally to account for the near-complete slip.
Under the usual hydrodynamic assumptions, namely, incompressible Newtonian fluids, no-slip boundary condition and smooth rigid walls, there is a velocity discontinuity at the moving contact line and the tangential force exerted by the fluids on the solid surface in the vicinity of the contact line becomes infinite which will result in non-integrability.
In order to remove the stress singularity at the moving contact line, while retaining the Newtonian behavior of stress and restoring the continuity of velocity field,
Qian-Wang-Sheng \cite{27, 25, 26} proposed GNBC as a generalization of the Navier boundary conditions.

\vskip1mm
When the fluids flow very slow, one can ignore the velocity, i.e. ${\bf u}=0$.
Then (\ref{2*})$_3$ with (\ref{L-phi}) reduces to the relaxation boundary condition
\begin{align}\label{RBC}
\phi_t+\partial_{\bf n}\phi+\gamma^\prime_{fs}(\phi)=0, \quad \rm on~\Gamma\times(0,T).
\end{align}
To date, the only rigorous analysis of (\ref{RBC}) appears in the seminal work by Chen-Wang-Xu \cite{C-W-X}, who established the global existence of solutions for the Cahn-Hilliard equations with relaxation boundary condition and justified the sharp interface limit through matched asymptotic expansions.
Substituting (\ref{L-phi-D}) for (\ref{L-phi}) in (\ref{2*})$_3$ yields the (Allen-Cahn type) dynamic boundary condition
\begin{align}\label{DBC}
\phi_t-\gamma\Delta_{\boldsymbol\tau}\phi+\partial_{\bf n}\phi+\gamma^\prime_{fs}(\phi)=0, \quad \rm on~\Gamma\times(0,T).
\end{align}
Obviously, the supplementary diffusion term introduces smoothing regularization along the boundary, a mechanism explaining why this class of boundary conditions have garnered significant attention.
The theoretical research for the Cahn-Hilliard equations with dynamic boundary condition (\ref{DBC})
or fourth order Cahn-Hilliard type dynamic boundary conditions can be found in \cite{L-W, M-Z, G-M-S, K-S, P-R-Z, R-Z}.
The Allen-Cahn equations equipped with (\ref{DBC}) have been studied in \cite{C-C, C-G-N-S, G-G}.

\vskip1mm
The NSCH/NSAC system with no-slip boundary condition and Neumann boundary condition have been extensively
studied, such as \cite{P-G-S, G-G-10, A-09, B-99, X-Z-L, G-M-T} and the references therein.
For diffuse interface models supplemented by the GNBC and relaxation boundary condition (i.e. (\ref{2*}) with (\ref{L-phi})),
we just find two rigorous studies:
Ding et al. \cite{D-L-L-Y} addressed the NSCH system in 2D channels, establishing the existence of unique local strong solutions.
And Chen et al. \cite{C-H-H-S} deduced a compressible NSAC system and proved the existence and uniqueness of local strong solutions in 3D bounded domains.
Then let us turn to review some investigations on diffuse interface models with the GNBC and dynamic boundary condition
(i.e. (\ref{2*}) with (\ref{L-phi-D})).
Gal-Grasselli-Miranville \cite{G-G-M} first proved the existence of a suitable global energy solution of NSCH system and also established the convergence of any such solution to a single equilibrium.
Gal-Grasselli-Wu \cite{G-G-W-19} obtained the existence of a global weak solution for the NSCH system with arbitrary initial data in both 2D and 3D.
Meanwhile, Cherfils et al. \cite{C-F-M-M-P-P} considered the compressible NSCH system and demonstrated the existence of global weak solutions for any finite energy initial data.
Recently, Giorgini-Knopf \cite{G-K} proposed a Chan-Hilliard type dynamic boundary conditions for NSCH models with unmatched densities, then proved the existence result for fluids with matched densities.
For NSAC system with GNBC and dynamic boundary condition, Gal-Grasselli-Poiatti \cite{G-G-P} proposed the Voigt approximation, based on which
a global weak solution satisfying the energy inequality was established.
In summary, existing literature has only established the existence of weak solutions, with no uniqueness results reported.

In this paper, we are interested in the existence and uniqueness of strong solutions to the NSAC system (\ref{1})
under the conditions (\ref{I})--(\ref{2*}) with (\ref{L-phi}) or (\ref{L-phi-D}).
Roughly speaking, we aim to achieve the following results:
\vskip1mm
\noindent
$\bullet$\quad Let the initial data satisfy $({\bf u}_0, \phi_0, \psi_0)\in H^2\times H^3\times H^3(\Gamma)$. Then the NSAC system with
\vskip1mm
~ GNBC and dynamic boundary condition, i.e. (\ref{1})--(\ref{2*}) with (\ref{L-phi-D})
\vskip1mm
\quad~ (i)~has a unique local-in-time strong solution in $\Omega\subset \mathbb R^N(N=2,3)$;
\vskip1mm
\quad~ (ii)~admits a unique global strong solution in $\Omega=\mathbb{T}\times(-1,1)\subset\mathbb R^2$.
\vskip1mm
\noindent
$\bullet$\quad The NSAC system with GNBC and relaxation boundary condition, i.e. (\ref{1})--(\ref{2*}) with (\ref{L-phi})
\vskip1mm
\quad~ (i)~has a unique local-in-time strong solution in $\Omega=\mathbb{T}^2\times(-1, 1)\subset\mathbb R^3$, if the initial data
\vskip1mm
\qquad~ \quad
satisfy $({\bf u}_0, \phi_0, \psi_0)\in H^2\times H^3\times H^{\frac52}(\Gamma)$;
\vskip1mm
\quad~ (ii)~admits a unique global strong solution in $\Omega=\mathbb{T}^2\times(-1, 1)\subset\mathbb R^3$, provided the initial
\vskip1mm
\qquad~ \quad
data $({\bf u}_0, \nabla\phi_0, \nabla_{\boldsymbol \tau}\psi_0)\in H^2\times H^2\times H^{\frac32}(\Gamma)$, $\phi_0^2-1\in L^2$
and for a small constant $\varepsilon_0>0$,
\begin{align*}
\|{\bf u}_0\|_{H^2}^2+\|\nabla\phi_0\|_{H^2}^2+\|\mu_0-\bar\mu_0\|_{H^1}^2+\|\phi_0^2-1\|_{L^2}^2
+|\nu\cos\theta_s|\le \varepsilon_0.
\end{align*}

\qquad~ \quad
Moreover, we also obtain some exponential decay estimates.

We first consider the NSAC system with GNBC and dynamic boundary condition (the problem (\ref{1})-(\ref{2*}), (\ref{L-phi-D})).
Due to the interfacial diffusion effect, we can discuss the local well-posedness on any bounded domain in both 2D and 3D cases.
Given the generality of the domain, when handling higher-order spatial derivative a priori estimates, we perform interior and boundary estimates separately by using cut-off functions (see Lemma \ref{L-222} and \ref{Lemma4}), which differs from the NSCH system case in \cite{D-L-L-Y}.
After obtaining the local existence of a unique strong solution via the contraction mapping principle, we aim to extend it globally.
The obstacle arises from the lack of information on pressure $p$ and its derivatives, which we seek to eliminate by ensuring terms containing them equal to zero, such as
\begin{align} \label{p-equality}
\int_\Omega\nabla\nabla_ {\boldsymbol\tau}p\cdot\nabla_{\boldsymbol\tau}{\bf u}{\rm d}x=-\int_\Omega\nabla_ {\boldsymbol\tau}p\cdot\nabla_{\boldsymbol\tau}{\rm div}{\bf u}{\rm d}x+\int_\Gamma\nabla_{\boldsymbol\tau}p\cdot\nabla_{\boldsymbol\tau}{\bf u}\cdot{\bf n}{\rm d}S=0.
\end{align}
Then the a priori estimates will proceed smoothly.
From boundary condition (\ref{2*})$_1$, we derive
$\nabla_{\boldsymbol\tau}{\bf u}\cdot{\bf n}+{\bf u}\cdot\nabla_{\boldsymbol\tau}{\bf n}=\nabla_{\boldsymbol\tau}({\bf u}\cdot{\bf n})=0$ on $\Gamma$.
Hence if $\nabla_{\boldsymbol\tau}{\bf n}\cdot{\bf u}=0$ holds on $\Gamma$, the identity (\ref{p-equality}) is valid.
We also note that in 2D case, $\nabla_{\boldsymbol\tau}{\bf n}=\kappa{\boldsymbol\tau}$ holds, where $\kappa$ denotes the boundary curvature
(Lemma \ref{poincare U}).
Consequently, for 2D channel domains, identity (\ref{p-equality}) naturally holds due to the geometric property $\kappa\equiv0$.
This processing parallels that in \cite{D-L-Y} for the NSCH system.
Unlike the fourth-order NSCH system in \cite{D-L-Y}, where standard approaches can obtain the regularity of the gradient of chemical potential $\nabla\mu$ and the regularity can be directly transferred to the phase field variable $\phi$, our regularity for $\mu$ remains limited.
To resolve this problem, we reformulate the governing equations into an elliptic problem for $\phi$. With the help of the elliptic regularity estimate in Lemma \ref{estimates}, we derive the required bounds for $\phi$, see Lemma \ref{Phi2}--\ref{4*}.
Unfortunately, since the global bounds depend on the surface diffusion coefficient $\gamma$, we cannot achieve global existence for relaxation boundary condition via limiting processes.

\vskip1mm
For the NSAC system with GNBC and relaxation boundary condition (the problem (\ref{1})-(\ref{2*}), (\ref{L-phi})),
the boundary condition for phase field variable $\phi$ becomes hyperbolic, which prevents regularity improvements on the boundary.
To overcome this difficulty, we introduce a surface diffusion term $\delta\Delta_{\boldsymbol\tau}\phi$ into the boundary condition to formulate a $\delta$-approximation problem.
However, this modification alone fails to yield uniform estimates for the approximate solutions independent of the surface diffusion coefficient $\delta$.
Inspired by the approach in \cite{D-L-L-Y} for the NSCH system, we incorporate the tangential Laplace operator $\delta\Delta_{\boldsymbol\tau}\phi=\delta(\partial_x^2+\partial_y^2)\phi$ into the equation.
By restricting the problem to a channel $\Omega=\mathbb{T}^2\times(-1,1)$, this operator counteracts boundary diffusion when the interior equation is restricted to the boundary, preserving the critical boundary structure:
\begin{align*}
\bar\mu-\mu=-L(\phi)=-\partial_{\bf n}\phi-\gamma_{fs}^\prime(\phi),\quad  \rm on~\Gamma.
\end{align*}
This equality links the chemical potential $\mu$ to the normal derivative of $\phi$ and interfacial forces, providing essential tools for deriving higher-order estimates of $\phi$.
Consequently, we establish uniform-in-$\delta$ estimates for the approximate solutions.
The key of proving the existence and uniqueness of local solutions lies in Proposition \ref{TH im}, which also serves as the foundation for subsequent global analysis.
To extend local solutions, we impose small initial value assumptions so that the system perturbs near the phase separation state ($\phi=\pm1$).
This ensures the phase field variable $\phi$ satisfies $3\phi^2-1\ge 1>0$.
Using this property, we first derive time-independent uniform estimates, then extend local solutions to global ones by using continuation method, and obtain exponential decay properties.
The main distinction between NSAC and NSCH systems with GNBC and relaxation condition is that NSAC results apply to 3D channels, while NSCH conclusions are for 2D channels.

\vskip1mm
We want to point out that, the mass-conserving property of the coupled Allen-Chan equations is crucial to our analysis.
The following properties will frequently be used throughout this paper:
$$
\int_\Omega\phi_t{\rm d}x=0, \qquad \int_\Omega(\mu-\bar\mu){\rm d}x=0, \qquad\int_\Omega(\mu-\bar\mu)_t{\rm d}x=0.
$$

\vskip1mm
The paper is organized as follows. In Section 2, we introduce the notations and lemmas required for the proofs.
Subsequently, we address the well-posedness of the NSAC system with GNBC under dynamic boundary condition and relaxation boundary condition in Sections 3 and 4, respectively.

\setcounter{section}{1}
\setcounter{equation}{0}
\section{Notations and preliminaries}

We first explain some notations and conventions used throughout the paper.
For a domain $\Omega\subset \mathbb R^N(N=2,3), 1\le p\le\infty$, and $s>0$, we set $L^p=L^p(\Omega)$, $W^{s,p}=W^{s,p}(\Omega)$ and $H^s=W^{s,2}$ while $\|\cdot\|$ indicates the induced norm. In the sequel, we will note the spatial average of a generic function $f$ defined over $\Omega$ as $\bar f=\frac{1}{|\Omega|}\int_\Omega f {\rm d}x$ and define the inner product of two functions that $\langle f,g\rangle=\int_\Omega fg {\rm d}x$. Moreover, we set
\begin{align}
\mathbb D_1&=\left\{({\bf u},\phi,\psi)\in H^2\times H^3\times H^3(\Gamma)\big|{\bf u},\phi,\psi\;\rm satisfy\; (\ref{NSAC})_2\; and\; (\ref{dynamic})_{1,2,4}
\right\}, \nonumber
\\
 \mathbb D_2&= \left\{({\bf u},\phi, \psi)\in H^2\times H^3 \times H^{\frac{5}{2}}(\Gamma)\big|{\bf u},\phi, \psi \;\rm satisfy\; (\ref{NSAC1})_2\; and\; (\ref{boundary})_{1,2,4}
\right\} , \nonumber
\\
\mathbb D_3&= \left\{({\bf u},\nabla\phi, \nabla_{\boldsymbol\tau}\psi )\in H^2\times H^2 \times H^{\frac{3}{2}}(\Gamma) \big|{\bf u},\phi, \psi\;\rm satisfy\; (\ref{NSAC1})_2\; and\; (\ref{boundary})_{1,2,4}
\right\}. \nonumber
\end{align}
\begin{Remark}
When $s$ is non-integer, $W^{s,p}$ is fractional Sobolev space.
\end{Remark}

In this paper, we take the fluid-solid interfacial free energy $\gamma_{fs}(\phi)$ as follows (\cite{24, 26, 27})
\[
\gamma_{fs}(\phi)=-\frac{\nu}{2}\cos\theta_s\sin\left(\frac{\pi\phi}{2}\right),
\]
where $\theta_s$ denotes the static contact angle constant and $\nu>0$ is the fluid-fluid interface tension constant.

\allowdisplaybreaks
\vskip2mm
Next, we list some auxiliary lemmas and inequalities that will be used in subsequent sections.
\begin{Lemma}
\label{Trace} {\rm(Trace Imbedding Theorem, \cite{9})}
Let $\Omega$ be a domain in ${\mathbb R}^N$ that satisfying the $C^{k-1,1}$-regularity condition, and suppose $u\in W^{k,p}(\Omega), p>1$, $k$ is an integer. If $l\le k-1$, then there is a trace operator $\rm Tr\big|_{\Gamma}$ such that
\begin{equation}
\label{L1}
\rm Tr\big|_\Gamma:W^{k,p}(\Omega)\hookrightarrow W^{k-l-\frac{1}{p},p}(\Gamma)
\end{equation}
holds. If $u\big|_\Gamma=\varphi$, we denote $\rm Tr\big|_\Gamma=\varphi$.
\end{Lemma}
\begin{Lemma}
\label{GN} {\rm(Gagliardo-Nirenberg inequality, \cite{3})}
Let $\Omega$ be either ${\mathbb R}^N$ or a half space or a Lipschitz bounded domain in ${\mathbb R}^N$. Assume that the real number $0\le s_1< s\le s_2$, $1\le p_1,p_2,p\le\infty$ and $0<\theta<1$ satisfy the relations
\begin{equation}
s=\theta s_1+(1-\theta)s_2
\;and\;\frac{1}{p}=\frac{\theta}{p_1}+\frac{1-\theta}{p_2}.
\end{equation}
Then there exists a constant $C$ depending on $s_1,s_2,p_1,p_2,\theta$ and $\Omega$ such that
\begin{equation}
\label{L2}
\|f\|_{W^{s,p}}\le C\|f\|_{W^{s_1,p_1}}^\theta\|f\|_{W^{s_2,p_2}}^{1-\theta},\;
\forall f\in W^{s_1,p_1}(\Omega)\cap W^{s_2,p_2}(\Omega)
\end{equation}
holds if and only if
\begin{equation}
s_2 \;is\;an\;integer\ge1,p_2=1\;and\;s_2-s_1\le 1-\frac{1}{p_1}
\end{equation}
fail.
\vskip2mm
In particular, for $k\in\mathbb N$, it holds that
\begin{equation}
\label{L21}
\|f\|_{H^{k+\frac{1}{2}}}^2\le C\|f\|_{H^k}\|f\|_{H^{k+1}}
\end{equation}
\end{Lemma}
\begin{Lemma}
\label{Gronwall} {\rm(Gronwall inequality, \cite{4})}
Let $\eta(\cdot)$ be a nonnegative, absolutely continuous function on $[0,T]$, which satisfies for a.e. t the differential inequality
\begin{equation}
\eta^\prime(t)\le\phi(t)\eta(t)+\psi(t),
\end{equation}
where $\phi(t)$ and $\psi(t)$ are nonnegative, summable function on $[0,T]$. Then
\begin{equation}
\eta(t)\le e^{\int_{0}^{T}\phi(s){\rm d}S}\left[\eta(0)+\int_{0}^{T}\psi(s){\rm d}S
\right].
\end{equation}
In particular, if
\begin{equation}
\eta^\prime(t)\le\phi(t)\eta(t)\;{\rm on}\;[0,T]\;and\;\eta(0)=0,
\end{equation}
then
\begin{equation}
\eta=0\;{\rm on}\;[0,T].
\end{equation}
\end{Lemma}
\begin{Lemma}
\label{Korn} {\rm (Korn's inequality, \cite{2})}
There exists a constant $C>0$, there holds
\begin{equation}
\label{L4}
\|{\mathbb S}({\bf u})\|_{L^2(\Omega)}\ge C\|\nabla{\bf u}\|_{L^2(\Omega)},
\end{equation}
for any ${\bf u}\in V=\{{\bf u}\in H^1(\Omega)\:\big|\:\rm div{\bf u}=0,{\bf u}\cdot{\bf n}=0$ {\rm on} $\Gamma$\}.
\end{Lemma}
\begin{Lemma}
 \label{product} {\rm(Estimates of product of functions, \cite{7})}
The following hold on sufficiently smooth subsets of ${\mathbb R}^N$.
\\
{\rm 1.} Let $0\le r\le s_1\le s_2$ be a such that $s_1>\frac{N}{2}$. Let $f\in H^{s_1}$, $g\in H^{s_2}$. Then $fg\in H^r$ and
\begin{equation}
\label{L5.1}
\|fg\|_{H^r}\le C\|f\|_{H^{s_1}}\|g\|_{H^{s_2}}.
\end{equation}
{\rm 2.} Let $0\le r\le s_1\le s_2$ be a such that $s_2>r+\frac{N}{2}$. Let $f\in H^{s_1}$, $g\in H^{s_2}$. Then $fg\in H^r$ and
\begin{equation}
\label{L5.2}
\|fg\|_{H^r}\le C\|f\|_{H^{s_1}}\|g\|_{H^{s_2}}.
\end{equation}
\end{Lemma}
\begin{Lemma} 
\label{estimates} {\rm(Elliptic estimates, \cite{M-Z})} Let $\Omega\subset\mathbb R^N(N=2,3)$ be a bounded domain with smooth boundary $\Gamma$. Consider the following linear elliptic problem
\begin{equation}
\begin{cases}
\label{problem}
-\Delta\phi=H,&{\rm on}\;\Omega,
\\
-\Delta_{\boldsymbol\tau}\psi+\psi+\partial_{\bf n}\phi=h,&{\rm on}\;\Gamma,
\\
\phi\big|_\Gamma=\psi,&{\rm on}\;\Gamma,
\end{cases}
\end{equation}
where $(H,h)\in H^s(\Omega)\times H^s(\Gamma)$ for any $s\ge 0$ and $s+\frac{1}{2}\notin\mathbb N$. Then every solution $(\phi,\psi)$ to problem $(\ref{problem})$ satisfies the estimates
\begin{align}
\|\phi\|_{H^{s+2}(\Omega)}+\|\psi\|_{H^{s+2}(\Gamma)}\le C\left(
\|H\|_{H^s(\Omega)}+\|h\|_{H^s(\Gamma)}
\right)
\end{align}
for some constant $C>0$ that may depend on $s,\Omega$ and $\Gamma$, but is independent of the solution $(\phi,\psi)$.
\end{Lemma}
\begin{Lemma} {\rm(See \cite{2})}
\label{poincare U}
Let $\Omega = \mathbb{T}^{N} \times (-1,1), N=1, 2$, then there exists a constant $C>0$, it holds
\begin{align}
\label{poincare u}
\|{\bf u}\|_{L^2(\Omega)}\le C\|\nabla_{\bf n}{\bf u}\|_{L^2(\Omega)},
\end{align}
for any ${\bf u}\in V=\{{\bf u}\in H^1(\Omega)\:\big|\:\rm div{\bf u}=0,{\bf u}\cdot{\bf n}=0$ {\rm on} $\Gamma$\}.
\end{Lemma}
\begin{Lemma} {\rm(See \cite{12, 13})}
\label{regularity u}
Let $\Omega\subset\mathbb R^N(N=2,3)$ be a bounded domain with smooth boundary $\Gamma$. Then there exists a constant $C>0$ independent of ${\bf u}$, such that
\begin{align}
\|{\bf u}\|_{H^s}\le C\left(\|\nabla\times{\bf u}\|_{H^{s-1}}+\|{\rm div}{\bf u}\|_{H^{s-1}}+\|{\bf u}\cdot{\bf n}\|_{H^{s-\frac{1}{2}}(\Gamma)}+\|{\bf u}\|_{H^{s-1}}
\right),
\end{align}
for any ${\bf u}\in H^s(\Omega)$, $s\ge 1$.
\end{Lemma}
\begin{Lemma} {\rm(See \cite{14, 16})}
\label{ellptic}
Let $\Omega\subset\mathbb R^N(N=2,3)$ and $\omega=\nabla\times{\bf u}$. Suppose ${\bf u}\in H^2$ and ${\bf u}\cdot{\bf n}=0$ on $\Gamma$. Then we have
\begin{align}
\left[\mathbb S({\bf u})\cdot{\bf n}
\right]\cdot{\boldsymbol\tau}=(\omega\times{\bf n})\cdot{\boldsymbol\tau}-2u\cdot\frac{\partial{\bf n}}{\partial{\boldsymbol\tau}},\quad for\;N=3,
\end{align}
and
\begin{align}
\left[\mathbb S({\bf u})\cdot{\bf n}
\right]\cdot{\boldsymbol\tau}=\omega-2u\cdot\frac{\partial{\bf n}}{\partial{\boldsymbol\tau}},\quad for\;N=2,
\end{align}
where $\Big|\dfrac{\partial{\bf n}}{\partial{\boldsymbol\tau}}\Big|$ is the normal curvature in the $\boldsymbol\tau$ direction when $N=3$, $\dfrac{\partial{\bf n}}{\partial{\boldsymbol\tau}}=\kappa{\boldsymbol\tau}$, while $\kappa$ is the curvature of $\Gamma$ when $N=2$.
\end{Lemma}
\begin{Lemma}
\label{PGN}
Let $\Omega=\mathbb T\times(-1,1)\subset\mathbb R^2$, and $U=\Omega$ or $\Gamma$, for $k\ge 1$, it holds that
\begin{align}
&\|\nabla_{\boldsymbol\tau}^k F\|_{L^2(U)}\le C\|\nabla\nabla_{\boldsymbol\tau}^kF\|_{L^2(U)},\quad\quad\quad\quad \forall\;\nabla_{\boldsymbol\tau}^kF\in H^1(U),\nonumber
\\
&\|F\|_{L^4(\Omega)}\le C\|F\|_{L^2(\Omega)}^\frac{1}{2}\|F\|_{H^1(\Omega)}^\frac{1}{2},\;\:\quad\quad\quad\forall\;F\in H^1(\Omega),\nonumber
\\
&\|F\|_{L^\infty(\Omega)}\le C\|F\|_{L^2(\Omega)}^\frac{1}{2}\|F\|_{H^2(\Omega)}^\frac{1}{2},\;\quad\quad\quad\forall\;F\in H^2(\Omega),\nonumber
\\
&\|F\|_{L^\infty(\Gamma)}\le C\|F\|_{L^2(\Gamma)}^\frac{1}{2}\|F\|_{H^1(\Gamma)}^\frac{1}{2},\;\;\quad\quad\quad\forall\;F\in H^1(\Gamma),\nonumber
\\
&\|F\|_{L^2(\Gamma)}\le C\|F\|_{L^2(\Omega)}^\frac{1}{2}\|F\|_{H^1(\Omega)}^\frac{1}{2},\quad\quad\quad\;\;\;\forall\;F\in H^1(\Omega).\nonumber
\end{align}
\end{Lemma}
\begin{proof}[\bf Proof.]
We only show that the first and the last inequalities, since other inequalities follow from Gagliardo-Nirenberg inequality. Note that $\int_\Omega\nabla_{\boldsymbol\tau}^kF{\rm d}x=\int_\Gamma\nabla_{\boldsymbol\tau}^kF{\rm d}S=0$. Thus, the first inequality holds. It follows from Lemma \ref{Trace} and (\ref{L21}) that
\[
\|F\|_{L^2(\Gamma)}\le C\|F\|_{H^\frac{1}{2}(\Omega)}\le C\|F\|_{L^2(\Omega)}^\frac{1}{2}\|F\|_{H^1(\Omega)}^\frac{1}{2}.
\]
This completes the proof of Lemma \ref{PGN}.
\end{proof}

\setcounter{equation}{0}
\section{NSAC system with GNBC and dynamic boundary condition}
\label{section3}
In this section, we consider the following NSAC system
\begin{equation}
\label{NSAC}
\begin{cases}
{\bf u}_t+({\bf u}\cdot\nabla){\bf u}+\nabla p -{\rm div}{\mathbb S}({\bf u})=-\rm div(\nabla\phi\otimes\nabla\phi) ,&\rm in\quad\Omega\times(0,T),
\\
{\rm div}{\bf u}=0,&\rm in\quad\Omega\times(0,T),
\\
\phi_t+{\bf u}\cdot\nabla\phi=\bar\mu-\mu,&\rm in\quad\Omega\times(0,T),
\\
\mu=-\Delta\phi+f(\phi),&\rm in\quad\Omega\times(0,T),
\end{cases}
\end{equation}
equipped with GNBC and dynamic boundary condition, i.e. ${\mathcal L}(\phi)=-\gamma\Delta_{\boldsymbol\tau}\phi+\partial_{\bf n}\phi+\gamma_{fs}^\prime(\phi)$,
\begin{equation}
\label{dynamic}
\begin{cases}
{\bf u}\cdot{\bf n}=0,&\rm on\quad\Gamma\times(0,T),
\\
\beta{\bf u}_{\boldsymbol {\tau}}+(\mathbb S({\bf u})\cdot{\bf n})_{\boldsymbol{\tau}}=\left(-\gamma\Delta_{\boldsymbol\tau}\psi
+\partial_{\bf n}\phi+\gamma_{fs}^\prime(\psi)
\right)\nabla_{\boldsymbol\tau}\psi,&\rm on\quad\Gamma\times(0,T),
\\
\psi_t+{\bf u}_{\boldsymbol{\tau}}\cdot\nabla_{\boldsymbol{\tau}}\psi
=\gamma\Delta_{\boldsymbol\tau}\psi
-\partial_{\bf n}\phi-\gamma_{fs}^\prime(\psi),&\rm on\quad\Gamma\times(0,T),
\\
\phi\big|_\Gamma=\psi,&\rm on\quad(0,T),
\end{cases}
\end{equation}
and the initial conditions
\begin{equation}
\label{initial condition}
({\bf u},\phi)\big|_{t=0}=({\bf u}_0,\phi_0),\quad \rm in\;\Omega,\quad\psi\big|_{t=0}=\phi_0\big|_\Gamma=\psi_0,\quad\rm on\;\Gamma.
\end{equation}

The main results of this section can be stated as follows:
\begin{Theorem}
\label{local}
Let $\Omega\subset\mathbb{R}^N (N=2, 3)$ be a bounded domain with smooth boundary $\Gamma$, assume that the initial data $({\bf u}_0,\phi_0,\psi_0)\in{\mathbb D}_1$. Then there is a positive constant $T^*$, which may depends on the initial data but is independent of $\gamma$, such that there exists a unique local-in-time strong solution $({\bf u},\phi,\psi)$ to the initial boundary value problem $(\ref{NSAC})$--$(\ref{initial condition})$, satisfying
\[
\begin{aligned}
&{\bf u}\in L^\infty(0,T^*;H^2)\cap L^2(0,T^*;H^3),\qquad\qquad {\bf u}_t\in L^\infty(0,T^*;L^2)\cap L^2(0,T^*;H^1),
\\
&\phi\in L^\infty(0,T^{*};H^3),\;\Delta\phi\in L^2(0,T^*;H^2),
\qquad\phi_t\in L^\infty(0,T^*;H^1)\cap L^2(0,T^*;H^2),
\\
&\psi\in L^\infty(0,T^*;H^3(\Gamma))\cap L^2(0,T^*;H^4(\Gamma)),~\quad\psi_{tt}\in L^2(0,T^*;L^2(\Gamma)),
\\
&\psi_t\in L^\infty(0,T^*;H^1(\Gamma))\cap L^2(0,T^*;H^2(\Gamma)),\quad(\mu-\bar\mu)_t\in L^2(0,T^*;L^2),
\\
&\mu-\bar\mu\in L^\infty(0,T^*;H^1)\cap L^2(0,T^*;H^2),\qquad\  \nabla p\in L^\infty(0,T^*;L^2)\cap L^2(0,T^*;H^1).
\end{aligned}
\]
\end{Theorem}
\begin{Theorem}
\label{dynamic global}
Assume that $\Omega=\mathbb T\times(-1,1)\subset\mathbb R^2$, and the initial data $({\bf u}_0,\phi_0,\psi_0)\in{\mathbb D}_1$. Then there exist a unique global-in-time strong solution $({\bf u},\phi,\psi)$ to the initial boundary value problem $(\ref{NSAC})$--$(\ref{initial condition})$ such that, for all $T\in(0,+\infty)$,
\begin{align}
&{\bf u}\in L^\infty(0,T;H^2)\cap L^2(0,T;H^3),\hspace{4em}{\bf u}_t\in L^\infty(0,T;L^2)\cap L^2(0,T;H^1),                                 \nonumber
\\
&\phi\in L^\infty(0,T;H^3)\cap L^2(0,T;H^4),\hspace{4em}\phi_t\in L^\infty(0,T;H^1)\cap L^2(0,T;H^2),                 \nonumber
\\
&\psi\in L^\infty(0,T;H^3(\Gamma))\cap L^2(0,T;H^4(\Gamma)),\quad\psi_t\in L^\infty(0,T;H^1(\Gamma))\cap L^2(0,T;H^2(\Gamma)),              \nonumber
\\
&\mu-\bar\mu\in L^\infty(0,T;H^1)\cap L^2(0,T;H^2),\qquad(\mu-\bar\mu)_t\in L^2(0,T;L^2).
\nonumber
\end{align}
\end{Theorem}
\begin{Remark}
Under the assumptions of Theorem $\ref{relaxation local}$, the existence of global strong solutions in the 3D case can also be established, with a comparatively simpler proof.
\end{Remark}
\subsection{Local well-posedness}\quad
To obtain the local well-posedness of $(\ref{NSAC})$--$(\ref{initial condition})$, we use linearization to decouple the original system into two linearized equations.
First, we define the following spaces which will be used in this section:
\[
Y_1 = \mleft\{\,{\bf u} \,\middle|
\begin{array}{c}
{\bf u}\in L^\infty(0,T;H^2)\cap L^2(0,T;H^3),\\
{\bf u}_t\in L^\infty(0,T;L^2)\cap L^2(0,T;H^1), \\
\rm div{\bf u}=0, \quad\rm in\;\Omega\times(0,T),\\
{\bf u}\cdot{\bf n}=0, \quad\rm on\;\Gamma\times(0,T),\\
{\bf u}\big|_{t=0}={\bf u}_0,\quad\rm {in}\;\Omega.
\end{array}
\mright\}
\]
with the norm
\[
\|{\bf u}\|_{Y_1}^2 =\sup\limits_{0\le t \le T}\left(\|{\bf u}\|_{H^2}^2+\|{\bf u}_t\|_{L^2}^2\right)+\int_{0}^{T}\left(\|{\bf u}_{H^3}^2+\|{\bf u}_t\|_{H^1}^2\right) {\rm d}t,
\]
and
\[
Y_2 = \mleft\{\,(\phi,\psi) \,\middle|
\begin{array}{c}
\phi\in L^\infty(0,T;H^3),\;\phi_t\in L^\infty(0,T;H^1)\cap L^2(0,T;H^2),\\
\psi\in L^\infty(0,T;H^3(\Gamma))\cap L^2(0,T;H^4(\Gamma)), \\
\psi_t\in L^\infty(0,T;H^1(\Gamma))\cap L^2(0,T;H^2(\Gamma)),\\
\phi\big|_{\Gamma}=\psi,\quad\phi\big|_{t=0}=\phi_0,\;\rm {in}\; \Omega,\quad\psi\big|_{t=0}=\psi_0,\;\rm on\; \Gamma.
\end{array}
\mright\}
\]
with the norm
\begin{equation*}
\begin{aligned}
\|(\phi,\psi)\|_{Y_2}^2 =&\sup\limits_{0\le t \le T}\left(\|\phi\|_{H^3}^2+\|\phi_t\|_{H^1}^2+\|\psi_t\|_{H^1(\Gamma)}^2
+\|\psi\|_{H^3(\Gamma)}^2\right) \\
&+\int_{0}^{T}\left(\|\phi_t\|_{H^2}^2+\|\psi\|_{H^4(\Gamma)}^2+\|\psi_t\|_{H^2(\Gamma)}^2
\right) {\rm d}t.
\end{aligned}
\end{equation*}
Hence, for some given positive constants $K_i\:(i=1,2)$ which will be determined later, we define
\[
X(0,T;K_1,K_2) = \mleft\{\,({\bf u},\phi,\psi)\in Y_1\times Y_2 \,\middle|
\begin{array}{c}
\|{\bf u}\|_{Y_1}^2\le K_1,
 \|(\phi,\psi)\|_{Y_2}^2\le K_2.
\end{array}
\mright\}
\]
with the norm
\begin{equation*}
\begin{aligned}
\|({\bf u},\phi,\psi)\|_X^2 &=\|{\bf u}\|_{Y_1}^2+\|(\phi,\psi)\|_{Y_2}^2.
\end{aligned}
\end{equation*}
\subsubsection{\bf The linearized Allen-Cahn equations}\quad
We next discuss the initial boundary problem for the following linearized problem of the Allen-Cahn equations
\begin{align}
\label{LAC}
\begin{cases}
\phi_t={\bar{\mu}}-\mu+H,&\rm in\quad\Omega\times(0,T),
\\
\mu=-\Delta\phi+f,&\rm in\quad\Omega\times(0,T),
\\
\psi_t-\gamma\Delta_{\boldsymbol {\tau}}\psi=-\partial_{\bf n}\phi+h,&\rm on\quad\Gamma\times(0,T),
\\
\phi\big|_\Gamma=\psi,&\rm in\quad(0,T),
\\
\phi\big|_{t=0}=\phi_0,&\rm in\quad\Omega,
\\
\psi\big|_{t=0}=\psi_0,&\rm on\quad\Gamma,
\end{cases}
\end{align}
where $H=-\tilde{\bf u}\cdot\nabla\tilde\phi$, $f=\tilde{\phi}^{3}-\tilde{\phi}$, $h=-\gamma_{fs}^\prime(\tilde\psi)-\tilde{\bf u}_{\boldsymbol\tau}\cdot\nabla_{\boldsymbol\tau}\tilde\psi$ and $(\tilde{\bf u},\tilde\phi,\tilde\psi)\in X(0,T;K_1,K_2)$.
\begin{Proposition}
\label{Prop3.1}
Let $(\tilde{\bf u},\tilde\phi,\tilde\psi)\in X(0,T;K_1,K_2)$, $H\in L^\infty(0,T;H^1)\cap L^2(0,T;H^2)$, $H_t\in L^2(0,T;L^2)$, $f\in L^\infty(0,T;H^1)\cap L^2(0,T;H^2)$, $h\in L^\infty(0,T;L^2(\Gamma))\cap L^2(0,T;H^2(\Gamma))$, $ f_{t} \in { L^2(0,T; L^2) } $, and $h_t\in L^2(0,T;L^2(\Gamma))$ for some fixed $T ( 0<T<1 )$. Assume that the initial condition $({\bf u}_0,\phi_0,\psi_0)\in {\mathbb D}_1$, then there exists a unique solution $(\phi,\psi)$ of $(\ref{LAC})$ such that
\begin{align}
\label{P1}
\|(\phi,\psi)\|_{Y_2}^2&\le C\left(\|\psi_0\|_{H^3(\Gamma)}^4
+\|\phi_0\|_{H^3}^6+\|{\bf u}_0\|_{H^2}^4+1
\right)+C\left(\|H\|_{H^1}^2+\|f\|_{H^1}^2+\|h\|_{L^2(\Gamma)}^2
\right)                             \nonumber
\\
&\quad+C\int_{0}^{T}\left(
\|H\|_{H^2}^2+\|f\|_{H^2}^2+\|h\|_{H^2(\Gamma)}^2
+\|H_t\|_{L^2}^2+\|f_t\|_{L^2}^2+\|h_t\|_{L^2(\Gamma)}^2
\right) {\rm d}t,
\end{align}
where $C$ is a positive constant depending on $\Omega$ and $\gamma^{-1}$, but independent of $K_1$ and $K_2$.
\end{Proposition}
The existence and uniqueness of the solutions to the problem (\ref{LAC}) can be derived from \cite{C-C} (see Theorem 2.3 and 2.4 therein).
To complete the analysis, it remains for us to establish the a priori estimates outlined in (\ref{P1}).
The proof starts with deriving lower-order a priori bounds for the pair $(\phi,\psi)$.
\begin{Lemma}
\label{Lemma1}
Let $(\phi,\psi)$ be the smooth solution to $(\ref{LAC})$ on $[0,T]$, then it holds that
\begin{align}
\label{eq.1}
&\sup\limits_{0\le t \le T}\left(\|\phi\|_{H^1}^2+\|\psi\|_{L^2(\Gamma)}^2+\gamma\|\nabla_{\boldsymbol\tau}\psi\|_{L^2(\Gamma)}^2
\right)+\int_{0}^{T}\left(\|\psi_t\|_{L^2(\Gamma)}^2+\|\mu-\bar \mu\|_{L^2}^2
\right) {\rm d}t                                                \nonumber
\\
&\le C\left( \|\phi_0\|_{H^1}^2+\|\psi_0\|_{H^1(\Gamma)}^2
\right)
+C\int_{0}^{T}\left(
\|H\|_{L^2}^2+\|f\|_{L^2}^2+\|h\|_{L^2(\Gamma)}^2
\right) {\rm d}t,
\end{align}
where $C$ is a generic constant that only depends on $\Omega$ and $\gamma^{-1}$, but not on $K_1$ and $K_2$.
\end{Lemma}

\begin{proof}[\bf Proof.]
Multiplying $(\ref{LAC})_1$ by $\phi$, integrating the result over $\Omega$ by parts, and using $(\ref{LAC})_2$ and $(\ref{LAC})_3$, we have
\begin{align}
\label{eq9}
&\frac{1}{2}\frac{\rm d}{{\rm d}t}\left(
\|\phi\|_{L^2}^2+\|\psi\|_{L^2(\Gamma)}^2
\right)
+\|\nabla\phi\|_{L^2}^2
+\gamma\|\nabla_{\boldsymbol\tau}\psi\|_{L^2(\Gamma)}^2 \nonumber          \\
&=\bar\mu\int_{\Omega}\phi {\rm d}x+\int_{\Omega}(H-f)\phi {\rm d}x+\int_\Gamma h\psi {\rm d}S      \nonumber           \\
&\le C\|\phi\|_{L^2}^2+C\|\psi\|_{L^2(\Gamma)}^2
+\frac{1}{4}\|\psi_t\|_{L^2(\Gamma)}^2
+C\left(\|H\|_{L^2}^2+\|f\|_{L^2}^2+\|h\|_{L^2(\Gamma)}^2
\right),
\end{align}
where we have used the following fact that
\begin{align}
\label{bar mu}
\bar\mu&=-\overline{\Delta\phi}+\bar f
=\frac{1}{|\Omega|}\int_\Gamma-\partial_{\bf n}\phi {\rm d}S+\bar f            \nonumber
\\
&=\frac{1}{|\Omega|}\int_\Gamma\left(-\gamma\Delta_
{\boldsymbol\tau}\psi+\psi_t-h
\right) {\rm d}S+\bar f             \nonumber
\\
&=\frac{1}{|\Omega|}\int_\Gamma\left(\psi_t-h\right){\rm d}S+\bar f \nonumber
\\
&\le C(|\Omega|,|\Gamma|)\left(
\|\psi_t\|_{L^2(\Gamma)}+\|h\|_{L^2(\Gamma)}
+\|f\|_{L^2}
\right).
\end{align}
Multiplying $(\ref{LAC})_1$ by $\Delta\phi$, integrating the result over $\Omega$, we have
\begin{equation}
\label{n1}
\int_\Omega\phi_t\Delta\phi {\rm d}x=\int_\Omega(\bar\mu-\mu)\Delta\phi {\rm d}x+\int_\Omega H\Delta\phi {\rm d}x.
\end{equation}
The left-hand side of \eqref{n1} can be written as
\begin{align}
\int_\Omega\phi_t\Delta\phi {\rm d}x&=-\int_\Omega\nabla\phi_t\cdot\nabla\phi {\rm d}x+\int_\Gamma\partial_{\bf n}\phi\psi_t {\rm d}S            \nonumber
\\
&=-\frac{1}{2}\frac{\rm d}{{\rm d}t}\|\nabla\phi\|_{L^2}^2
+\int_\Gamma\psi_t(\gamma\Delta_{\boldsymbol\tau}
\psi-\psi_t+h){\rm d}S                                           \nonumber
\\
&=-\frac{1}{2}\frac{\rm d}{{\rm d}t}\left(\|\nabla\phi\|_{L^2}^2
+\gamma\|\nabla_{\boldsymbol\tau}\psi\|_{L^2(\Gamma)}^2 \right)-\|\psi_t\|_{L^2(\Gamma)}^2+\int_\Gamma\psi_t h {\rm d}S.                                                     \nonumber
\end{align}
Due to the incompressibility of ${\bf u}$ and boundary conditions ${\bf u} \cdot {\bf n} = 0$ on $\Gamma$,   there is
\begin{align}
\label{H}
\int_\Omega H{\rm d}x=-\int_\Omega \tilde{\bf u}\cdot\nabla\tilde\phi {\rm d}x=\int_\Omega \rm div\tilde{\bf u}\tilde\phi {\rm d}x-\int_\Gamma\tilde{\bf u}\cdot{\bf n}\tilde\phi {\rm d}x=0.
\end{align}
Note that
\begin{align}
\label{mu-bar mu}
\int_\Omega(\mu-\bar\mu){\rm d}x=\int_\Omega\mu {\rm d}x-\int_\Omega\bar\mu {\rm d}x=0.
\end{align}
Substituting $(\ref{LAC})_2$ into the right-hand side of (\ref{n1}), using \eqref{H}  and \eqref{mu-bar mu}, we see that
\begin{align}
&\int_\Omega(\bar\mu-\mu)\Delta\phi {\rm d}x+\int_\Omega H\Delta\phi {\rm d}x \nonumber
\\
&=\int_\Omega\left[(\bar\mu-\mu)+H\right]\left[(\bar\mu-\mu)-\bar\mu+f\right] {\rm d}x                     \nonumber
\\
&=\|\mu-\bar\mu\|_{L^2}^2-\bar\mu\int_\Omega(\bar\mu-\mu) {\rm d}x +\int_\Omega f(\bar\mu-\mu){\rm d}x   \nonumber
\\
&\quad+\int_\Omega H(\bar\mu-\mu){\rm d}x-\bar\mu\int_\Omega H
{\rm d}x+\int_\Omega fH{\rm d}x                       \nonumber
\\
&=\|\mu-\bar\mu\|_{L^2}^2+\int_\Omega f(\bar\mu-\mu){\rm d}x+\int_\Omega H(\bar\mu-\mu){\rm d}x+\int_\Omega fH{\rm d}x.   \nonumber
\end{align}
Then, we obtain
\begin{align}
\label{eq23}
&\frac{1}{2}\frac{\rm d}{{\rm d}t}\left(
\|\nabla\phi\|_{L^2}^2+\gamma\|\nabla_{\boldsymbol\tau}\psi\|_{L^2(\Gamma)}^2
\right)
+\|\mu-\bar\mu\|_{L^2}^2
+\|\psi_t\|_{L^2(\Gamma)}^2                     \nonumber
\\
&=\int_{\Omega}f(\mu-\bar\mu) {\rm d}x-\int_{\Omega}Hf {\rm d}x+\int_{\Omega}H(\mu-\bar\mu) {\rm d}x+\int_\Gamma h\psi_t {\rm d}S \nonumber
\\
&\le \frac{1}{2}\|\mu-\bar\mu\|_{L^2}^2+
\frac{1}{2}\|\psi_t\|_{L^2(\Gamma)}^2
+C\left(\|H\|_{L^2}^2+\|f\|_{L^2}^2+\|h\|_{L^2(\Gamma)}^2
\right).
\end{align}
Combining (\ref{eq9}) with (\ref{eq23}), it follows from Gronwall inequality that (\ref{eq.1}) holds. Then, we complete the proof of Lemma \ref{Lemma1}.
\end{proof}
The following lemma provides the highest-order a priori estimates, which are crucial for our subsequent arguments.
\begin{Lemma}\label{d-tt}
Let $(\phi,\psi)$ be the smooth solution to $(\ref{LAC})$ on $[0,T]$, then it holds that
\begin{align}
\label{eq4}
&\sup\limits_{0\le t \le T}\left( \|\phi_t\|_{H^1}^2+\|\psi_t\|_{L^2(\Gamma)}^2+\gamma\|\nabla_{\boldsymbol\tau}\psi_t\|_{L^2(\Gamma)}^2
+\|\mu-\bar\mu\|_{H^1}^2
\right)  \nonumber
\\
&\quad+\int_{0}^{T}\left(\|(\mu-\bar\mu)_t\|_{L^2}^2+
\|\psi_{tt}\|_{L^2(\Gamma)}^2+\|\Delta\phi_t\|_{L^2}^2
\right) {\rm d}t                                        \nonumber
\\
&\le C \left(
\|\psi_0\|_{H^3(\Gamma)}^4
+\|\phi_0\|_{H^3}^6+\|{\bf u}_0\|_{H^2}^4+\|H\|_{H^1}^2+1
\right)  \nonumber
\\
&\quad+C\int_{0}^{T}\left(
\|H_t\|_{L^2}^2+\|f_t\|_{L^2}^2+\|h_t\|_{L^2(\Gamma)}^2
\right) {\rm d}t                                          \nonumber
\\
&=C_1 ,
\end{align}
where $C$ is a generic constant that only depends on $\Omega$ and $\gamma^{-1}$, but not on $K_1$ and $K_2$.
\end{Lemma}
\begin{proof}[\bf Proof.]
First, differentiating $(\ref{LAC})_{1,2,3}$ with respect to $t$, we get
\begin{align}
\label{partial t}
\begin{cases}
\phi_{tt}=({\bar{\mu}}-\mu)_t+H_t,&\rm in\quad\Omega\times(0,T),
\\
\mu_t=-\Delta\phi_t+f_t,&\rm in\quad\Omega\times(0,T),
\\
\psi_{tt}-\gamma\Delta_{\boldsymbol {\tau}}\psi_t=-\partial_{\bf n}\phi_t+h_t,&\rm on\quad\Gamma\times(0,T).
\end{cases}
\end{align}
It follows from $(\ref{LAC})_1$, (\ref{mu-bar mu}) and integration by part that
\begin{align}
\label{phit}
\int_\Omega\phi_t {\rm d}x&=\int_\Omega\left[(\bar\mu-\mu)-\tilde{\bf u}\cdot\nabla\tilde\phi\right] {\rm d}x                        \nonumber
\\
&=\int_\Omega(\bar\mu-\mu) {\rm d}x+\int_\Omega{\rm div}\tilde{\bf u}\cdot\tilde\phi {\rm d}x-\int_\Gamma\tilde{\bf u}\cdot{\bf n}\tilde\phi {\rm d}S=0.
\end{align}
Multiplying $(\ref{partial t})_1$ by $\phi_t$, integrating the result over $\Omega$ by parts, and then using $(\ref{partial t})_2$ and $(\ref{phit})$, one has
\begin{align}
\label{eq26}
&\frac{1}{2}\frac{\rm d}{{\rm d}t}\left(
\|\phi_t\|_{L^2}^2+\gamma\|\psi_t\|_{L^2(\Gamma)}^2
\right)+\|\nabla\phi_t\|_{L^2}^2+\gamma\|\nabla_{\boldsymbol\tau}\psi_t\|
_{L^2(\Gamma)}^2                                \nonumber
\\
&=\int_\Gamma\psi_t h_t {\rm d}S-\int_\Omega f_t\phi_t {\rm d}x+\int_\Omega H_t\phi_t {\rm d}x                             \nonumber
\\
&\le C\|\psi_t\|_{L^2(\Gamma)}^2+C\|\phi_t\|_{L^2}^2
+C\left(\|H_t\|_{L^2}^2+\|f_t\|_{L^2}^2+\|h_t\|_{L^2(\Gamma)}^2
\right).
\end{align}
\vskip2mm
Next, thanks to (\ref{mu-bar mu}), we have
\begin{equation}
\label{eq34}
-\int_\Omega\mu_t(\bar\mu-\mu)_t{\rm d}x=\|(\mu-\bar\mu)_t\|_{L^2}^2-\bar\mu_t\int_\Omega(\bar\mu-\mu)_t{\rm d}x=\|(\mu-\bar\mu)_t\|_{L^2}^2.
\end{equation}
Multiplying $(\ref{partial t})_1$, $(\ref{partial t})_2$ and $(\ref{partial t})_3$ by $-\mu_t$, $\phi_{tt}$ and $\psi_{tt}$ respectively, integrating the results over $\Omega$ by parts and using (\ref{eq34}), we get
\begin{align}
\label{eq3}
&\frac{1}{2}\frac{\rm d}{{\rm d}t}\left(
\|\nabla\phi_t\|_{L^2}^2+\gamma\|\nabla_{\boldsymbol\tau}\psi_t\|_{L^2(\Gamma)}^2
\right)
+\|(\mu-\bar\mu)_t\|_{L^2}^2+\|\psi_{tt}\|_{L^2(\Gamma)}^2
\nonumber      \\
&=\int_{\Omega}H_t\mu_t {\rm d}x-\int_{\Omega}f_t\phi_{tt} {\rm d}x+\int_\Gamma \psi_{tt}h_t {\rm d}S                  \nonumber
\\
&=\int_{\Omega}\left[H_t(\mu-\bar\mu)_t+\bar\mu_t H_t\right] {\rm d}x-\int_{\Omega}f_t\left[(\bar\mu-\mu)_t+H_t\right] {\rm d}x
+\int_\Gamma \psi_{tt}h_t {\rm d}S                    \nonumber
\\
&=\int_{\Omega}H_t(\mu-\bar\mu)_t {\rm d}x +0 -\int_{\Omega}f_t(\bar\mu-\mu)_t{\rm d}x-\int_{\Omega}f_tH_t {\rm d}x
+\int_\Gamma \psi_{tt}h_t {\rm d}S                    \nonumber
\\
&\le \frac{1}{2}\|(\mu-\bar\mu)_t\|_{L^2}^2
+\frac{1}{2}\|\psi_{tt}\|_{L^2(\Gamma)}^2
+C\left(\|H_t\|_{L^2}^2+\|f_t\|_{L^2}^2+\|h_t\|_{L^2(\Gamma)}^2
\right),
\end{align}
where we have used (\ref{H}).
Thus, summing (\ref{eq3}) together with (\ref{eq26}), we obtain
\begin{align}
\label{eq14}
&\sup\limits_{0\le t \le T}\left( \|\phi_t\|_{H^1}^2+\|\psi_t\|_{L^2(\Gamma)}^2
+\gamma\|\nabla_{\boldsymbol\tau}\psi_t\|_{L^2(\Gamma)}^2
\right)+\int_{0}^{T}\left(
\|(\mu-\bar\mu)_t\|_{L^2}^2+\|\psi_{tt}\|_{L^2(\Gamma)}^2
\right) {\rm d}t                                              \nonumber
\\
&\le C \left(
\|\psi_0\|_{H^3(\Gamma)}^4+\|\phi_0\|_{H^3}^6+\|{\bf u}_0\|_{H^2}^4+1
\right)+C\int_{0}^{T}\left(
\|H_t\|_{L^2}^2+\|f_t\|_{L^2}^2+\|h_t\|_{L^2(\Gamma)}^2
\right) {\rm d}t.
\end{align}

Furthermore, it follows from $(\ref{LAC})_1$ and (\ref{eq14}) that
\begin{align}
\label{eq27}
&\|\mu-\bar\mu\|_{H^1}^2 \le C\left(\|\phi_t\|_{H^1}^2+\|H\|_{H^1}^2\right)
\le C_1.
\end{align}
Moreover, we deduce from $(\ref{LAC})_2$ that
\begin{align}
\bar\mu_t&=-\overline{\Delta\phi_t}+\bar f_t=\frac{1}{|\Omega|}\int_{\Gamma}-\partial_{\bf n}\phi_t {\rm d}S+\bar f_t            \nonumber
\\
&=\frac{1}{|\Omega|}\int_\Gamma\left(
\psi_{tt}-\gamma\Delta_{\boldsymbol\tau}\psi_t-h_t
\right) {\rm d}S+\bar f_t          \nonumber
\\
&=\frac{1}{|\Omega|}\int_\Gamma\left(
\psi_{tt}-h_t\right) {\rm d}S+\bar f_t          \nonumber
\\
&\le C(|\Omega|,|\Gamma|)
\left(
\|\psi_{tt}\|_{L^2(\Gamma)}+\|h_t\|_{L^2(\Gamma)}
+\|f_t\|_{L^2}
\right),                                                \nonumber
\end{align}
which together with (\ref{eq14}) yields
\begin{equation}
\label{eq8}
\int_{0}^{T}\|\Delta\phi_t\|_{L^2}^2 {\rm d}t
\le \int_{0}^{T}\left( \|(\mu-\bar\mu)_t\|_{L^2}^2
+\|\bar\mu_t\|_{L^2}^2+\|f_t\|_{L^2}^2
\right) {\rm d}t
\le C_1.
\end{equation}
Combining (\ref{eq8}) with (\ref{eq14}) and (\ref{eq27}) leads to (\ref{eq4}). Hence, the proof of Lemma \ref{d-tt} is completed.
\end{proof}

Now we proceed to derive spatial-derivative estimates for $\phi$ and $\psi$.
\begin{Lemma} \label{L-222}
Let $(\phi,\psi)$ be the smooth solution to $(\ref{LAC})$ on $[0,T]$, then it holds that
\begin{align}
\label{eq lamma3}
&\sup\limits_{0\le t \le T}
\left(
\|\nabla^2\phi\|_{L^2}^2
+\gamma\|\nabla_{\boldsymbol\tau}^2\psi\|_{L^2(\Gamma)}^2
\right)
+\int_{0}^{T}\left(
\|\nabla\phi_t\|_{L^2}^2
+\|\nabla_{\boldsymbol\tau}\psi_t\|_{L^2(\Gamma)}^2
\right) {\rm d}t
\nonumber \\
&\le CC_1+C\left(\|f\|_{L^2}^2+\|h\|_{L^2(\Gamma)}^2
\right)+C\int_{0}^{T}\left(
\|\nabla H\|_{L^2}^2+\|\nabla f\|_{L^2}^2
+\|\nabla_{\boldsymbol\tau} h\|_{L^2(\Gamma)}^2
\right) {\rm d}t
\nonumber \\
& = C_{2},
\end{align}
where $C$ is a generic constant that only depends on $\Omega$ and  $\gamma^{-1}$, but not on $K_1$ and $K_2$.
\end{Lemma}

\noindent
{\bf Proof.} \quad The proof consists of two parts.
\\
{\it\bfseries Step 1. Interior estimates.} Assume that $\zeta=\zeta(x)$ is smooth and $\zeta=0$ on $\Gamma$. Adding $(\ref{LAC})_2$ into $(\ref{LAC})_1$, we obtain
\begin{equation}
\label{1+2}
\phi_t=\bar\mu+\Delta\phi-f+H.
\end{equation}
Differentiating it with respect to $x$ leads to
\begin{equation}
\label{eq13}
\nabla\phi_t=\nabla\Delta\phi-\nabla f+\nabla H.
\end{equation}
Then, testing (\ref{eq13}) by $\zeta^2\nabla\phi_t$, we get
\begin{align*}
&\frac{1}{2}\frac{\rm d}{{\rm d}t}\|\zeta\nabla^2\phi\|_{L^2}^2
+\|\zeta\nabla\phi_t\|_{L^2}^2
\\
&=-\int_\Omega 2\zeta\nabla\phi_t \cdot \nabla^2\phi  \cdot \nabla\zeta {\rm d}x+\int_\Omega\zeta^2\nabla\phi_t\cdot(\nabla H-\nabla f){\rm d}x
\\
&\le C\|\zeta\nabla^2\phi\|_{L^2}^2
+C\left( \|\nabla\phi_t\|_{L^2}^2+\|\nabla H\|_{L^2}^2+\|\nabla f\|_{L^2}^2 \right).
\end{align*}
Then, together with Gronwall inequality and (\ref{eq4}), we obtain
\begin{align}
\label{eq11}
\sup\limits_{0\le t \le T} \|\zeta\nabla^2\phi\|_{L^2}^2+\int_{0}^{T}\|\zeta\nabla\phi_t\|_{L^2}^2{\rm d}t
\le CC_1+C\int_{0}^{T}\left(
\|\nabla H\|_{L^2}^2+\|\nabla f\|_{L^2}^2
\right){\rm d}t.
\end{align}
{\it\bfseries Step 2. Near the boundary estimates.} We begin with estimating tangential derivative. Assume a point $p_{0} \in\Gamma$ and $\tilde\zeta=\tilde\zeta(x)$ is smooth function vanishing outside of a small neighborhood of $p_{0}$. Taking the tangential derivative of (\ref{1+2}) and $(\ref{LAC})_3$ with respect to $x$ leads to
\begin{align}
\begin{cases}
\label{eq5}
\nabla_{\boldsymbol\tau}\phi_t=\nabla_{\boldsymbol\tau}\Delta\phi
-\nabla_{\boldsymbol\tau} f+\nabla_{\boldsymbol\tau} H,&\rm in\quad\Omega\times(0,T),
\\
\nabla_{\boldsymbol\tau}\psi_t-\gamma\nabla_{\boldsymbol\tau}\Delta_{\boldsymbol \tau}\psi=-\nabla_{\boldsymbol\tau}\partial_{\bf n}\phi+\nabla_{\boldsymbol\tau}h,&\rm on\quad\Gamma\times(0,T).
\end{cases}
\end{align}
And then, testing $(\ref{eq5})_1$ by $\tilde\zeta^2\nabla_{\boldsymbol\tau}\phi_t$, noting that $\tilde\zeta(x)=1$ on the boundary, then integration by parts and $(\ref{eq5})_2$ show that
\begin{align}
&\frac{1}{2}\frac{\rm d}{{\rm d}t}\left(
\|\tilde\zeta\nabla_{\boldsymbol\tau}\nabla\phi\|_{L^2}^2
+\gamma\|\nabla_{\boldsymbol\tau}^2\psi\|_{L^2(\Gamma)}^2
\right)
+\|\tilde\zeta\nabla_{\boldsymbol\tau}\phi_t\|_{L^2}^2 +\|\nabla_{\boldsymbol\tau}\psi_t\|_{L^2(\Gamma)}^2
\nonumber
\\
&=-\int_\Omega2\tilde\zeta\nabla_{\boldsymbol\tau}\phi_t \cdot \nabla_{\boldsymbol\tau}\nabla\phi \cdot \nabla\tilde\zeta {\rm d}x +\int_\Omega\tilde\zeta^2\nabla_{\boldsymbol\tau}\phi_t\cdot(\nabla_{\boldsymbol\tau} H-\nabla_{\boldsymbol\tau} f){\rm d}x
 +\int_\Gamma\nabla_{\boldsymbol\tau}\psi_t
\cdot\nabla_{\boldsymbol\tau}h {\rm d}S                       \nonumber
\\
&\le \frac{1}{2}
\|\nabla_{\boldsymbol\tau}\psi_t\|_{L^2(\Gamma)}^2
+C\|\tilde\zeta\nabla_{\boldsymbol\tau}\nabla\phi\|_{L^2}^2+C\left(
\|\nabla\phi_t\|_{L^2}^2+\|\nabla_{\boldsymbol\tau} H\|_{L^2}^2+\|\nabla_{\boldsymbol\tau} f\|_{L^2}^2+\|\nabla_{\boldsymbol\tau} h\|_{L^2(\Gamma)}^2
\right),                                               \nonumber
\end{align}
which together with Gronwall inequality and (\ref{eq4}) implies
\begin{align}
\label{eq7}
&\sup\limits_{0\le t \le T}
\left(
\|\tilde\zeta\nabla_{\boldsymbol\tau}\nabla\phi\|_{L^2}^2
+\gamma\|\nabla_{\boldsymbol\tau}^2\psi\|_{L^2(\Gamma)}^2
\right)
+\int_{0}^{T}\left(
\|\tilde\zeta\nabla_{\boldsymbol\tau}\phi_t\|_{L^2}^2
+\|\nabla_{\boldsymbol\tau}\psi_t\|_{L^2(\Gamma)}^2
\right) {\rm d}t                                           \nonumber
\\
&\le CC_1+C\int_{0}^{T}\left(
\|\nabla H\|_{L^2}^2+\|\nabla f\|_{L^2}^2
+\|\nabla_{\boldsymbol\tau} h\|_{L^2(\Gamma)}^2
\right) {\rm d}t.
\end{align}
\vskip2mm
Finally, the differentiation in the normal direction can be estimated by using the differential equation $(\ref{LAC})_2$, there is
\begin{align}
\label{eq10}
\|\tilde\zeta\partial_{\bf n}\partial_{\bf n}\phi\|_{L^2}^2
&=\|\tilde\zeta(-\Delta_{\boldsymbol\tau}\phi-\mu+f)\|_{L^2}^2
\nonumber
\\
&\le C\left(
\|\tilde\zeta\nabla_{\boldsymbol\tau}\nabla\phi\|_{L^2}^2
+\|\mu-\bar\mu\|_{L^2}^2+\|\bar\mu\|_{L^2}^2+\|f\|_{L^2}^2
\right)                                              \nonumber
\\
&\le CC_1+C\left(\|f\|_{L^2}^2+\|h\|_{L^2(\Gamma)}^2
\right)+C\int_{0}^{T}\left(
\|\nabla H\|_{L^2}^2+\|\nabla f\|_{L^2}^2
+\|\nabla_{\boldsymbol\tau} h\|_{L^2(\Gamma)}^2
\right) {\rm d}t,
\end{align}
where we have used (\ref{bar mu}), (\ref{eq4}) and (\ref{eq7}). Consequently, (\ref{eq10}) together with (\ref{eq11}) and (\ref{eq7}) shows (\ref{eq lamma3}). This completes the proof of Lemma \ref{L-222}.
\hfill$\Box$
\begin{Lemma}
\label{Lemma4}
Let $(\phi,\psi)$ be the smooth solution to $(\ref{LAC})$ on $[0,T]$, then it holds that
\begin{align}
\label{eq32}
&\sup\limits_{0\le t \le T}
\left(
\|\nabla^3\phi\|_{L^2}^2
+\gamma\|\nabla_{\boldsymbol\tau}^3\psi\|_{L^2(\Gamma)}^2
\right)                                           \nonumber
\\
&+\int_{0}^{T}\left(
\|\nabla^2\phi_t\|_{L^2}^2+\|\nabla^2\Delta\phi\|_{L^2}^2
+\|\nabla^2_{\boldsymbol\tau}\psi_t\|_{L^2(\Gamma)}^2
+\gamma^2\|\nabla^4_{\boldsymbol\tau}\psi\|
_{L^2(\Gamma)}^2+\|\mu-\bar\mu\|_{H^2}^2
\right) {\rm d}t                                        \nonumber
\\
&\le  C( \|\phi_0\|_{H^3}^2 + C_{2} )+C\int_{0}^{T} \left(\| \nabla^{2}H\|_{L^2}^2+\| \nabla^{2}f\|_{L^2}^2 +\|\nabla^2_{\boldsymbol\tau} h\|_{L^2(\Gamma)}^2
\right) {\rm d}t,
\end{align}
where $C$ is a generic constant that only depends on $\Omega$ and $\gamma^{-1}$, but not on $K_1$ and $K_2$.
\end{Lemma}
\begin{proof}[\bf Proof.] The proof consists of two parts.
\\
{\it\bfseries Step 1. Interior estimates.} Assume that $\zeta=\zeta(x)$ is smooth and $\zeta=0$ on $\Gamma$. Multiplying (\ref{eq13}) by $\zeta^2\Delta\nabla\phi_t$ and integrating the result over $\Omega$ by parts, one has
\begin{align}
\label{eq15}
& \frac{1}{2}\frac{\rm d}{{\rm d}t}\|\zeta\nabla\Delta\phi\|_{L^2}^2
+\|\zeta\nabla^2\phi_t\|_{L^2}^2
\nonumber \\
&=-\int_\Omega2\zeta\nabla\phi_t\cdot\nabla^2\phi_t\cdot\nabla\zeta {\rm d}x+\int_\Omega\zeta^2\nabla^2\phi_t:(\nabla^2 H-\nabla^2 f) {\rm d}x       \nonumber
\\
&\quad+\int_\Omega2\zeta(\nabla H-\nabla f)\cdot\nabla^2\phi_t\cdot\nabla\zeta {\rm d}x                                     \nonumber \\
&\le \frac{1}{2}\|\zeta\nabla^2\phi_t\|_{L^2}^2+C\left(\|\nabla\phi_t\|_{L^2}^2
+\|H\|_{H^2}^2+\|f\|_{H^2}^2
\right).
\end{align}
Integrating (\ref{eq15}) over $(0,T)$, which together with (\ref{eq4}) yields
\begin{equation}
\label{eq41}
\sup\limits_{0\le t \le T}\|\zeta\nabla\Delta\phi\|_{L^2}^2
+\int_{0}^{T}\|\zeta\nabla^2\phi_t\|_{L^2}^2 {\rm d}t
\le \|\phi_0\|_{H^3}^2 + C_{2} +C\int_{0}^{T} \left(\| \nabla^{2}H\|_{L^2}^2+\| \nabla^{2}f\|_{L^2}^2
\right) {\rm d}t .
\end{equation}
Then, direct calculation indicates
\begin{align}
\label{eq40}
\|\zeta\nabla^3\phi\|_{L^2}^2
&=-\int_\Omega\zeta^2\nabla^2\Delta\phi:\nabla^2\phi {\rm d}x-\int_\Omega2\zeta
\nabla^2\phi : \nabla^3\phi \cdot \nabla\zeta {\rm d}x
\nonumber
\\
&=\int_\Omega\zeta^2\nabla\Delta\phi\cdot\nabla\Delta\phi {\rm d}x +\int_\Omega 2\zeta
\nabla\Delta\phi \cdot \nabla^2\phi \cdot \nabla\zeta {\rm d}x-\int_\Omega2\zeta
\nabla^2\phi : \nabla^3\phi \cdot \nabla\zeta {\rm d}x                      \nonumber
\\
&=\|\zeta\nabla\Delta\phi\|_{L^2}^2 +\int_\Omega 2\zeta
\nabla\Delta\phi \cdot \nabla^2\phi \cdot \nabla\zeta {\rm d}x-\int_\Omega2\zeta
\nabla^2\phi : \nabla^3\phi \cdot \nabla\zeta {\rm d}x                    \nonumber
\\
&\le \frac{1}{2}\|\zeta\nabla^3\phi\|_{L^2}^2
+C\|\zeta\nabla\Delta\phi\|_{L^2}^2+C\|\nabla^2\phi\|_{L^2}^2
\nonumber
\\
&\le \frac{1}{2}\|\zeta\nabla^3\phi\|_{L^2}^2
 + C( \|\phi_0\|_{H^3}^2 + C_{2} )+C\int_{0}^{T} \left(\| \nabla^{2}H\|_{L^2}^2+\| \nabla^{2}f\|_{L^2}^2
\right) {\rm d}t  ,
\end{align}
where we have used (\ref{eq lamma3}) and (\ref{eq41}).
It follows from (\ref{eq41}) and (\ref{eq40}) that
\begin{equation}
\label{eq21}
\sup\limits_{0\le t \le T}\|\zeta\nabla^3\phi\|_{L^2}^2
+\int_{0}^{T}\|\zeta\nabla^2\phi_t\|_{L^2}^2 {\rm d}t\le C( \|\phi_0\|_{H^3}^2 + C_{2} )+C\int_{0}^{T} \left(\| \nabla^{2}H\|_{L^2}^2+\| \nabla^{2}f\|_{L^2}^2
\right) {\rm d}t .
\end{equation}
\vskip2mm
Moreover, taking the derivative of (\ref{eq13}) with respect to $x$ leads to
\begin{equation*}
\nabla^2\phi_t=\nabla^2\Delta\phi-\nabla^2 f+\nabla^2 H,
\end{equation*}
for which we have
\begin{align}
\label{eq30}
\int_{0}^{T}\|\zeta\nabla^2\Delta\phi\|_{L^2}^2 {\rm d}t & \le C \int_{0}^{T}\left\Vert\zeta\left(\nabla^2\phi_t+\nabla^2 f-\nabla^2 H\right)\right\Vert_{L^2}^2 {\rm d}t
\nonumber \\
& \le C( \|\phi_0\|_{H^3}^2 + C_{2} )+C\int_{0}^{T} \left(\| \nabla^{2}H\|_{L^2}^2+\| \nabla^{2}f\|_{L^2}^2
\right) {\rm d}t .
\end{align}
{\it\bfseries Step 2. Near the boundary estimates.} We also start by estimating tangential derivative. Assume a point $ p_{0} \in\Gamma$ and $\tilde\zeta=\tilde\zeta(x)$ is smooth function vanishing outside of a small neighborhood of $p_{0}$. Taking the tangential derivative of (\ref{1+2}) and $(\ref{LAC})_3$ with respect to $x$ twice leads to
\begin{align}
\begin{cases}
\label{eq16}
\nabla^2_{\boldsymbol\tau}\phi_t=\nabla^2_{\boldsymbol\tau}\Delta\phi
-\nabla^2_{\boldsymbol\tau} f+\nabla^2_{\boldsymbol\tau} H,&\rm in\quad\Omega\times(0,T),
\\
\nabla^2_{\boldsymbol\tau}\psi_t-\gamma\nabla^2_{\boldsymbol\tau}\Delta_{\boldsymbol \tau}\psi=-\nabla^2_{\boldsymbol\tau}\partial_{\bf n}\phi+\nabla^2_{\boldsymbol\tau}h,&\rm on\quad\Gamma\times(0,T).
\end{cases}
\end{align}
Thereafter, multiplying $(\ref{eq16})_1$ by $\tilde\zeta^2\nabla^2_{\boldsymbol\tau}\phi_t$, integrating over $\Omega$ by parts, one has
\begin{align}
\label{eq20}
&\frac{1}{2}\frac{\rm d}{{\rm d}t} \left(
\|\tilde\zeta\nabla^2_{\boldsymbol\tau}\nabla\phi\|_{L^2}^2
+\gamma\|\nabla^3_{\boldsymbol\tau}\psi\|_{L^2(\Gamma)}^2
\right)
+\|\tilde\zeta\nabla^2_{\boldsymbol\tau}\phi_t\|_{L^2}^2
+\|\nabla^2_{\boldsymbol\tau}\psi_t\|_{L^2(\Gamma)}^2
\nonumber
\\
&=-\int_\Omega2\tilde\zeta \nabla^2_{\boldsymbol\tau}\phi_t :
\nabla^2_{\boldsymbol\tau} \nabla\phi \cdot \nabla\tilde{\zeta} {\rm d}x  +\int_\Omega\tilde\zeta^2\nabla^2_{\boldsymbol\tau}\phi_t
:(\nabla^2_{\boldsymbol\tau} H-\nabla^2_{\boldsymbol\tau} f) {\rm d}x
+\int_\Gamma\nabla^2_{\boldsymbol\tau}\psi_t
:\nabla^2_{\boldsymbol\tau} h {\rm d}S                      \nonumber
\\
&\le
\frac{1}{2}\|\nabla^2_{\boldsymbol\tau}\psi_t\|_{L^2(\Gamma)}^2
+C\|\tilde\zeta\nabla^2_{\boldsymbol\tau}\nabla\phi\|_{L^2}^2
+\varepsilon \|\nabla^2\phi_t\|_{L^2}^2+C\left(
\|\nabla^2 H\|_{L^2}^2+\|\nabla^2 f\|_{L^2}^2+\|\nabla^2_{\boldsymbol\tau} h\|_{L^2(\Gamma)}^2
\right).
\end{align}
To arrive at (\ref{eq20}), it is sufficient to estimate $\|\nabla^2\phi_t\|_{L^2}^2$.
Directly calculated by using integration by parts formula, we get
\begin{align}
\label{eq51}
\|\zeta\nabla^2\phi_t\|_{L^2}^2&=\|\zeta\Delta\phi_t\|_{L^2}^2
+\int_\Omega2\zeta\nabla\zeta\cdot\nabla\phi_t\Delta\phi_t {\rm d}x-\int_\Omega 2\zeta\nabla\phi_t \cdot \nabla^2\phi_t \cdot \nabla\zeta {\rm d}x
\nonumber
\\
&\le\frac{1}{2}\|\zeta\nabla^2\phi_t\|_{L^2}^2
+C\|\zeta\Delta\phi_t\|_{L^2}^2+C\|\nabla\phi_t\|_{L^2}^2.
\end{align}
For the estimates near the boundary, we infer from $\Delta\phi_t=\Delta_{\boldsymbol\tau}\phi_t+\partial_{\bf n}\partial_{\bf n}\phi_t$ that
\begin{equation*}
\|\tilde\zeta\partial_{\bf n}\partial_{\bf n}\phi_t\|_{L^2}^2\le C\left(
\|\tilde\zeta\Delta_{\boldsymbol\tau}\phi_t\|_{L^2}^2
+\|\tilde\zeta\Delta\phi_t\|_{L^2}^2
\right)
\le C\left(\|\tilde\zeta\nabla_{\boldsymbol\tau}^2\phi_t\|_{L^2}^2
+\|\tilde\zeta\Delta\phi_t\|_{L^2}^2
\right),
\end{equation*}
which together with (\ref{eq51}) and finite covering theorem yields
\begin{align}
\label{eq19}
\|\nabla^2\phi_t\|_{L^2}^2 &\le\tilde C\left(
\|\zeta\nabla^2\phi_t\|_{L^2}^2
+\|\tilde\zeta\partial_{\bf n}\partial_{\bf n}\phi_t\|_{L^2}^2
+\|\tilde\zeta\nabla_{\boldsymbol\tau}^2\phi_t\|_{L^2}^2
\right)                                                 \nonumber
\\
&\le\tilde C\left(
\|\Delta\phi_t\|_{L^2}^2+\|\nabla\phi_t\|_{L^2}^2
+\|\tilde\zeta\nabla_{\boldsymbol\tau}^2\phi_t\|_{L^2}^2
\right).
\end{align}
Putting \eqref{eq19} into \eqref{eq20}, and then choosing $\varepsilon$ enough small such that $\tilde{C}\varepsilon < 1/2$, we infer  from Gronwall inequality and (\ref{eq4}) that
\begin{align}
\label{eq22}
&\sup\limits_{0\le t \le T}\left(
\|\tilde\zeta\nabla^2_{\boldsymbol\tau}\nabla\phi\|_{L^2}^2
+\gamma\|\nabla^3_{\boldsymbol\tau}\psi\|_{L^2(\Gamma)}^2
\right)+\int_{0}^{T}\left(\|\tilde\zeta\nabla^2_{\boldsymbol\tau}\phi_t\|_{L^2}^2
+\|\nabla^2_{\boldsymbol\tau}\psi_t\|_{L^2(\Gamma)}^2
\right) {\rm d}t
\nonumber \\
& \le  C( \|\phi_0\|_{H^3}^2 + C_{2} )+C\int_{0}^{T} \left(\| \nabla^{2}H\|_{L^2}^2+\| \nabla^{2}f\|_{L^2}^2 +\|\nabla^2_{\boldsymbol\tau} h\|_{L^2(\Gamma)}^2
\right) {\rm d}t .
\end{align}
Further, collecting all of the estimates $(\ref{eq4})$, $(\ref{eq19})$ and $(\ref{eq22})$, one gets that
\begin{align}
\label{eq39}
\|\nabla^2\phi_t\|_{L^2}^2 & \le C\left(
\|\Delta\phi_t\|_{L^2}^2+\|\nabla\phi_t\|_{L^2}^2
+\|\tilde\zeta\nabla_{\boldsymbol\tau}^2\phi_t\|_{L^2}^2
\right)
\nonumber \\
& \le C( \|\phi_0\|_{H^3}^2 + C_{2} )+C\int_{0}^{T} \left(\| \nabla^{2}H\|_{L^2}^2+\| \nabla^{2}f\|_{L^2}^2 +\|\nabla^2_{\boldsymbol\tau} h\|_{L^2(\Gamma)}^2
\right) {\rm d}t,
\end{align}
which implies
\begin{align} \label{mu-222}
\int_{0}^{T}\|\nabla^2(\mu-\bar\mu)\|_{L^2}^2{\rm d}t & \le C\int_{0}^{T}\left(
\|\nabla^2\phi_t\|_{L^2}^2+\|\nabla^2H\|_{L^2}^2
\right) {\rm d}t
\nonumber \\
& \le C( \|\phi_0\|_{H^3}^2 + C_{2} )+C\int_{0}^{T} \left(\| \nabla^{2}H\|_{L^2}^2+\| \nabla^{2}f\|_{L^2}^2 +\|\nabla^2_{\boldsymbol\tau} h\|_{L^2(\Gamma)}^2
\right) {\rm d}t.
\end{align}
\vskip2mm
Finally, differentiating $(\ref{LAC})_2$ with respect to $x$ leads to
\[
\nabla\mu=-\nabla\Delta\phi+\nabla f.
\]
Let us rewrite it as
\[
\nabla\partial_{\bf n}\partial_{\bf n}\phi=-\nabla\Delta_{\boldsymbol\tau}\phi-\nabla\mu+\nabla f.
\]
which yields
\begin{align}
\|\tilde\zeta\nabla\partial_{\bf n}\partial_{\bf n}\phi\|_{L^2}^2
&=\|\tilde\zeta\left(
-\nabla\Delta_{\boldsymbol\tau}\phi-\nabla\mu+\nabla f
\right)\|_{L^2}^2                                     \nonumber
\\
&\le C\left(
\|\tilde\zeta\nabla^2_{\boldsymbol\tau}\nabla\phi\|_{L^2}^2
+\|\nabla(\mu-\bar\mu)\|_{L^2}^2+\|\nabla f\|_{L^2}^2
\right)
\le C_2,                                          \nonumber
\end{align}
where we have used (\ref{eq4}) and (\ref{eq22}). Besides, we deduce from (\ref{LAC}) that
\begin{align*}
\begin{cases}
-\Delta\phi = \mu - f,&\rm in\quad\Omega\times(0,T),
\\
-\gamma\Delta_{\boldsymbol {\tau}}\psi + \partial_{\bf n}\phi + \psi= - \psi_t +h + \psi,&\rm on\quad\Gamma\times(0,T),
\\
\phi\big|_\Gamma=\psi,&\rm in \quad(0,T),
\end{cases}
\end{align*}
which together with Lemma \ref{estimates}, $(\ref{bar mu})$, $(\ref{eq4})$, $(\ref{eq lamma3})$, $(\ref{eq22})$, \eqref{mu-222},  implies
\begin{align}\label{phi-psi-44}
& \int_{0}^{T} \left( \| \nabla^{4} \phi\|_{L^{2}}^{2} + \gamma^{2} \| \nabla_{\boldsymbol {\tau}}^{4} \psi\|_{L^{2}(\Gamma)}^{2}  \right) {\rm d}t
\nonumber \\
& \le   C  \int_{0}^{T} \left( \| \mu \|_{H^{2}}^{2} +  \| f \|_{H^{2}}^{2} +  \| \psi_{t} \|_{H^{2}(\Gamma)}^{2} + \| h \|_{H^{2}(\Gamma)}^{2} + \| \psi \|_{H^{2}(\Gamma)}^{2} \right) {\rm d}t
\nonumber \\
& \le C  \int_{0}^{T} \left( \| \mu - \bar{\mu} \|_{H^{2}}^{2} + \| \bar{\mu} \|_{L^{2}}^{2} +  \| f \|_{H^{2}}^{2} +  \| \psi_{t} \|_{H^{2}(\Gamma)}^{2} + \| h \|_{H^{2}(\Gamma)}^{2} + \| \psi \|_{H^{2}(\Gamma)}^{2} \right) {\rm d}t
\nonumber \\
& \le   C( \|\phi_0\|_{H^3}^2 + C_{2} )+C\int_{0}^{T} \left(\| \nabla^{2}H\|_{L^2}^2+\| \nabla^{2}f\|_{L^2}^2 +\|\nabla^2_{\boldsymbol\tau} h\|_{L^2(\Gamma)}^2
\right) {\rm d}t.
\end{align}
Combining with (\ref{eq21}), (\ref{eq30}) and $(\ref{eq22})$--\eqref{phi-psi-44} yields (\ref{eq32}). This completes the proof of Lemma \ref{Lemma4}.
\end{proof}
\subsubsection{\bf The linearized Navier-Stokes equations}\quad
Next, we consider the initial boundary problem for the linearized Navier-Stokes equations
\begin{align}
\label{LNS}
\begin{cases}
{\bf u}_t-{\rm div}{\mathbb S}({\bf u})+\nabla p=G,&{\rm in}\quad\Omega\times(0,T),
\\
{\rm div}{\bf u}=0,&{\rm in}\quad\Omega\times(0,T),
\\
{\bf u}\cdot{\bf n}=0,&{\rm on}\quad\Gamma\times(0,T),
\\
\beta{\bf u}_{\boldsymbol\tau}+({\mathbb S}({\bf u})\cdot{\bf n})_{\boldsymbol{\tau}}=g,&{\rm on}\quad\Gamma\times(0,T),
\\
{\bf u}\big|_{t=0}={\bf u}_0,&{\rm in}\quad\Omega,
\end{cases}
\end{align}
where $G=-\tilde{\bf u}\cdot\nabla\tilde{\bf u}-\rm div(\nabla\tilde\phi\otimes\nabla\tilde\phi)$, $g=\left(-\gamma\Delta_{\boldsymbol\tau}\tilde\psi+\partial_{\bf n}\tilde\phi+\gamma_{fs}^{\prime}(\tilde\psi)\right)
\nabla_{\boldsymbol\tau}\tilde\psi$ and $(\tilde{\bf u},\tilde\phi,\tilde\psi)\in X(0,T;K_1,K_2)$.
\begin{Proposition}
\label{Prop3.2}
Let $(\tilde{\bf u},\tilde\phi,\tilde\psi)\in X(0,T;K_1,K_2)$, $G\in L^\infty(0,T;L^2)\cap L^2(0,T;H^1)$, $G_t\in L^2(0,T;L^2)$, $g\in L^\infty(0,T;H^\frac{1}{2}(\Gamma))\cap L^2(0,T;H^\frac{3}{2}(\Gamma))$ and $g_t\in L^2(0,T;L^2(\Gamma))$ for some fixed $T(0<T<1)$. Assume that the initial condition $({\bf u}_0,\phi_0,\psi_0)\in {\mathbb D}_1$, then there exists a unique solution ${\bf u}$ of $(\ref{LNS})$ such that
\begin{align}
\label{P2}
\|{\bf u}\|_{Y_1}^2
&\le Ce^{\epsilon^{-1}K_1^2K_2^2T}\left(\|{\bf u}_0\|_{H^2}^4+\|\phi_0\|_{H^3}^4+1+\|G\|_{L^2}^2
+\|g\|_{H^{\frac{1}{2}}(\Gamma)}^2
\right)                                              \nonumber
\\
&\quad+Ce^{\epsilon^{-1}K_1^2K_2^2T}\int_{0}^{T}
\left[\epsilon K_1^{-2}K_2^{-2}\left(\|G_t\|_{L^2}^2+\|g_t\|_{L^2(\Gamma)}^2\right)
+\|G\|_{H^1}^2+\|g\|_{H^{\frac{3}{2}}(\Gamma)}^2
\right]{\rm d}t,
\end{align}
where $C$ is a positive constant depending on $\Omega$ and $\beta$, but independent of $K_1$ and $K_2$.
\end{Proposition}
The well-posedness of the problem (\ref{LNS}) can be found in \cite{6} (see Theorem 2.3 therein).
Now we only need to do the a priori estimates in (\ref{P2}).
The first estimate is about the $L^\infty(0,T;L^2)$-norm of the velocity $\bf u$.
\begin{Lemma}
\label{LNS L1}
Let ${\bf u}$ be the smooth solution to $(\ref{LNS})$ on $[0,T]$, then it holds that
\begin{align}
\label{L2.0}
&\sup\limits_{0\le t \le T}
\|{\bf u}\|_{L^2}^2+\int_{0}^{T}\left(
\|\nabla{\bf u}\|_{L^2}^2
+\|{\bf u}_{\boldsymbol\tau}\|_{L^2(\Gamma)}^2
\right) {\rm d}t                                               \nonumber
\\
&\le C(\beta^{-1})\left[\|{\bf u}_0\|_{L^2}^2+C\int_{0}^{T}\left(
\|G\|_{L^2}^2+\|g\|_{L^2(\Gamma)}^2
\right) {\rm d}t
\right],
\end{align}
where $C$ is a generic constant that depends on $\Omega$ and $\beta$, but not on $K_1$ and $K_2$.
\end{Lemma}
\begin{proof}[\bf Proof.]
Multiplying $(\ref{LNS})_1$ by ${\bf u}$, and integrating the result over $\Omega$ lead to
\begin{equation}
\label{eq29}
\int_\Omega{\bf u}_t\cdot{\bf u} {\rm d}x-\int_\Omega{\rm div}{\mathbb S}({\bf u})\cdot{\bf u} {\rm d}x+\int_\Omega\nabla p\cdot{\bf u} {\rm d}x =\int_\Omega G\cdot{\bf u} {\rm d}x,
\end{equation}
First, due to integration by parts and boundary condition $(\ref{LNS})_4$, one has
\begin{align}
\label{eq24}
-\int_\Omega{\rm div}{\mathbb S}({\bf u})\cdot{\bf u}{\rm d}x&=-\int_\Omega\partial_i(\partial_i u_j+\partial_j u_i)u_j {\rm d}x
\nonumber
\\
&=\int_\Omega(\partial_i u_j+\partial_j u_i)\partial_i u_j
{\rm d}x-\int_\Gamma(\partial_i u_j+\partial_j u_i)n_i u_j {\rm d}S
\nonumber
\\
&=\frac{1}{2}\int_\Omega(\partial_i u_j+\partial_j u_i)(\partial_i u_j+\partial_j u_i)
{\rm d}x-\int_\Gamma [{\mathbb S}({\bf u})\cdot{\bf n}]\cdot {\bf u}{\rm d}S                                              \nonumber
\\
&=\frac{1}{2}\int_\Omega|{\mathbb S}({\bf u})|^2 {\rm d}x-\int_\Gamma({\mathbb S}({\bf u})\cdot{\bf n})_{\boldsymbol\tau}\cdot{\bf u}_{\boldsymbol\tau}{\rm d}S  \nonumber
\\
&=\frac{1}{2}\|{\mathbb S}({\bf u})\|_{L^2}^2-\int_\Gamma
(g-\beta{\bf u}_{\boldsymbol\tau})\cdot{\bf u}_{\boldsymbol\tau}{\rm d}S  \nonumber
\\
&=\frac{1}{2}\|{\mathbb S}({\bf u})\|_{L^2}^2+\beta\|{\bf u}_{\boldsymbol\tau}\|_{L^2(\Gamma)}^2-\int_\Gamma
g\cdot{\bf u}_{\boldsymbol\tau}{\rm d}S,
\end{align}
where we have used the following fact
\begin{align}
\label{su}
\int_\Gamma({\mathbb S}({\bf u})\cdot{\bf n})\cdot{\bf u}{\rm d}S&=\int_\Gamma \Big\{  \Big[ \Big({\mathbb S}({\bf u})\cdot{\bf n}\Big)\cdot{\bf n}  \Big] {\bf n}+\Big({\mathbb S}({\bf u})\cdot{\bf n}\Big)_{\boldsymbol\tau}  \Big\} \cdot\Big\{ ({\bf u}\cdot{\bf n}) {\bf n}+{\bf u}_{\boldsymbol\tau} \Big\}{\rm d}S      \nonumber
\\
&=\int_\Gamma({\mathbb S}({\bf u})\cdot{\bf n})_{\boldsymbol\tau}\cdot{\bf u}_{\boldsymbol\tau}{\rm d}S,
\end{align}
since ${\bf u}\cdot{\bf n}=0$ on $\Gamma$.
Then, by the incompressibility of ${\bf u}$, we deduce that
\begin{align}
\label{eq25}
\int_\Omega\nabla p\cdot{\bf u}{\rm d}x=-\int_\Omega p{\rm div}{\bf u}{\rm d}x+\int_\Gamma p{\bf u}\cdot{\bf n}{\rm d}S=0.
\end{align}
Then, substituting (\ref{eq24}) and (\ref{eq25}) into (\ref{eq29}), using Cauchy-Schwartz's inequality, we obtain
\begin{align}
\frac{1}{2}\frac{\rm d}{{\rm d}t}\|{\bf u}\|_{L^2}^2+\beta\|{\bf u}_{\boldsymbol\tau}\|_{L^2(\Gamma)}^2+\frac{1}{2}\|{\mathbb S}({\bf u})\|_{L^2}^2&=\int_\Omega{\bf u}\cdot G {\rm d}x+\int_\Gamma g\cdot{\bf u}_{\boldsymbol\tau} {\rm d}S
\nonumber
\\
&\le \frac{\beta}{2}\|{\bf u}_{\boldsymbol\tau}\|_{L^2(\Gamma)}^2
+C(\beta^{-1})\|g\|_{L^2(\Gamma)}^2+C\|{\bf u}\|_{L^2}^2+C\|G\|_{L^2}^2,
\nonumber
\end{align}
together with Gronwall inequality and Korn's inequality, we arrive at
\begin{align}
&\sup\limits_{0\le t \le T}
\|{\bf u}\|_{L^2}^2+\int_{0}^{T}\left(
\|\nabla{\bf u}\|_{L^2}^2
+\|{\bf u}_{\boldsymbol\tau}\|_{L^2(\Gamma)}^2
\right) {\rm d}t                                              \nonumber
\\
&\le C(\beta^{-1})\left[\|{\bf u}_0\|_{L^2}^2+C\int_{0}^{T}\left(
\|G\|_{L^2}^2+\|g\|_{L^2(\Gamma)}^2
\right) {\rm d}t
\right]. \nonumber
\end{align}
Hence, the proof of Lemma \ref{LNS L1} is completed.
\end{proof}
\begin{Lemma}
\label{LNS L2}
Let ${\bf u}$ be the smooth solution to $(\ref{LNS})$ on $[0,T]$, then it holds that
\begin{align}
\label{L2.1}
&\sup\limits_{0\le t \le T}
\|{\bf u}_t\|_{L^2}^2+\int_{0}^{T}\left(
\|\nabla{\bf u}_t\|_{L^2}^2
+\beta\|{\bf u}_{{\boldsymbol\tau}t}\|_{L^2(\Gamma)}^2
\right) {\rm d}t                                            \nonumber
\\
&\le Ce^{\epsilon^{-1}K_1^2K_2^2T}\left[\|{\bf u}_0\|_{H^2}^4+\|\phi_0\|_{H^3}^4+1
+\int_{0}^{T}
\epsilon K_1^{-2}K_2^{-2}\left(\|G_t\|_{L^2}^2+\|g_t\|_{L^2(\Gamma)}^2
\right) {\rm d}t \right]                                            \nonumber
\\
& =C_3,
\end{align}
where $C$ is a generic constant that depends on $\Omega$ and $\beta$, but not on $K_1$ and $K_2$.
\end{Lemma}
\begin{proof}[\bf Proof.]
Differentiating $(\ref{LNS})_1$ with respect to $t$, there is
\begin{equation*}
{\bf u}_{tt}-{\rm div}{\mathbb S}({\bf u}_t)+\nabla p_t=G_t.
\end{equation*}
Multiplying it by ${\bf u}_t$, and integrating the result over $\Omega$ by parts, similar to (\ref{eq24}) and (\ref{eq25}), we obtain
\begin{align}
-\int_\Omega{\bf u}_t\cdot{\rm div}{\mathbb S}({\bf u}_t)
{\rm d}x&=\frac{1}{2}\int_\Omega|{\mathbb S}({\bf u}_t)|^2
{\rm d}x-\int_\Gamma \left({\mathbb S}({\bf u}_t)\cdot{\bf n}\right)\cdot {\bf u}_t {\rm d}S                             \nonumber
\\
&=\frac{1}{2}\|{\mathbb S}({\bf u}_t)\|_{L^2}^2
-\int_\Gamma\left(g_t-\beta{\bf u}_{{\boldsymbol\tau}t}
\right)
\cdot{\bf u}_{{\boldsymbol\tau}t} {\rm d}S                    \nonumber
\\
&=\frac{1}{2}\|{\mathbb S}({\bf u}_t)\|_{L^2}^2
+\beta\|{\bf u}_{{\boldsymbol\tau}t}\|_{L^2(\Gamma)}^2
-\int_\Gamma g_t\cdot{\bf u}_{{\boldsymbol\tau}t} {\rm d}S,
\nonumber
\end{align}
and
\begin{align}
\label{nabla pt}
\int_\Omega{\bf u}_t\cdot\nabla p_t {\rm d}x  = -\int_\Omega p_t {\rm div} {\bf u}_t  {\rm d}x+\int_\Gamma p_t {\bf u}_t \cdot {\bf n} {\rm d}S=0.
\end{align}
According to the above equalities, Lemma \ref{Trace}, (\ref{L21}) and (\ref{poincare u}), we get
\begin{align}
\label{L2.2}
&\frac{1}{2}\frac{\rm d}{{\rm d}t}\|{\bf u}_t\|_{L^2}^2+\frac{1}{2}\|{\mathbb S}({\bf u}_t)\|_{L^2}^2+\beta\|{\bf u}_{{\boldsymbol\tau}t}\|_{L^2(\Gamma)}^2=\int_\Omega{\bf u}_t\cdot G_t {\rm d}x+\int_\Gamma g_t\cdot{\bf u}_{{\boldsymbol\tau}t} {\rm d}S                           \nonumber
\\
&\le\epsilon\left(K_1^2K_2^2\right)^{-1}\|G_t\|_{L^2}^2+\epsilon^{-1}K_1^2K_2^2\|{\bf u}_t\|_{L^2(\Gamma)}^2 \nonumber
\\
&\quad+\epsilon^{-1}K_1^2K_2^2
\|{\bf u}_{{\boldsymbol\tau}t}\|_{L^2(\Gamma)}^2
+\epsilon\left(K_1^2K_2^2\right)^{-1}\|g_t\|_{L^2(\Gamma)}^2                       \nonumber
\\
&\le C\epsilon \left(K_1^2K_2^2\right)^{-1}\|G_t\|_{L^2}^2+\epsilon^{-1}K_1^2K_2^2\|{\bf u}_t\|_{L^2(\Gamma)}^2 \nonumber
\\
&\quad+C\epsilon^{-1}K_1^2K_2^2\|{\bf u}_t\|_{L^2}\|{\bf u}_t\|_{H^1}
+\epsilon \left(K_1^2K_2^2\right)^{-1}\|g_t\|_{L^2(\Gamma)}^2                      \nonumber
\\
&\le \frac{1}{4}\|\mathbb S({\bf u}_t)\|_{L^2}^2+C\epsilon^{-1}K_1^2K_2^2\|{\bf u}_t\|_{L^2}^2+\epsilon \left(K_1^2K_2^2\right)^{-1}\left(\|G_t\|_{L^2}^2+\|g_t\|_{L^2(\Gamma)}^2\right).
\end{align}
In addition, the initial data $\|\partial_t{\bf u}(\cdot,0)\|_{L^2}$ will be constructed via $({\bf u}_0,\phi_0,\psi_0)$ as follows
\begin{align}
&\|\partial_t{\bf u}(\cdot,0)\|_{L^2}^2=\lim_{t \to 0+}\|\partial_t{\bf u}\|_{L^2}^2=\lim_{t \to 0+}\langle-{\bf u}\cdot\nabla{\bf u}-\nabla p+{\rm div}\mathbb S({\bf u})-{\rm div}(\nabla\phi\otimes\nabla\phi),\partial_t{\bf u}\rangle,  \nonumber
\end{align}
Therefore, one can infer
\begin{align}
\label{initial ut}
\|\partial_t{\bf u}(\cdot,0)\|_{L^2}^2&\le C\|{\bf u}_0\|_{H^2}^2\left(\|{\bf u}_0\|_{H^2}^2+1\right)+C\|\phi_0\|_{H^3}^4,
\end{align}
which together with \eqref{L2.2}, Gronwall inequality, Korn's inequality  yields (\ref{L2.1}). Hence, we complete the proof of Lemma \ref{LNS L2}.
\end{proof}
\vskip2mm
Now, we will use Lemma \ref{LNS L2} to prove the bound on the $L^\infty(0,T;L^2)$-norm of $\nabla{\bf u}$:
\begin{Lemma}
\label{LNS L3}
Let ${\bf u}$ be the smooth solution to $(\ref{LNS})$ on $[0,T]$, then it holds that
\begin{align}
\label{eq444}
&\sup\limits_{0\le t \le T}\left(
\|\nabla{\bf u}\|_{L^2}^2+\beta\|{\bf u}_{\boldsymbol\tau}\|_{L^2(\Gamma)}^2
\right) +\int_{0}^{T}\|{\bf u}_t\|_{L^2}^2 {\rm d}t
\le CC_3+C \int_{0}^{T}\left(\|G\|_{L^2}^2+\|g\|_{L^2(\Gamma)}^2
\right) {\rm d}t ,
\end{align}
where $C$ is a generic constant that depends on $\Omega$ and $\beta$, but not on $K_1$ and $K_2$.
\end{Lemma}
\begin{proof}[\bf Proof.]
Multiplying $(\ref{LNS})_1$ by ${\bf u}_t$, and integrating the result over $\Omega$ lead to
\begin{equation}
\label{eq45}
\|{\bf u}_t\|_{L^2}^2-\int_\Omega{\bf u}_t\cdot{\rm div}{\mathbb S}({\bf u}){\rm d}x+\int_\Omega{\bf u}_t\cdot\nabla p {\rm d}x =\int_\Omega{\bf u}_t\cdot G {\rm d}x.
\end{equation}
Using integration by parts and boundary conditions $(\ref{LNS})_{2,4}$, we obtain
\begin{align}
-\int_\Omega{\bf u}_t\cdot{\rm div}{\mathbb S}({\bf u}){\rm d}x
&=\int_\Omega(\partial_i u_j+\partial_j u_i)\partial_t\partial_i u_j
{\rm d}x-\int_\Gamma(\partial_i u_j+\partial_j u_i)n_i \partial_t u_j {\rm d}S                                       \nonumber
\\
&=\frac{1}{2}\int_\Omega(\partial_i u_j+\partial_j u_i)\partial_t(\partial_i u_j+\partial_j u_i)
{\rm d}x-\int_\Gamma({\mathbb S}({\bf u})\cdot{\bf n})\cdot {\bf u}_t{\rm d}S                                               \nonumber
\\
&=\frac{1}{4}\frac{\rm d}{{\rm d}t}\|{\mathbb S}({\bf u})\|_{L^2}^2-\int_\Gamma({\mathbb S}({\bf u})\cdot{\bf n})_{\boldsymbol\tau}\cdot{\bf u}_{t{\boldsymbol\tau}}{\rm d}S  \nonumber
\\
&=\frac{1}{4}\frac{\rm d}{{\rm d}t}\|{\mathbb S}({\bf u})\|_{L^2}^2
-\int_\Gamma(g-\beta{\bf u}_{\boldsymbol\tau})
\cdot{\bf u}_{t{\boldsymbol\tau}}{\rm d}S                       \nonumber
\\
&=\frac{1}{4}\frac{\rm d}{{\rm d}t}\|{\mathbb S}({\bf u})\|_{L^2}^2+\frac{\beta}{2}\frac{\rm d}{{\rm d}t}\|{\bf u}_{\boldsymbol\tau}\|_{L^2(\Gamma)}^2
-\int_\Gamma g\cdot{\bf u}_{t{\boldsymbol\tau}}{\rm d}S,\nonumber
\end{align}
and
\begin{align}
\label{nabla p}
\int_\Omega{\bf u}_t\cdot\nabla p {\rm d}x   =-\int_\Omega {\rm div} {\bf u}_{t}  p {\rm d}x+\int_\Gamma {\bf u}_{t} \cdot {\bf n} p {\rm d}S=0.
\end{align}
Then, we obtain
\begin{align}
&\frac{1}{2}\frac{\rm d}{{\rm d}t}\left(\frac{1}{2}\|{\mathbb S}({\bf u})\|_{L^2}^2+\beta\|{\bf u}_{\boldsymbol\tau}\|_{L^2(\Gamma)}^2
\right)+\|{\bf u}_t\|_{L^2}^2                         \nonumber
\\
&=\int_\Omega{\bf u}_t\cdot G {\rm d}x+\int_\Gamma g\cdot{\bf u}_{t{\boldsymbol\tau}}{\rm d}S                             \nonumber
\\
&\le \frac{1}{2}\|{\bf u}_t\|_{L^2}^2+C\left(\|G\|_{L^2}^2+\|g\|_{L^2(\Gamma)}^2
+\|{\bf u}_{t{\boldsymbol\tau}}\|_{L^2(\Gamma)}^2
\right)                                               \nonumber
\\
&\le \frac{1}{2}\|{\bf u}_t\|_{L^2}^2+C\left(\|G\|_{L^2}^2+\|g\|_{L^2(\Gamma)}^2
+\|{\bf u}_t\|_{H^1}^2
\right),                                              \nonumber
\end{align}
where we have used Cauchy-Schwartz's inequality and Lemma \ref{Trace}. Together with Gronwall inequality, 
 and (\ref{L2.1}), we have
\begin{align}
&\sup\limits_{0\le t \le T}\left(\|\nabla{\bf u}\|_{L^2}^2+\beta\|{\bf u}_{\boldsymbol\tau}\|_{L^2(\Gamma)}^2
\right)
+\int_{0}^{T}\|{\bf u}_t\|_{L^2}^2 {\rm d}t                   \nonumber
\\
&\le C\int_{0}^{T}\left(\|G\|_{L^2}^2+\|g\|_{L^2(\Gamma)}^2
+\|{\bf u}_t\|_{H^1}^2
\right) {\rm d}t                                              \nonumber
\\
&\le CC_3+C \int_{0}^{T}\left(\|G\|_{L^2}^2+\|g\|_{L^2(\Gamma)}^2
\right) {\rm d}t.                                             \nonumber
\end{align}
Thus, the proof of Lemma \ref{LNS L3} is completed.
\end{proof}
\begin{Lemma}
\label{LNS L4}
Let ${\bf u}$ be the smooth solution to $(\ref{LNS})$ on $[0,T]$, then it holds that
\begin{align}
\label{eq49}
&\sup\limits_{0\le t \le T}\left(\|{\bf u}\|_{H^2}^2+\|\nabla p\|_{L^2}^2
\right)
+\int_{0}^{T}\left(\|{\bf u}\|_{H^3}^2+\|\nabla p\|_{H^1}^2
\right) {\rm d}t                                         \nonumber
\\
&\le C\left(C_3+\|G\|_{L^2}^2+\|g\|_{H^{\frac{1}{2}}(\Gamma)}^2 \right) +C\int_{0}^{T}\left(\|G\|_{H^1}^2+\|g\|_{H^{\frac{3}{2}}(\Gamma)}^2
\right) {\rm d}t,
\end{align}
where $C$ is a generic constant that depends on $\Omega$ and $\beta$, but not on $K_1$ and $K_2$.
\end{Lemma}
\begin{proof}[\bf Proof.]
Thanks to the linearity of \eqref{LNS}, the proof of $(\ref{eq49})$ is a straightforward application of
\cite{1} (see Theorem 1.2 therein). Indeed, by rewriting the equations for $\bf u$ as
\begin{align}
\begin{cases}
-{\rm div}{\mathbb S}({\bf u})+\nabla p=G-{\bf u}_t,&\rm in\quad\Omega\times(0,T),
\\
{\rm div}{\bf u}=0,&\rm in\quad\Omega\times(0,T),
\\
{\bf u}\cdot{\bf n}=0,&\rm on\quad\Gamma\times(0,T),
\\
\beta{\bf u}_{\boldsymbol\tau}+({\mathbb S}({\bf u})\cdot{\bf n})_{\boldsymbol{\tau}}=g,&\rm on\quad\Gamma\times(0,T),
\nonumber
\end{cases}
\end{align}
then, we deduce for almost any $t\in[0,T]$,
\begin{align}
\label{eq50}
||{\bf u}||_{H^2}^2+||\nabla p||_{L^2}^2
&\le C||G||_{L^2}^2+C||{\bf u}_t||_{L^2}^2+C||g||_{H^{\frac{1}{2}}(\Gamma)}^2 \nonumber
\\
&\le CC_3+C||G||_{L^2}^2+C||g||_{H^{\frac{1}{2}}(\Gamma)}^2,
\end{align}
where we have used (\ref{L2.1}) and (\ref{eq444}). Exploiting \cite{1} (Theorem 1.2) again, there is
\begin{align}
||\nabla{\bf u}||_{H^2}^2+||\nabla p||_{H^1}^2&\le C\left(||\nabla G||_{L^2}^2+||\nabla{\bf u}_t||_{L^2}^2+||\nabla_{\boldsymbol\tau} g||_{H^{\frac{1}{2}}(\Gamma)}^2
\right) \nonumber
\\
&\le C\left(||G||_{H^1}^2+||\nabla{\bf u}_t||_{L^2}^2+||g||_{H^{\frac{3}{2}}(\Gamma)}^2
\right),                                                 \nonumber
\end{align}
which together with $(\ref{L2.0})$ and $(\ref{L2.1})$ implies
\begin{align}
\label{eq46}
\int_{0}^{T}\left(||{\bf u}||_{H^3}^2+||\nabla p||_{H^1}^2
\right) {\rm d}t
&\le C\int_{0}^{T}\left(||G||_{H^1}^2+||\nabla{\bf u}_t||_{L^2}^2+||g||_{H^{\frac{3}{2}}(\Gamma)}^2
\right){\rm d}t                                           \nonumber
\\
&\le CC_3+ C\int_{0}^{T}\left(||G||_{H^1}^2+||g||_{H^{\frac{3}{2}}(\Gamma)}^2
\right){\rm d}t.
\end{align}
Combining (\ref{eq50}) with (\ref{eq46}), we complete the proof of Lemma \ref{LNS L4}.
\end{proof}
\subsubsection{\bf Proof of Theorem \ref{local}}\quad
For $(\tilde{\bf u},\tilde\phi,\tilde\psi)\in X(0,T;K_1,K_2)$, the Allen-Cahn equations (\ref{LAC}) and the Navier-Stokes equations (\ref{LNS}) admit a unique solution $(\phi,\psi)\in Y_2$ and ${\bf u}\in Y_1$ respectively. Thus, we can easily deduce $({\bf u},\phi,\psi)\in Y$. Now we consider the mapping $\Lambda$ defined on $Y$ as following
\begin{align}
\Lambda: (\tilde{\bf u},\tilde\phi,\tilde\psi)\mapsto({\bf u},\phi,\psi).                                       \nonumber
\end{align}
We shall prove that $\Lambda$ has a fixed point in $X(0,T;K_1,K_2)$ for $T$ small enough, where the positive constants $K_1$, $K_2$ and $T$ will be determined below. Let $({\bf u},\phi,\psi)=\Lambda(\tilde{\bf u},\tilde\phi,\tilde\psi)$,
some a priori estimates are requested which allow us to accomplish the proof.
\vskip2mm
First, we prove $\Lambda$ is a self-mapping. To this aim, we need to estimate linearized terms of Allen-Cahn equations (\ref{LAC}).
Direct calculation shows that
\begin{align}
\int_{0}^{T}\|H\|_{H^2}^2 {\rm d}t
&=\int_{0}^{T}\left(\|H\|_{L^2}^2+\|\nabla H\|_{L^2}^2+\|\nabla^2 H\|_{L^2}^2
\right) {\rm d}t                                              \nonumber
\\
&=\int_{0}^{T}\left(\|-\tilde{\bf u}\cdot\nabla\tilde\phi\|_{L^2}^2+\|-\nabla\tilde{\bf u}\cdot\nabla\tilde\phi-\tilde{\bf u}\cdot\nabla^2\tilde\phi\|_{L^2}^2
\right) {\rm d}t                                              \nonumber
\\
&\quad +C\int_{0}^{T}\|-\nabla^2\tilde{\bf u}\cdot\nabla\tilde\phi-2\nabla\tilde{\bf u}\cdot\nabla^2\tilde\phi-\tilde{\bf u}\cdot\nabla^3\tilde\phi\|_{L^2}^2{\rm d}t                                                      \nonumber
\\
&\le C\int_{0}^{T}\left(\|\tilde{\bf u}\|_{L^\infty}^2\|\nabla\tilde\phi\|_{L^2}^2+\|\nabla\tilde{\bf u}\|_{L^3}^2\|\nabla\tilde\phi\|_{L^6}^2+\|\tilde{\bf u}\|_{L^\infty}^2\|\nabla^2\tilde\phi\|_{L^2}^2
\right) {\rm d}t                                              \nonumber
\\
&\quad +C\int_{0}^{T}\left(\|\nabla^2\tilde{\bf u}\|_{L^2}^2\|\nabla\tilde\phi\|_{L^\infty}^2+\|\nabla\tilde{\bf u}\|_{L^3}^2\|\nabla^2\tilde\phi\|_{L^6}^2+\|\tilde{\bf u}\|_{L^\infty}^2\|\nabla^3\tilde\phi\|_{L^2}^2
\right) {\rm d}t                                              \nonumber
\\
&\le C\int_{0}^{T}\|\tilde{\bf u}\|_{H^2}^2\|\nabla\tilde\phi\|_{H^2}^2{\rm d}t \nonumber
\\
&\le CK_1K_2T,                                          \nonumber
\\[1em]
\int_{0}^{T}\|H_t\|_{L^2}^2 {\rm d}t
&=\int_{0}^{T}\|(-\tilde{\bf u}\cdot\nabla\tilde\phi)_t\|_{L^2}^2 {\rm d}t
\le C\int_{0}^{T}\left(\|\tilde{\bf u}_t\cdot\nabla\tilde\phi\|_{L^2}^2+\|\tilde{\bf u}\cdot\nabla\tilde\phi_t\|_{L^2}^2
\right) {\rm d}t                                              \nonumber
\\
&\le C\int_{0}^{T}\left(\|\tilde{\bf u}_t\|_{L^2}^2\|\nabla\tilde\phi\|_{L^\infty}^2+\|\tilde{\bf u}\|_{L^\infty}^2\|\nabla\tilde\phi_t\|_{L^2}^2
\right) {\rm d}t                                              \nonumber
\\
&\le C\int_{0}^{T}\left(\|\tilde{\bf u}_t\|_{L^2}^2\|\nabla\tilde\phi\|_{H^2}^2+\|\tilde{\bf u}\|_{H^2}^2\|\nabla\tilde\phi_t\|_{L^2}^2
\right) {\rm d}t                                              \nonumber
\\
&\le CK_1K_2T.                                          \nonumber
\end{align}
Similarly, we have
\begin{align}
\int_{0}^{T}\|\nabla H_t\|_{L^2}^2 {\rm d}t
& \le C\int_{0}^{T}\left(\|\nabla\tilde{\bf u}_t\cdot\nabla\tilde\phi\|_{L^2}^2+\|\tilde{\bf u}_t\cdot\nabla^2\tilde\phi\|_{L^2}^2+\|\nabla\tilde{\bf u}\cdot\nabla\tilde\phi_t\|_{L^2}^2+\|\tilde{\bf u}\cdot\nabla^2\tilde\phi_t\|_{L^2}^2
\right) {\rm d}t                                              \nonumber
\\
&\le C\int_{0}^{T}\left(\|\nabla\tilde{\bf u}_t\|_{L^2}^2\|\nabla\tilde\phi\|_{L^\infty}^2+\|\tilde{\bf u}_t\|_{L^6}^2\|\nabla^2\tilde\phi\|_{L^3}^2\right) {\rm d}t \nonumber
\\
&\quad+C\int_{0}^{T}\left(\|\nabla\tilde{\bf u}\|_{L^6}^2\|\nabla\tilde\phi_t\|_{L^3}^2+\|\tilde{\bf u}\|_{L^\infty}^2\|\nabla^2\tilde\phi_t\|_{L^2}^2
\right) {\rm d}t                                              \nonumber
\\
&\le C\int_{0}^{T}\left(\|\tilde{\bf u}_t\|_{H^1}^2\|\nabla\tilde\phi\|_{H^2}^2+\|\tilde{\bf u}\|_{H^2}^2\|\nabla\tilde\phi_t\|_{H^1}^2\right) {\rm d}t \nonumber
\\
&\le CK_1K_2,                                           \nonumber
\end{align}
which implies
\begin{align}
\|H\|_{H^1}
&=\left\Vert\int_{0}^{t}H_t(x,s){\rm d}s+H(x,0)\right\Vert_{H^1}
\le \int_{0}^{t}\|H_t\|_{H^1}{\rm d}s+\|H_0\|_{H^1}
\nonumber \\
&\le T^{\frac{1}{2}}\|H_t\|_{L^2(0,T;H^1)}+\|H_0\|_{H^1}
\le CK_1^{\frac{1}{2}}K_2^{\frac{1}{2}}T^{\frac{1}{2}}+C\|{\bf u}_0\|_{H^2}^2+C\|\phi_0\|_{H^2}^2,
\nonumber
\end{align}
where
\begin{align}
\|H_0\|_{H^1}&=\|H_0\|_{L^2}+\|\nabla H_0\|_{L^2}  \nonumber
\\
&=\|-\tilde{\bf u}_0\cdot\nabla\tilde\phi_0\|_{L^2}+\|-\nabla\tilde{\bf u}_0\cdot\nabla\tilde\phi_0-\tilde{\bf u}_0\cdot\nabla^2\tilde\phi_0\|_{L^2}
\nonumber
\\
&\le C\|\tilde{\bf u}_0\|_{L^6}\|\nabla\tilde\phi_0\|_{L^3}+C\left(\|\nabla\tilde{\bf u}_0\|_{L^6}\|\nabla\tilde\phi_0\|_{L^3}
+\|\tilde{\bf u}_0\|_{L^\infty}\|\nabla^2\tilde\phi_0\|_{L^2}
\right)                                               \nonumber
\\
&\le C\|\tilde{\bf u}_0\|_{H^1}\|\nabla\tilde\phi_0\|_{H^1}
+C\left(\|\nabla\tilde{\bf u}_0\|_{H^1}\|\nabla\tilde\phi_0\|_{H^1}
+\|\tilde{\bf u}_0\|_{H^2}\|\nabla^2\tilde\phi_0\|_{L^2}
\right)                                               \nonumber
\\
&\le C\|{\bf u}_0\|_{H^2}^2+C\|\phi_0\|_{H^2}^2.      \nonumber
\end{align}
On the related estimates of $f$, we get
\begin{align}
\int_{0}^{T}\|f\|_{H^2}^2 {\rm d}t
&=\int_{0}^{T}\left(\|\tilde\phi^3-\tilde\phi\|_{L^2}^2
+\|(3\tilde\phi^2-1)\nabla\tilde\phi\|_{L^2}^2
+\| 6\tilde\phi|\nabla\tilde\phi|^2+(3\tilde\phi^2-1)
\nabla^2\tilde\phi\|_{L^2}^2
\right) {\rm d}t                                              \nonumber
\\
&\le C\int_{0}^{T}\left(\|\tilde\phi\|_{L^6}^6+\|\tilde\phi\|_{L^2}^2
+\|3\tilde\phi^2-1\|_{L^\infty}^2\|\nabla\tilde\phi\|_{L^2}^2
\right) {\rm d}t                                              \nonumber
\\
&\quad
+C\int_{0}^{T}\left(\|\tilde\phi\|_{L^\infty}^2\|\nabla\tilde\phi\|_{L^4}^4
+\|3\tilde\phi^2-1\|_{L^\infty}^2\|\nabla^2\tilde\phi\|_{L^2}^2
\right) {\rm d}t                                              \nonumber
\\
&\le C\int_{0}^{T}\left(\|\phi\|_{H^2}^6+\|\phi\|_{H^2}^2\right) {\rm d}t \nonumber
\\
&\le CK_2^3T.                                           \nonumber
\end{align}
Meanwhile
\begin{align}
\int_{0}^{T}\|f_t\|_{H^1}^2 {\rm d}t
&=\int_{0}^{T}\left(\|f_t\|_{L^2}^2+\|\nabla f_t\|_{L^2}^2
\right) {\rm d}t                                              \nonumber
\\
&=\int_{0}^{T}\left(
\|(3\tilde\phi^2-1)\tilde\phi_t\|_{L^2}^2
+\|6\tilde\phi\nabla\tilde\phi\tilde\phi_t+(3\tilde\phi^2-1)
\nabla\tilde\phi_t\|_{L^2}^2
\right) {\rm d}t                                              \nonumber
\\
&\le C\int_{0}^{T}\left(
\|3\tilde\phi^2-1\|_{L^\infty}^2\|\tilde\phi_t\|_{L^2}^2
+\|\tilde\phi\|_{L^\infty}^2\|\nabla\tilde\phi\|_{L^\infty}^2\|\tilde\phi_t\|_{L^2}^2
\right) {\rm d}t \nonumber
\\
&\quad+C\int_{0}^{T}\|3\tilde\phi^2-1\|_{L^\infty}^2\|\nabla\tilde\phi_t\|_{L^2}^2{\rm d}t                                              \nonumber
\\
&\le C\int_{0}^{T}\left(\|\tilde\phi\|_{H^2}^4\|\tilde\phi_t\|_{H^1}^2+\|\tilde\phi_t\|_{H^1}^2
\right) {\rm d}t  \nonumber
\\
&\le CK_2^3T,                                           \nonumber
\end{align}
which implies
\begin{align}
\|f(\tilde\phi)\|_{H^1}
&=\left\Vert\int_{0}^{t}f_t(\tilde\phi(x,s)){\rm d}s
+f(\tilde\phi(x,s))\right\Vert_{H^1}                    \nonumber
\\
&\le \int_{0}^{t}\|f_t(\tilde\phi)\|_{H^1}{\rm d}s
+\|f(\tilde\phi(x,0))\|_{H^1}  \le CK_2^{\frac{3}{2}}T+C\left(\|\phi_0\|_{H^2}^3+\|\phi_0\|_{H^1}
\right),                                                \nonumber
\end{align}
where we have used the following fact:
\begin{align}
\|f(\tilde\phi(x,0))\|_{H^1}&=\|\tilde\phi_0^3-\tilde\phi_0\|_{L^2}
+\|(3\tilde\phi_0^2-1)\nabla\tilde\phi_0\|_{L^2}
\nonumber
\\
&\le C\left(\|\tilde\phi_0\|_{L^6}^3+\|\tilde\phi_0\|_{L^2}
+\|3\tilde\phi_0^2-1\|_{L^\infty}\|\nabla\tilde\phi_0\|_{L^2}
\right)                                                 \nonumber
\\
&\le C\left(\|\tilde\phi_0\|_{H^1}^3
+\|\tilde\phi_0\|_{H^2}^2\|\nabla\tilde\phi_0\|_{L^2}+\|\tilde\phi_0\|_{H^1}
\right)                                                 \nonumber
\\
&\le C\left(\|\phi_0\|_{H^2}^3+\|\phi_0\|_{H^1}
\right).                                                \nonumber
\end{align}
It follows from Lemma \ref{Trace} that
\begin{align}
\|h\|_{H^1(\Gamma)}^2
&=\left\Vert-\gamma_{fs}^\prime(\tilde\psi)-\tilde{\bf u}_{\boldsymbol\tau}\cdot\nabla_{\boldsymbol\tau}\tilde\psi\right\Vert
_{L^2(\Gamma)}^2+\left\Vert-\gamma_{fs}^{(2)}(\tilde\psi)
\nabla_{\boldsymbol\tau}\tilde\psi
-\nabla_{\boldsymbol\tau}\tilde{\bf u}_{\boldsymbol\tau}\cdot\nabla_{\boldsymbol\tau}\tilde\psi
-\tilde{\bf u}_{\boldsymbol\tau}\cdot\nabla^2_{\boldsymbol\tau}\tilde\psi\right\Vert
_{L^2(\Gamma)}^2                                  \nonumber
\\
&\le C\left(1+\|\tilde{\bf u}_{\boldsymbol\tau}\|_{L^3(\Gamma)}^2\|\nabla_{\boldsymbol\tau}\tilde\psi\|
_{L^6(\Gamma)}^2+\|\nabla_{\boldsymbol\tau}\tilde\psi\|_{L^2(\Gamma)}^2
\right)                                                   \nonumber
\\
&\quad + C\left(
\|\nabla_{\boldsymbol\tau}\tilde{\bf u}_{\boldsymbol\tau}\|_{L^2(\Gamma)}^2\|\nabla_{\boldsymbol\tau}\tilde\psi\|
_{L^\infty(\Gamma)}^2
+\|\tilde{\bf u}_{\boldsymbol\tau}\|_{L^6(\Gamma)}^2
\|\nabla^2_{\boldsymbol\tau}\tilde\psi\|_{L^3(\Gamma)}^2
\right)                                                   \nonumber
\\
&\le C\left(1+\|\tilde{\bf u}\|_{H^2}^2\|\nabla_{\boldsymbol\tau}\tilde\psi\|_{H^2(\Gamma)}^2
+\|\nabla_{\boldsymbol\tau}\tilde\psi\|_{L^2(\Gamma)}^2\right) \nonumber
\\
&\le CK_2^2.                                              \nonumber
\end{align}
Moreover, noting that $\Gamma$ is a two-dimensional manifold, using (\ref{L21}) we have
\begin{align}
\|\nabla^2_{\boldsymbol\tau}h\|_{L^2(\Gamma)}^2
&=\left\Vert\nabla_{\boldsymbol\tau}\left(-\gamma_{fs}^{(2)}(\tilde\psi)
\nabla_{\boldsymbol\tau}\tilde\psi
-\nabla_{\boldsymbol\tau}\tilde{\bf u}_{\boldsymbol\tau}\cdot\nabla_{\boldsymbol\tau}\tilde\psi
-\tilde{\bf u}_{\boldsymbol\tau}\cdot\nabla^2_{\boldsymbol\tau}\tilde\psi\right)\right\Vert
_{L^2(\Gamma)}^2                                  \nonumber
\\
&\le C\left(\|\nabla_{\boldsymbol\tau}\tilde\psi\|_{L^4(\Gamma)}^4
+\|\nabla^2_{\boldsymbol\tau}\tilde\psi\|_{L^2(\Gamma)}^2
+\|\nabla^2_{\boldsymbol\tau}\tilde{\bf u}_{\boldsymbol\tau}\|_{L^2(\Gamma)}^2
\|\nabla_{\boldsymbol\tau}\tilde\psi\|_{L^\infty(\Gamma)}^2
\right)                                                  \nonumber
\\
&\quad + C\left(\|\nabla_{\boldsymbol\tau}\tilde{\bf u}_{\boldsymbol\tau}\|_{L^3(\Gamma)}^2
\|\nabla^2_{\boldsymbol\tau}\tilde\psi\|_{L^6(\Gamma)}^2
+\|\tilde{\bf u}_{\boldsymbol\tau}\|_{L^\infty(\Gamma)}^2
\|\nabla^3_{\boldsymbol\tau}\tilde\psi\|_{L^2(\Gamma)}^2
\right)                                                  \nonumber
\\
&\le C\left(\|\nabla_{\boldsymbol\tau}\tilde\psi\|_{H^1(\Gamma)}^4
+\|\nabla^2_{\boldsymbol\tau}\tilde\psi\|_{L^2(\Gamma)}^2+\|\tilde{\bf u}_{\boldsymbol\tau}\|_{H^2(\Gamma)}^2\|\nabla_{\boldsymbol\tau}\tilde\psi\|_{H^2(\Gamma)}^2
\right)  \nonumber
\\
&\le C\left(K_2^2+K_1K_2+K_2\|\nabla^2_{\boldsymbol\tau}\tilde{\bf u}_{\boldsymbol\tau}\|_{L^2(\Gamma)}^2
\right)                                                  \nonumber
\\
&\le C\left(K_2^2+K_1K_2+K_2\|\nabla^2\tilde{\bf u}\|_{H^{\frac{1}{2}}}^2
\right)                                                  \nonumber
\\
&\le C\left(K_2^2+K_1K_2+K_1^{\frac{1}{2}}K_2\|\nabla^2\tilde{\bf u}\|_{H^1}
\right),                                                 \nonumber
\end{align}
which implies
\begin{align}
\int_{0}^{T}\|h\|_{H^2(\Gamma)}^2 {\rm d}t
&\le C\int_{0}^{T}\left(K_2^2+K_1K_2+K_1^{\frac{1}{2}}K_2\|\nabla^2\tilde{\bf u}\|_{H^1}
\right) {\rm d}t                                                \nonumber
\\
&\le C\left[K_2^2T+K_1K_2T+K_1^{\frac{1}{2}}K_2\left(\int_{0}^{T}1^2 {\rm d}t
\right)^{\frac{1}{2}} \left(\int_{0}^{T}\|\nabla^2\tilde{\bf u}\|_{H^1}^2 {\rm d}t
\right)^{\frac{1}{2}}
\right]                                                   \nonumber
\\
&\le C\left(K_2^2T+K_1K_2T^{\frac{1}{2}}
\right).                                                  \nonumber
\end{align}
Besides, from
\begin{align}
\int_{0}^{T}\|h_t\|_{L^2(\Gamma)}^2 {\rm d}t
&=\int_{0}^{T}\left\Vert-\gamma_{fs}^{(2)}(\tilde\psi)\tilde\psi_t-\tilde{\bf u}_{{\boldsymbol\tau}t}\cdot\nabla_{\boldsymbol\tau}\tilde\psi
-\tilde{\bf u}_{\boldsymbol\tau}\cdot\nabla_{\boldsymbol\tau}\tilde\psi_t
\right\Vert_{L^2(\Gamma)}^2{\rm d}t                   \nonumber
\\
&\le C\int_{0}^{T}\left(\|\tilde\psi_t\|_{L^2(\Gamma)}^2
+\|\tilde{\bf u}_{{\boldsymbol\tau}t}\|_{L^2(\Gamma)}^2\|\nabla
_{\boldsymbol\tau}\tilde\psi\|
_{L^\infty(\Gamma)}^2+\|\tilde{\bf u}_{\boldsymbol\tau}\|_{L^\infty(\Gamma)}^2\|\nabla
_{\boldsymbol\tau}\tilde\psi_t\|
_{L^2(\Gamma)}^2
\right) {\rm d}t                                                \nonumber
\\
&\le C\int_{0}^{T}\left(K_2+K_1K_2+K_2\|\tilde{\bf u}_t\|_{H^{\frac{1}{2}}}^2+K_2\|\nabla^2
\tilde{\bf u}\|_{H^{\frac{1}{2}}}^2
\right) {\rm d}t                                                \nonumber
\\
&\le C\int_{0}^{T}\left(K_2+K_1K_2+K_2\|\tilde{\bf u}_t\|_{L^2}\|\tilde{\bf u}_t\|_{H^1}+K_2\|\nabla^2\tilde{\bf u}\|_{L^2}\|\nabla^2\tilde{\bf u}\|_{H^1}
\right) {\rm d}t                                                \nonumber
\\
&\le C\left(K_2T+K_1K_2T^{\frac{1}{2}}
\right),                                                \nonumber
\end{align}
we deduce that
\begin{align}
\|h\|_{L^2(\Gamma)}
&=\left\Vert\int_{0}^{t}h_t(x,s){\rm d}s
+h(x,0)\right\Vert_{L^2(\Gamma)}
\le \int_{0}^{t}\|h_t\|_{L^2(\Gamma)}{\rm d}s
+\|h_0\|_{L^2(\Gamma)}                          \nonumber
\\
&\le T^{\frac{1}{2}}\|h_t\|_{L^2(0,T;L^2(\Gamma))}+C\left(1+\|{\bf u}_0\|_{H^2}
\|\phi_0\|_{H^3}\right),                                \nonumber
\end{align}
where
\begin{align}
\|h_0\|_{L^2(\Gamma)}&=\left\Vert-\gamma_{fs}^\prime(\tilde\psi_0)
-\tilde{\bf u}_{{\boldsymbol\tau}0}\cdot\nabla_{\boldsymbol\tau}\tilde\psi_0
\right\Vert_{L^2(\Gamma)}                       \nonumber
\\
&\le C\left(1+\|\tilde{\bf u}_{{\boldsymbol\tau}0}\|_{L^6(\Gamma)}
\|\nabla_{\boldsymbol\tau}\tilde\psi_0\|_{L^3(\Gamma)}
\right)                                                 \nonumber
\\
&\le C\left(1+\|{\bf u}_0\|_{H^2}
\|\phi_0\|_{H^3}
\right).                                                \nonumber
\end{align}
\vskip2mm
Then, taking the above estimates into consideration, we obtain
\begin{align}
\label{A.1}
&\|H\|_{H^1}^2+\|f\|_{H^1}^2+\|h\|_{L^2(\Gamma)}^2
+\int_{0}^{T}\left(
\|H\|_{H^2}^2+\|f\|_{H^2}^2+\|h\|_{H^2(\Gamma)}^2
\right) {\rm d}t                                              \nonumber
\\
&\quad +\int_{0}^{T}\left(
\|H_t\|_{L^2}^2+ \|f_t\|_{L^2}^2 +\|h_t\|_{L^2(\Gamma)}^2
\right) {\rm d}t                                              \nonumber
\\
&\le C\left(\|{\bf u}_0\|_{H^2}^4+\|\phi_0\|_{H^3}^6+1+K_2^3T+K_1K_2T^{\frac{1}{2}}
\right).
\end{align}
It follows from (\ref{P1}) in Proposition \ref{Prop3.1} and (\ref{A.1}) that
\begin{align}
&\|(\phi,\psi)\|_{Y_2}^2\le\tilde C_1\left(K_2^3T+K_1K_2T^{\frac{1}{2}}+1
\right),                                              \nonumber
\end{align}
where $\tilde C_1=\tilde C_1\left(\|\phi_0\|_{H^3}^6,
\|{\bf u}\|_{H^2}^4,\|\psi_0\|_{H^3(\Gamma)}^4\right)$. Choosing
\begin{align}
K_2=3\tilde C_1                               \nonumber
\end{align}
and $T$ suitable small such that
\begin{align}
K_2^3T<1\;{\rm and}\;K_1K_2T^\frac{1}{2}<1,    \nonumber
\end{align}
then we deduce
\begin{align}
\label{A.2}
\|(\phi,\psi)\|_{Y_2}^2\le K_2.
\end{align}
\vskip2mm
Next, we are going to estimate linearized terms of Navier-Stokes equations (\ref{LNS}).
Calculating directly then we get
\begin{align}
\int_{0}^{T}\|G\|_{H^1}^2 {\rm d}t
&=\int_{0}^{T}\left(\|-\tilde{\bf u}\cdot\nabla\tilde{\bf u}-\nabla\tilde\phi\cdot\Delta\tilde\phi\|_{L^2}^2
+\|\nabla(-\tilde{\bf u}\cdot\nabla\tilde{\bf u}-\nabla\tilde\phi\cdot\Delta\tilde\phi)\|_{L^2}^2
\right) {\rm d}t                                             \nonumber
\\
&\le C\int_{0}^{T}\left(\|\tilde{\bf u}\|_{L^\infty}^2\|\nabla\tilde{\bf u}\|_{L^2}^2+\|\nabla\tilde\phi\|_{L^\infty}^2\|\Delta\phi\|_{L^2}^2
+\|\nabla\tilde{\bf u}\|_{L^4}^4+\|\tilde{\bf u}\|_{L^\infty}^2\|\nabla^2\tilde{\bf u}\|_{L^2}^2
\right) {\rm d}t                                             \nonumber
\\
&\quad +C\int_{0}^{T}\left(\|\nabla^2\tilde\phi\|_{L^6}^2\|\Delta\tilde\phi\|_{L^3}^2
+\|\nabla\tilde\phi\|_{L^\infty}^2\|\nabla\Delta\tilde\phi\|_{L^2}^2
\right) {\rm d}t                                             \nonumber
\\
&\le C\int_{0}^{T}\left(\|\tilde{\bf u}\|_{H^2}^4+\|\nabla\tilde\phi\|_{H^2}^4\right) {\rm d}t \le C(K_1^2+K_2^2)T,                                  \nonumber
\\[1em]
\int_{0}^{T}\|G_t\|_{L^2}^2 {\rm d}t
&\le C\int_{0}^{T}\left(\|\tilde{\bf u}_t\|_{L^6}^2\|\nabla\tilde{\bf u}\|_{L^3}^2+\|\tilde{\bf u}\|_{L^\infty}^2\|\nabla\tilde{\bf u}_t\|_{L^2}^2\right) {\rm d}t
\nonumber
\\
&\quad+C\int_{0}^{T}\left(\|\nabla\tilde\phi_t\|_{L^6}^2\|\Delta\phi\|_{L^3}^2
+\|\nabla\tilde\phi\|_{L^\infty}^2\|\Delta\tilde\phi_t\|_{L^2}^2
\right) {\rm d}t                                               \nonumber
\\
&\le C\int_{0}^{T}\left(K_1\|\tilde{\bf u}_t\|_{H^1}^2+K_2\|\nabla\tilde\phi_t\|_{H^1}^2
\right) {\rm d}t \le C(K_1^2+K_2^2), \nonumber
\end{align}
from which we can infer that
\begin{align}
\|G\|_{L^2}
&=\left\Vert\int_{0}^{t}G_t(x,s){\rm d}s+G(x,0)\right\Vert_{L^2}
\nonumber
\\
&\le \int_{0}^{t}\|G_t\|_{L^2}{\rm d}s+\|G_0\|_{L^2}           \nonumber
\\
&\le C(K_1^2+K_2^2)^\frac{1}{2}T^{\frac{1}{2}}+C\left(\|{\bf u}_0\|_{H^2}^2+\|\phi_0\|_{H^3}^2\right),               \nonumber
\end{align}
where
\begin{align}
\|G_0\|_{L^2}&=\|-\tilde{\bf u}_0\cdot\nabla\tilde{\bf u}_0-\nabla\tilde\phi_0\cdot\Delta\tilde\phi_0\|_{L^2}
\nonumber
\\
&\le C\left(\|\tilde{\bf u}\|_{L^\infty}\|\nabla\tilde{\bf u}_0\|_{L^2}+\|\nabla\tilde\phi_0\|_{L^\infty}\|\Delta\tilde\phi_0\|_{L^2}
\right)                                                  \nonumber
\\
&\le C\left(\|{\bf u}_0\|_{H^2}^2+\|\phi_0\|_{H^3}^2
\right).                                                 \nonumber
\end{align}
By using Lemma \ref{Trace} and (\ref{L5.2}) for $r=s_1=\frac{1}{2}$, $r+\frac{2}{2}=\frac{3}{2}<2=s_2$, we obtain
\begin{align}
\|g\|_{H^\frac{1}{2}(\Gamma)}^2&=\left\Vert\left(-\gamma\Delta_{\boldsymbol\tau}\tilde\psi
+\partial_{\bf n}\tilde\phi+\gamma_{fs}^\prime(\tilde\psi)
\right)\nabla_{\boldsymbol\tau}\tilde\psi\right\Vert_{H^\frac{1}{2}(\Gamma)}^2
\nonumber
\\
&\le\|-\gamma\Delta_{\boldsymbol\tau}\tilde\psi
+\partial_{\bf n}\tilde\phi+\gamma_{fs}^\prime(\tilde\psi)\|_{H^\frac{1}{2}(\Gamma)}^2
\|\nabla_{\boldsymbol\tau}\tilde\psi\|_{H^2(\Gamma)}^2          \nonumber
\\
&\le C\left(\|\gamma\Delta_{\boldsymbol\tau}\tilde\psi\|_{H^\frac{1}{2}(\Gamma)}^2
+\|\partial_{\bf n}\tilde\phi\|_{H^\frac{1}{2}(\Gamma)}^2+1
\right)\|\tilde\psi\|_{H^3(\Gamma)}^2                         \nonumber
\\
&\le C\left(\|\Delta\tilde\phi\|_{H^1}^2
+\|\nabla\tilde\phi\|_{H^1}^2+1
\right)\|\tilde\psi\|_{H^3(\Gamma)}^2                         \nonumber
\\
&\le CK_2^2.                                                          \nonumber
\end{align}
Similarly, combining (\ref{product}) for $k=1$ with (\ref{L5.1}) for $r=s_1=s_2=\frac{3}{2}>\frac{2}{2}$, it follows from Lemma \ref{Trace} that
\begin{align}
\int_{0}^{T}\|g\|_{H^{\frac{3}{2}}(\Gamma)}^2{\rm d}t
&\le C\int_{0}^{T}\left(
\|\gamma\Delta_{\boldsymbol\tau}\tilde\psi\|_{H^{\frac{3}{2}}(\Gamma)}^2
+\|\partial_{\bf n}\tilde\phi\|_{H^{\frac{3}{2}}(\Gamma)}^2
+\|\gamma_{fs}^\prime(\tilde\psi)\|_{H^{\frac{3}{2}}(\Gamma)}^2
\right)
\|\nabla_{\boldsymbol\tau}\tilde\psi\|_{H^{\frac{3}{2}}(\Gamma)}^2{\rm d}t
\nonumber
\\
&\le C\int_{0}^{T}\left(\|\Delta_{\boldsymbol\tau}\tilde\psi\|_{H^1(\Gamma)}
\|\Delta_{\boldsymbol\tau}\tilde\psi\|_{H^2(\Gamma)}
+\|\nabla\tilde\phi\|_{H^2}^2+1
\right)\|\nabla\tilde\phi\|_{H^2}^2{\rm d}t                  \nonumber
\\
&\le C\int_{0}^{T}K_2^\frac{3}{2}\|\tilde\psi\|_{H^4(\Gamma)}{\rm d}t+CK_2^2T
\nonumber
\\
&\le CK_2^\frac{3}{2}\left(\int_{0}^{T}1^2 {\rm d}t
\right)^{\frac{1}{2}} \left(\int_{0}^{T}\|\tilde\psi\|_{H^4(\Gamma)}^2 {\rm d}t
\right)^{\frac{1}{2}}+CK_2^2T  \le CK_2^2T^\frac{1}{2}.   \nonumber
\end{align}
Calculating directly shows
\begin{align}
g_t=\left(-\gamma\Delta_{\boldsymbol\tau}\tilde\psi_t+\partial_{\bf n}\tilde\phi_t
+\gamma_{fs}^{(2)}(\tilde\psi)\tilde\psi_t
\right)\nabla_{\boldsymbol\tau}\tilde\psi+\left(-\gamma\Delta_{\boldsymbol\tau}\tilde\psi+\partial_{\bf n}\tilde\phi
+\gamma_{fs}^\prime(\tilde\psi)
\right)\nabla_{\boldsymbol\tau}\tilde\psi_t,                        \nonumber
\end{align}
which together with Lemma \ref{Trace} yields
\begin{align}
\int_{0}^{T}\|g_t\|_{L^2(\Gamma)}^2 {\rm d}t
&\le \int_{0}^{T}\left(\|\gamma\Delta_{\boldsymbol\tau}\tilde\psi_t\|_{L^2(\Gamma)}^2
+\|\partial_{\bf n}\tilde\phi_t\|_{L^2(\Gamma)}^2
+C\|\tilde\psi_t\|_{L^2(\Gamma)}^2
\right)\|\nabla_{\boldsymbol\tau}\tilde\psi\|_{L^\infty(\Gamma)}^2 {\rm d}t
\nonumber
\\
&\quad +\int_{0}^{T}\left(\|\gamma\Delta_{\boldsymbol\tau}\tilde\psi\|_{L^4(\Gamma)}^2
+\|\partial_{\bf n}\tilde\phi\|_{L^4(\Gamma)}^2
+C\|\tilde\psi\|_{L^4(\Gamma)}^2
\right)\|\nabla_{\boldsymbol\tau}\tilde\psi_t\|_{L^4(\Gamma)}^2 {\rm d}t
\nonumber
\\
&\le C\int_{0}^{T}\left(\|\nabla_{\boldsymbol\tau}^2\tilde\psi_t\|_{L^2(\Gamma)}^2
+\|\nabla\tilde\phi_t\|_{H^1}^2
+C\|\tilde\psi_t\|_{L^2(\Gamma)}^2
\right)\|\nabla_{\boldsymbol\tau}\tilde\psi\|_{H^2(\Gamma)}^2 {\rm d}t
\nonumber
\\
&\quad +C\int_{0}^{T}\left(\|\Delta_{\boldsymbol\tau}\tilde\psi\|_{H^1(\Gamma)}^2
+\|\nabla\tilde\phi\|_{H^2}^2
+C\|\tilde\psi\|_{H^1(\Gamma)}^2
\right)\|\nabla_{\boldsymbol\tau}\tilde\psi_t\|_{H^1(\Gamma)}^2 {\rm d}t
\nonumber
\\
&\le C\int_{0}^{T}K_2\left(\|\tilde\psi_t\|_{H^2(\Gamma)}^2
+\|\nabla\tilde\phi_t\|_{H^1}^2    \nonumber
\right) {\rm d}t \le CK_2^2. \nonumber
\end{align}
\vskip2mm
Now, putting all the above estimates together, we have
\begin{align}
&\|G\|_{L^2}^2+\|g\|_{H^\frac{1}{2}(\Gamma)}^2+\int_{0}^{T}\left[
K_1^{-2}K_2^{-2}\left(\|G_t\|_{L^2}^2+\|g_t\|_{L^2(\Gamma)}^2
\right)+\|G\|_{H^1}^2+\|g\|_{H^{\frac{3}{2}}(\Gamma)}^2
\right]{\rm d}t                                              \nonumber
\\
&\le C\left(K_1^2T+K_2^2T^\frac{1}{2}+K_2^2+\|{\bf u}_0\|_{H^2}^4+\|\phi_0\|_{H^3}^4
\right),             \nonumber
\end{align}
which together with (\ref{P2}) in Proposition \ref{Prop3.2}, where we choosing $\epsilon=1$ yields
\begin{align}
\|{\bf u}\|_{Y_1}^2\le\tilde C_2e^{K_1^2K_2^2T}\left(1+K_1^2T+K_2^2T^\frac{1}{2}
\right)+\tilde C_2e^{K_1^2K_2^2T}K_2^2,                \nonumber
\end{align}
where $\tilde C_2=\tilde C_2\left(\|{\bf u}_0\|_{H^2}^4,\|\phi_0\|_{H^3}^4
\right)$. Choosing
\begin{align}
K_1>3\tilde C_2e+9\tilde C_2e\tilde C_1^2            \nonumber
\end{align}
and $T$ suitable small such that
\begin{align}
K_1^2K_2^2T<1,\;K_1^2T<1,\;{\rm and}\;K_2^2T^\frac{1}{2}<1.    \nonumber
\end{align}
Then, we deduce that
\begin{align}
\label{A.4}
\|{\bf u}\|_{Y_1}^2\le K_1.
\end{align}
Therefore, choosing $K_1$, $K_2$ and $T$ in this way, defining $T_1^*={\rm min}\{K_2^{-4},(K_1K_2)^{-2}\}$. From (\ref{A.2}) and (\ref{A.4}), for every $T<T_1^*$, we conclude that $\Lambda$ is a self-mapping, which maps from $X(0,T;K_1,K_2)$ to itself.
\vskip2mm
And then, we shall prove that $\Lambda$ is a contraction mapping with respect to the norm $\|\cdot\|_X$. Let $(\tilde{\bf u}_i,\tilde\phi_i,\tilde\psi_i)\in X(0,T;K_1,K_2), i=1,2$, and $\Lambda(\tilde{\bf u}_i,\tilde\phi_i,\tilde\psi_i)=({\bf u}_i,\phi_i,\psi_i)$. Then let ${\bf u}={\bf u}_1-{\bf u}_2$, $\phi=\phi_1-\phi_2$, $\psi=\psi_1-\psi_2$, $\mu=\mu_1-\mu_2$, $p=p_1-p_2$, $\tilde{\bf u}=\tilde{\bf u}_1-\tilde{\bf u}_2$, $\tilde\phi=\tilde\phi_1-\tilde\phi_2$, $\tilde\psi=\tilde\psi_1-\tilde\psi_2$, $\tilde\mu=\tilde\mu_1-\tilde\mu_2$ and $\tilde p=\tilde p_1-\tilde p_2$ satisfy
\begin{align*}
\begin{cases}
\phi_t={\bar{\mu}}-\mu+\bar H,&\rm in\quad\Omega\times(0,T),
\\
\mu=-\Delta\phi+\bar f,&\rm in\quad\Omega\times(0,T),
\\
\psi_t-\gamma\Delta_{\boldsymbol {\tau}}\psi=-\partial_{\bf n}\phi+\bar h,&\rm on\quad\Gamma\times(0,T),
\\
\phi\big|_\Gamma=\psi,&\rm in\quad(0,T),
\\
\phi\big|_{t=0}=0,&\rm in\quad\Omega,
\\
\psi\big|_{t=0}=0,&\rm on\quad\Gamma,
\end{cases}
\end{align*}
and
\begin{align*}
\begin{cases}
{\bf u}_t-{\rm div}{\mathbb S}({\bf u})+\nabla p=\bar G,&\rm in\quad\Omega\times(0,T),
\\
{\rm div}{\bf u}=0,&\rm in\quad\Omega\times(0,T),
\\
{\bf u}\cdot{\bf n}=0,&\rm on\quad\Gamma\times(0,T),
\\
\beta{\bf u}_{\boldsymbol\tau}+({\mathbb S}({\bf u})\cdot{\bf n})_{\boldsymbol{\tau}}=\bar g,&\rm on\quad\Gamma\times(0,T),
\\
{\bf u}\big|_{t=0}=0,&\rm in\quad\Omega,
\end{cases}
\end{align*}
where
\begin{align}
\bar H&=-\tilde{\bf u}_1\cdot\nabla\tilde\phi_1-(-\tilde{\bf u}_2\cdot\nabla\tilde\phi_2)                             =-\tilde{\bf u}\cdot\nabla\tilde\phi_1-\tilde{\bf u}_2\cdot\nabla\tilde\phi.
\nonumber
\\
\bar f&=(\tilde\phi_1^3+\tilde\phi_1)-(\tilde\phi_2^3+\tilde\phi_2)
=\tilde\phi(\tilde\phi_1^2+\tilde\phi_1\tilde\phi_2+\tilde\phi_2^2-1).
\nonumber
\\
\bar h&=-\gamma_{fs}^\prime(\tilde\psi_1)-\tilde{\bf u}_{1{\boldsymbol\tau}}\cdot\nabla_{\boldsymbol\tau}\tilde\psi_1
-(-\gamma_{fs}^\prime(\tilde\psi_2)-\tilde{\bf u}_{2{\boldsymbol\tau}}\cdot\nabla_{\boldsymbol\tau}\tilde\psi_2)
\nonumber
\\
&=\gamma_{fs}^\prime\left(\theta\tilde\psi_1+(1-\theta)\tilde\psi_2\right)\tilde\psi-\tilde{\bf u}_{\boldsymbol\tau}\cdot\nabla_{\boldsymbol\tau}\tilde\psi_1-\tilde{\bf u}_{2{\boldsymbol\tau}}\cdot\nabla_{\boldsymbol\tau}\tilde\psi,\quad\theta\in(0,1).
\nonumber
\\
\bar G&=-\tilde{\bf u}_1\cdot\nabla\tilde{\bf u}_1-\rm div(\nabla\tilde\phi_1\otimes\nabla\tilde\phi_1)-\left(-\tilde{\bf u}_2\cdot\nabla\tilde{\bf u}_2-\rm div(\nabla\tilde\phi_2\otimes\nabla\tilde\phi_2)
\right)                                                  \nonumber
\\
&=-\tilde{\bf u}\cdot\nabla\tilde{\bf u}_1-\tilde{\bf u}_2\cdot\nabla\tilde{\bf u}-\nabla\tilde\phi\cdot\Delta\tilde\phi_1
-\nabla\tilde\phi_2\cdot\Delta\tilde\phi.                \nonumber
\\
\bar g&=\left[-\gamma\Delta_{\boldsymbol\tau}\tilde\psi_1+\partial_{\bf n}\tilde\phi_1
+\gamma_{fs}^\prime(\tilde\psi_1)
\right]\nabla_{\boldsymbol\tau}\tilde\psi_1-\left[
-\gamma\Delta_{\boldsymbol\tau}\tilde\psi_2+\partial_{\bf n}\tilde\phi_2
+\gamma_{fs}^\prime(\tilde\psi_2)
\right]\nabla_{\boldsymbol\tau}\tilde\psi_2              \nonumber
\\
&=-\gamma\Delta_{\boldsymbol\tau}\tilde\psi\nabla_{\boldsymbol\tau}\tilde\psi_1
-\gamma\Delta_{\boldsymbol\tau}\tilde\psi_2\nabla_{\boldsymbol\tau}\tilde\psi
+\partial_{\bf n}\tilde\phi\nabla_{\boldsymbol\tau}\tilde\psi_1
+\partial_{\bf n}\tilde\phi_2\nabla_{\boldsymbol\tau}\tilde\psi  \nonumber
\\
&\quad+\gamma_{fs}^\prime\left(\theta\tilde\psi_1+(1-\theta)\tilde\psi_2\right)\tilde\psi
\nabla_{\boldsymbol\tau}\tilde\psi_1+\gamma_{fs}^\prime(\tilde\psi_2)
\nabla_{\boldsymbol\tau}\tilde\psi,\quad\theta\in(0,1).                       \nonumber
\end{align}
Moreover, there are some key properties
\begin{align}
\label{A.5}
\langle\phi\rangle=\langle\bar H\rangle=\bar H(x,0)=\bar h(x,0)=\bar f(x,0)=\bar G(x,0)=\bar g(x,0)=0.
\end{align}
\vskip2mm
Then, we should establish some necessary a priori estimates to complete the proof. Direct calculations show that
\begin{align}
\int_{0}^{T}\|\bar H\|_{H^2}^2 {\rm d}t
&\le C\int_{0}^{T}\left(\|\tilde{\bf u}\|_{L^\infty}^2\|\nabla\tilde\phi_1\|_{L^2}^2+\|\tilde{\bf u}_2\|_{L^\infty}^2\|\nabla\tilde\phi\|_{L^2}^2+\|\nabla\tilde{\bf u}\|_{L^6}^2\|\nabla\tilde\phi_1\|_{L^3}^2
\right) {\rm d}t                                              \nonumber
\\
&\quad +C\int_{0}^{T}\left(
\|\tilde{\bf u}\|_{L^\infty}^2\|\nabla^2\tilde\phi_1\|_{L^2}^2
+\|\nabla\tilde{\bf u}_2\|_{L^6}^2\|\nabla\tilde\phi\|_{L^3}^2+\|\tilde{\bf u}_2\|_{L^\infty}^2\|\nabla^2\tilde\phi\|_{L^2}^2
\right) {\rm d}t                                              \nonumber
\\
&\quad +C\int_{0}^{T}\left(
\|\nabla^2\tilde{\bf u}\|_{L^2}^2\|\nabla\tilde\phi_1\|_{L^\infty}^2
+\|\nabla\tilde{\bf u}\|_{L^6}^2\|\nabla^2\tilde\phi_1\|_{L^3}^2
+\|\tilde{\bf u}\|_{L^\infty}^2\|\nabla^3\tilde\phi_1\|_{L^2}^2
\right) {\rm d}t                                              \nonumber
\\
&\quad +C\int_{0}^{T}\left( \|\nabla^2\tilde{\bf u}_2\|_{L^2}^2\|\nabla\tilde\phi\|_{L^\infty}^2+\|\nabla\tilde{\bf u}_2\|_{L^3}^2\|\nabla^2\tilde\phi\|_{L^6}^2+\|\tilde{\bf u}_2\|_{L^\infty}^2\|\nabla^3\tilde\phi\|_{L^2}^2
\right) {\rm d}t                                              \nonumber
\\
&\le C(K_1+K_2)T\|(\tilde{\bf u}, \tilde\phi, \tilde\psi)\|_X^2.
\nonumber
\\[1em]
\int_{0}^{T}\|\bar H_t\|_{L^2}^2 {\rm d}t
&\le C\int_{0}^{T}\left(\|\tilde{\bf u}_t\|_{L^2}^2\|\nabla\tilde\phi_1\|_{L^\infty}^2+\|\tilde{\bf u}\|_{L^\infty}^2\|\nabla\tilde\phi_{1t}\|_{L^2}^2
\right) {\rm d}t                                              \nonumber
\\
&\quad +\int_{0}^{T}\left(\|\tilde{\bf u}_{2t}\|_{L^2}^2\|\nabla\tilde\phi\|_{L^\infty}^2+\|\tilde{\bf u}_2\|_{L^\infty}^2\|\nabla\tilde\phi_t\|_{L^2}^2
\right) {\rm d}t                                              \nonumber
\\
&\le C(K_1+K_2)T\|(\tilde{\bf u}, \tilde\phi, \tilde\psi)\|_X^2.
\nonumber
\\[1em]
\int_{0}^{T}\|\nabla\bar H_t\|_{L^2}^2 {\rm d}t
&\le C\int_{0}^{T}\left(\|\nabla\tilde{\bf u}_t\|_{L^2}^2\|\nabla\tilde\phi_1\|_{L^\infty}^2+\|\tilde{\bf u}_t\|_{L^6}^2\|\nabla^2\tilde\phi_1\|_{L^3}^2+\|\tilde{\bf u}\|_{L^6}^2\|\nabla\tilde\phi_{1t}\|_{L^3}^2
\right) {\rm d}t                                              \nonumber
\\
&\quad +\int_{0}^{T}\left(\|\tilde{\bf u}\|_{L^\infty}^2\|\nabla^2\tilde\phi_{1t}\|_{L^2}^2+\|\nabla\tilde{\bf u}_{2t}\|_{L^2}^2\|\nabla\tilde\phi\|_{L^\infty}^2+\|\tilde{\bf u}_{2t}\|_{L^6}^2\|\nabla^2\tilde\phi\|_{L^3}^2
\right) {\rm d}t                                              \nonumber
\\
&\quad+\int_{0}^{T}\left(\|\tilde{\bf u}_2\|_{L^6}^2\|\nabla\tilde\phi_t\|_{L^3}^2+\|\tilde{\bf u}_2\|_{L^\infty}^2\|\nabla^2\tilde\phi_t\|_{L^2}^2
\right) {\rm d}t                                              \nonumber
\\
&\le C(K_1+K_2)\|(\tilde{\bf u}, \tilde\phi, \tilde\psi)\|_X^2,
\nonumber
\end{align}
which together with (\ref{A.5}) implies
\begin{align}
\|\bar H\|_{H^1}
&=\left\Vert\int_{0}^{t}\bar H_t(x,s){\rm d}s\right\Vert_{H^1}
\le \int_{0}^{t}\|\bar H_t\|_{H^1}{\rm d}s
\le T^{\frac{1}{2}}\|\bar H_t\|_{L^2(0,T;H^1)}          \nonumber
\\
&\le C(K_1+K_2)^{\frac{1}{2}}T^{\frac{1}{2}}\|(\tilde{\bf u}, \tilde\phi, \tilde\psi)\|_X.
\nonumber
\end{align}
Similarly, we obtain
\begin{equation*}
\int_{0}^{T}\left(\|\bar f\|_{H^2}^2+\|\bar f_t\|_{H^1}^2
\right) {\rm d}t
\le CK_2^2 T\|(\tilde{\bf u}, \tilde\phi, \tilde\psi)\|_X^2,
\end{equation*}
which yields
\begin{align}
\|\bar f\|_{H^1}
&=\left\Vert\int_{0}^{t}f_t(\tilde\phi(x,s)){\rm d}s\right\Vert_{H^1}
\le \int_{0}^{t}\|f_t(\tilde\phi)\|_{H^1}{\rm d}s \nonumber
\\
&\le C T^{\frac{1}{2}}\|f_t(\tilde\phi)\|_{L^2(0,T;H^1)}\le CK_2T\|(\tilde{\bf u}, \tilde\phi, \tilde\psi)\|_X. \nonumber
\end{align}
It follows from Lemma \ref{Trace} that
\begin{align}
\int_{0}^{T}\|\bar h\|_{H^1(\Gamma)}^2 {\rm d}t
&\le C\int_{0}^{T}\left(\|\tilde\psi\|_{L^2(\Gamma)}^2+\|\tilde{\bf u}_{\boldsymbol\tau}\|_{L^2(\Gamma)}^2
\|\nabla_{\boldsymbol\tau}\tilde\psi_1\|_{L^\infty(\Gamma)}^2
\right) {\rm d}t                                           \nonumber
\\
&\quad +C\int_{0}^{T}\left(\|\tilde{\bf u}_{2{\boldsymbol\tau}}\|_{L^2(\Gamma)}^2
\|\nabla_{\boldsymbol\tau}\tilde\psi\|_{L^\infty(\Gamma)}^2+\|\nabla\tilde\psi\|_{L^2(\Gamma)}^2
\right) {\rm d}t                                           \nonumber
\\
&\quad +C\int_{0}^{T}\left(\|\nabla_{\boldsymbol\tau}\tilde{\bf u}_{\boldsymbol\tau}\|_{L^2(\Gamma)}^2
\|\nabla_{\boldsymbol\tau}\tilde\psi_1\|_{L^\infty(\Gamma)}^2+\|\tilde{\bf u}_{\boldsymbol\tau}\|_{L^3(\Gamma)}^2
\|\nabla^2_{\boldsymbol\tau}\tilde\psi_1\|_{L^6(\Gamma)}^2
\right) {\rm d}t                                           \nonumber
\\
&\quad+C\int_{0}^{T}\left(
\|\nabla_{\boldsymbol\tau}\tilde{\bf u}_{2{\boldsymbol\tau}}\|_{L^2(\Gamma)}^2
\|\nabla_{\boldsymbol\tau}\tilde\psi\|_{L^\infty(\Gamma)}^2+\|\tilde{\bf u}_{2{\boldsymbol\tau}}\|_{L^3(\Gamma)}^2
\|\nabla^2_{\boldsymbol\tau}\tilde\psi\|_{L^6(\Gamma)}^2\right) {\rm d}t \nonumber
\\
&\le C(K_1+K_2)T\|(\tilde{\bf u}, \tilde\phi, \tilde\psi)\|_X^2.
\nonumber
\\[1em]
\int_{0}^{T}\|\nabla^2_{\boldsymbol\tau}\bar h\|_{L^2(\Gamma)}^2 {\rm d}t &\le C\int_{0}^{T}\left(
\|\nabla^2_{\boldsymbol\tau}\tilde\psi\|_{L^2(\Gamma)}^2
+\|\nabla^2_{\boldsymbol\tau}\tilde{\bf u}_{\boldsymbol\tau}\|_{L^2(\Gamma)}^2
\|\nabla_{\boldsymbol\tau}\tilde\psi_1\|_{L^\infty(\Gamma)}^2
\right) {\rm d}t                                           \nonumber
\\
&\quad +C\int_{0}^{T}\left(
\|\nabla_{\boldsymbol\tau}\tilde{\bf u}_{\boldsymbol\tau}\|_{L^3(\Gamma)}^2
\|\nabla^2_{\boldsymbol\tau}\tilde\psi_1\|_{L^6(\Gamma)}^2
+\|\tilde{\bf u}_{\boldsymbol\tau}\|_{L^\infty(\Gamma)}^2
\|\nabla^3_{\boldsymbol\tau}\tilde\psi_1\|_{L^2(\Gamma)}^2
\right) {\rm d}t                                           \nonumber
\\
&\quad +C\int_{0}^{T}\left(\|\tilde{\bf u}_{2{\boldsymbol\tau}}\|_{L^\infty(\Gamma)}^2
\|\nabla^3_{\boldsymbol\tau}\tilde\psi\|_{L^2(\Gamma)}^2+\|\nabla_{\boldsymbol\tau}\tilde{\bf u}_{2{\boldsymbol\tau}}\|_{L^3(\Gamma)}^2
\|\nabla^2_{\boldsymbol\tau}\tilde\psi\|_{L^6(\Gamma)}^2
\right) {\rm d}t                                           \nonumber
\\
&\quad +C\int_{0}^{T}\|\nabla^2_{\boldsymbol\tau}\tilde{\bf u}_{2{\boldsymbol\tau}}\|_{L^2(\Gamma)}^2
\|\nabla_{\boldsymbol\tau}\tilde\psi\|_{L^\infty(\Gamma)}^2 {\rm d}t \nonumber
\\
&\le C\int_{0}^{T}\left(
\|\nabla^2_{\boldsymbol\tau}\tilde\psi\|_{L^2(\Gamma)}^2
+K_2\|\tilde{\bf u}\|_{H^2}\|\tilde{\bf u}\|_{H^3}
+K_1^{\frac{1}{2}}\|\tilde{\bf u}_2\|_{H^3}\|\nabla_{\boldsymbol\tau}\tilde\psi\|_{H^2(\Gamma)}^2
\right) {\rm d}t                                            \nonumber
\\
&\le C(K_1+K_2)T^{\frac{1}{2}}\|(\tilde{\bf u}, \tilde\phi, \tilde\psi)\|_X^2,
\nonumber
\end{align}
where we have used (\ref{L21}) for $k=2$. Similarly, due to (\ref{L21}) for $k=0,2$, we get
\begin{align}
\int_{0}^{T}\|\bar h_t\|_{L^2(\Gamma)}^2 {\rm d}t
&\le C\int_{0}^{T}\left(\|\tilde\psi_t\|_{L^2(\Gamma)}^2
+\|\tilde{\bf u}_{{\boldsymbol\tau}t}\|_{L^2(\Gamma)}^2\|\nabla
_{\boldsymbol\tau}\tilde\psi_1\|
_{L^\infty(\Gamma)}^2
\right) {\rm d}t                                             \nonumber
\\
&\quad +C\int_{0}^{T}\left(\|\tilde{\bf u}_{\boldsymbol\tau}\|_{L^\infty(\Gamma)}^2\|\nabla
_{\boldsymbol\tau}\tilde\psi_{1t}\|
_{L^2(\Gamma)}^2+\|\tilde{\bf u}_{2{\boldsymbol\tau}t}\|_{L^2(\Gamma)}^2\|\nabla
_{\boldsymbol\tau}\tilde\psi\|
_{L^\infty(\Gamma)}^2\right) {\rm d}t \nonumber
\\
&\quad+C\int_{0}^{T}\|\tilde{\bf u}_{2{\boldsymbol\tau}}\|_{L^\infty(\Gamma)}^2\|\nabla
_{\boldsymbol\tau}\tilde\psi_t\|
_{L^2(\Gamma)}^2 {\rm d}t                  \nonumber
\\
&\le C\int_{0}^{T}\left(
\|\tilde\psi_t\|_{L^2(\Gamma)}^2+K_2\|\tilde{\bf u}_t\|_{L^2}\|\tilde{\bf u}_t\|_{H^1}+K_2\|\tilde{\bf u}\|_{H^2}\|\tilde{\bf u}\|_{H^3}
\right) {\rm d}t                                              \nonumber
\\
&\quad +C\int_{0}^{T}\left(K_1^{\frac{1}{2}}\|\tilde{\bf u}_{2t}\|_{H^1}
\|\nabla_{\boldsymbol\tau}\tilde\psi\|_{H^2(\Gamma)}^2
+K_1^{\frac{1}{2}}\|\tilde{\bf u}_2\|_{H^3}
\|\nabla_{\boldsymbol\tau}\tilde\psi_t\|_{L^2(\Gamma)}^2
\right) {\rm d}t                                             \nonumber
\\
&\le C(K_1+K_2)T^{\frac{1}{2}}\|(\tilde{\bf u}, \tilde\phi, \tilde\psi)\|_X^2,
\nonumber
\end{align}
Go a step further, we obtain
\begin{align}
\|\bar h\|_{L^2(\Gamma)}
&=\left\Vert\int_{0}^{t}\bar h_t(x,s){\rm d}s\right\Vert_{L^2(\Gamma)}
\le \int_{0}^{t}\|\bar h_t\|_{L^2(\Gamma)}{\rm d}s
\le T^{\frac{1}{2}}\|\bar h_t\|_{L^2(0,T;L^2(\Gamma)}
\nonumber
\\
&\le C(K_1+K_2)^{\frac{1}{2}}T^{\frac{3}{4}}\|(\tilde{\bf u}, \tilde\phi, \tilde\psi)\|_X.                                       \nonumber
\end{align}
Combining with the above estimates, we draw a conclusion
\begin{align}
&\|\bar H\|_{H^1}^2+\|\bar f\|_{H^1}^2+\|\bar h\|_{L^2(\Gamma)}^2
+\int_{0}^{T}\left(
\|\bar H\|_{H^2}^2+\|\bar f\|_{H^2}^2+\|\bar h\|_{H^2(\Gamma)}^2
\right) {\rm d}t                                             \nonumber
\\
&\quad +\int_{0}^{T}\left(
\|\bar H_t\|_{L^2}^2+ \|\bar f_t\|_{L^2}^2+\|\bar h_t\|_{L^2(\Gamma)}^2
\right) {\rm d}t                                             \nonumber
\\
&\le C(K_1+K_2^2)T^{\frac{1}{2}}\|(\tilde{\bf u}, \tilde\phi, \tilde\psi)\|_X^2.
\nonumber
\end{align}
Recalling (\ref{P1}) in proposition \ref{Prop3.1}, replacing $(H,f,h)$ by $(\bar H,\bar f,\bar h)$, thus we obtain
\begin{align}
\label{A.9}
\|(\phi,\psi)\|_{Y_2}^2\le C\left(K_1+K_2^2
\right)T^\frac{1}{2}\|(\tilde{\bf u},\tilde\phi,\tilde\psi)\|_X^2.
\end{align}
\vskip2mm
Proceeding in the same way for estimating $\|\bar G\|_{L^\infty(0,T;L^2)\cap L^2(0,T;H^1)}^2$, $\|\bar G_t\|_{L^2(0,T;L^2)}^2$,\\ $\|\bar g\|_{L^\infty(0,T;H^\frac{1}{2}(\Gamma))\cap L^2(0,T;H^{\frac{3}{2}}(\Gamma))}^2$ and $\|\bar g_t\|_{L^2(0,T;L^2(\Gamma))}^2$, we obtain
\begin{align*}
\int_{0}^{T}\|\bar G\|_{H^1}^2 {\rm d}t \le C(K_1+K_2)T\|(\tilde{\bf u}, \tilde\phi, \tilde\psi)\|_X^2,
\\[1em]
\int_{0}^{T}\|\bar G_t\|_{L^2}^2 {\rm d}t \le C(K_1+K_2)\|(\tilde{\bf u}, \tilde\phi, \tilde\psi)\|_X^2,
\end{align*}
which implies
\begin{align}
\|\bar G\|_{L^2}
=\left\Vert\int_{0}^{t}\bar G_t(x,s){\rm d}s\right\Vert_{L^2}
&\le \int_{0}^{t}\|\bar G_t\|_{L^2}{\rm d}s
\le T^{\frac{1}{2}}\|\bar G_t\|_{L^2(0,T;L^2)} \nonumber
\\
&\le C(K_1+K_2)^{\frac{1}{2}}T^{\frac{1}{2}}\|(\tilde{\bf u}, \tilde\phi, \tilde\psi)\|_X.
\end{align}
It follows from (\ref{L21}), (\ref{L5.1}) and Lemma \ref{Trace} that
\begin{align}
\int_{0}^{T}\|\bar g\|_{H^{\frac{3}{2}}(\Gamma)}^2 {\rm d}t
&\le C\int_{0}^{T}\left(
\|\gamma\Delta_{\boldsymbol\tau}\tilde\psi\|_{H^{\frac{3}{2}}(\Gamma)}^2
\|\nabla_{\boldsymbol\tau}\tilde\psi_1\|_{H^{\frac{3}{2}}(\Gamma)}^2
+\|\gamma\Delta_{\boldsymbol\tau}\tilde\psi_2\|_{H^{\frac{3}{2}}(\Gamma)}^2
\|\nabla_{\boldsymbol\tau}\tilde\psi\|_{H^{\frac{3}{2}}(\Gamma)}^2
\right) {\rm d}t                                            \nonumber
\\
&\quad+C\int_{0}^{T}\left(\|\partial_{\bf n}\tilde\phi\|_{H^{\frac{3}{2}}(\Gamma)}^2
\|\nabla_{\boldsymbol\tau}\tilde\psi_1\|_{H^{\frac{3}{2}}(\Gamma)}^2
+\|\partial{\bf_n}\tilde\phi_2\|_{H^{\frac{3}{2}}(\Gamma)}^2
\|\nabla_{\boldsymbol\tau}\tilde\psi\|_{H^{\frac{3}{2}}(\Gamma)}^2
\right) {\rm d}t                                              \nonumber
\\
&\quad+C\int_{0}^{T}\left(\|\tilde\psi\|_{H^{\frac{3}{2}}(\Gamma)}^2
\|\nabla_{\boldsymbol\tau}\tilde\psi_1\|_{H^{\frac{3}{2}}(\Gamma)}^2
+\|\gamma_{fs}^\prime(\tilde\psi_2)\|_{H^{\frac{3}{2}}(\Gamma)}^2
\|\nabla_{\boldsymbol\tau}\tilde\psi\|_{H^{\frac{3}{2}}(\Gamma)}^2
\right) {\rm d}t                                                \nonumber
\\
&\le C\int_{0}^{T}
\|\Delta_{\boldsymbol\tau}\tilde\psi\|_{H^1(\Gamma)}
\|\Delta_{\boldsymbol\tau}\tilde\psi\|_{H^2(\Gamma)}
\|\nabla\tilde\phi_1\|_{H^2}^2 {\rm d}t           \nonumber
\\
&\quad+C\int_{0}^{T}\|\Delta_{\boldsymbol\tau}\tilde\psi_2\|_{H^1(\Gamma)}
\|\Delta_{\boldsymbol\tau}\tilde\psi_2\|_{H^2(\Gamma)}
\|\nabla\tilde\phi\|_{H^2}^2 {\rm d}t  \nonumber
\\
&\quad+C\int_{0}^{T}\left(
\|\nabla\tilde\phi\|_{H^2}^2\|\nabla\tilde\phi_1\|_{H^2}^2
+\|\nabla\tilde\phi_2\|_{H^2}^2\|\nabla\tilde\phi\|_{H^2}^2
\right) {\rm d}t                                             \nonumber
\\
&\quad+C\int_{0}^{T}\left(\|\tilde\phi\|_{H^2}^2
\|\nabla\tilde\phi_1\|_{H^2}^2+C\|\nabla\tilde\phi\|_{H^2}^2\right) {\rm d}t \nonumber
\\
&\le CK_2\|(\tilde{\bf u}, \tilde\phi, \tilde\psi)\|_X\int_{0}^{T}
\|\Delta_{\boldsymbol\tau}\tilde\psi\|_{H^2(\Gamma)} {\rm d}t
+CK_2^\frac{1}{2}\|(\tilde{\bf u}, \tilde\phi, \tilde\psi)\|_X^2
\int_{0}^{T}\|\Delta_{\boldsymbol\tau}\tilde\psi_2\|_{H^2(\Gamma)} {\rm d}t
\nonumber
\\
&\quad+ C(K_2+1)T\|(\tilde{\bf u}, \tilde\phi, \tilde\psi)\|_X^2
\nonumber
\\
&\le C(K_2T^\frac{1}{2}+1)\|(\tilde{\bf u}, \tilde\phi, \tilde\psi)\|_X^2.
\nonumber
\\[1em]
\int_{0}^{T}\|\bar g_t\|_{L^2(\Gamma)}^2 {\rm d}t
&\le C\int_{0}^{T}\left(\|\gamma\Delta_{\boldsymbol\tau}\tilde\psi_t\|_{L^2(\Gamma)}^2
\|\nabla_{\boldsymbol\tau}\tilde\psi_1\|_{L^\infty(\Gamma)}^2
+\|\gamma\Delta_{\boldsymbol\tau}\tilde\psi\|_{L^\infty(\Gamma)}^2
\|\nabla_{\boldsymbol\tau}\tilde\psi_{1t}\|_{L^2(\Gamma)}^2
\right) {\rm d}t                                            \nonumber
\\
&\quad +C\int_{0}^{T}\left(\|\gamma\Delta_{\boldsymbol\tau}\tilde\psi_{2t}\|_{L^2(\Gamma)}^2
\|\nabla_{\boldsymbol\tau}\tilde\psi\|_{L^\infty(\Gamma)}^2
+\|\gamma\Delta_{\boldsymbol\tau}\tilde\psi_2\|_{L^\infty(\Gamma)}^2
\|\nabla_{\boldsymbol\tau}\tilde\psi_t\|_{L^2(\Gamma)}^2
\right) {\rm d}t                                            \nonumber
\\
&\quad+C\int_{0}^{T}\left(\|\partial_{\bf n}\tilde\phi_t\|_{L^2(\Gamma)}^2
\|\nabla_{\boldsymbol\tau}\tilde\psi_1\|_{L^\infty(\Gamma)}^2
+\|\partial_{\bf n}\tilde\phi\|_{L^4(\Gamma)}^2
\|\nabla_{\boldsymbol\tau}\tilde\psi_{1t}\|_{L^4(\Gamma)}^2
\right) {\rm d}t                                             \nonumber
\\
&\quad+C\int_{0}^{T}\left(\|\partial_{\bf n}\tilde\phi_{2t}\|_{L^2(\Gamma)}^2
\|\nabla_{\boldsymbol\tau}\tilde\psi\|_{L^\infty(\Gamma)}^2
+\|\partial_{\bf n}\tilde\phi_2\|_{L^4(\Gamma)}^2
\|\nabla_{\boldsymbol\tau}\tilde\psi_t\|_{L^4(\Gamma)}^2
\right) {\rm d}t                                             \nonumber
\\
&\quad +C\int_{0}^{T}
\left(\|\tilde\psi_t\|_{L^2(\Gamma)}^2
\|\nabla_{\boldsymbol\tau}\tilde\psi_1\|_{L^\infty(\Gamma)}^2
+\|\tilde\psi\|_{L^\infty(\Gamma)}^2
\|\nabla_{\boldsymbol\tau}\tilde\psi_{1t}\|_{L^2(\Gamma)}^2
+\|\nabla_{\boldsymbol\tau}\tilde\psi_t\|_{L^2(\Gamma)}^2
\right) {\rm d}t                                              \nonumber
\\
&\le CK_2\|(\tilde{\bf u}, \tilde\phi, \tilde\psi)\|_X^2.  \nonumber
\end{align}
In the end, we shall estimate $\|\bar g\|_{L^\infty(0,T;H^{\frac{1}{2}}(\Gamma))}^2$. Direct calculation shows
\begin{equation*}
\|\bar g\|_{H^1(\Gamma)}
\le CK_2^\frac{1}{2}\|(\tilde{\bf u}, \tilde\phi, \tilde\psi)\|_X,
\end{equation*}
and
\begin{align}
\|\bar g\|_{L^2(\Gamma)}
&=\left\Vert\int_{0}^{t}\bar g_t(x,s){\rm d}s\right\Vert_{L^2(\Gamma)}
\le \int_{0}^{t}\|\bar g_t\|_{L^2(\Gamma)}{\rm d}s  \nonumber
\\
&\le T^{\frac{1}{2}}\|\bar g_t\|_{L^2(0,T;L^2(\Gamma))}
\le CK_2^\frac{1}{2}T^\frac{1}{2}\|(\tilde{\bf u}, \tilde\phi, \tilde\psi)\|_X, \nonumber
\end{align}
from which we obtain
\begin{equation*}
\|\bar g\|_{H^{\frac{1}{2}}(\Gamma)}^2
\le C\|\bar g\|_{L^2(\Gamma)}\|\bar g\|_{H^1(\Gamma)}
\le CK_2T^{\frac{1}{2}}\|(\tilde{\bf u}, \tilde\phi, \tilde\psi)\|_X^2,
\end{equation*}
where we have used (\ref{L21}) for $k=0$.
Through the above estimates, we conclude
\begin{align}
\label{A.8}
&\|\bar G\|_{L^2}^2+\|\bar g\|_{H^\frac{1}{2}(\Gamma)}^2
+\int_{0}^{T}\left[
\epsilon K_1^{-2}K_2^{-2}\left(\|\bar G_t\|_{L^2}^2+\|\bar g_t\|_{L^2(\Gamma)}^2
\right)
+\|\bar G\|_{H^1}^2+\|\bar g\|_{H^{\frac{3}{2}}(\Gamma)}^2
\right] {\rm d}t                                            \nonumber
\\
&\le C\left(K_1T+K_2T^{\frac{1}{2}}+\epsilon
\right)\|(\tilde{\bf u}, \tilde\phi, \tilde\psi)\|_X^2.
\end{align}
We go back to (\ref{P2}) in Proposition \ref{Prop3.2}, replace $(G,g)$ by $(\bar G,\bar g)$, use (\ref{A.8}) then we have
\begin{align}
\label{A.6}
\|{\bf u}\|_{Y_1}^2\le Ce^{\epsilon^{-1}K_1^2K_2^2T}\left(K_1T+K_2T^{\frac{1}{2}}+\epsilon
\right)\|(\tilde{\bf u}, \tilde\phi, \tilde\psi)\|_X^2.
\end{align}
Hence, (\ref{A.6}) together with (\ref{A.9}), we obtain
\begin{align}
\|({\bf u},\phi,\psi)\|_X^2\le\tilde C_3e^{\epsilon^{-1}K_1^2K_2^2T}\left(K_1T+K_2^2T^{\frac{1}{2}}+\epsilon
\right)\|(\tilde{\bf u}, \tilde\phi, \tilde\psi)\|_X^2.  \nonumber
\end{align}
Choosing $T$ suitable small such that $\epsilon^{-1}K_1^2K_2^2T<1$, and further satisfy
\begin{align}
\tilde C_3eK_1T^\frac{1}{2}<\frac{1}{6},\;\tilde C_3eK_2^2T^\frac{1}{2}<\frac{1}{6},
\end{align}
then, we choosing $\epsilon=\frac{1}{6e\tilde C_3}$ to satisfy $\tilde C_3e\epsilon<\frac{1}{6}$. Thus, we take $T$ as following
\begin{align}
\left(6\tilde C_3eK_1K_2\right)^4T<1,\;{\rm i.e.}\; T<T_2^*=\left(
6e\tilde C_3K_1K_2
\right)^{-4},                                 \nonumber
\end{align}
which meets the above conditions and implies
\begin{align}
\|\Lambda(\tilde{\bf u},\tilde\phi,\tilde\psi)\|_X^2=\|({\bf u},\phi,\psi)\|_X^2\le\frac{1}{2}\|(\tilde{\bf u},\tilde\phi,\tilde\psi)\|_X^2,
\nonumber
\end{align}
Finally, define $T^*={\rm min}\{T_1^*,T_2^*\}={\rm min}\Big\{K_2^{-4},\left(6e\tilde C_3K_1K_2
\right)^{-4}\Big\}=\left(6e\tilde C_3K_1K_2
\right)^{-4}$, we have proved $\Lambda$ is a contraction mapping. Hence, by the Banach's fixed point theorem, the map $\Lambda$ has a fixed point $({\bf u},\phi,\psi)$ which is the solution of system $(\ref{NSAC})$--$(\ref{initial condition})$. The proof of Theorem \ref{local} is completed.
$\hfill\square$
\subsection{Global well-posedness}
Given the established local existence of the unique strong solutions and the a priori global estimates, the existence of a global solution follows directly via a standard continuation argument. Furthermore, the uniqueness of the solution can be rigorously established through energy estimates.
So we only give the a priori global estimates to the problem $(\ref{NSAC})$--$(\ref{initial condition})$.
Now, we start with the following basic energy equality.
\begin{Lemma}\label{Energy}
Let $({\bf u},\phi,\psi)$ be the smooth solution of $(\ref{NSAC})$--$(\ref{initial condition})$ on $[0,T]$, then it holds that
\begin{align}
\label{energy}
\sup\limits_{0\le t \le T} E(t)+\int_{0}^{T}\left( \frac{1}{2} \|\mathbb S({\bf u})\|_{L^2}^2+\beta\|{\bf u}_{\boldsymbol\tau}\|_{L^2(\Gamma)}^2+\|\mu-\bar\mu\|_{L^2}^2+\|\mathcal L(\psi)\|_{L^2(\Gamma)}^2
\right) {\rm d}t= E(0),
\end{align}
where
\begin{align}
 E(t)=\|{\bf u}\|_{L^2}^2+\|\nabla\phi\|_{L^2}^2+\gamma\|\nabla_{\boldsymbol\tau}\psi\|_{L^2(\Gamma)}^2
+\frac{1}{2}\|\phi^2-1\|_{L^2}^2+2\int_\Gamma\gamma_{fs}(\psi){\rm d}S.\nonumber
\end{align}
\end{Lemma}
\begin{proof}[\bf Proof.]
First, multiplying $(\ref{NSAC})_3$ by $(\mu-\bar\mu)$ and integrating the result over $\Omega$, one has
\begin{align}
\label{e1}
\int_\Omega\phi_t(\mu-\bar\mu){\rm d}x+\int_\Omega{\bf u}\cdot\nabla\phi(\mu-\bar\mu)
{\rm d}x=-\|\mu-\bar\mu\|_{L^2}^2.
\end{align}
The first term on the left-hand side together with (\ref{phit}) leads to
\begin{align}
&\int_\Omega\phi_t(\mu-\bar\mu){\rm d}x+\int_\Omega{\bf u}\cdot\nabla\phi(\mu-\bar\mu){\rm d}x \nonumber
\\
&=-\int_\Omega\Delta\phi\phi_t {\rm d}x+\int_\Omega f(\phi)\phi_t {\rm d}x+\int_\Omega{\bf u}\cdot\nabla\phi(-\Delta\phi+f) {\rm d}x  \nonumber
\\
&=-\int_\Omega\Delta\phi\phi_t {\rm d}x+\int_\Omega(\phi^3-\phi)\phi_t{\rm d}x-\int_\Omega{\bf u}\cdot\nabla\phi\Delta\phi {\rm d}x+\int_\Omega{\bf u}\cdot\nabla\phi(\phi^3-\phi) {\rm d}x  \nonumber
\\
&=\int_\Omega\nabla\phi\cdot\nabla\phi_t {\rm d}x-\int_\Gamma\psi_t\partial_{\bf n}\phi {\rm d}S+\int_\Omega(\phi^3-\phi)\phi_t{\rm d}x-\int_\Omega{\bf u}\cdot\nabla\phi\Delta\phi {\rm d}x \nonumber
\\
&=\frac{1}{2}\left(\|\nabla\phi\|_{L^2}^2+\frac{1}{2}\|\phi^2-1\|_{L^2}^2\right) -\int_\Gamma\psi_t\partial_{\bf n}\phi {\rm d}S-\int_\Omega{\rm div}(\nabla\phi\otimes\nabla\phi)\cdot{\bf u}{\rm d}x,\nonumber
\end{align}
where we have used the following fact:
\begin{align}
\int_\Omega(\phi^3-\phi)\nabla\phi\cdot{\bf u} {\rm d}x  &=\frac{1}{4}\int_\Omega\nabla(\phi^2-1)^2\cdot{\bf u} {\rm d}x
\nonumber
\\
&=-\frac{1}{4}\int_\Omega(\phi^2-1)^2{\rm div}{\bf u}{\rm d}x+\frac{1}{4}\int_\Gamma(\phi^2-1)^2{\bf u}\cdot{\bf n}{\rm d}S=0.                                            \nonumber
\end{align}
Then, substituting the above equalities into (\ref{e1}), we get
\begin{align}
\label{e4}
&\frac{1}{2}\frac{\rm d}{{\rm d}t}\left(\|\nabla\phi\|_{L^2}^2
+\frac{1}{2}\|\phi^2-1\|_{L^2}^2
\right)
+\|\mu-\bar\mu\|_{L^2}^2=\int_\Gamma\partial_{\bf n}\phi\psi_t{\rm d}S+\int_\Omega{\rm div}(\nabla\phi\otimes\nabla\phi)\cdot{\bf u}{\rm d}x.
\end{align}
\vskip2mm
Next, multiplying $(\ref{NSAC})_1$ by $\bf u$, and integrating the result over $\Omega$ by parts,  one finds
\begin{align}
\label{e2}
&\frac{1}{2}\frac{\rm d}{{\rm d}t}\|{\bf u}\|_{L^2}^2+\frac{1}{2}\|{\mathbb S}({\bf u})\|_{L^2}^2+\beta\|{\bf u}_{\boldsymbol\tau}\|_{L^2(\Gamma)}^2
\nonumber
\\
&=\int_\Gamma\mathcal L(\psi)\nabla_{\boldsymbol\tau}\psi\cdot{\bf u}_{\boldsymbol\tau}{\rm d}S-\int_\Omega{\rm div}(\nabla\phi\otimes\nabla\phi)\cdot{\bf u} {\rm d}x \nonumber
\\
&=\int_\Gamma\mathcal L(\psi)\left(-\mathcal L(\psi)-\psi_t
\right){\rm d}S-\int_\Omega{\rm div}(\nabla\phi\otimes\nabla\phi)\cdot{\bf u} {\rm d}x
\nonumber
\\
&=-\|\mathcal L(\psi)\|_{L^2(\Gamma)}^2-\int_\Gamma\left(
-\gamma\Delta_{\boldsymbol\tau}\psi+\partial_{\bf n}\phi+\gamma_{fs}^\prime(\psi)
\right)\psi_t{\rm d}S-\int_\Omega{\rm div}(\nabla\phi\otimes\nabla\phi)\cdot{\bf u} {\rm d}x
\nonumber
\\
&=-\|\mathcal L(\psi)\|_{L^2(\Gamma)}^2-\frac{1}{2}\frac{\rm d}{{\rm d}t}\left(\gamma
\|\nabla_{\boldsymbol\tau}\psi\|_{L^2(\Gamma)}^2+2\int_\Gamma\gamma_{fs}(\psi){\rm d}S
\right)                  \nonumber
\\
&\quad-\int_\Gamma\partial_{\bf n}\phi\psi_t{\rm d}S
-\int_\Omega{\rm div}(\nabla\phi\otimes\nabla\phi)\cdot{\bf u} {\rm d}x.
\end{align}
Combining (\ref{e4}) with (\ref{e2}), then together with Gronwall inequality shows (\ref{energy}). Hence, we complete the proof of  Lemma \ref{Energy}.
\end{proof}
\begin{Lemma}
\label{Phi2}
Let $({\bf u},\phi,\psi)$ be the smooth solution to $(\ref{NSAC})$--$(\ref{initial condition})$ on $[0,T]$, then it holds that
\begin{align}
\label{Phi h2}
\|\phi\|_{H^2}^2+\gamma\|\psi\|_{H^2(\Gamma)}^2\le C\left(\|\mu-\bar\mu\|_{L^2}^2+\|\mathcal L(\psi)\|_{L^2(\Gamma)}^2+1
\right),
\end{align}
where $C$ is a positive constant that depends on $\Omega$, $\gamma^{-1}$ and $T$.
\end{Lemma}

\begin{proof}[\bf Proof.]
First, using Poincar$\rm\acute{e}$ inequality and the fact that $\int_\Omega\phi_t {\rm d}x=0$, we deduce
\begin{align}
\label{Phi}
\|\phi\|_{L^2}^2\le \|\phi-\bar\phi\|_{L^2}^2+\|\bar\phi\|_{L^2}^2\le C\|\nabla\phi\|_{L^2}^2+C|\bar\phi_0|^2\le C.
\end{align}
Then, from $(\ref{NSAC})_4$ and the definition of $\mathcal L(\psi)=-\gamma\Delta_{\boldsymbol\tau}\phi+\partial_{\bf n}\phi+\gamma_{fs}^\prime(\psi)$, we infer that
\begin{equation}
\label{ty}
\begin{cases}
-\Delta\phi=\mu-(\phi^3-\phi),\;&{\rm in}\;\Omega,
\\
-\gamma\Delta_{\boldsymbol\tau}\psi+\psi+\partial_{\bf n}\phi=\mathcal L(\psi)-\gamma_{fs}^\prime(\psi)+\psi,\;&{\rm on}\;\Gamma,
\\
\phi\big|_\Gamma=\psi,\;&{\rm on}\;\Gamma,
\end{cases}
\end{equation}
whence, together Lemma \ref{estimates}, we obtain
\begin{align}
\label{H4}
&\|\phi\|_{H^2}^2+\gamma\|\psi\|_{H^2(\Gamma)}^2 \nonumber
\\
&\le C(\gamma^{-1})\left(\|\mu\|_{L^2}^2+\|\phi^3\|_{L^2}^2+\|\phi\|_{L^2}^2+\|\mathcal L(\psi)\|_{L^2(\Gamma)}^2+\|\psi\|_{L^2(\Gamma)}^2+1
\right)
\nonumber
\\
&\le C(\gamma^{-1})\left(\|\mu-\bar\mu\|_{L^2}^2+\|\mathcal L(\psi)\|_{L^2(\Gamma)}^2+1
\right),
\end{align}
where we have used (\ref{energy}), (\ref{Phi}), Lemma \ref{Trace} and the following fact that
\begin{align}
\bar\mu&=\frac{1}{|\Omega|}\int_\Omega(-\Delta\phi+f){\rm d}x=\frac{1}{|\Omega|}\int_\Gamma
-\partial_{\bf n}\phi {\rm d}S+\frac{1}{|\Omega|}\int_\Omega(\phi^3-\phi){\rm d}x    \nonumber
\\
&=\frac{1}{|\Omega|}\int_\Gamma\left[-\mathcal L(\psi)-\gamma\Delta_{\boldsymbol\tau}\psi+\gamma_{fs}^\prime(\psi)
\right]{\rm d}S+\frac{1}{|\Omega|}\int_\Omega(\phi^3-\phi){\rm d}x    \nonumber
\\
&=\frac{1}{|\Omega|}\int_\Gamma\left[-\mathcal L(\psi)+\gamma_{fs}^\prime(\psi)
\right]{\rm d}S+\frac{1}{|\Omega|}\int_\Omega(\phi^3-\phi){\rm d}x    \nonumber
\\
&\le C\left(\|\mathcal L(\psi)\|_{L^2(\Gamma)}+\|\phi^3\|_{L^2}+\|\phi\|_{L^2}+1
\right).       \nonumber
\end{align}
Thus, we complete the proof of Lemma \ref{Phi2}.
\end{proof}
\begin{Lemma}
\label{3}
Let $({\bf u},\phi,\psi)$ be the smooth solution to $(\ref{NSAC})$--$(\ref{initial condition})$ on $[0,T]$, then it holds that
\begin{align}
\label{H10}
&\sup\limits_{0\le t \le T}\left(
\|\nabla_{\boldsymbol\tau}\nabla\phi\|_{L^2}^2+\gamma\|\nabla_{\boldsymbol\tau}^2\psi\|_{L^2(\Gamma)}^2
+\|\nabla_{\boldsymbol\tau}{\bf u}\|_{L^2}^2
\right)                          \nonumber
\\
&\quad+\int_{0}^{T}\left(\|\nabla_{\boldsymbol\tau}\mathcal L(\psi)\|_{L^2(\Gamma)}^2+\|\mathbb S(\nabla_{\boldsymbol\tau}{\bf u})\|_{L^2}^2+\|\nabla_{\boldsymbol\tau}{\bf u}_{\boldsymbol\tau}\|_{L^2(\Gamma)}^2
\right) {\rm d}t\le C,
\end{align}
where $C$ is a positive constant that depends on $\Omega$, $\beta^{-1}$, $\gamma^{-1}$ and $T$.
\end{Lemma}
\begin{proof}[\bf Proof.]
Taking the tangential derivative of $(\ref{NSAC})_3$, $(\ref{NSAC})_4$ and $(\ref{dynamic})_3$ with respect with $x$ respectively leads to
 \begin{equation}
 \label{tangential 1}
 \begin{cases}
 \nabla_{\boldsymbol\tau}\phi_t+\nabla_{\boldsymbol\tau}({\bf u}\cdot\nabla\phi)=\nabla_{\boldsymbol\tau}(\bar\mu-\mu),\;&{\rm in}\;\Omega,
 \\
 \nabla_{\boldsymbol\tau}\mu=-\nabla_{\boldsymbol\tau}\Delta\phi+\nabla_{\boldsymbol\tau}f,\;&{\rm in}\;\Omega,
 \\
 \nabla_{\boldsymbol\tau}\psi_t+\nabla_{\boldsymbol\tau}({\bf u}_{\boldsymbol\tau}\cdot\nabla_{\boldsymbol\tau}\psi)=-\nabla_{\boldsymbol\tau}\mathcal L(\psi)=\gamma\nabla_{\boldsymbol\tau}\Delta_{\boldsymbol\tau}\psi
 -\nabla_{\boldsymbol\tau}\partial_{\bf n}\phi-\gamma_{fs}^{(2)}(\psi)\nabla_{\boldsymbol\tau}\psi,\;&{\rm on}\;\Gamma,
 \end{cases}
 \end{equation}
Multiplying $(\ref{tangential 1})_1$, $(\ref{tangential 1})_2$  and $(\ref{tangential 1})_3$ by $-\nabla_{\boldsymbol\tau}\mu$, $\nabla_{\boldsymbol\tau}\phi_t$ and $\nabla_{\boldsymbol\tau}\psi_t$ respectively, integrating it over $\Omega$ by parts, and then adding the result together, we get
\begin{align}
\label{H1}
&\frac{1}{2}\frac{\rm d}{{\rm d}t}\left(\|\nabla_{\boldsymbol\tau}\nabla\phi\|_{L^2}^2
+\gamma\|\nabla_{\boldsymbol\tau}^2\psi\|_{L^2(\Gamma)}^2
\right)+\|\nabla_{\boldsymbol\tau}(\mu-\bar\mu)\|_{L^2}^2+\|\nabla_{\boldsymbol\tau}\mathcal L(\psi)\|_{L^2(\Gamma)}^2               \nonumber
\\
&=-\int_\Omega\nabla_{\boldsymbol\tau}({\bf u}\cdot\nabla\phi)\nabla_{\boldsymbol\tau}\mu {\rm d}x-\int_\Omega(3\phi^2-1)\nabla_{\boldsymbol\tau}\phi\nabla_{\boldsymbol\tau}(\bar\mu-\mu){\rm d}x
\nonumber
\\
&\quad-\int_\Omega\left[6\phi|\nabla_{\boldsymbol\tau}\phi|^2+(3\phi^2-1)
\nabla_{\boldsymbol\tau}^2\phi
\right]{\bf u}\cdot\nabla\phi {\rm d}x     \nonumber
\\
&\quad-\int_{\Gamma}\nabla_{\boldsymbol\tau}\mathcal L(\psi)\left({\bf u}_{\boldsymbol\tau}\cdot\nabla_{\boldsymbol\tau}^2\psi+\nabla_{\boldsymbol\tau}{\bf u}_{\boldsymbol\tau}\cdot\nabla_{\boldsymbol\tau}\psi
\right){\rm d}S                            \nonumber
\\
&\quad+\int_{\Gamma}\gamma_{fs}^{(2)}(\psi)\nabla_{\boldsymbol\tau}\psi\left(
\nabla_{\boldsymbol\tau}\mathcal L(\psi)+\nabla_{\boldsymbol\tau}{\bf u}_{\boldsymbol\tau}\cdot\nabla_{\boldsymbol\tau}\psi+{\bf u}_{\boldsymbol\tau}\cdot\nabla_{\boldsymbol\tau}^2\psi
\right){\rm d}S.
\end{align}
\vskip2mm
Then, taking the tangential derivative of $(\ref{NSAC})_1$ and $(\ref{dynamic})_2$ respectively gives
\begin{equation}
\label{tangential 2}
\begin{cases}
\nabla_{\boldsymbol\tau}{\bf u}_t+\nabla_{\boldsymbol\tau}{\bf u}\cdot\nabla{\bf u}+{\bf u}\cdot\nabla\nabla_{\boldsymbol\tau}{\bf u}+\nabla\nabla_{\boldsymbol\tau}p-{\rm div}\mathbb S(\nabla_{\boldsymbol\tau}{\bf u})=-\nabla_{\boldsymbol\tau}\Delta\phi\nabla\phi-\Delta\phi\nabla_{\boldsymbol\tau}\nabla\phi,
\;&{\rm in}\;\Omega,
\\
\beta\nabla_{\boldsymbol\tau}{\bf u}_{\boldsymbol\tau}+\left(\mathbb S(\nabla_{\boldsymbol\tau}{\bf u})\cdot{\bf n}
\right)_{\boldsymbol\tau}=\nabla_{\boldsymbol\tau}\mathcal L(\psi)\nabla_{\boldsymbol\tau}\psi+\mathcal L(\psi)\nabla_{\boldsymbol\tau}^2\psi,\;&{\rm on}\;\Gamma.
\end{cases}
\end{equation}
Testing $(\ref{tangential 2})_1$ by $\nabla_{\boldsymbol\tau}{\bf u}$, it is worthy noting that $\Omega$ is a channel domain, due to the incompressibility of velocity ${\bf u}$ and boundary condition, we have
\begin{align}
\int_\Omega{\bf u}\cdot\nabla\nabla_{\boldsymbol\tau}{\bf u}\cdot\nabla_{\boldsymbol\tau}{\bf u}{\rm d}x&=\frac{1}{2}\int_\Omega{\bf u}\cdot\nabla|\nabla_{\boldsymbol\tau}{\bf u}|^2 {\rm d}x \nonumber
\\
&=-\frac{1}{2}\int_\Omega{\rm div}{\bf u}|\nabla_{\boldsymbol\tau}{\bf u}|^2 {\rm d}x+\frac{1}{2}\int_\Gamma{\bf u}\cdot{\bf n}|\nabla_{\boldsymbol\tau}{\bf u}|^2{\rm d}S=0,
\nonumber
\end{align}
and
\begin{align}
\int_\Omega\nabla\nabla_{\boldsymbol\tau}p\cdot\nabla_{\boldsymbol\tau}{\bf u}{\rm d}x=-\int_\Omega\nabla_{\boldsymbol\tau}p\cdot\nabla_{\boldsymbol\tau}{\rm div}{\bf u}{\rm d}x+\int_\Gamma\nabla_{\boldsymbol\tau}p\cdot\nabla_{\boldsymbol\tau}{\bf u}\cdot{\bf n}{\rm d}S=0. \nonumber
\end{align}
Thus, we infer that
\begin{align}
\label{H2}
&\frac{1}{2}\frac{\rm d}{{\rm d}t}\|\nabla_{\boldsymbol\tau}{\bf u}\|_{L^2}^2+\frac{1}{2}\|\mathbb S(\nabla_{\boldsymbol\tau}{\bf u})\|_{L^2}^2+\beta\|\nabla_{\boldsymbol\tau}{\bf u}_{\boldsymbol\tau}\|_{L^2(\Gamma)}^2  \nonumber
\\
&=-\int_\Omega\left(\nabla_{\boldsymbol\tau}{\bf u}\cdot\nabla{\bf u}
+\nabla_{\boldsymbol\tau}\Delta\phi\nabla\phi+\Delta\phi\nabla_{\boldsymbol\tau}\nabla\phi
\right)
\cdot\nabla_{\boldsymbol\tau}{\bf u}{\rm d}x             \nonumber
\\
&\quad+\int_\Gamma\nabla_{\boldsymbol\tau}\mathcal L(\psi)\nabla_{\boldsymbol\tau}\psi\cdot\nabla_{\boldsymbol\tau}{\bf u}_{\boldsymbol\tau} {\rm d}S
+\int_\Gamma\mathcal L(\psi)\nabla_{\boldsymbol\tau}^2\psi
\cdot\nabla_{\boldsymbol\tau}{\bf u}_{\boldsymbol\tau}{\rm d}S.
\end{align}
Adding (\ref{H1}) and (\ref{H2}) together, we obtain
\begin{align}
\label{H3}
&\frac{1}{2}\frac{\rm d}{{\rm d}t}\left(\|\nabla_{\boldsymbol\tau}\nabla\phi\|_{L^2}^2
+\gamma\|\nabla_{\boldsymbol\tau}^2\psi\|_{L^2(\Gamma)}^2+\|\nabla_{\boldsymbol\tau}{\bf u}\|_{L^2}^2
\right)     \nonumber
\\
&\quad+\|\nabla_{\boldsymbol\tau}(\mu-\bar\mu)\|_{L^2}^2+\|\nabla_{\boldsymbol\tau}\mathcal L(\psi)\|_{L^2(\Gamma)}^2+\frac{1}{2}\|\mathbb S(\nabla_{\boldsymbol\tau}{\bf u})\|_{L^2}^2+\beta\|\nabla_{\boldsymbol\tau}{\bf u}_{\boldsymbol\tau}\|_{L^2(\Gamma)}^2    \nonumber
\\
&=-\int_\Omega\nabla_{\boldsymbol\tau}({\bf u}\cdot\nabla\phi)\nabla_{\boldsymbol\tau}\mu {\rm d}x-\int_\Omega\left(\nabla_{\boldsymbol\tau}{\bf u}\cdot\nabla{\bf u}
+\nabla_{\boldsymbol\tau}\Delta\phi\nabla\phi+\Delta\phi\nabla_{\boldsymbol\tau}\nabla\phi
\right)
\cdot\nabla_{\boldsymbol\tau}{\bf u}{\rm d}x
\nonumber
\\
&\quad-\int_\Omega(3\phi^2-1)\nabla_{\boldsymbol\tau}\phi\nabla_{\boldsymbol\tau}(\bar\mu-\mu){\rm d}x
-\int_\Omega\left[6\phi|\nabla_{\boldsymbol\tau}\phi|^2+(3\phi^2-1)
\nabla_{\boldsymbol\tau}^2\phi
\right]{\bf u}\cdot\nabla\phi {\rm d}x \nonumber
\\
&\quad
-\int_\Gamma\nabla_{\boldsymbol\tau}\mathcal L(\psi){\bf u}_{\boldsymbol\tau}\cdot\nabla_{\boldsymbol\tau}^2\psi {\rm d}S-\int_\Gamma\mathcal L(\psi)\nabla_{\boldsymbol\tau}^2\psi\cdot\nabla_{\boldsymbol\tau}{\bf u}_{\boldsymbol\tau}{\rm d}S             \nonumber
\\
&\quad+\int_\Gamma\gamma_{fs}^{(2)}(\psi)\nabla_{\boldsymbol\tau}\psi\left(
\nabla_{\boldsymbol\tau}\mathcal L(\psi)+\nabla_{\boldsymbol\tau}{\bf u}_{\boldsymbol\tau}\cdot\nabla_{\boldsymbol\tau}\psi+{\bf u}_{\boldsymbol\tau}\cdot\nabla_{\boldsymbol\tau}^2\psi
\right){\rm d}S                          \nonumber
\\
&=\sum_{i=1}^{7}I_i.
\end{align}
\vskip2mm
Now, we focus on the estimates of the terms on the right-hand side of (\ref{H3}). First, due to Lemma \ref{PGN}, (\ref{energy}), (\ref{Phi}) and Korn's inequality, we obtain
\begin{align}
I_1&=-\int_\Omega\left(\nabla_{\boldsymbol\tau}{\bf u}\cdot\nabla\phi+{\bf u}\cdot\nabla_{\boldsymbol\tau}\nabla\phi\right)
\nabla_{\boldsymbol\tau}(\mu-\bar\mu){\rm d}x      \nonumber
\\
&\le\left(\|\nabla_{\boldsymbol\tau}{\bf u}\|_{L^4}\|\nabla\phi\|_{L^4}+\|{\bf u}\|_{L^4}\|\nabla_{\boldsymbol\tau}\nabla\phi\|_{L^4}
\right)\|\nabla_{\boldsymbol\tau}(\mu-\bar\mu)\|_{L^2} \nonumber
\\
&\le\frac{1}{8}\|\nabla_{\boldsymbol\tau}(\mu-\bar\mu)\|_{L^2}^2 +C\|\nabla_{\boldsymbol\tau}{\bf u}\|_{L^2}\|\nabla_{\boldsymbol\tau}{\bf u}\|_{H^1}\|\nabla\phi\|_{L^2}\|\nabla\phi\|_{H^1} \nonumber
\\
&\quad+C\|{\bf u}\|_{L^2}\|{\bf u}\|_{H^1}\|\nabla_{\boldsymbol\tau}\nabla\phi\|_{L^2}\|\nabla_{\boldsymbol\tau}\nabla\phi\|_{H^1} \nonumber
\\
&\le\frac{1}{8}\|\nabla_{\boldsymbol\tau}(\mu-\bar\mu)\|_{L^2}^2 +C\left(\|\nabla_{\boldsymbol\tau}{\bf u}\|_{L^2}\|\nabla_{\boldsymbol\tau}{\bf u}\|_{H^1}\|\nabla\phi\|_{H^1}+\|{\bf u}\|_{H^1}\|\nabla_{\boldsymbol\tau}\nabla\phi\|_{L^2}\|\nabla_{\boldsymbol\tau}\nabla\phi\|_{H^1}
\right)                                         \nonumber
\\
&\le\frac{1}{8}\|\nabla_{\boldsymbol\tau}(\mu-\bar\mu)\|_{L^2}^2+\frac{1}{16}\|\mathbb S(\nabla_{\boldsymbol\tau}{\bf u})\|_{L^2}^2
+\varepsilon\|\nabla_{\boldsymbol\tau}\phi\|_{H^2}^2 \nonumber
\\
&\quad+C\|\nabla{\bf u}\|_{L^2}^2\|\nabla_{\boldsymbol\tau}\nabla\phi\|_{L^2}^2+C\|\nabla_{\boldsymbol\tau}{\bf u}\|_{L^2}^2\|\phi\|_{H^2}^2.
\nonumber
\\[1em]
I_2&\le\|\nabla_{\boldsymbol\tau}{\bf u}\|_{L^4}^2\|\nabla{\bf u}\|_{L^2}+\|\nabla_{\boldsymbol\tau}\Delta\phi\|_{L^2}\|\nabla\phi\|_{L^4}\|\nabla_{\boldsymbol\tau}{\bf u}\|_{L^4}
+\|\Delta\phi\|_{L^2}\|\nabla_{\boldsymbol\tau}\nabla\phi\|_{L^4}
\|\nabla_{\boldsymbol\tau}{\bf u}\|_{L^4}        \nonumber
\\
&\le C\|\nabla_{\boldsymbol\tau}{\bf u}\|_{L^2}\|\nabla_{\boldsymbol\tau}{\bf u}\|_{H^1}\|\nabla{\bf u}\|_{L^2} +C\|\nabla_{\boldsymbol\tau}\phi\|_{H^2}\|\nabla\phi\|_{L^2}^\frac{1}{2}
\|\nabla\phi\|_{H^1}^\frac{1}{2}
\|\nabla_{\boldsymbol\tau}{\bf u}\|_{L^2}^\frac{1}{2}\|\nabla_{\boldsymbol\tau}{\bf u}\|_{H^1}^\frac{1}{2}
\nonumber
\\
&\quad+C\|\phi\|_{H^2}\|\nabla_{\boldsymbol\tau}\nabla\phi\|_{L^2}^\frac{1}{2}
\|\nabla_{\boldsymbol\tau}\nabla\phi\|_{H^1}^\frac{1}{2}\|\nabla_{\boldsymbol\tau}{\bf u}\|_{L^2}^\frac{1}{2}\|\nabla_{\boldsymbol\tau}{\bf u}\|_{H^1}^\frac{1}{2}
\nonumber
\\
&\le C\|\nabla_{\boldsymbol\tau}{\bf u}\|_{L^2}\|\nabla_{\boldsymbol\tau}\nabla{\bf u}\|_{L^2}\|\nabla{\bf u}\|_{L^2}+\varepsilon\|\nabla_{\boldsymbol\tau}\phi\|_{H^2}^2+C\|\nabla\phi\|_{H^1}
\|\nabla_{\boldsymbol\tau}{\bf u}\|_{L^2}\|\nabla_{\boldsymbol\tau}\nabla{\bf u}\|_{L^2}
\nonumber
\\
&\quad+C\|\phi\|_{H^2}\|\nabla_{\boldsymbol\tau}\nabla\phi\|_{L^2}
\|\nabla_{\boldsymbol\tau}\nabla\phi\|_{H^1}+C\|\phi\|_{H^2}\|\nabla_{\boldsymbol\tau}{\bf u}\|_{L^2}\|\nabla_{\boldsymbol\tau}\nabla{\bf u}\|_{L^2}      \nonumber
\\
&\le\frac{1}{16}\|\mathbb S(\nabla_{\boldsymbol\tau}{\bf u})\|_{L^2}^2+\varepsilon\|\nabla_{\boldsymbol\tau}\phi\|_{H^2}^2+C
\|\nabla_{\boldsymbol\tau} {\bf u}\|_{L^2}^2\|\nabla{\bf u}\|_{L^2}^2 \nonumber
\\
&\quad+C\|\phi\|_{H^2}^2\|\nabla_{\boldsymbol\tau}{\bf u}\|_{L^2}^2+C\|\phi\|_{H^2}^2\|\nabla_{\boldsymbol\tau}\nabla\phi\|_{L^2}^2
\nonumber
\\
&\le\frac{1}{16}\|\mathbb S(\nabla_{\boldsymbol\tau}{\bf u})\|_{L^2}^2+\varepsilon\|\nabla_{\boldsymbol\tau}\phi\|_{H^2}^2+C
\left(\|\nabla{\bf u}\|_{L^2}^2+\|\phi\|_{H^2}^2\right)
\|\nabla_{\boldsymbol\tau}{\bf u}\|_{L^2}^2
+C\|\phi\|_{H^2}^2\|\nabla_{\boldsymbol\tau}\nabla\phi\|_{L^2}^2  \nonumber
\\[1em]
I_3&\le\|3\phi^2-1\|_{L^\infty} \|\nabla_{\boldsymbol\tau}\phi\|_{L^2}
\|\nabla_{\boldsymbol\tau}(\mu-\bar\mu)\|_{L^2}  \nonumber
\\
&\le\frac{1}{8}\|\nabla_{\boldsymbol\tau}(\mu-\bar\mu)\|_{L^2}^2 +C\left(\|\phi\|_{H^2}^2+1\right)\|\nabla_{\boldsymbol\tau}\phi\|_{L^2}^2  \nonumber
\\
&\le\frac{1}{8}\|\nabla_{\boldsymbol\tau}(\mu-\bar\mu)\|_{L^2}^2 +C\|\phi\|_{H^2}^2+C.
\nonumber
\\[1em]
I_4
&\le\left(\|\phi\|_{L^\infty}\|\nabla_{\boldsymbol\tau}\phi\|_{L^4}^2+
\|3\phi^2-1\|_{L^\infty}
\|\nabla_{\boldsymbol\tau}^2\phi\|_{L^2}\right)\|{\bf u}\|_{L^4}\|\nabla\phi\|_{L^4}  \nonumber
\\
&\le C\left(\|\phi\|_{H^2}\|\nabla_{\boldsymbol\tau}\phi\|_{L^2}
\|\nabla_{\boldsymbol\tau}\phi\|_{H^1}
+\|\nabla_{\boldsymbol\tau}^2\phi\|_{L^2}\right)\|{\bf u}\|_{L^2}^\frac{1}{2}\|{\bf u}\|_{H^1}^\frac{1}{2}\|\nabla\phi\|_{L^2}^\frac{1}{2}\|\nabla\phi\|_{H^1}^\frac{1}{2}
\nonumber
\\
&\quad+C\|\phi\|_{L^2}\|\phi\|_{H^2}\|\nabla_{\boldsymbol\tau}^2\phi\|_{L^2}\|{\bf u}\|_{L^2}^\frac{1}{2}\|{\bf u}\|_{H^1}^\frac{1}{2}\|\nabla\phi\|_{L^2}^\frac{1}{2}\|\nabla\phi\|_{H^1}^\frac{1}{2}
\nonumber
\\
&\le C\|\phi\|_{H^2}
\|\nabla_{\boldsymbol\tau}\phi\|_{H^1}\|{\bf u}\|_{H^1}^\frac{1}{2}\|\nabla\phi\|_{H^1}^\frac{1}{2}
+C\left(\|\phi\|_{H^2}+1
\right)
\|\nabla_{\boldsymbol\tau}^2\phi\|_{L^2}\|{\bf u}\|_{H^1}^\frac{1}{2}\|\nabla\phi\|_{H^1}^\frac{1}{2}
\nonumber
\\
&\le C\left(\|\phi\|_{H^2}+1
\right)
\|\nabla_{\boldsymbol\tau}\phi\|_{H^1}\|{\bf u}\|_{H^1}^\frac{1}{2}\|\nabla\phi\|_{H^1}^\frac{1}{2}   \nonumber
\\
&\le C\left(\|\phi\|_{H^2}^2+1\right)
\|\nabla_{\boldsymbol\tau}\nabla\phi\|_{L^2}^2+C\|\nabla{\bf u}\|_{L^2}^2+C\|\phi\|_{H^2}^2.   \nonumber
\end{align}
Next, on the basis of Lemma \ref{Trace} and Lemma \ref{PGN}, $I_5-I_7$ can be estimated as
\begin{align}
I_5
&\le\|\nabla_{\boldsymbol\tau}\mathcal L(\psi)\|_{L^2(\Gamma)}\|{\bf u}_{\boldsymbol\tau}\|_{L^\infty(\Gamma)}\|\nabla_{\boldsymbol\tau}^2\psi\|_{L^2(\Gamma)}
\nonumber
\\
&\le C\|\nabla_{\boldsymbol\tau}\mathcal L(\psi)\|_{L^2(\Gamma)}\|{\bf u}_{\boldsymbol\tau}\|_{L^2(\Gamma)}^\frac{1}{2}\|{\bf u}_{\boldsymbol\tau}\|_{H^1(\Gamma)}^\frac{1}{2}\|\nabla_{\boldsymbol\tau}^2\psi\|_{L^2(\Gamma)}
\nonumber
\\
&\le C\|\nabla_{\boldsymbol\tau}\mathcal L(\psi)\|_{L^2(\Gamma)}\left(
\|{\bf u}_{\boldsymbol\tau}\|_{L^2(\Gamma)}+\|{\bf u}_{\boldsymbol\tau}\|_{L^2(\Gamma)}^\frac{1}{2}\|\nabla_{\boldsymbol\tau}{\bf u}_{\boldsymbol\tau}\|_{L^2(\Gamma)}^\frac{1}{2}
\right)\|\nabla_{\boldsymbol\tau}^2\psi\|_{L^2(\Gamma)}  \nonumber
\\
&\le C\|\nabla_{\boldsymbol\tau}\mathcal L(\psi)\|_{L^2(\Gamma)}\left(
\|\nabla{\bf u}\|_{L^2}^\frac{1}{2}+\|\nabla{\bf u}\|_{L^2}^\frac{1}{4}\|\nabla_{\boldsymbol\tau}{\bf u}\|_{L^2}^\frac{1}{4}
\|\nabla\nabla_{\boldsymbol\tau}{\bf u}\|_{L^2}^\frac{1}{4}
\right)\|\nabla_{\boldsymbol\tau}^2\psi\|_{L^2(\Gamma)}  \nonumber
\\
&\le\frac{1}{12}\|\nabla_{\boldsymbol\tau}\mathcal L(\psi)\|_{L^2(\Gamma)}^2 + C\|\nabla{\bf u}\|_{L^2}  \|\nabla_{\boldsymbol\tau}^2\psi\|_{L^2(\Gamma)}^2
\nonumber \\
& \quad +C\|\nabla{\bf u}\|_{L^2}^\frac{1}{2}\|\nabla_{\boldsymbol\tau}{\bf u}\|_{L^2}^\frac{1}{2}\|\nabla\nabla_{\boldsymbol\tau}{\bf u}\|_{L^2}^\frac{1}{2} \|\nabla_{\boldsymbol\tau}^2\psi\|_{L^2(\Gamma)}^2
\nonumber
\\
&\le\frac{1}{12}\|\nabla_{\boldsymbol\tau}\mathcal L(\psi)\|_{L^2(\Gamma)}^2+\frac{1}{16}\|\mathbb S(\nabla_{\boldsymbol\tau}{\bf u})\|_{L^2}^2+C\|\nabla_{\boldsymbol\tau}^2\psi\|_{L^2(\Gamma)}^4                  \nonumber
\\
&\quad+C\|\nabla{\bf u}\|_{L^2}^2+C\|\nabla{\bf u}\|_{L^2}^2\|\nabla_{\boldsymbol\tau}{\bf u}\|_{L^2}^2.
\nonumber
\\[1em]
I_6
&\le\|\mathcal L(\psi)\|_{L^\infty(\Gamma)}\|\nabla_{\boldsymbol\tau}^2\psi\|_{L^2(\Gamma)}
\|\nabla_{\boldsymbol\tau}{\bf u}_{\boldsymbol\tau}\|_{L^2(\Gamma)}    \nonumber
\\
&\le C\|\mathcal L(\psi)\|_{L^2(\Gamma)}^\frac{1}{2}\|\mathcal L(\psi)\|_{H^1(\Gamma)}^\frac{1}{2}\|\nabla_{\boldsymbol\tau}^2\psi\|_{L^2(\Gamma)}
\|\nabla_{\boldsymbol\tau}{\bf u}\|_{H^\frac{1}{2}}    \nonumber
\\
&\le C\left(\|\mathcal L(\psi)\|_{L^2(\Gamma)}+\|\mathcal L(\psi)\|_{L^2(\Gamma)}^\frac{1}{2}\|\nabla_{\boldsymbol\tau}\mathcal L(\psi)\|_{L^2(\Gamma)}^\frac{1}{2}
\right)\|\nabla_{\boldsymbol\tau}^2\psi\|_{L^2(\Gamma)}\|\nabla_{\boldsymbol\tau}{\bf u}\|_{L^2}^\frac{1}{2}\|\nabla\nabla_{\boldsymbol\tau}{\bf u}\|_{L^2}^\frac{1}{2}
\nonumber
\\
&\le C\|\mathcal L(\psi)\|_{L^2(\Gamma)}^2\|\nabla_{\boldsymbol\tau}^2\psi\|_{L^2(\Gamma)}^2
+C\|\nabla_{\boldsymbol\tau}{\bf u}\|_{L^2}\|\nabla\nabla_{\boldsymbol\tau}{\bf u}\|_{L^2 }
\nonumber
\\
&\quad+C\|\mathcal L(\psi)\|_{L^2(\Gamma)}\|\nabla_{\boldsymbol\tau}\mathcal L(\psi)\|_{L^2(\Gamma)}\|\nabla_{\boldsymbol\tau}{\bf u}\|_{L^2}+C\|\nabla_{\boldsymbol\tau}^2\psi\|_{L^2(\Gamma)}^2\|\nabla\nabla_{\boldsymbol\tau}{\bf u}\|_{L^2}           \nonumber
\\
&\le\frac{1}{12}\|\nabla_{\boldsymbol\tau}\mathcal L(\psi)\|_{L^2(\Gamma)}^2+\frac{1}{16}\|\mathbb S(\nabla_{\boldsymbol\tau}{\bf u})\|_{L^2}^2+C\|\mathcal L(\psi)\|_{L^2(\Gamma)}^2\|\nabla_{\boldsymbol\tau}^2\psi\|_{L^2(\Gamma)}^2
\nonumber
\\
&\quad+C\|\nabla_{\boldsymbol\tau}{\bf u}\|_{L^2}^2+C\|\mathcal L(\psi)\|_{L^2(\Gamma)}^2\|\nabla_{\boldsymbol\tau}{\bf u}\|_{L^2}^2+C\|\nabla_{\boldsymbol\tau}^2\psi\|_{L^2(\Gamma)}^4  \nonumber
\\
&\le\frac{1}{12}\|\nabla_{\boldsymbol\tau}\mathcal L(\psi)\|_{L^2(\Gamma)}^2+\frac{1}{16}\|\mathbb S(\nabla_{\boldsymbol\tau}{\bf u})\|_{L^2}^2+C\left(\|\mathcal L(\psi)\|_{L^2(\Gamma)}^2+1\right)
\|\nabla_{\boldsymbol\tau}{\bf u}\|_{L^2}^2    \nonumber
\\
&\quad+C\left(\|\nabla_{\boldsymbol\tau}^2\psi\|_{L^2(\Gamma)}^2+\|\mathcal L(\psi)\|_{L^2(\Gamma)}^2
\right)\|\nabla_{\boldsymbol\tau}^2\psi\|_{L^2(\Gamma)}^2.   \nonumber
\\[1em]
I_7
&\le C\|\nabla_{\boldsymbol\tau}\psi\|_{L^2(\Gamma)}\|\nabla_{\boldsymbol\tau}\mathcal L(\psi)\|_{L^2(\Gamma)}+C\|\nabla_{\boldsymbol\tau}\psi\|_{L^\infty(\Gamma)}\|\nabla_{\boldsymbol\tau}{\bf u}_{\boldsymbol\tau}\|_{L^2(\Gamma)}\|\nabla_{\boldsymbol\tau}\psi\|_{L^2(\Gamma)}
\nonumber
\\
&\quad+C\|\nabla_{\boldsymbol\tau}\psi\|_{L^2(\Gamma)}\|{\bf u}_{\boldsymbol\tau}\|_{L^\infty(\Gamma)}\|\nabla_{\boldsymbol\tau}^2\psi\|_{L^2(\Gamma)}
\nonumber
\\
&\le C\|\nabla_{\boldsymbol\tau}\psi\|_{L^2(\Gamma)}\|\nabla_{\boldsymbol\tau}\mathcal L(\psi)\|_{L^2(\Gamma)}+C
\|\nabla_{\boldsymbol\tau}\psi\|_{H^1(\Gamma)}\|\nabla_{\boldsymbol\tau}{\bf u}_{\boldsymbol\tau}\|_{L^2(\Gamma)}      \nonumber
\\
&\quad+C\|{\bf u}_{\boldsymbol\tau}\|_{L^2(\Gamma)}^\frac{1}{2}\|{\bf u}_{\boldsymbol\tau}\|_{H^1(\Gamma)}^\frac{1}{2}\|\nabla_{\boldsymbol\tau}^2\psi\|_{L^2(\Gamma)}
\nonumber
\\
&\le C\|\nabla_{\boldsymbol\tau}\psi\|_{L^2(\Gamma)}\|\nabla_{\boldsymbol\tau}\mathcal L(\psi)\|_{L^2(\Gamma)}
+  C \|\nabla_{\boldsymbol\tau}^2\psi\|_{L^2(\Gamma)}
\|\nabla_{\boldsymbol\tau}{\bf u}_{\boldsymbol\tau}\|_{L^2(\Gamma)}
\nonumber
\\
&\quad+C(\|{\bf u}_{\boldsymbol\tau}\|_{L^2(\Gamma)}+\|{\bf u}_{\boldsymbol\tau}\|_{L^2(\Gamma)}^\frac{1}{2}\|\nabla_{\boldsymbol\tau}{\bf u}_{\boldsymbol\tau}\|_{L^2(\Gamma)}^\frac{1}{2} )\|\nabla_{\boldsymbol\tau}^2\psi\|_{L^2(\Gamma)}
\nonumber
\\
&\le\frac{1}{12}\|\nabla_{\boldsymbol\tau}\mathcal L(\psi)\|_{L^2(\Gamma)}^2+C\|\nabla_{\boldsymbol\tau}^2\psi\|_{L^2(\Gamma)}^2
+\frac{\beta}{4}\|\nabla_{\boldsymbol\tau}{\bf u}_{\boldsymbol\tau}\|_{L^2(\Gamma)}^2    \nonumber
\\
& \quad  +C\|{\bf u}_{\boldsymbol\tau}\|_{L^2(\Gamma)}\|\nabla_{\boldsymbol\tau}{\bf u}_{\boldsymbol\tau}\|_{L^2(\Gamma)} +C\|{\bf u}_{\boldsymbol\tau}\|_{L^2(\Gamma)}^2
\nonumber \\
&\le\frac{1}{12}\|\nabla_{\boldsymbol\tau}\mathcal L(\psi)\|_{L^2(\Gamma)}^2+\frac{\beta}{2}\|\nabla_{\boldsymbol\tau}{\bf u}_{\boldsymbol\tau}\|_{L^2(\Gamma)}^2+C\|\nabla_{\boldsymbol\tau}^2\psi\|_{L^2(\Gamma)}^2
+C\|{\bf u}_{\boldsymbol\tau}\|_{L^2(\Gamma)}^2.   \nonumber
\end{align}
It follows from (\ref{ty}) and Lemma \ref{estimates} that
\begin{align}
&\|\nabla_{\boldsymbol\tau}\phi\|_{H^2}^2+\gamma\|\nabla_{\boldsymbol\tau}\psi\|_{H^2(\Gamma)}^2 \nonumber
\\
&\le C(\gamma^{-1})\left(\|\nabla_{\boldsymbol\tau}(\mu-\bar\mu)\|_{L^2}^2+\|(3\phi^2-1)
\nabla_{\boldsymbol\tau}\phi\|_{L^2}^2+\|\nabla_{\boldsymbol\tau}\mathcal L(\psi)\|_{L^2(\Gamma)}^2+\|\nabla_{\boldsymbol\tau}\psi\|_{L^2(\Gamma)}^2+1
\right)                           \nonumber
\\
&\le C(\gamma^{-1})\left(\|\nabla_{\boldsymbol\tau}(\mu-\bar\mu)\|_{L^2}^2
+\|\phi^2-1\|_{L^2}^2\|\nabla_{\boldsymbol\tau}\phi\|_{L^\infty}^2+\|\nabla_{\boldsymbol\tau}\mathcal L(\psi)\|_{L^2(\Gamma)}^2+\|\nabla_{\boldsymbol\tau}\psi\|_{L^2(\Gamma)}^2+1
\right)                           \nonumber
\\
&\le C(\gamma^{-1})\left(\|\nabla_{\boldsymbol\tau}(\mu-\bar\mu)\|_{L^2}^2
+\|\nabla_{\boldsymbol\tau}\phi\|_{L^2}\|\nabla_{\boldsymbol\tau}\phi\|_{H^2}
+\|\nabla_{\boldsymbol\tau}\mathcal L(\psi)\|_{L^2(\Gamma)}^2+1
\right)                           \nonumber
\\
&\le \frac{\gamma}{2} \|\nabla_{\boldsymbol\tau}\phi\|_{H^2}^2+C(\gamma^{-1})\left(
\|\nabla_{\boldsymbol\tau}(\mu-\bar\mu)\|_{L^2}^2+\|\nabla_{\boldsymbol\tau}\mathcal L(\psi)\|_{L^2(\Gamma)}^2+1
\right),    \nonumber
\end{align}
which implies
\begin{align}
\label{tan phi h2}
\|\nabla_{\boldsymbol\tau}\phi\|_{H^2}^2+\gamma\|\nabla_{\boldsymbol\tau}\psi\|_{H^2(\Gamma)}^2\le C\left(\|\nabla_{\boldsymbol\tau}(\mu-\bar\mu)\|_{L^2}^2+\|\nabla_{\boldsymbol\tau}\mathcal L(\psi)\|_{L^2(\Gamma)}^2+1
\right).
\end{align}
Substituting $I_1$ -- $I_7$ into (\ref{H3}) and taking (\ref{tan phi h2}) into account, then  choosing $C\varepsilon<\frac{1}{4}$, we arrive at
\begin{align}
&\frac{\rm d}{{\rm d}t}\left(\|\nabla_{\boldsymbol\tau}\nabla\phi\|_{L^2}^2
+\gamma\|\nabla_{\boldsymbol\tau}^2\psi\|_{L^2(\Gamma)}^2+\|\nabla_{\boldsymbol\tau}{\bf u}\|_{L^2}^2
\right)     \nonumber
\\
&\quad+\|\nabla_{\boldsymbol\tau}(\mu-\bar\mu)\|_{L^2}^2+\|\nabla_{\boldsymbol\tau}\mathcal L(\psi)\|_{L^2(\Gamma)}^2+\|\mathbb S(\nabla_{\boldsymbol\tau}{\bf u})\|_{L^2}^2+\beta\|\nabla_{\boldsymbol\tau}{\bf u}_{\boldsymbol\tau}\|_{L^2(\Gamma)}^2    \nonumber
\\
&\le C\mathcal A_1(t)\left(\|\nabla_{\boldsymbol\tau}\nabla\phi\|_{L^2}^2
+\|\nabla_{\boldsymbol\tau}^2\psi\|_{L^2(\Gamma)}^2+\|\nabla_{\boldsymbol\tau}{\bf u}\|_{L^2}^2
\right)+C\mathcal B_1(t),
\end{align}
where
\begin{align}
&\mathcal A_1(t)=\|\psi\|_{H^2(\Gamma)}^2+\|\mathcal L(\psi)\|_{L^2(\Gamma)}^2+\|\nabla{\bf u}\|_{L^2}^2+\|\phi\|_{H^2}^2+1
\nonumber
\\
&\mathcal B_1(t)=\|{\bf u}_{\boldsymbol\tau}\|_{L^2}^2+\|\nabla{\bf u}\|_{L^2}^2+\|\phi\|_{H^2}^2+1.      \nonumber
\end{align}
It follows from (\ref{energy}), (\ref{Phi h2}) and Korn's inequality that $\mathcal A_1(t)\in L^1(0,T)$ and $\mathcal B_1(t)\in L^1(0,T)$. Finally, together with Gronwall inequality, we complete the proof of Lemma \ref{3}.
\end{proof}
\begin{Lemma}
\label{4*}
Let $({\bf u},\phi,\psi)$ be the smooth solution to $(\ref{NSAC})$--$(\ref{initial condition})$ on $[0,T]$, then it holds that
\begin{align}
\label{H13}
&\sup\limits_{0\le t \le T}\left(
\|\phi_t\|_{H^1}^2+\gamma\|\psi_t\|_{H^1(\Gamma)}^2
+\|{\bf u}_t\|_{L^2}^2
\right)                          \nonumber
\\
&\quad+\int_{0}^{T}\left(\|\partial_t\mathcal L(\psi)\|_{L^2(\Gamma)}^2+\|(\mu-\bar\mu)_t\|_{L^2}^2+\|\mathbb S({\bf u}_t)\|_{L^2}^2+\|{\bf u}_{t{\boldsymbol\tau}}\|_{L^2(\Gamma)}^2
\right) {\rm d}t\le C,
\end{align}
where $C$ is a positive constant that depends on $\Omega$, $\beta^{-1}$, $\gamma^{-1}$ and $T$.
\end{Lemma}
\begin{proof}[\bf Proof.]
First, using Poincar$\rm\acute{e}$ inequality and the fact $\int_\Omega\phi_t{\rm d}x=0$ to obtain
\begin{align}
\label{Phi-t}
\|\phi_t\|_{L^2}^2=\|\phi_t-\bar\phi_t\|_{L^2}^2\le C\|\nabla\phi_t\|_{L^2}^2,
\end{align}
which together with (\ref{ty}) and Lemma \ref{estimates} yields
\begin{align}
\label{H11}
&\|\phi_t\|_{H^2}^2+\gamma\|\psi_t\|_{H^2(\Gamma)}^2  \nonumber
\\
&\le C(\gamma^{-1})\left(
\|(\mu-\bar\mu)_t\|_{L^2}^2+\|\bar\mu_t\|_{L^2}^2+\|(3\phi^2-1)\phi_t\|_{L^2}^2
+\|\partial_t\mathcal L(\psi)\|_{L^2(\Gamma)}^2+\|\psi_t\|_{L^2(\Gamma)}^2
\right)                          \nonumber
\\
&\le C(\gamma^{-1})\left(
\|(\mu-\bar\mu)_t\|_{L^2}^2+\|(3\phi^2-1)\phi_t\|_{L^2}^2
+\|\partial_t\mathcal L(\psi)\|_{L^2(\Gamma)}^2+\|\psi_t\|_{L^2(\Gamma)}^2
\right)                          \nonumber
\\
&\le C(\gamma^{-1})\left(\|(\mu-\bar\mu)_t\|_{L^2}^2+\|\phi^2\|_{L^4}^2\|\phi_t\|_{L^4}^2+\|\phi_t\|_{L^2}^2
+\|\partial_t\mathcal L(\psi)\|_{L^2(\Gamma)}^2+\|\phi_t\|_{H^1}^2
\right)                          \nonumber
\\
&\le C(\gamma^{-1})\left(\|(\mu-\bar\mu)_t\|_{L^2}^2+\|\nabla\phi_t\|_{L^2}^2
+\|\partial_t\mathcal L(\psi)\|_{L^2(\Gamma)}^2
\right),
\end{align}
where we have used (\ref{energy}), Lemma \ref{Trace} and the following fact:
\begin{align}
\bar\mu_t&=\dfrac{1}{|\Omega|}\int_\Omega\mu_t{\rm d}x=\dfrac{1}{|\Omega|}\int_\Omega\left(
-\Delta\phi_t+f^\prime(\phi)\phi_t
\right){\rm d}x    \nonumber
\\
&=\dfrac{1}{|\Omega|}\int_\Gamma\partial_{\bf n}\phi_t {\rm d}S+\dfrac{1}{|\Omega|}\int_\Omega
(3\phi^2-1)\phi_t{\rm d}x  \nonumber
\\
&=\dfrac{1}{|\Omega|}\int_\Gamma\left(\gamma\Delta_{\boldsymbol\tau}\psi_t+\partial_t\mathcal L(\psi)-\gamma_{fs}^{(2)}(\psi)\psi_t
\right){\rm d}S+\dfrac{1}{|\Omega|}\int_\Omega(3\phi^2-1)\phi_t{\rm d}x  \nonumber
\\
&=\dfrac{1}{|\Omega|}\int_\Gamma\left(\partial_t\mathcal L(\psi)-\gamma_{fs}^{(2)}(\psi)\psi_t
\right){\rm d}S+\dfrac{1}{|\Omega|}\int_\Omega(3\phi^2-1)\phi_t{\rm d}x.  \nonumber
\end{align}
\vskip2mm
Next, differentiating $(\ref{NSAC})_1$ with respect to $t$ leads to
\begin{align}
{\bf u}_{tt}+{\bf u}_t\cdot\nabla{\bf u}+{\bf u}\cdot\nabla{\bf u}_t+\nabla p_t-{\rm div}\mathbb S({\bf u}_t)=-\Delta\phi_t\nabla\phi-\Delta\phi\nabla\phi_t.\nonumber
\end{align}
Testing it by ${\bf u}_t$, reviewing (\ref{nabla pt}) and noting
\begin{align}
-\int_\Omega{\rm div}\mathbb S({\bf u}_t)\cdot{\bf u}_t{\rm d}x&=\frac{1}{2}\|\mathbb S({\bf u}_t)\|_{L^2}^2-\int_\Gamma\left(\mathbb S({\bf u}_t)\cdot{\bf n}
\right)_{\boldsymbol\tau}\cdot{\bf u}_{t{\boldsymbol\tau}}{\rm d}S \nonumber
\\
&=\frac{1}{2}\|\mathbb S({\bf u}_t)\|_{L^2}^2-\int_\Gamma\left[
\partial_t\mathcal L(\psi)\nabla_{\boldsymbol\tau}\psi+\mathcal L(\psi)\nabla_{\boldsymbol\tau}\psi_t-\beta{\bf u}_{t{\boldsymbol\tau}}
\right]\cdot{\bf u}_{t{\boldsymbol\tau}}{\rm d}S                \nonumber
\\
&=\frac{1}{2}\|\mathbb S({\bf u}_t)\|_{L^2}^2+\beta\|{\bf u}_{t{\boldsymbol\tau}}\|_{L^2(\Gamma)}^2-\int_\Gamma\left[
\partial_t\mathcal L(\psi)\nabla_{\boldsymbol\tau}\psi+\mathcal L(\psi)\nabla_{\boldsymbol\tau}\psi_t
\right]\cdot{\bf u}_{t{\boldsymbol\tau}}{\rm d}S,   \nonumber
\end{align}
and
\begin{align}
\label{777}
\int_\Omega{\bf u}\cdot\nabla{\bf u}_t\cdot{\bf u}_t {\rm d}x=\frac{1}{2}\int_\Omega{\bf u}\cdot\nabla|{\bf u}_t|^2 {\rm d}x=-\frac{1}{2}\int_\Omega{\rm div}{\bf u}|{\bf u}_t|^2 {\rm d}x+\frac{1}{2}\int_\Gamma{\bf u}\cdot{\bf n}|{\bf u}_t|^2{\rm d}S=0.
\end{align}
Thus, we obtain
\begin{align}
\label{H8}
&\frac{1}{2}\frac{\rm d}{{\rm d}t}\|{\bf u}_t\|_{L^2}^2+\frac{1}{2}\|\mathbb S({\bf u}_t)\|_{L^2}^2+\beta\|{\bf u}_{t{\boldsymbol\tau}}\|_{L^2(\Gamma)}^2 \nonumber
\\
&=-\int_\Omega{\bf u}_t\cdot\nabla{\bf u}\cdot{\bf u}_t{\rm d}x-\int_\Omega\Delta\phi_t\nabla\phi\cdot{\bf u}_t{\rm d}x-\int_\Omega\Delta\phi\nabla\phi_t\cdot{\bf u}_t{\rm d}x \nonumber
\\
&\quad+\int_\Gamma\left[
\partial_t\mathcal L(\psi)\nabla_{\boldsymbol\tau}\psi+\mathcal L(\psi)\nabla_{\boldsymbol\tau}\psi_t
\right]\cdot{\bf u}_{t{\boldsymbol\tau}}{\rm d}S     \nonumber
\\
&=-\int_\Omega{\bf u}_t\cdot\nabla{\bf u}\cdot{\bf u}_t{\rm d}x+\int_\Omega\mu_t\nabla\phi\cdot{\bf u}_t{\rm d}x-\int_\Omega f^\prime(\phi)\phi_t\nabla\phi\cdot{\bf u}_t{\rm d}x   \nonumber
\\
&\quad-\int_\Omega\Delta\phi\nabla\phi_t\cdot{\bf u}_t{\rm d}x+\int_{\Gamma}\left[
\partial_t\mathcal L(\psi)\nabla_{\boldsymbol\tau}\psi+\mathcal L(\psi)\nabla_{\boldsymbol\tau}\psi_t
\right]\cdot{\bf u}_{t{\boldsymbol\tau}}{\rm d}S.
\end{align}
\vskip2mm
Then, differentiating $(\ref{NSAC})_3$, $(\ref{NSAC})_4$ and $(\ref{dynamic})_3$ with respect to $t$, one has
 \begin{equation}
 \label{Partial t}
 \begin{cases}
 \phi_{tt}=(\bar\mu-\mu)_t-{\bf u}_t\cdot\nabla\phi-{\bf u}\cdot\nabla\phi_t,\;&{\rm in}\;\Omega,
 \\
 \mu_t=-\Delta\phi_t+f^\prime(\phi)\phi_t,\;&{\rm in}\;\Omega,
 \\
 \psi_{tt}+{\bf u}_{t{\boldsymbol\tau}}\cdot\nabla_{\boldsymbol\tau}\psi+{\bf u}_{\boldsymbol\tau}\cdot\nabla_{\boldsymbol\tau}\psi_t=-\partial_t\mathcal L(\psi)=\gamma\Delta_{\boldsymbol\tau}\psi_t-\partial_{\bf n}\phi_t-\gamma_{fs}^{(2)}(\psi)\psi_t,\;&{\rm on}\;\Gamma,
 \end{cases}
 \end{equation}
Multiplying $(\ref{Partial t})_1$, and $(\ref{Partial t})_2$ by $-\mu_t$ and $\phi_{tt}$ respectively, adding the result together, and integrating it over $\Omega$ by parts, we have
\begin{align}
0&=-\int_\Omega(\bar\mu-\mu)_t\mu_t{\rm d}x+\int_\Omega{\bf u}_t\cdot\nabla\phi\mu_t{\rm d}x+\int_\Omega{\bf u}\cdot\nabla\phi_t\mu_t{\rm d}x-\int_\Omega\Delta\phi_t\phi_{tt}{\rm d}x+\int_\Omega f^\prime(\phi)\phi_t\phi_{tt}{\rm d}x   \nonumber
\\
&=\|(\mu-\bar\mu)_t\|_{L^2}^2 - \bar\mu_t\int_\Omega(\bar\mu-\mu)_t{\rm d}x+\int_\Omega{\bf u}_t\cdot\nabla\phi\mu_t{\rm d}x+\int_\Omega{\bf u}\cdot\nabla\phi_t\mu_t{\rm d}x \nonumber
\\
&\quad+\frac{1}{2}\frac{\rm d}{{\rm d}t}\|\nabla\phi_t\|_{L^2}^2
-\int_\Gamma\psi_{tt}\partial_{\bf n}\phi_t {\rm d}S+\int_\Omega f^\prime(\phi)\phi_t\phi_{tt}{\rm d}x, \nonumber
\end{align}
in which we obtain
\begin{align}
\label{H7}
&\frac{1}{2}\frac{\rm d}{{\rm d}t}\|\nabla\phi_t\|_{L^2}^2+\|(\mu-\bar\mu)_t\|_{L^2}^2 \nonumber
\\
&=-\int_\Omega{\bf u}_t\cdot\nabla\phi\mu_t{\rm d}x-\int_\Omega{\bf u}\cdot\nabla\phi_t\mu_t{\rm d}x+\int_\Gamma\psi_{tt}\partial_{\bf n}\phi_t {\rm d}S-\int_\Omega f^\prime(\phi)\phi_t\phi_{tt}{\rm d}x.
\end{align}
Multiplying $(\ref{Partial t})_3$ by $\psi_{tt}$, and integrating the result over $\Gamma$, one shows that
\begin{align}
\label{H5}
-\int_\Gamma\partial_t\mathcal L(\psi)\psi_{tt}{\rm d}S=\gamma\int_\Gamma\Delta_{\boldsymbol\tau}\psi_t\psi_{tt}{\rm d}S
-\int_\Gamma\partial_{\bf n}\phi_t\psi_{tt}{\rm d}S
-\int_\Gamma\gamma_{fs}^\prime(\psi)\psi_t\psi_{tt}{\rm d}S.
\end{align}
Substituting $\psi_{tt}$ with $(\ref{Partial t})_3$ into the left-hand side of (\ref{H5}), we get
\begin{align}
\label{H33}
-\int_\Gamma\partial_t\mathcal L(\psi)\psi_{tt}{\rm d}S&=\int_\Gamma\partial_t\mathcal L(\psi)\left[
\partial_t\mathcal L(\psi)+{\bf u}_{t{\boldsymbol\tau}}\cdot\nabla_{\boldsymbol\tau}\psi+{\bf u}_{\boldsymbol\tau}\cdot\nabla_{\boldsymbol\tau}\psi_t
\right]{\rm d}S                      \nonumber
\\
&=\|\partial_t\mathcal L(\psi)\|_{L^2(\Gamma)}^2+\int_\Gamma\partial_t\mathcal L(\psi)\left(
{\bf u}_{t{\boldsymbol\tau}}\cdot\nabla_{\boldsymbol\tau}\psi+{\bf u}_{\boldsymbol\tau}\cdot\nabla_{\boldsymbol\tau}\psi_t
\right){\rm d}S.
\end{align}
Proceeding in the same way for the last term on the right-hand side of (\ref{H5}), substituting the result and (\ref{H33}) into (\ref{H5}), rearranging terms gives
\begin{align}
\label{H6}
&\frac{\gamma}{2}\frac{\rm d}{{\rm d}t}\|\nabla_{\boldsymbol\tau}\psi_t\|_{L^2(\Gamma)}^2+\|\partial_t\mathcal L(\psi)\|_{L^2(\Gamma)}^2      \nonumber
\\
&=-\int_\Gamma\partial_{\bf n}\phi_t\psi_{tt}{\rm d}S-\int_\Gamma\partial_t\mathcal L(\psi)\left(
{\bf u}_{t{\boldsymbol\tau}}\cdot\nabla_{\boldsymbol\tau}\psi+{\bf u}_{\boldsymbol\tau}\cdot\nabla_{\boldsymbol\tau}\psi_t
\right){\rm d}S                \nonumber
\\
&\quad+\int_\Gamma\gamma_{fs}^\prime(\psi)\psi_t\left[\partial_t\mathcal L(\psi)+{\bf u}_{t{\boldsymbol\tau}}\cdot\nabla_{\boldsymbol\tau}\psi+{\bf u}_{\boldsymbol\tau}\cdot\nabla_{\boldsymbol\tau}\psi_t
\right]{\rm d}S.
\end{align}
Summing (\ref{H8}) and (\ref{H7}) together with (\ref{H6}), we conclude
\begin{align}
\label{H9}
&\frac{1}{2}\frac{\rm d}{{\rm d}t}\left(\|\nabla\phi_t\|_{L^2}^2
+\gamma\|\nabla_{\boldsymbol\tau}\psi_t\|_{L^2(\Gamma)}^2+\|{\bf u}_t\|_{L^2}^2
\right)           \nonumber
\\
&\quad+\|(\mu-\bar\mu)_t\|_{L^2}+\|\partial_t\mathcal L(\psi)\|_{L^2(\Gamma)}^2+\frac{1}{2}\|\mathbb S({\bf u}_t)\|_{L^2}^2+\beta\|{\bf u}_{t{\boldsymbol\tau}}\|_{L^2(\Gamma)}^2  \nonumber
\\
&=-\int_\Omega{\bf u}_t\cdot\nabla{\bf u}\cdot{\bf u}_t{\rm d}x-\int_\Omega\Delta\phi\nabla\phi_t\cdot{\bf u}_t{\rm d}x-\int_\Omega{\bf u}\cdot\nabla\phi_t\mu_t{\rm d}x \nonumber
\\
&\quad-\int_\Omega f^\prime(\phi)\phi_t\phi_{tt}{\rm d}x-\int_\Omega f^\prime(\phi)\phi_t\nabla\phi\cdot{\bf u}_t{\rm d}x  \nonumber
\\
&\quad+\int_\Gamma\mathcal L(\psi)\nabla_{\boldsymbol\tau}\psi_t
\cdot{\bf u}_{t{\boldsymbol\tau}}{\rm d}S-\int_\Gamma\partial_t\mathcal L(\psi){\bf u}_{\boldsymbol\tau}\cdot\nabla_{\boldsymbol\tau}\psi_t{\rm d}S   \nonumber
\\
&\quad-\int_\Gamma\gamma_{fs}^\prime(\psi)\psi_t\left[\partial_t\mathcal L(\psi)+{\bf u}_{t{\boldsymbol\tau}}\cdot\nabla_{\boldsymbol\tau}\psi+{\bf u}_{\boldsymbol\tau}\cdot\nabla_{\boldsymbol\tau}\psi_t
\right]{\rm d}S                       \nonumber
\\
&=\sum_{i=1}^{8}J_i.
\end{align}
\vskip2mm
Now, we estimate each term on the right-hand side of (\ref{H9}). Due to Lemma \ref{PGN} and Korn's inequality, one has
\begin{align}
J_1&\le C\|\nabla{\bf u}\|_{L^2}\|{\bf u}_t\|_{L^4}^2\le C\|\nabla{\bf u}\|_{L^2}\|{\bf u}_t\|_{L^2}\|{\bf u}_t\|_{H^1} \nonumber
\\
&\le\frac{1}{12}\|\mathbb S({\bf u}_t)\|_{L^2}^2+C\left(\|\nabla{\bf u}\|_{L^2}^2+1
\right)\|{\bf u}_t\|_{L^2}^2,   \nonumber
\\[1em]
J_2&\le\|\Delta\phi\|_{L^2}\|\nabla\phi_t\|_{L^4}\|{\bf u}_t\|_{L^4}  \nonumber
\\
&\le C\|\phi\|_{H^2}\|\nabla\phi_t\|_{L^2}^\frac{1}{2}\|\nabla\phi_t\|_{H^1}^\frac{1}{2}\|{\bf u}_t\|_{L^2}^\frac{1}{2}\|{\bf u}_t\|_{H^1}^\frac{1}{2}   \nonumber
\\
&\le C\|\phi\|_{H^2}\|\nabla\phi_t\|_{L^2}\|{\bf u}_t\|_{H^1}+C\|\phi\|_{H^2}\|\nabla\phi_t\|_{H^1}\|{\bf u}_t\|_{L^2} \nonumber
\\
&\le\frac{1}{12}\|\mathbb S({\bf u}_t)\|_{L^2}^2+\varepsilon\|\phi_t\|_{H^2}^2
+C\|\phi\|_{H^2}^2\|\nabla\phi_t\|_{L^2}^2+C\left(
\|\phi\|_{H^2}^2+1\right)\|{\bf u}_t\|_{L^2}^2.        \nonumber
\end{align}
Then, it follows from the incompressibility and the boundary condition that
\begin{align}
J_3&=-\int_\Omega{\bf u}\cdot\nabla\phi_t(\mu-\bar\mu)_t{\rm d}x \nonumber
\\
&\le\|{\bf u}\|_{L^4}\|\nabla\phi_t\|_{L^4}\|(\mu-\bar\mu)_t\|_{L^2} \nonumber
\\
&\le C\|{\bf u}\|_{L^2}^\frac{1}{2}\|{\bf u}\|_{H^1}^\frac{1}{2}\|\nabla\phi_t\|_{L^2}^\frac{1}{2}\|\nabla\phi_t\|_{H^1}^\frac{1}{2}
\|(\mu-\bar\mu)_t\|_{L^2}              \nonumber
\\
&\le\frac{1}{8}\|(\mu-\bar\mu)_t\|_{L^2}^2+C\|\nabla{\bf u}\|_{L^2}\|\nabla\phi_t\|_{L^2}\|\nabla\phi_t\|_{H^1}  \nonumber
\\
&\le\frac{1}{8}\|(\mu-\bar\mu)_t\|_{L^2}^2+\varepsilon\|\phi_t\|_{H^2}^2+C\|\nabla{\bf u}\|_{L^2}^2\|\nabla\phi_t\|_{L^2}^2.  \nonumber
\end{align}
By Poincar{\'e} inequality and (\ref{Phi-t}), $J_4$ and $J_5$ can be rewritten and estimated as
\begin{align}
J_4+J_5&=-\int_\Omega f^\prime(\phi)\phi_t\left(\phi_{tt}+\nabla\phi\cdot{\bf u}_t
\right){\rm d}x  \nonumber
\\
&=-\int_\Omega(3\phi^2-1)\phi_t\left[(\bar\mu-\mu)_t-{\bf u}\cdot\nabla\phi_t
\right]{\rm d}x  \nonumber
\\
&\le C\left(\|\phi^2\|_{L^\infty}+1
\right)\|\phi_t\|_{L^2}\|(\mu-\bar\mu)_t\|_{L^2}+C\left(\|\phi^2\|_{L^\infty}+1
\right)\|\phi_t\|_{L^4}\|{\bf u}\|_{L^4}\|\nabla\phi_t\|_{L^2} \nonumber
\\
&\le C\left(\|\phi\|_{L^2}\|\phi\|_{H^2}+1
\right)\|\phi_t\|_{L^2}\|(\mu-\bar\mu)_t\|_{L^2} \nonumber
\\
&\quad+C\left(\|\phi\|_{L^2}\|\phi\|_{H^2}+1
\right)\|\phi_t\|_{H^1}\|{\bf u}\|_{H^1}\|\nabla\phi_t\|_{L^2} \nonumber
\\
&\le C\left(\|\phi\|_{H^2}+1
\right)\|\nabla\phi_t\|_{L^2}\|(\mu-\bar\mu)_t\|_{L^2}+C\left(\|\phi\|_{H^2}+1
\right) \|\nabla\phi_t\|_{L^2}^{2} \|\nabla{\bf u}\|_{L^2}  \nonumber
\\
&\le\frac{1}{8}\|(\mu-\bar\mu)_t\|_{L^2}^2+ C\left(\|\phi\|_{H^2}^2 + \|\nabla{\bf u}\|_{L^2}^2 +1\right)\|\nabla\phi_t\|_{L^2}^2 . \nonumber
\end{align}
Thanks to the embedding $H^\frac{1}{2}(\Gamma)\hookrightarrow L^p(\Gamma)$ for $1\le p < \infty$, one achieves from Lemma \ref{PGN}, Lemma \ref{Trace}, (\ref{energy}), (\ref{H10}) and (\ref{Phi-t}) that
\begin{align}
J_6&\le\|\mathcal L(\psi)\|_{L^\infty(\Gamma)}\|\nabla_{\boldsymbol\tau}\psi_t\|_{L^2(\Gamma)}\|{\bf u}_{t{\boldsymbol\tau}}\|_{L^2(\Gamma)}\le C\|\mathcal L(\psi)\|_{H^1(\Gamma)}\|\nabla_{\boldsymbol\tau}\psi_t\|_{L^2(\Gamma)}\|{\bf u}_{t{\boldsymbol\tau}}\|_{L^2(\Gamma)}          \nonumber
\\
&\le\frac{\beta}{4}\|{\bf u}_{t{\boldsymbol\tau}}\|_{L^2(\Gamma)}^2+C\|\mathcal L(\psi)\|_{H^1(\Gamma)}^2\|\nabla_{\boldsymbol\tau}\psi_t\|_{L^2(\Gamma)}^2.
\nonumber
\\[1em]
J_7&\le \|\partial_t\mathcal L(\psi)\|_{L^2(\Gamma)}\|{\bf u}_{\boldsymbol\tau}\|_{L^\infty(\Gamma)}\|\nabla_{\boldsymbol\tau}\psi_t\|_{L^2(\Gamma)} \nonumber
\\
&\le\frac{1}{8}\|\partial_t\mathcal L(\psi)\|_{L^2(\Gamma)}^2+C\|{\bf u}_{\boldsymbol\tau}\|_{H^1(\Gamma)}^2\|\nabla_{\boldsymbol\tau}\psi_t\|_{L^2(\Gamma)}^2.
\nonumber
\\[1em]
J_8&\le C\|\psi_t\|_{L^2(\Gamma)}\left(\|\partial_t\mathcal L(\psi)\|_{L^2(\Gamma)}+\|{\bf u}_{t{\boldsymbol\tau}}\|_{L^2(\Gamma)}\|\nabla_{\boldsymbol\tau}\psi\|_{L^\infty(\Gamma)}
+\|{\bf u}_{\boldsymbol\tau}\|_{L^\infty(\Gamma)}\|\nabla_{\boldsymbol\tau}\psi_t\|_{L^2(\Gamma)}
\right)                 \nonumber
\\
&\le C\|\psi_t\|_{L^2(\Gamma)}\left(\|\partial_t\mathcal L(\psi)\|_{L^2(\Gamma)}+\|{\bf u}_{t{\boldsymbol\tau}}\|_{L^2(\Gamma)}\|\nabla_{\boldsymbol\tau}\psi\|_{H^1(\Gamma)}
+\|{\bf u}_{\boldsymbol\tau}\|_{H^1(\Gamma)}\|\nabla_{\boldsymbol\tau}\psi_t\|_{L^2(\Gamma)}
\right)                 \nonumber
\\
&\le\frac{1}{8}\|\partial_t\mathcal L(\psi)\|_{L^2(\Gamma)}^2+\frac{\beta}{4}\|{\bf u}_{t{\boldsymbol\tau}}\|_{L^2(\Gamma)}^2+C\|\psi_t\|_{L^2(\Gamma)}^2
\nonumber
\\
&\quad+C\|\psi_t\|_{L^2(\Gamma)}^2\|\nabla_{\boldsymbol\tau}\psi\|_{H^1(\Gamma)}^2+C\|{\bf u}_{\boldsymbol\tau}\|_{H^1(\Gamma)}^2\|\nabla_{\boldsymbol\tau}\psi_t\|_{L^2(\Gamma)}^2
\nonumber
\\
&\le\frac{1}{8}\|\partial_t\mathcal L(\psi)\|_{L^2(\Gamma)}^2+\frac{\beta}{4}\|{\bf u}_{t{\boldsymbol\tau}}\|_{L^2(\Gamma)}^2+C\|\psi_t\|_{L^2(\Gamma)}^2
\nonumber
\\
&\quad+C\|\nabla\phi_t\|_{L^2}^2\|\psi\|_{H^2(\Gamma)}^2+C\|{\bf u}_{\boldsymbol\tau}\|_{H^1(\Gamma)}^2\|\nabla_{\boldsymbol\tau}\psi_t\|_{L^2(\Gamma)}^2
\nonumber
\\
&\le\frac{1}{8}\|\partial_t\mathcal L(\psi)\|_{L^2(\Gamma)}^2+\frac{\beta}{4}\|{\bf u}_{t{\boldsymbol\tau}}\|_{L^2(\Gamma)}^2+C\|\nabla\phi_t\|_{L^2}^2
\nonumber
\\
&\quad+C\|\nabla\phi_t\|_{L^2}^2\|\psi\|_{H^2(\Gamma)}^2+C\|{\bf u}_{\boldsymbol\tau}\|_{H^1(\Gamma)}^2\|\nabla_{\boldsymbol\tau}\psi_t\|_{L^2(\Gamma)}^2. \nonumber
\end{align}
Then substituting $J_1$ -- $J_8$ into (\ref{H9}), and taking (\ref{H11}) into account, after choosing $C\varepsilon<\frac{1}{4}$, we obtain
\begin{align}
&\frac{1}{2}\frac{\rm d}{{\rm d}t}\left(\|\nabla\phi_t\|_{L^2}^2
+\gamma\|\nabla_{\boldsymbol\tau}\psi_t\|_{L^2(\Gamma)}^2+\|{\bf u}_t\|_{L^2}^2
\right)           \nonumber
\\
&\quad+\|(\mu-\bar\mu)_t\|_{L^2}+\|\partial_t\mathcal L(\psi)\|_{L^2(\Gamma)}^2+\frac{1}{2}\|\mathbb S({\bf u}_t)\|_{L^2}^2+\beta\|{\bf u}_{t{\boldsymbol\tau}}\|_{L^2(\Gamma)}^2  \nonumber
\\
&\le C\mathcal A_2(t)\left(\|\nabla\phi_t\|_{L^2}^2+\|{\bf u}_t\|_{L^2}^2+\|\nabla_{\boldsymbol\tau}\psi_t\|_{L^2(\Gamma)}
\right),
\end{align}
where
\begin{align}
\mathcal A_2(t)&=\|\phi\|_{H^2}^2+\|\psi\|_{H^2(\Gamma)}^2+\|\nabla{\bf u}\|_{L^2}^2+\|\mathcal L(\psi)\|_{H^1(\Gamma)}^2+\|{\bf u}_{\boldsymbol\tau}\|_{H^1(\Gamma)}^2+1.
\nonumber
\end{align}
It follows from (\ref{energy}), (\ref{Phi h2}) and (\ref{H10}) that $\mathcal A_2(t)\in L^1(0,T)$. Thus, using Gronwall inequality and (\ref{Phi-t}), and noting $\|\psi_t\|_{L^2(\Gamma)}^2\le C\|\phi_t\|_{H^1}^2$, we arrive at Lemma \ref{4*}. This completes the proof.
\end{proof}
\begin{Lemma}
Let $({\bf u},\phi,\psi)$ be the smooth solution to $(\ref{NSAC})$--$(\ref{initial condition})$ on $[0,T]$, then it holds that
\begin{align}
\label{H14}
&\sup\limits_{0\le t \le T}\left(
\|\mu-\bar\mu\|_{L^2}^2+\|\mathcal L(\psi)\|_{L^2(\Gamma)}^2
+\|\phi\|_{H^2}^2+\|\psi\|_{H^2(\Gamma)}^2
\right)                          \nonumber
\\
&\quad+\int_{0}^{T}\left(\|\phi_t\|_{H^1}^2+\gamma\|\nabla_{\boldsymbol\tau}\psi_t\|_{L^2(\Gamma)}^2
\right) {\rm d}t\le C,
\end{align}
where $C$ is a positive constant that depends on $\Omega$, $\beta^{-1}$, $\gamma^{-1}$ and $T$.
\end{Lemma}
\begin{proof}[\bf Proof.]
We start testing $(\ref{NSAC})_3$ by $(\bar\mu-\mu)_t$ to obtain
\begin{align}
-\int_\Omega\phi_t\mu_t{\rm d}x=\int_\Omega\phi_t(\bar\mu-\mu)_t{\rm d}x
&=\frac{1}{2}\frac{\rm d}{{\rm d}t}\|\mu-\bar\mu\|_{L^2}^2-\int_\Omega{\bf u}\cdot\nabla\phi(\bar\mu-\mu)_t{\rm d}x.    \nonumber
\end{align}
Next, testing $(\ref{partial t})_2$ by $\phi_t$, one has
\begin{align}
\int_\Omega\mu_t\phi_t{\rm d}x&=-\int_\Omega\Delta\phi_t\phi_t{\rm d}x+\int_\Omega (3\phi^2-1)|\phi_t|^2{\rm d}x    \nonumber
\\
&=\|\nabla\phi_t\|_{L^2}^2-\int_\Gamma\partial_{\bf n}\phi_t\psi_t{\rm d}S+\int_\Omega (3\phi^2-1)|\phi_t|^2{\rm d}x.    \nonumber
\end{align}
Then, testing $(\ref{partial t})_3$ by $\psi_t$, we get
\begin{align}
\int_\Gamma\partial_t\mathcal L(\psi)\psi_t{\rm d}S&=-\gamma\int_\Gamma\Delta_{\boldsymbol\tau}\psi_t\psi_t{\rm d}S
+\int_\Gamma\partial_{\bf n}\phi_t\psi_t{\rm d}S
+\int_\Gamma\gamma_{fs}^{(2)}(\psi)|\psi_t|^2{\rm d}S  \nonumber
\\
&=\gamma\|\nabla_{\boldsymbol\tau}\psi_t\|_{L^2(\Gamma)}^2
+\int_\Gamma\partial_{\bf n}\phi_t\psi_t{\rm d}S
+\int_\Gamma\gamma_{fs}^{(2)}(\psi)|\psi_t|^2{\rm d}S,  \nonumber
\end{align}
whereas the left-hand side can be substituting $\psi_t$ with $\psi_t=-{\mathcal L}(\psi)-{\bf u}_{\boldsymbol\tau}\cdot\nabla_{\boldsymbol\tau}\psi$, thus we get
\begin{align}
\int_\Gamma\partial_t\mathcal L(\psi)\psi_t{\rm d}S
&=\int_\Gamma\partial_t\mathcal L(\psi)\left(-\mathcal L(\psi)-{\bf u}_{\boldsymbol\tau}\cdot\nabla_{\boldsymbol\tau}\psi
\right){\rm d}S    \nonumber
\\
&=-\frac{1}{2}\frac{\rm d}{{\rm d}t}\|\mathcal L(\psi)\|_{L^2(\Gamma)}^2-\int_\Gamma\partial_t\mathcal L(\psi){\bf u}_{\boldsymbol\tau}\cdot\nabla_{\boldsymbol\tau}\psi {\rm d}S. \nonumber
\end{align}
Adding the above equalities together, we obtain
\begin{align}
\label{H12}
&\frac{1}{2}\frac{\rm d}{{\rm d}t}\left(\|\mu-\bar\mu\|_{L^2}^2+\|\mathcal L(\psi)\|_{L^2(\Gamma)}^2
\right)+\|\nabla\phi_t\|_{L^2}^2+\gamma\|\nabla_{\boldsymbol\tau}\psi_t\|_{L^2(\Gamma)}^2
+\int_\Omega(3\phi^2-1)|\phi_t|^2{\rm d}x     \nonumber
\\
&=\int_\Omega{\bf u}\cdot\nabla\phi(\bar\mu-\mu)_t{\rm d}x-\int_\Gamma\gamma_{fs}^{(2)}(\psi)|\psi_t|^2{\rm d}S
-\int_\Gamma\partial_t\mathcal L(\psi){\bf u}_{\boldsymbol\tau}\cdot\nabla_{\boldsymbol\tau}\psi {\rm d}S.
\end{align}
Multiplying $(\ref{NSAC})_3$ by $\phi_t$, integrating the result over $\Omega$ leads to
\begin{align}
\label{H26}
2\|\phi_t\|_{L^2}^2&=2\int_\Omega(\bar\mu-\mu)\phi_t{\rm d}x-2\int_\Omega{\bf u}\cdot\nabla\phi\phi_t{\rm d}x.
\end{align}
Adding (\ref{H26}) to (\ref{H12}), since $f^\prime(\phi)=3\phi^2-1\ge -1$, using Lemma \ref{PGN}, Lemma \ref{poincare u}, (\ref{energy}), (\ref{H13}) and (\ref{Phi-t}), we have
\begin{align}
&\frac{1}{2}\frac{\rm d}{{\rm d}t}\left(\|\mu-\bar\mu\|_{L^2}^2+\|\mathcal L(\psi)\|_{L^2(\Gamma)}^2
\right)+\|\phi_t\|_{H^1}^2+\gamma\|\nabla_{\boldsymbol\tau}\psi_t\|_{L^2(\Gamma)}^2     \nonumber
\\
&\le 2\int_\Omega(\bar\mu-\mu)\phi_t{\rm d}x-2\int_\Omega{\bf u}\cdot\nabla\phi\phi_t{\rm d}x+\int_\Omega{\bf u}\cdot\nabla\phi(\bar\mu-\mu)_t{\rm d}x \nonumber
\\
&\quad-\int_\Gamma\gamma_{fs}^{(2)}(\psi)|\psi_t|^2{\rm d}S
-\int_\Gamma\partial_t\mathcal L(\psi){\bf u}_{\boldsymbol\tau}\cdot\nabla_{\boldsymbol\tau}\psi {\rm d}S   \nonumber
\\
&\le C\|\mu-\bar\mu\|_{L^2}\|\phi_t\|_{L^2}+C\|{\bf u}\|_{L^4}\|\nabla\phi\|_{L^4}\left(\|\phi_t\|_{L^2}+\|(\mu-\bar\mu)_t\|_{L^2}
\right)        \nonumber
\\
&\quad+C\|\psi_t\|_{L^2(\Gamma)}^2+\|\partial_t\mathcal L(\psi)\|_{L^2(\Gamma)}\|{\bf u}_{\boldsymbol\tau}\|_{L^\infty(\Gamma)}\|\nabla_{\boldsymbol\tau}\psi\|_{L^2(\Gamma)}
\nonumber
\\
&\le C\|\mu-\bar\mu\|_{L^2}\|\phi_t\|_{L^2}+C\|{\bf u}\|_{H^1}^\frac{1}{2}\|\nabla\phi\|_{H^1}^\frac{1}{2}\left(\|\phi_t\|_{L^2}+\|(\mu-\bar\mu)_t\|_{L^2}
\right)        \nonumber
\\
&\quad+C\|\partial_t\mathcal L(\psi)\|_{L^2(\Gamma)}\|{\bf u}_{\boldsymbol\tau}\|_{H^1(\Gamma)}\|\nabla_{\boldsymbol\tau}\psi\|_{L^2(\Gamma)}
+C      \nonumber
\\
&\le\frac{1}{2}\|\phi_t\|_{H^1}^2+C\|\mu-\bar\mu\|_{L^2}^2+C\left(\|\nabla{\bf u}\|_{L^2}^2+\|\phi\|_{H^2}^2+\|(\mu-\bar\mu)_t\|_{L^2}^2
\right) \nonumber
\\
&\quad+C\left(\|\partial_t\mathcal L(\psi)\|_{L^2(\Gamma)}^2+\|\nabla_{\boldsymbol\tau}{\bf u}_{\boldsymbol\tau}\|_{L^2(\Gamma)}^2+1
\right),       \nonumber
\end{align}
which together with Gronwall inequality, (\ref{energy}), (\ref{Phi h2}), (\ref{H10}) and (\ref{H13}) yields
\begin{align}
\sup\limits_{0\le t \le T}&\left(
\|\mu-\bar\mu\|_{L^2}^2+\|\mathcal L(\psi)\|_{L^2(\Gamma)}^2
\right)+\int_{0}^{T}\left(\|\phi_t\|_{H^1}^2+\gamma\|\nabla_{\boldsymbol\tau}\psi_t\|_{L^2(\Gamma)}^2
\right) {\rm d}t\le C(T).     \nonumber
\end{align}
 It combined with (\ref{Phi h2}) leads to (\ref{H14}). This completes the proof.
\end{proof}
\begin{Lemma}
\label{final}
Let $({\bf u},\phi,\psi)$ be the smooth solution to $(\ref{NSAC})$--$(\ref{initial condition})$ on $[0,T]$, then it holds that
\begin{align}
\label{H32}
\sup\limits_{0\le t \le T} E_1(t)+\int_{0}^{T} D_1(t) {\rm d}t\le C,
\end{align}
where
\begin{align}
& E_1=\|{\bf u}\|_{H^2}^2+\|\phi\|_{H^3}^2+\gamma\|\psi\|_{H^3(\Gamma)}^2+\|\mu-\bar\mu\|_{H^1}^2
+\|\nabla_{\boldsymbol\tau}\mathcal L(\psi)\|_{L^2}^2,  \nonumber
\\
& D_1=\|{\bf u}\|_{H^3}^2+\|\phi\|_{H^4}^2+\|\phi_t\|_{H^2}^2+\|\psi\|_{H^4(\Gamma)}^2
+\|\psi_t\|_{H^2(\Gamma)}^2+\|\mu-\bar\mu\|_{H^2}^2
+\|\mathcal L(\psi)\|_{H^2(\Gamma)}^2.  \nonumber
\end{align}
 and $C$ is a positive constant that depends on $\Omega$, $\beta^{-1}$, $\gamma^{-1}$ and $T$.
\end{Lemma}
\begin{proof}[\bf Proof.]
The proof will be completed by several steps.
\\
{\it\bfseries Step 1. Estimates of $\|\nabla{\bf u}\|_{L^\infty(0,T;L^2)}$.}
Multiplying $(\ref{NSAC})_1$ by ${\bf u}_t$, and integrating the result over $\Omega$, one has
\begin{align}
\|{\bf u}_t\|_{L^2}^2+\int_\Omega{\bf u}\cdot\nabla{\bf u}\cdot{\bf u}_t{\rm d}x+\int_\Omega\nabla p\cdot{\bf u}_t{\rm d}x-\int_\Omega{\rm div}\mathbb S({\bf u})\cdot{\bf u}_t{\rm d}x=-\int_\Omega{\rm div}(\nabla\phi\otimes\nabla\phi)\cdot{\bf u}_t{\rm d}x.
\nonumber
\end{align}
By using boundary condition, we get
\begin{align}
-\int_\Omega{\rm div}\mathbb S({\bf u})\cdot{\bf u}_t{\rm d}x&=\frac{1}{4}\frac{\rm d}{{\rm d}t}\|\mathbb S({\bf u})\|_{L^2}^2-\int_\Gamma\left(\mathbb S({\bf u})\cdot{\bf n}
\right)\cdot{\bf u}_t{\rm d}S    \nonumber
\\
&=\frac{1}{4}\frac{\rm d}{{\rm d}t}\|\mathbb S({\bf u})\|_{L^2}^2-\int_\Gamma\left(\mathbb S({\bf u})\cdot{\bf n}
\right)_{\boldsymbol\tau}\cdot{\bf u}_{t{\boldsymbol\tau}}{\rm d}S \nonumber
\\
&=\frac{1}{4}\frac{\rm d}{{\rm d}t}\|\mathbb S({\bf u})\|_{L^2}^2-\int_\Gamma\left[
\mathcal L(\psi)\nabla_{\boldsymbol\tau}\psi-\beta{\bf u}_{\boldsymbol\tau}
\right]\cdot{\bf u}_{t{\boldsymbol\tau}}{\rm d}S    \nonumber
\\
&=\frac{1}{2}\frac{\rm d}{{\rm d}t}\left(\frac{1}{2}\|\mathbb S({\bf u})\|_{L^2}^2+\beta\|{\bf u}_{\boldsymbol\tau}\|_{L^2(\Gamma)}^2
\right)-\int_\Gamma\mathcal L(\psi)\nabla_{\boldsymbol\tau}\psi\cdot{\bf u}_{t{\boldsymbol\tau}}{\rm d}S. \nonumber
\end{align}
Thus, recalling (\ref{nabla p}), due to (\ref{H14}) and Lemma \ref{poincare u}, we obtain
\begin{align}
&\frac{1}{2}\frac{\rm d}{{\rm d}t}\left(\frac{1}{2}\|\mathbb S({\bf u})\|_{L^2}^2+\beta\|{\bf u}_{\boldsymbol\tau}\|_{L^2(\Gamma)}^2
\right)+\|{\bf u}_t\|_{L^2}^2               \nonumber
\\
&=-\int_\Omega{\bf u}\cdot\nabla{\bf u}\cdot{\bf u}_t {\rm d}x+\int_\Gamma\mathcal L(\psi)\nabla_{\boldsymbol\tau}\psi\cdot{\bf u}_{t{\boldsymbol\tau}}{\rm d}S-\int_\Omega\Delta\phi\nabla\phi\cdot{\bf u}_t{\rm d}x
\nonumber
\\
&\le\|{\bf u}\|_{L^4}\|\nabla{\bf u}\|_{L^2}\|{\bf u}_t\|_{L^4}+\|\mathcal L(\psi)\|_{L^\infty(\Gamma)}\|\nabla_{\boldsymbol\tau}\psi\|_{L^2(\Gamma)}\|{\bf u}_{t{\boldsymbol\tau}}\|_{L^2(\Gamma)}
+\|\nabla\phi\|_{L^4}\|\Delta\phi\|_{L^2}\|{\bf u}_t\|_{L^4}   \nonumber
\\
&\le C\|{\bf u}\|_{H^1}\|\nabla{\bf u}\|_{L^2}\|{\bf u}_t\|_{H^1}+C\|\mathcal L(\psi)\|_{H^1(\Gamma)}\|\nabla_{\boldsymbol\tau}\psi\|_{L^2(\Gamma)}\|{\bf u}_{t\boldsymbol\tau}\|_{L^2(\Gamma)}+C\|\nabla\phi\|_{H^1}\|\phi\|_{H^2}\|{\bf u}_t\|_{H^1}
\nonumber
\\
&\le C\|\nabla{\bf u}\|_{L^2}^2\left(\|\nabla{\bf u}_t\|_{L^2}^2+1\right)
+C\|\nabla_{\boldsymbol\tau}\mathcal L(\psi)\|_{L^2(\Gamma)}^2+C\|{\bf u}_{t\boldsymbol\tau}\|_{L^2(\Gamma)}^2+C\|\nabla{\bf u}_t\|_{L^2}^2+C,  \nonumber
\end{align}
which together with Gronwall inequality, (\ref{H10}), (\ref{H13}) and  Korn's inequality yields
\begin{align}
\label{H17}
\sup\limits_{0\le t \le T}\left(\|\nabla{\bf u}\|_{L^2}^2+\beta\|{\bf u}_{\boldsymbol\tau}\|_{L^2(\Gamma)}^2
\right)+\int_{0}^{T}\|{\bf u}_t\|_{L^2}^2{\rm d}t\le C(T).
\end{align}
\\
{\it\bfseries Step 2. Higher order estimates of $\phi$ and $\psi$.}
In order to obtain the higher order estimates of ${\bf u}$, we need to derive the higher order regularity of $\phi$ and $\psi$. First, it follows from $(\ref{NSAC})_3$ that
\begin{align}
\|\mu-\bar\mu\|_{H^1}^2&\le C\|\phi_t\|_{H^1}^2+C\|{\bf u}\cdot\nabla\phi\|_{H^1}^2
\nonumber
\\
&\le C+C\left(\|{\bf u}\|_{L^4}^2\|\nabla\phi\|_{L^4}^2+\|\nabla{\bf u}\|_{L^2}^2\|\nabla\phi\|_{L^\infty}^2+\|{\bf u}\|_{L^4}^2\|\nabla^2\phi\|_{L^4}^2
\right)    \nonumber
\\
&\le C+C\left(\|{\bf u}\|_{H^1}^2\|\nabla\phi\|_{H^1}^2+\|\nabla{\bf u}\|_{L^2}^2\|\nabla\phi\|_{L^2}\|\nabla\phi\|_{H^2}+\|{\bf u}\|_{H^1}^2\|\nabla^2\phi\|_{L^2}\|\nabla^2\phi\|_{H^1}
\right)    \nonumber
\\
&\le\varepsilon\|\phi\|_{H^3}^2+C,  \nonumber
\end{align}
which together with (\ref{ty}) and Lemma \ref{ellptic} implies
\begin{align}
\label{H16}
&\|\phi\|_{H^3}^2+\gamma\|\psi\|_{H^3(\Gamma)}^2  \nonumber
\\
&\le C\left(\|\mu-\bar\mu\|_{H^1}^2+\|\phi^3\|_{H^1}^2+\|\phi\|_{H^1}^2+\|\mathcal L(\psi)\|_{H^1(\Gamma)}^2+\|\psi\|_{H^1(\Gamma)}^2+1
\right)   \nonumber
\\
&\le C\left(\|\mu-\bar\mu)\|_{H^1}^2+\|\nabla_{\boldsymbol\tau}\mathcal L(\psi)\|_{L^2(\Gamma)}^2+1
\right)   \nonumber
\\
&\le C\varepsilon\|\phi\||_{H^3}^2+C\left(1+\|\nabla_{\boldsymbol\tau}\mathcal L(\psi)\|_{L^2(\Gamma)}^2
\right),
\end{align}
where we have used Lemma \ref{PGN}, (\ref{energy}), (\ref{H14}) and (\ref{H17}).
Then, taking $C\varepsilon<\frac{1}{2}$ yields
\begin{align}
\label{H28}
\|\phi\|_{H^3}^2+\gamma\|\psi\|_{H^3(\Gamma)}^2\le C\left(1+\|\nabla_{\boldsymbol\tau}\mathcal L(\psi)\|_{L^2(\Gamma)}^2
\right).
\end{align}
\vskip2mm
Besides, recalling (\ref{H11}), which together with (\ref{H13}) and (\ref{H14}) implies
\begin{align}
\label{H24}
&\int_{0}^{T}\left(\|\phi_t\|_{H^2}^2+\gamma\|\psi_t\|_{H^2(\Gamma)}^2
\right) {\rm d}t        \nonumber
\\
&\le C\int_{0}^{T}\left(\|(\mu-\bar\mu)_t\|_{L^2}^2+\|\nabla\phi_t\|_{L^2}^2
+\|\partial_t\mathcal L(\psi)\|_{L^2(\Gamma)}^2
\right) {\rm d}t         \nonumber
\\
&\le C\int_{0}^{T}\left(\|(\mu-\bar\mu)_t\|_{L^2}^2
+\|\partial_t\mathcal L(\psi)\|_{L^2(\Gamma)}^2+1
\right) {\rm d}t \le C(T).
\end{align}
\\
{\it\bfseries Step 3. Estimates of $\|{\bf u}\|_{L^\infty(0,T;H^2)}$.}
We start with recalling the basic interpolation inequality (see \cite{13})
\begin{align}
\|gh\|_{H^{2\theta}(\Gamma)}\le C\|g\|_{H^{\theta+\frac{1}{2}}(\Gamma)}\|h\|_{H^{\theta+\frac{1}{2}}(\Gamma)},
\quad\forall g,h\in H^{\theta+\frac{1}{2}}(\Gamma),\quad\theta\in\left(0,\frac{1}{2}
\right)  \nonumber
\end{align}
Taking the operator $\nabla\times$ to $(\ref{NSAC})_1$, let $\omega=\nabla\times{\bf u}$, we obtain
\begin{equation}
\label{curl}
\begin{cases}
\omega_t-\Delta\omega=-{\bf u}\cdot\nabla\omega-\nabla\times(\Delta\phi\nabla\phi),\quad&{\rm in}\;\Omega,
\\
\omega=\left[(2\kappa-\beta){\bf u}_{\boldsymbol\tau}+\mathcal L(\psi)\nabla_{\boldsymbol\tau}\psi
\right]\cdot{\boldsymbol\tau},\quad&{\rm on}\;\Gamma.
\end{cases}
\end{equation}
Applying the $H^k$-regularity theorem to $(\ref{curl})_1$ with Dirichlet boundary condition $(\ref{curl})_2$, one has
\begin{align}
\|\omega\|_{H^k}&\le C\left(\|-\omega_t+\nabla\times(\Delta\phi\nabla\phi)\|_{H^{k-2}}+\left\Vert
\left[(2\kappa-\beta){\bf u}_{\boldsymbol\tau}+\mathcal L(\psi)\nabla_{\boldsymbol\tau}\psi
\right]\cdot{\boldsymbol\tau}
\right\Vert_{H^{k-\frac{1}{2}}(\Gamma)}+\|\omega\|_{L^2}
\right) \nonumber
\\
&\le C(\kappa,\beta)\left(\|\nabla{\bf u}_t\|_{H^{k-2}}+\|\nabla\times(\Delta\phi\nabla\phi)\|_{H^{k-2}}+\|{\bf u}_{\boldsymbol\tau}\|_{H^{k-\frac{1}{2}}(\Gamma)}\right) \nonumber
\\
&\quad+C(\kappa,\beta)\left(\|\mathcal L(\psi)\nabla_{\boldsymbol\tau}\psi\|_{H^{k-\frac{1}{2}}(\Gamma)}
+\|\nabla{\bf u}\|_{L^2}
\right). \nonumber
\end{align}
Thus, using Lemma \ref{regularity u} and Lemma \ref{poincare u}, it follows from interpolation inequality for $\theta=\frac{1}{4}$ that
\begin{align}
\label{H27}
\|{\bf u}\|_{H^2}^2&\le C\left(\|\omega\|_{H^1}^2+\|{\bf u}\|_{H^1}^2
\right)  \nonumber
\\
&\le C(\kappa,\beta)\left(\|\nabla{\bf u}_t\|_{H^{-1}}^2+\|\nabla\times(\Delta\phi\nabla\phi)\|_{H^{-1}}^2+\|{\bf u}_{\boldsymbol\tau}\|_{H^\frac{1}{2}(\Gamma)}^2\right) \nonumber
\\
&\quad+C(\kappa,\beta)\left(\|\mathcal L(\psi)\nabla_{\boldsymbol\tau}\psi\|_{H^\frac{1}{2}(\Gamma)}^2
+\|\nabla{\bf u}\|_{L^2}^2
\right)  \nonumber
\\
&\le C(\kappa,\beta)\left(\|{\bf u}_t\|_{L^2}^2+\|\Delta\phi\nabla\phi\|_{L^2}^2+\|{\bf u}\|_{H^1}^2+\|\mathcal L(\psi)\|_{H^\frac{3}{4}(\Gamma)}^2\|\nabla_{\boldsymbol\tau}\psi\|_{H^\frac{3}{4}(\Gamma)}^2
+\|\nabla{\bf u}\|_{L^2}^2\right)   \nonumber
\\
&\le C(\kappa,\beta)\left(\|{\bf u}_t\|_{L^2}^2+\|\nabla{\bf u}\|_{L^2}^2+\|\Delta\phi\|_{L^4}^2\|\nabla\phi\|_{L^4}^2+\|\mathcal L(\psi)\|_{H^1(\Gamma)}^2\|\nabla_{\boldsymbol\tau}\psi\|_{H^1(\Gamma)}^2
\right)  \nonumber
\\
&\le C(\kappa,\beta)\left(\|\phi\|_{H^3}^2
+\|\nabla_{\boldsymbol\tau}\mathcal L(\psi)\|_{L^2(\Gamma)}^2+1
\right)  \nonumber
\\
&\le C(\kappa,\beta)\left(\|\nabla_{\boldsymbol\tau}\mathcal L(\psi)\|_{L^2(\Gamma)}^2+1
\right),
\end{align}
where we have used (\ref{H13}), (\ref{H14}), (\ref{H17}) (\ref{H16}) and (\ref{poincare u}).
Thanks to $(\ref{tangential 1})_3$, we infer that
\begin{align}
\label{H19}
\|\nabla_{\boldsymbol\tau}\mathcal L(\psi)\|_{L^2(\Gamma)}^2&\le C\left(
\|\nabla_{\boldsymbol\tau}\psi_t\|_{L^2(\Gamma)}^2
+\|\nabla_{\boldsymbol\tau}{\bf u}_{\boldsymbol\tau}\|_{L^2(\Gamma)}^2\|\nabla_{\boldsymbol\tau}\psi\|_{L^\infty(\Gamma)}^2
+\|{\bf u}_{\boldsymbol\tau}\|_{L^\infty(\Gamma)}^2\|\nabla^2_{\boldsymbol\tau}\psi\|_{L^2(\Gamma)}^2
\right)  \nonumber
\\
&\le C\left(
\|\nabla_{\boldsymbol\tau}\psi_t\|_{L^2(\Gamma)}^2
+\|\nabla_{\boldsymbol\tau}{\bf u}\|_{H^\frac{1}{2}}^2\|\nabla_{\boldsymbol\tau}\psi\|_{H^1(\Gamma)}^2
+\|{\bf u}\|_{H^\frac{3}{2}}^2\|\nabla^2_{\boldsymbol\tau}\psi\|_{L^2(\Gamma)}^2
\right)  \nonumber
\\
&\le C\left(\|\nabla{\bf u}\|_{L^2}\|\nabla{\bf u}\|_{H^1}+\|{\bf u}\|_{H^1}\|{\bf u}\|_{H^2}+1
\right)     \nonumber
\\
&\le\varepsilon\|\nabla^2{\bf u}\|_{L^2}^2+C(T).
\end{align}
Then, substituting (\ref{H19}) into (\ref{H27}), after choosing $C\varepsilon<\frac{1}{2}$, we conclude that
\begin{align}
\label{H18}
\sup\limits_{0\le t \le T}\left(\|{\bf u}\|_{H^2}^2+\|\nabla_{\boldsymbol\tau}\mathcal L(\psi)\|_{L^2(\Gamma)}^2\right)\le C(T).
\end{align}
\\
{\it\bfseries Step 4. Estimates of $\|\phi\|_{L^\infty(0,T;H^3)}$ and $\|\psi\|_{L^\infty(0,T;H^3(\Gamma))}$.}
With the regularity of ${\bf u}$ at hand, we can get the bounds of $\|\phi\|_{L^\infty(0,T;H^3)}$ and $\|\psi\|_{L^\infty(0,T;H^3(\Gamma))}$. It follows from (\ref{H28}) and (\ref{H18}) that
\begin{align}
\label{H20}
&\|\phi\|_{H^3}^2+\gamma\|\psi\|_{H^3(\Gamma)}^2\le C\left(1+\|\nabla_{\boldsymbol\tau}\mathcal L(\psi)\|_{L^2(\Gamma)}^2
\right)\le C(T),
\end{align}
Taking the operator $\nabla_{\boldsymbol\tau}^2$ to $(\ref{dynamic})_3$, one has
\begin{align}
\nabla_{\boldsymbol\tau}^2\mathcal L(\psi)=-\nabla_{\boldsymbol\tau}^2\psi_t-{\bf u}_{\boldsymbol\tau}\cdot\nabla_{\boldsymbol\tau}^3\psi-\nabla_{\boldsymbol\tau}{\bf u}_{\boldsymbol\tau}\cdot\nabla_{\boldsymbol\tau}^2\psi-\nabla_{\boldsymbol\tau}^2{\bf u}_{\boldsymbol\tau}\cdot\nabla_{\boldsymbol\tau}\psi, \nonumber
\end{align}
which together with (\ref{H18}), (\ref{H20}) and Lemma \ref{poincare u} implies
\begin{align}
\label{H23}
\|\mathcal L(\psi)\|_{H^2(\Gamma)}^2&\le C\left(
\|\nabla_{\boldsymbol\tau}^2\psi_t\|_{L^2(\Gamma)}^2+\|{\bf u}_{\boldsymbol\tau}\|_{L^\infty(\Gamma)}^2\|\nabla_{\boldsymbol\tau}^3\psi\|_{L^2(\Gamma)}^2
+\|\nabla_{\boldsymbol\tau}{\bf u}_{\boldsymbol\tau}\|_{L^2(\Gamma)}^2\|\nabla_{\boldsymbol\tau}^2\psi\|_{L^\infty(\Gamma)}^2
\right)  \nonumber
\\
&\quad+C\left(\|\nabla_{\boldsymbol\tau}^2{\bf u}_{\boldsymbol\tau}\|_{L^2(\Gamma)}^2\|\nabla_{\boldsymbol\tau}\psi\|_{L^\infty(\Gamma)}^2+1
\right) \nonumber
\\
&\le C\left(
\|\nabla_{\boldsymbol\tau}^2\psi_t\|_{L^2(\Gamma)}^2+\|{\bf u}_{\boldsymbol\tau}\|_{H^1(\Gamma)}^2\|\nabla_{\boldsymbol\tau}^3\psi\|_{L^2(\Gamma)}^2
+\|\nabla_{\boldsymbol\tau}{\bf u}_{\boldsymbol\tau}\|_{L^2(\Gamma)}^2\|\nabla_{\boldsymbol\tau}^2\psi\|_{H^1(\Gamma)}^2
\right)  \nonumber
\\
&\quad+C\left(\|\nabla_{\boldsymbol\tau}^2{\bf u}_{\boldsymbol\tau}\|_{L^2(\Gamma)}^2\|\nabla_{\boldsymbol\tau}\psi\|_{H^1(\Gamma)}^2+1
\right) \nonumber
\\
&\le C\left(
\|\nabla_{\boldsymbol\tau}^2\psi_t\|_{L^2(\Gamma)}^2+\|{\bf u}\|_{H^2}^2\|\psi\|_{H^3(\Gamma)}^2+\|\nabla^2{\bf u}\|_{L^2}\|\nabla^2{\bf u}\|_{H^1}\|\nabla_{\boldsymbol\tau}\psi\|_{H^1(\Gamma)}^2+1
\right)  \nonumber
\\
&\le C\left(\|\nabla_{\boldsymbol\tau}^2\psi_t\|_{L^2(\Gamma)}^2+\|\psi\|_{H^3(\Gamma)}^2+\|\nabla^2{\bf u}\|_{L^2}\|\nabla^2{\bf u}\|_{H^1}+1 \right)  \nonumber
\\
&\le\varepsilon\|{\bf u}\|_{H^3}^2+C\left(\|\nabla_{\boldsymbol\tau}^2\psi_t\|_{L^2(\Gamma)}^2+1
\right).
\end{align}
\vskip2mm
We use the same method as in step 3 above to estimate $\|\nabla^3{\bf u}\|_{L^2(0,T;L^2)}$. It follows from Lemma \ref{regularity u} and (\ref{L5.1}) for the case $r=s_1=s_2=\frac{3}{2}>\frac{2}{2}$ that
\begin{align}
\label{H21}
\|{\bf u}\|_{H^3}^2&\le C\left(\|\omega\|_{H^2}^2+\|{\bf u}\|_{H^2}^2
\right)    \nonumber
\\
&\le C(\kappa,\beta)\left(\|\nabla{\bf u}_t\|_{L^2}^2+\|\nabla(\Delta\phi\nabla\phi)\|_{L^2}^2+\|{\bf u}_{\boldsymbol\tau}\|_{H^\frac{3}{2}(\Gamma)}^2+\|\mathcal L(\psi)\nabla_{\boldsymbol\tau}\psi\|_{H^\frac{3}{2}(\Gamma)}^2+\|{\bf u}\|_{H^2}^2
\right)   \nonumber
\\
&\le C(\kappa,\beta)\left(\|\nabla{\bf u}_t\|_{L^2}^2+\|\nabla\Delta\phi\|_{L^2}^2\|\nabla\phi\|_{L^\infty}^2
+\|\Delta\phi\|_{L^4}^2\|\nabla^2\phi\|_{L^4}^2
\right)   \nonumber
\\
&\quad+C(\kappa,\beta)\left(\|{\bf u}\|_{H^2}^2+\|\mathcal L(\psi)\|_{H^\frac{3}{2}(\Gamma)}^2
\|\nabla_{\boldsymbol\tau}\psi\|_{H^\frac{3}{2}(\Gamma)}^2
\right)  \nonumber
\\
&\le C(\kappa,\beta)\left(\|\nabla{\bf u}_t\|_{L^2}^2+\|\nabla\Delta\phi\|_{L^2}^2\|\nabla\phi\|_{H^2}^2
+\|\Delta\phi\|_{H^1}^2\|\nabla^2\phi\|_{H^1}^2
\right)   \nonumber
\\
&\quad+C(\kappa,\beta)\left(\|{\bf u}\|_{H^2}^2+\|\mathcal L(\psi)\|_{H^1(\Gamma)}\|\mathcal L(\psi)\|_{H^2(\Gamma)}
\|\nabla_{\boldsymbol\tau}\psi\|_{H^1(\Gamma)}\|\nabla_{\boldsymbol\tau}\psi\|_{H^2(\Gamma)}
\right)  \nonumber
\\
&\le C(\kappa,\beta)\left(\|\nabla{\bf u}_t\|_{L^2}^2+\|\mathcal L(\psi)\|_{H^2(\Gamma)}^2+1
\right),
\end{align}
where we have used (\ref{H18}) and (\ref{H20}).
Then, substituting (\ref{H21}) into (\ref{H23}), taking $C\varepsilon<\frac{1}{2}$, we infer that
\begin{align}
\label{H30}
\int_{0}^{T}\|{\mathcal L}(\psi)\|_{H^2(\Gamma)}^2{\rm d}t\le C\int_{0}^{T}\left(\|\nabla_{\boldsymbol\tau}^2\psi_t\|_{L^2(\Gamma)}^2
+\|\nabla{\bf u}_t\|_{L^2}^2+1
\right) {\rm d}t \le C(T).
\end{align}
So we can get the result
\begin{align}
\label{H25}
\int_{0}^{T}\|{\bf u}\|_{H^3}^2{\rm d}t\le C\int_{0}^{T}\left(\|\nabla{\bf u}_t\|_{L^2}^2+\|\mathcal L(\psi)\|_{H^2(\Gamma)}^2\right) {\rm d}t\le C(T),
\end{align}
where we have used (\ref{H13}) and (\ref{H24}).
\vskip2mm
Finally, differentiating $(\ref{NSAC})_3$ with respect to $x$ twice leads to
\begin{align}
\nabla^2\phi_t+\nabla^2{\bf u}\cdot\nabla\phi+2\nabla{\bf u}\cdot\nabla^2\phi+{\bf u}\cdot\nabla^3\phi=\nabla^2(\bar\mu-\mu),   \nonumber
\end{align}
which together with (\ref{H24}), (\ref{H18}) and (\ref{H20}) implies
\begin{align}
\label{mu2}
&\int_{0}^{T}\|\nabla^2(\mu-\bar\mu)\|_{L^2}^2{\rm d}t \nonumber
\\
&\le C\int_{0}^{T}\left(\|\nabla^2\phi_t\|_{L^2}^2+\|\nabla^2{\bf u}\|_{L^2}^2\|\nabla\phi\|_{L^\infty}^2+\|\nabla{\bf u}\|_{L^4}^2\|\nabla^2\phi\|_{L^4}^2+\|{\bf u}\|_{L^\infty}^2\|\nabla^3\phi\|_{L^2}^2
\right) {\rm d}t   \nonumber
\\
&\le C\int_{0}^{T}\left(\|\nabla^2\phi_t\|_{L^2}^2+\|\nabla^2{\bf u}\|_{L^2}^2\|\nabla\phi\|_{H^2}^2+\|\nabla{\bf u}\|_{H^1}^2\|\nabla^2\phi\|_{H^1}^2+\|{\bf u}\|_{H^2}^2\|\nabla^3\phi\|_{L^2}^2
\right) {\rm d}t   \nonumber
\\
&\le C\int_{0}^{T}\left(\|\nabla^2\phi_t\|_{L^2}^2+\|{\bf u}\|_{H^2}^2\|\phi\|_{H^3}^2
\right) {\rm d}t\le C(T).
\end{align}
Moreover, by exploiting (\ref{ty}) and Lemma \ref{estimates}, we find that
\begin{align}
\label{H31}
&\int_{0}^{T}\left(\|\phi\|_{H^4}^2+\gamma\|\psi\|_{H^4(\Gamma)}^2
\right) {\rm d}t  \nonumber
\\
&\le C\int_{0}^{T}\left(\|\mu-\bar\mu\|_{H^2}^2+\|\phi^3\|_{H^2}^2+\|\phi\|_{H^2}^2+\|\mathcal L(\psi)\|_{H^2(\Gamma)}^2+\|\psi\|_{H^2(\Gamma)}^2+1
\right) {\rm d}t\le C(T).
\end{align}
Collecting (\ref{H18}), (\ref{H20}), $(\ref{H30})$-$(\ref{H31})$ all together, we arrive at (\ref{H32}). Hence, we complete the proof of Lemma \ref{final}.
\end{proof}
\setcounter{equation}{0}
\section{NSAC system with GNBC and relaxation boundary condition}
\label{section4}
In this section, we will establish the well-posedness of the NSAC system with GNBC and relaxation boundary condition in 3D case.
The NSAC equations
\begin{equation}
\label{NSAC1}
\begin{cases}
{\bf u}_t+({\bf u}\cdot\nabla){\bf u}+\nabla p ={\rm div}{\mathbb S}({\bf u})-{\rm div}(\nabla\phi\otimes\nabla\phi) ,&\rm in\quad\Omega\times(0,T),
\\
{\rm div}{\bf u}=0,&\rm in\quad\Omega\times(0,T),
\\
\phi_t+{\bf u}\cdot\nabla\phi=\bar\mu-\mu,&\rm in\quad\Omega\times(0,T),
\\
\mu=-\Delta\phi+f(\phi),&\rm in\quad\Omega\times(0,T)
\end{cases}
\end{equation}
is supplemented by the GNBC and relaxation boundary condition, i.e. $L(\phi)=\partial_{\bf n}\phi+\gamma_{fs}^\prime(\phi)$,
\begin{equation}
\label{boundary}
\begin{cases}
{\bf u}\cdot{\bf n}=0,&\rm on\quad\Gamma\times(0,T),
\\
\beta{\bf u}_{\boldsymbol {\tau}}+({\mathbb S}({\bf u})\cdot{\bf n})_{\boldsymbol{\tau}}=\left(\partial_{\bf n}\phi+\gamma_{fs}^\prime(\psi)
\right)\nabla_{\boldsymbol\tau}\psi,&\rm on\quad\Gamma\times(0,T),
\\
\psi_t+{\bf u}_{\boldsymbol{\tau}}\cdot\nabla_{\boldsymbol{\tau}}\psi=
-\partial_{\bf n}\phi-\gamma_{fs}^\prime(\psi),&\rm on\quad\Gamma\times(0,T),
\\
\phi\big|_\Gamma=\psi,&\rm on\quad(0,T),
\end{cases}
\end{equation}
and the initial conditions
\begin{equation}
\label{initial1}
({\bf u},\phi)\big|_{t=0}=({\bf u}_0,\phi_0),\;{\rm in}\;\Omega,\quad\psi\big|_{t=0}=\phi_0\big|_\Gamma=\psi_0,\;\rm on\;\Gamma.
\end{equation}

Before introduce the main results, we first give the definition of the following energy functional
\begin{align}
\tilde{\mathcal E}(t)&=\|{\bf u}\|_{H^2}^2+\|{\bf u}_t\|_{L^2}^2+\|\phi^2-1\|_{H^2}^2+\|\nabla\phi\|_{H^2}^2+\|\phi_t\|_{H^1}^2
+\|\mu-\bar\mu\|_{H^1}^2+ \|L(\psi)\|_{H^1(\Gamma)}^2  , \nonumber
\end{align}
and dissipative functional
\begin{align}
\tilde{\mathcal D}(t)&=\|{\bf u}\|_{H^3}^2+\|{\bf u}_t\|_{H^1}^2+\|\nabla_{\boldsymbol\tau}\phi\|_{H^2}^2+\|\nabla_{\boldsymbol\tau}^2\Delta\phi\|_{L^2}^2
+\|\phi_t\|_{H^2}^2
\nonumber
\\
&\quad+\|\nabla_{\boldsymbol\tau}^k(\mu-\bar\mu)\|_{L^2}^2+\|(\mu-\bar\mu)_t\|_{L^2}^2+\|\nabla_{\boldsymbol\tau}\partial_{\bf n}\phi\|_{H^1(\Gamma)}^2+\|\partial_{\bf n}\phi_t\|_{L^2(\Gamma)}^2,\quad k=1,2.   \nonumber
\end{align}

Our main results of this section are stated as follows.
\begin{Theorem}
\label{relaxation local}
Let $\Omega=\mathbb{T}^2\times(-1,1)\subset\mathbb R^3$, assume that $({\bf u}_0,\phi_0, \psi_{0} )\in{\mathbb D}_2$, then there exist $T_0$ small enough and a unique local-in-time strong solution $({\bf u},\phi, \psi )$ to the initial boundary value problem $(\ref{NSAC1})$--$(\ref{initial1})$ on $[0,T_0]$ for $T\in(0,T_0]$ satisfying
\begin{align}
&{\bf u}\in L^\infty(0,T_0;H^2)\cap L^2(0,T_0;H^3),\;{\bf u}_t\in L^\infty(0,T_0;L^2)\cap L^2(0,T_0;H^1),                                 \nonumber
\\
&\phi\in L^\infty(0,T_0;H^3),\hspace{7em}\phi_t\in L^\infty(0,T_0;H^1)\cap L^2(0,T_0 ;H^2),
\nonumber
\\
& \psi\in L^\infty(0,T_0;H^{\frac{5}{2}}(\Gamma)),\hspace{5.5em}\phi_t\in L^\infty(0,T_0;H^{\frac{1}{2}}(\Gamma)) \cap L^2(0,T_0 ;H^{\frac{3}{2}}(\Gamma)).
\nonumber
\end{align}
\end{Theorem}
\begin{Theorem}
\label{Th rg}
Let $\Omega=\mathbb{T}^2\times(-1,1)\subset\mathbb R^3$, assume that the initial values satisfy $({\bf u}_0,\phi_0, \psi_{0} )\in{\mathbb D}_3$,   $ \phi^{2}_{0} - 1  \in L^{2} $ and
\begin{align}
\label{small initial}
\|{\bf u}_0\|_{H^2}^2+\|\nabla\phi_0\|_{H^2}^2+\|\mu_0-\bar\mu_0\|_{H^1}^2+\|\phi_0^2-1\|_{L^2}^2
+|\nu\cos\theta_s|\le \varepsilon_0,
\end{align}
where $\varepsilon_0$ is a small positive constant.
Then there exists a unique global-in-time strong solution $({\bf u},\phi,  \psi )$ to the initial boundary value problem $(\ref{NSAC1})$--$(\ref{initial1})$ for all $t\ge 0$ ,
\begin{align}
\tilde{\mathcal E}(t)+\int_{0}^{t} {\tilde{\mathcal D}(s) } {\rm d}s\le C\varepsilon_0,
\end{align}
Moreover, there exists a positive constant $\alpha$ such that
\begin{align}
\label{big time}
\|{\bf u}(\cdot,t)\|_{L^2}^2+ \| L(\psi)(\cdot,t)\|_{L^2(\Gamma)}^2  +\|(\mu-\bar\mu)(\cdot,t)\|_{L^2}^2\le C\tilde{\mathcal E}(0)e^{-\alpha t},\quad for\;all\;t\ge0.
\end{align}
\end{Theorem}
\begin{Remark}
Since the global estimates for solutions to the problem $(\ref{NSAC})$--$(\ref{initial condition})$ depend on the interface coefficient $\gamma$, the global existence result cannot be extended to the solutions governed by the relaxation boundary condition through limiting processes.
Consequently, the global existence of solutions to the problem $(\ref{NSAC1})$--$(\ref{initial1})$ with large initial data in 2D case remains an open problem even in channels.
\end{Remark}
\begin{Remark}
\label{remark boundary}
As in \cite{D-L-L-Y}, rewrite the equation $(\ref{NSAC1})_3$ as follows
\begin{align}
\label{E2}
\phi_t+{\bf u}_{\boldsymbol\tau}\cdot\nabla_{\boldsymbol\tau}\phi+({\bf u}\cdot{\bf n})\partial_{\bf n}\phi=\bar\mu-\mu,\quad \rm in \;\Omega,
\end{align}
where ${\bf u}_{\boldsymbol\tau}=(u_1,u_2)$, $\nabla_{\boldsymbol\tau}=(\partial_x,\partial_y)$.
Then, taking the trace for $(\ref{E2})$, we have
\begin{align}
\label{E9}
\psi_t+{\bf u}_{\boldsymbol\tau}\cdot\nabla_{\boldsymbol\tau}\psi=\bar\mu-\mu,\quad \rm on \;\Gamma,
\end{align}
where we have used ${\bf u}\cdot{\bf n}=0,\rm on\; \Gamma$. Combining $(\ref{E9})$ with $(\ref{boundary})_3$, we obtain
\begin{align}
\label{bb1}
\bar\mu-\mu=-L(\psi)=-\partial_{\bf n}\phi-\gamma_{fs}^\prime(\psi),\quad \rm on \;\Gamma,
\end{align}
which is crucial for dealing with the boundary estimates in this section.
\end{Remark}
\subsection{Local well-posedness}

\subsubsection{\bf The \texorpdfstring{$\delta$}{}-approximate problem}\quad
In order to find a solution $({\bf u},\phi,  \psi)$ to the original problem $(\ref{NSAC1})$--$(\ref{initial1})$, we introduce the {\it $\delta$-approximate problem} as follows
\begin{equation}
\label{delta NSAC_1}
\begin{cases}
{\bf u}_t+({\bf u}\cdot\nabla){\bf u}+\nabla p ={\rm div}{\mathbb S}({\bf u})-{\rm div}(\nabla\phi\otimes\nabla\phi) ,&\rm in\quad\Omega\times(0,T),
\\
{\rm div}{\bf u}=0,&\rm in\quad\Omega\times(0,T),
\\
\phi_t-\delta\Delta_{\boldsymbol\tau}\phi=\bar\mu-\mu-{\bf u}\cdot\nabla\phi,&\rm in\quad\Omega\times(0,T),
\\
\mu=-\Delta\phi+f(\phi),&\rm in\quad\Omega\times(0,T),
\end{cases}
\end{equation}
where $\Delta_ {\boldsymbol\tau}\phi=(\partial_x^2+\partial_y^2)\phi$, with the boundary conditions
\begin{equation}
\label{delta NBC_1}
\begin{cases}
{\bf u}\cdot{\bf n}=0,&\rm on\quad\Gamma\times(0,T),
\\
\beta{\bf u}_{\boldsymbol {\tau}}+({\mathbb S}({\bf u})\cdot{\bf n})_{\boldsymbol{\tau}}=\left(\partial_{\bf n}\phi+\gamma_{fs}^\prime(\psi)
\right)\nabla_{\boldsymbol\tau}\psi,&\rm on\quad\Gamma\times(0,T),
\\
\psi_t-\delta\Delta_{\boldsymbol\tau}\psi=-{\bf u}_{\boldsymbol{\tau}}\cdot\nabla_{\boldsymbol{\tau}}\psi
-\partial_{\bf n}\phi-\gamma_{fs}^\prime(\psi),&\rm on\quad\Gamma\times(0,T),
\\
\phi\big|_\Gamma=\psi,&\rm on\quad(0,T),
\end{cases}
\end{equation}
and the initial conditions (\ref{initial1}). Here $0<\delta\le1$.
\begin{Remark}
The term $\delta\Delta_{\boldsymbol\tau}\phi$ added in equation $(\ref{delta NSAC_1})_3$ is designed to counteract the surface diffusion in boundary condition $(\ref{delta NBC_1})_3$. So we need to constrain the problem in 3D channels.
\end{Remark}

The existence and uniqueness of strong solutions to the $\delta$-approximate problem $(\ref{delta NSAC_1})$, $(\ref{delta NBC_1})$, $(\ref{initial1})$ parallel with the proof presented in Section 3.
Therefore, we omit the details of the proof and directly present the result as below.
\begin{Theorem}
\label{delta problem}
Let $\Omega=\mathbb{T}^2\times(-1,1)\subset\mathbb R^3$, assume that $({\bf u}_0,\phi_0, \psi_{0} )\in{\mathbb D}_2$. Then there exists a unique local-in-time strong solution $({\bf u},\phi,  \psi )$ to the initial boundary value problem $(\ref{delta NSAC_1})$, $(\ref{delta NBC_1})$, $(\ref{initial1})$ such that
\begin{align}
&{\bf u}\in L^\infty(0,T^*;H^2)\cap L^2(0,T^*;H^3),\hspace{5.5em}{\bf u}_t\in L^\infty(0,T^*;L^2)\cap L^2(0,T^*;H^1),
\nonumber \\
&\phi\in L^\infty(0,T^*;H^3)\cap L^2(0,T^*;H^4),\hspace{5.5em} \phi_t\in L^\infty(0,T^*;H^1)\cap L^2(0,T^*;H^2), \nonumber \\
& \sqrt{\delta}\psi\in L^\infty(0,T^*;H^{\frac52}(\Gamma))\cap L^2(0,T^*;H^{\frac72}(\Gamma)),\quad \psi_t\in L^\infty(0,T^*;H^{\frac12}(\Gamma))\cap L^2(0,T^*;H^{\frac32}(\Gamma)), \nonumber
\end{align}
for some $T^*>0$, which may depends on the initial data and $\delta$.
\end{Theorem}
\begin{remark}
Thanks to  the channel domain, we can define fractional derivatives by Fourier transform. The regularity of $\psi$ in Theorem \ref{delta problem} can be further refined as above.
\end{remark}

\subsubsection{\bf \texorpdfstring{$\delta$}{}-independent estimates}\quad
Define the energy functional
\begin{align}
{\mathcal E}(t)&=1+\|\phi_t\|_{H^1}^2+\|\phi\|_{H^3}^2+\|{\bf u}_t\|_{L^2}^2+\|{\bf u}\|_{H^2}^2 \nonumber
\\
&\quad+\delta\|\nabla\nabla_{\boldsymbol\tau}^2\phi\|_{L^2}^2
+\|\nabla_{\boldsymbol\tau}L(\psi)\|_{L^2(\Gamma)}^2
+\|\nabla_{\boldsymbol\tau}(\mu-\bar\mu)\|_{L^2}^2, \nonumber
\end{align}
and dissipative functional
\begin{align}
{\mathcal D}(t)&=\|\phi_t\|_{H^2}^2+\|{\bf u}_t\|_{H^1}^2
+\|{\bf u}\|_{H^3}^2+\delta\|\nabla\nabla_{\boldsymbol\tau}\phi_t\|_{L^2}^2+\delta\|\nabla\nabla_{\boldsymbol\tau} \phi\|_{L^2}^2+  \delta\|\nabla\nabla_{\boldsymbol\tau}^3 \psi\|_{L^2(\Gamma)}^2
\nonumber
\\
&\quad+\|\partial_{\bf n}\phi_t\|_{L^2(\Gamma)}^2
+\|\nabla_{\boldsymbol\tau}\partial_{\bf n}\phi\|_{H^1(\Gamma)}^2
+\|(\mu-\bar\mu)_t\|_{L^2}^2
+\|\nabla_{\boldsymbol\tau}^2(\mu-\bar\mu)\|_{L^2}^2.   \nonumber
\end{align}
We will show that the energy functional ${\mathcal E}(t)$ remains uniformly bounded over a $\delta$-independent existence time interval.
\begin{Proposition}
\label{TH im}
Under the assumptions of Theorem $\ref{relaxation local}$. Then there exists a time $T$ independent of $\delta$ such that
\begin{align}
\sup\limits_{0\le t \le T}{\mathcal E}(t)+\int_{0}^{T}{\mathcal D}(t) {\rm d}t\le 2C_0,  \nonumber
\end{align}
where $C_0$ is a constant that depend on $\Omega$, $\beta$, $\hat{\mathcal E}(0)$ and the initial value, but not on $\delta$.
\vskip2mm
Moreover, under the assumptions of Theorem $\ref{Th rg}$, if there exists a universal positive constants $\sigma_0$ such that $\tilde{\mathcal E}(0)\le\sigma_0$, then the solution $({\bf u},\phi)$ to $(\ref{NSAC1})$--$(\ref{initial1})$ on $[0,T]$ satisfying
\begin{align}
\sup\limits_{0\le t \le T}\tilde{\mathcal E}(t)+\int_{0}^{T}\tilde{\mathcal D}(t) {\rm d}t\le C_1\tilde{\mathcal E}(0). \nonumber
\end{align}
\end{Proposition}
\vskip2mm
Proposition $\ref{TH im}$ is based on the following lemma $\ref{Lemma4.1}-\ref{Lemma4.3}$.
\begin{Lemma}
\label{Lemma4.1}
Let $({\bf u},\phi, \psi )$ be the smooth solution to $(\ref{delta NSAC_1})$, $(\ref{delta NBC_1})$, $(\ref{initial1})$ on $[0,T]$, then it holds that
\begin{align}
\label{E6}
\sup\limits_{0\le t \le T}\hat{\mathcal E}(t)+\int_{0}^{T}\hat{\mathcal D}(t) {\rm d}t
\le C\hat{\mathcal E}(0),
\end{align}
where
\begin{align}
&\hat{\mathcal E}(t)=\|{\bf u}\|_{L^2}^2+\|\phi\|_{H^1}^2+\frac{1}{2}\|\phi^2-1\|_{L^2}^2
+ 2\int_\Gamma\gamma_{fs}(\psi){\rm d}S   ,               \nonumber
\\
&\hat{\mathcal D}(T)=\|{\mathbb S}({\bf u})\|_{L^2}^2+\beta\|{\bf u}_{\boldsymbol\tau}\|_{L^2(\Gamma)}^2+\|\mu-\bar\mu\|_{L^2}^2
+\delta\|\nabla_{\boldsymbol\tau}\nabla\phi\|_{L^2}^2
+ \|L(\psi)\|_{L^2(\Gamma)}^2    ,                      \nonumber
\end{align}
and $C$ is constant that depends on $\Omega$, $\beta$, $\hat{\mathcal E}(0)$ and the initial value, but independent of $\delta$.
\end{Lemma}

\begin{proof}[\bf Proof.]
First, multiplying $(\ref{delta NSAC_1})_3$ by $(\mu-\bar\mu)$ and integrating the result over $\Omega$, one has
\begin{align}
\label{E1}
\int_\Omega\phi_t(\mu-\bar\mu){\rm d}x-\int_\Omega\delta
\Delta_{\boldsymbol\tau}\phi(\mu-\bar\mu){\rm d}x=-\|\mu-\bar\mu\|_{L^2}^2
-\int_\Omega{\bf u}\cdot\nabla\phi(\mu-\bar\mu)
{\rm d}x.
\end{align}
Recalling (\ref{phit}) and using integration by parts gives
\begin{align}
\label{E7}
\int_\Omega\phi_t {\rm d}x=\int_\Omega\delta\Delta_{\boldsymbol\tau}\phi {\rm d}x
+\int_\Omega(\bar\mu-\mu) {\rm d}x-\int_\Omega{\bf u}\cdot\nabla\phi {\rm d}x=0.
\end{align}Here we will only elaborate on the second term on the left of the (\ref{E1}), and the rest can be seen in Lemma \ref{Energy}. It follows from integration by parts that
\begin{align}
-\int_\Omega\delta
\Delta_{\boldsymbol\tau}\phi(\mu-\bar\mu){\rm d}x&=-\int_\Omega\delta
\Delta_{\boldsymbol\tau}\phi\mu {\rm d}x+\bar\mu\int_\Omega\delta
\Delta_{\boldsymbol\tau}\phi {\rm d}x                         \nonumber
\\
&=-\int_\Omega\delta
\Delta_{\boldsymbol\tau}\phi(-\Delta\phi+f){\rm d}x           \nonumber
\\
&=\int_\Omega\delta
\Delta_{\boldsymbol\tau}\phi\Delta\phi {\rm d}x-\delta\int_\Omega
\Delta_{\boldsymbol\tau}\phi f {\rm d}x                       \nonumber
\\
&=\delta\|\nabla_{\boldsymbol\tau}\nabla\phi\|_{L^2}^2
+ \int_\Gamma\delta\Delta_{\boldsymbol\tau}\psi\partial_{\bf n}\phi {\rm d}S   +\delta\int_\Omega(3\phi^2-1)|\nabla_{\boldsymbol\tau}\phi|^2 {\rm d}x.
\nonumber
\end{align}
Then, (\ref{E1}) could be rewritten as
\begin{align}
\label{E4}
&\frac{1}{2}\frac{\rm d}{{\rm d}t}\left(\|\nabla\phi\|_{L^2}^2
+\frac{1}{2}\|\phi^2-1\|_{L^2}^2
\right)
+\|\mu-\bar\mu\|_{L^2}^2+\delta\|\nabla_{\boldsymbol\tau}\nabla\phi\|_{L^2}^2
\nonumber
\\
&= \int_\Gamma(\psi_t-\delta\Delta_{\boldsymbol\tau}\psi)\partial_{\bf n}\phi {\rm d}S +\delta\int_\Omega(1-3\phi^2)|\nabla_{\boldsymbol\tau}\phi|^2 {\rm d}x+\int_\Omega{\rm div}(\nabla\phi\otimes\nabla\phi)\cdot{\bf u}{\rm d}x,
\end{align}
whereas the first term on the right of (\ref{E4}) can be rewritten from $L(\psi)=\partial_{\bf n}\phi+\gamma_{fs}^\prime(\psi)$ to
\begin{align}
&\int_\Gamma(\psi_t-\delta\Delta_{\boldsymbol\tau}\psi)\partial_{\bf n}\phi {\rm d}S=\int_\Gamma(\psi_t-\delta\Delta_{\boldsymbol\tau}\psi)
(L(\psi)-\gamma_{fs}^\prime(\psi)) {\rm d}S                 \nonumber
\\
&=-\frac{\rm d}{{\rm d}t}\int_\Gamma\gamma_{fs}(\psi){\rm d}S+\int_\Gamma\delta
\Delta_{\boldsymbol\tau}\psi\gamma_{fs}^\prime(\psi){\rm d}S
+\int_\Gamma(\psi_t-\delta\Delta_{\boldsymbol\tau}\psi)L(\psi) {\rm d}S                                                     \nonumber
\\
&=-\frac{\rm d}{{\rm d}t}\int_\Gamma\gamma_{fs}(\psi){\rm d}S
-\delta\int_\Gamma\gamma_{fs}^{(2)}(\psi)
|\nabla_{\boldsymbol\tau}\psi|^2{\rm d}S
+\int_\Gamma(-L(\psi)-{\bf u}_{\boldsymbol\tau}\cdot\nabla_{\boldsymbol\tau}\psi)L(\psi) {\rm d}S
\nonumber
\\
&=-\frac{\rm d}{{\rm d}t}\int_\Gamma\gamma_{fs}(\psi){\rm d}S
-\delta\int_\Gamma\gamma_{fs}^{(2)}(\psi)
|\nabla_{\boldsymbol\tau}\psi|^2{\rm d}S-\|L(\psi)\|_{L^2(\Gamma)}^2
-\int_\Gamma L(\psi){\bf u}_{\boldsymbol\tau}\cdot\nabla_{\boldsymbol\tau}\psi {\rm d}S.                    \nonumber
\end{align}
Hence, we find (\ref{E4}) becomes
\begin{align}
&\frac{1}{2}\frac{\rm d}{{\rm d}t}\left(
\|\nabla\phi\|_{L^2}^2+\frac{1}{2}\|\phi^2-1\|_{L^2}^2
+ 2\int_\Gamma\gamma_{fs}(\psi){\rm d}S
\right)+\|\mu-\bar\mu\|_{L^2}^2+\delta\|\nabla_{\boldsymbol\tau}\nabla\phi\|_{L^2}^2
+ \|L(\psi)\|_{L^2(\Gamma)}^2                      \nonumber
\\
&=\int_\Omega{\rm div}(\nabla\phi\otimes\nabla\phi)\cdot{\bf u}{\rm d}x- \int_\Gamma L(\psi){\bf u}_{\boldsymbol\tau}\cdot\nabla_{\boldsymbol\tau}\psi {\rm d}S  \nonumber
\\
&\quad+\delta\int_\Omega(1-3\phi^2)|\nabla_{\boldsymbol\tau}\phi|^2 {\rm d}x-\delta\int_\Gamma\gamma_{fs}^{(2)}(\psi)
|\nabla_{\boldsymbol\tau}\psi|^2{\rm d}S                \nonumber
\\
&\le \int_\Omega{\rm div}(\nabla\phi\otimes\nabla\phi)\cdot{\bf u}{\rm d}x- \int_\Gamma L(\psi){\bf u}_{\boldsymbol\tau}\cdot\nabla_{\boldsymbol\tau}\psi {\rm d}S +\delta\|\nabla_{\boldsymbol\tau}\phi\|_{L^2}^2
+C\delta\|\nabla_{\boldsymbol\tau}\psi\|_{L^2(\Gamma)}^2
\nonumber
\\
&\le \int_\Omega{\rm div}(\nabla\phi\otimes\nabla\phi)\cdot{\bf u}{\rm d}x- \int_\Gamma L(\psi){\bf u}_{\boldsymbol\tau}\cdot\nabla_{\boldsymbol\tau}\psi {\rm d}S+C\|\nabla\phi\|_{L^2}^2+C\delta\|\nabla_{\boldsymbol\tau}\phi\|
_{L^2}\|\nabla_{\boldsymbol\tau}\phi\|_{H^1}               \nonumber
\\
&\le \int_\Omega{\rm div}(\nabla\phi\otimes\nabla\phi)\cdot{\bf u}{\rm d}x- \int_\Gamma L(\psi){\bf u}_{\boldsymbol\tau}\cdot\nabla_{\boldsymbol\tau}\psi {\rm d}S +C\|\nabla\phi\|_{L^2}^2+ \frac{\delta}{2} \|\nabla\nabla_{\boldsymbol\tau}\phi\|
_{L^2}^2,  \nonumber
\end{align}
which together with (\ref{e2}) and Gronwall inequality yields
\begin{align}
\label{E8}
&\sup\limits_{0\le t \le T}\left (\|{\bf u}\|_{L^2}^2+\|\nabla\phi\|_{L^2}^2+\frac{1}{2}\|\phi^2-1\|_{L^2}^2
+2\int_\Gamma\gamma_{fs}(\psi){\rm d}S
\right)                                                \nonumber
\\
&\quad +\int_{0}^{T}\left(\|({\mathbb S}({\bf u})\|_{L^2}^2+\beta\|{\bf u}_{\boldsymbol\tau}\|_{L^2}^2+\|\mu-\bar\mu\|_{L^2}^2+ \|L(\psi)\|_{L^2(\Gamma)}^2
+\delta\|\nabla_{\boldsymbol\tau}\nabla\phi\|_{L^2}^2
\right) {\rm d}t  \le C\hat{\mathcal E}(0).
\end{align}
Furthermore, it follows from (\ref{E7}) and Poincar$\rm\acute{e}$ inequality that
\begin{align}
\|\phi\|_{L^2}^2\le \|\phi-\bar\phi\|_{L^2}^2
+\|\bar\phi\|_{L^2}^2\le C\left(\|\nabla\phi\|_{L^2}^2+|\bar\phi_0|^2
\right),                               \nonumber
\end{align}
which together with (\ref{E8}) shows (\ref{E6}). Therefore, we complete the proof.
\end{proof}
\begin{Lemma}\label{L-33}
Let $({\bf u},\phi, \psi )$ be the smooth solution to $(\ref{delta NSAC_1})$, $(\ref{delta NBC_1})$, $(\ref{initial1})$ on $[0,T]$, then it holds that
\begin{align}
\label{4.1}
&\sup\limits_{0\le t \le T} \left( \|\phi_t\|_{H^1}^2
+\|{\bf u}_t\|_{L^2}^2  \right)
+\int_{0}^{T} {\mathcal D}_1(t) {\rm d}t \le C\int_{0}^{T}  {\mathcal E}(t)^3  {\rm d}t+C,
\end{align}
where
\begin{align*}
{\mathcal D}_1(t)=\|\nabla {\bf u}_t\|_{L^2}^2 +\|\nabla {\bf u}\|_{H^2}^2 +\|{\bf u}_{t{\boldsymbol\tau}}\|_{L^2(\Gamma)}^2
 +\delta\|\nabla\nabla_{\boldsymbol\tau}\phi_t\|_{L^2}^2
+\|(\mu-\bar\mu)_t\|_{L^2}^2 +\|\partial_{\bf n}\phi_t\|_{L^2(\Gamma)}^2,
\end{align*}
and $C$ is constant that depends on $\Omega$, $\beta$, $\hat{\mathcal E}(0)$ and the initial value, but independent of $\delta$.
\end{Lemma}
\begin{proof}[\bf Proof.] The proof consists of the following steps.
\\
{\it\bfseries Step 1. Estimates of $\| {\bf u}_{t}\|_{L^\infty(0,T; L^2)}$.} Differentiating $(\ref{delta NSAC_1})_1$ and $(\ref{delta NBC_1})_2$ with respect to $t$ leads to
\begin{equation}
\label{ut}
\begin{cases}
{\bf u}_{tt}+{\bf u}_t\cdot\nabla{\bf u}+{\bf u}\cdot\nabla{\bf u}_t+\nabla p_t={\rm div}{\mathbb S}({\bf u}_t)-\Delta\phi_t\cdot\nabla\phi-\Delta\phi\cdot\nabla\phi_t,&\rm in\quad\Omega\times(0,T),
\\
\beta{\bf u}_{t\boldsymbol\tau}+({\mathbb S}({\bf u}_t)\cdot{\bf n})_{\boldsymbol{\tau}}= \left(\partial_{\bf n}\phi_t+\gamma_{fs}^{(2)}(\psi)\psi_t
\right)\nabla_{\boldsymbol\tau}\psi+\left(\partial_{\bf n}\phi+\gamma_{fs}^\prime(\psi)
\right)\nabla_{\boldsymbol\tau}\psi_t ,&\rm on\quad\Gamma\times(0,T).
\end{cases}
\end{equation}
Multiplying $(\ref{ut})_1$ by ${\bf u}_t$, integrating the result over $\Omega$ by parts, and recalling (\ref{nabla pt}) and (\ref{777}), we infer
\begin{align}
\label{D6}
&\frac{1}{2}\frac{\rm d}{{\rm d}t}\|{\bf u}_t\|_{L^2}^2+\frac{1}{2}\|{\mathbb S}({\bf u}_t)\|_{L^2}^2+\beta\|{\bf u}_{t{\boldsymbol\tau}}\|_{L^2(\Gamma)}^2        \nonumber
\\
&=-\int_\Omega{\bf u}_t\cdot\nabla{\bf u}\cdot{\bf u}_t {\rm d}x+ \int_\Gamma
\left(\partial_{\bf n}\phi_t+\gamma_{fs}^{(2)}(\psi)\psi_t\right
)\nabla_{\boldsymbol\tau}\psi \cdot{\bf u}_{t{\boldsymbol\tau}} {\rm d}S              \nonumber
\\
&\quad+\int_\Gamma
\left(\partial_{\bf n}\phi+\gamma_{fs}^\prime(\psi)\right)\nabla_{\boldsymbol\tau}\psi_t \cdot{\bf u}_{t{\boldsymbol\tau}}{\rm d}S -\int_\Omega\Delta\phi_t\nabla\phi\cdot{\bf u}_t{\rm d}x-\int_\Omega\Delta\phi\nabla\phi_t\cdot{\bf u}_t{\rm d}x \nonumber
\\
&=\sum_{i=1}^{5}N_i.
\end{align}
Now, we estimate each term on the right-hand side of (\ref{D6}).
First, due to Korn's inequality and Cauchy-Schwartz's inequality, one has
\begin{align}
N_1&=-\int_\Omega{\bf u}_t\cdot\nabla{\bf u}\cdot{\bf u}_t {\rm d}x
\le C\|{\bf u}_t\|_{L^2}\|\nabla{\bf u}\|_{L^4}\|{\bf u}_t\|_{L^4}
\le C\|{\bf u}_t\|_{L^2}\|\nabla{\bf u}\|_{H^1}\|{\bf u}_t\|_{H^1}
\nonumber
\\
&\le C\|{\bf u}_t\|_{L^2}^2\|{\bf u}\|_{H^2}^2+\varepsilon\|{\bf u}_t\|_{H^1}^2
\le\varepsilon\|\nabla{\bf u}_t\|_{L^2}^2+C{\mathcal E}(t)^2
\le\frac{1}{16}\|\mathbb S({\bf u}_t)\|_{L^2}^2+C{\mathcal E}(t)^2.\nonumber
\end{align}
By using Lemma \ref{Trace}, $N_2$, $N_3$ could be estimated as
\begin{align}
N_2&=\int_\Gamma
\left(\partial_{\bf n}\phi_t+\gamma_{f_s}^{(2)}(\psi)\psi_t\right)
\nabla_{\boldsymbol\tau}\psi\cdot{\bf u}_{t{\boldsymbol\tau}} {\rm d}S
\nonumber
\\
&\le \|\partial_{\bf n}\phi_t+\gamma_{fs}^{(2)}(\psi)\psi_t\|_{H^{-\frac{1}{2}}(\Gamma)}
\|\nabla_{\boldsymbol\tau}\psi\cdot{\bf u}_{t{\boldsymbol\tau}}\|_{H^\frac{1}{2}(\Gamma)}   \nonumber
\\
&\le\|\partial_{\bf n}\phi_t+\gamma_{fs}^{(2)}(\phi)\phi_t\|_{L^2}
\|\nabla\phi\cdot{\bf u}_t\|_{H^1} \nonumber
\\
&\le\left(\|\partial_{\bf n}\phi_t\|_{L^2}+\|\gamma_{fs}^{(2)}(\phi)\phi_t\|_{L^2}
\right) \left( \|\nabla\phi\|_{L^\infty}\|{\bf u}_t\|_{H^1} + \|\nabla^{2}\phi\|_{L^4}\|{\bf u}_t\|_{L^4} \right)        \nonumber
\\
& \le\left(\|\partial_{\bf n}\phi_t\|_{L^2}+\|\gamma_{fs}^{(2)}(\phi)\phi_t\|_{L^2}
\right)    \|\nabla \phi\|_{H^2} \|{\bf u}_t\|_{H^1}        \nonumber
\\
&\le\varepsilon\|{\bf u}_t\|_{H^1}^2
+C\left(\|\nabla\phi_t\|_{L^2}^2+\|\phi_t\|_{L^2}^2\right)\|\nabla\phi\|_{H^2}^2           \nonumber
\\
&\le\frac{1}{16}\|\mathbb S({\bf u}_t)\|_{L^2}^2+C{\mathcal E}(t)^2. \nonumber
\\[1em]
N_3&=\int_\Gamma
(\partial_{\bf n}\phi+\gamma_{fs}^\prime(\psi))\nabla_{\boldsymbol\tau}\psi_t \cdot{\bf u}_{t{\boldsymbol\tau}} {\rm d}S         \nonumber
\\
&\le \|\nabla_{\boldsymbol\tau}\psi_t\|_{H^{-\frac{1}{2}}(\Gamma)}\|(
\partial_{\bf n}\phi+\gamma_{fs}^\prime(\psi))
{\bf u}_{t{\boldsymbol\tau}}\|_{H^\frac{1}{2}(\Gamma)}    \nonumber
\\
&\le C\|\nabla\phi_t\|_{L^2}\|(
\partial_{\bf n}\phi+\gamma_{fs}^\prime(\phi))
{\bf u}_t\|_{H^1}   \nonumber
\\
&\le C\|\nabla\phi_t\|_{L^2}   \left( \|\partial_{\bf n}\phi+\gamma_{fs}^\prime(\phi)\|_{L^\infty}\|{\bf u}_t\|_{H^1} +  \|\nabla [\partial_{\bf n}\phi+\gamma_{fs}^\prime(\phi)]\|_{L^4}\|{\bf u}_t\|_{L^4}  \right)
\nonumber \\
& \le C\|\nabla\phi_t\|_{L^2} ( \|\nabla\phi\|_{H^2} +1 ) \|{\bf u}_t\|_{H^1}
\nonumber \\
&\le\varepsilon\|{\bf u}_t\|_{H^1}^2+C\left(
\|\nabla\phi\|_{H^2}^2+1\right)\|\nabla\phi_t\|_{L^2}^2  \nonumber
\\
&\le\frac{1}{16}\|\mathbb S({\bf u}_t)\|_{L^2}^2+C{\mathcal E}(t)^2. \nonumber
\end{align}
Note that
$$
\int_{\Omega} \nabla \phi \cdot {\bf u}_{t} {\rm d}x = - \int_{\Omega} \phi  {\rm div} {\bf u}_{t} {\rm d}x + \int_{\Gamma} \psi {\bf u}_{t}  \cdot {\bf n}  {\rm d}S = 0.
$$
Then, direct calculations show that
\begin{align}
N_4&= -\int_\Omega\Delta\phi_t\nabla\phi\cdot{\bf u}_t{\rm d}x =  \int_\Omega \left[ (\mu- \bar{\mu})_{t} - f_{t}\right]\nabla\phi\cdot{\bf u}_t{\rm d}x
\nonumber \\
& \le C\|\nabla\phi\|_{L^\infty} \|{\bf u}_t\|_{L^2} \left( \| (\mu- \bar{\mu})_{t} \|_{L^{2}} +  \|3\phi^{2}-1\|_{L^\infty} \|\phi_t\|_{L^2}  \right)
\nonumber \\
&\le \varepsilon_{1}\| (\mu- \bar{\mu})_{t} \|_{L^{2}}^{2} + C ( \| \phi\|_{H^2}^4 +1 ) \|\phi_t\|_{L^2}^{2} +C \| \nabla \phi\|_{H^2}^2 \|{\bf u}_t\|_{L^2}^2
\nonumber \\
&\le\varepsilon_{1} \| (\mu - \bar{\mu})_{t} \|_{L^{2}}^{2} +C {\mathcal E}(t)^3 .\nonumber
\\[1em]
N_5&=-\int_\Omega\Delta\phi\nabla\phi_t\cdot{\bf u}_t{\rm d}x
\le C\|\Delta\phi\|_{L^4}\|{\bf u}_t\|_{L^4}\|\nabla\phi_t\|_{L^2}
\nonumber \\
&\le C\|\Delta\phi\|_{H^1}\|{\bf u}_t\|_{H^1}\|\nabla\phi_t\|_{L^2} \le\varepsilon\|{\bf u}_t\|_{H^1}^2+C \|\phi\|_{H^3}^2  \|\nabla\phi_t\|_{L^2}^2
\nonumber
\\
&\le\frac{1}{16}\|\mathbb S({\bf u}_t)\|_{L^2}^2+C{\mathcal E}(t)^2. \nonumber
\end{align}
Substituting $N_1-N_5$ into (\ref{D6}), and using Korn's inequality, we conclude that
\begin{align}
\label{D7}
&\frac{\rm d}{{\rm d}t}\|{\bf u}_t\|_{L^2}^2+\|\nabla {\bf u}_t\|_{L^2}^2+\|{\bf u}_{t{\boldsymbol\tau}}\|_{L^2(\Gamma)}^2\le \varepsilon_{1} \tilde{C}_{1}\| (\mu- \bar{\mu})_{t} \|_{L^{2}}^{2}+C {\mathcal E}(t)^3 .
\end{align}
To use Gronwall inequality, we need to estimate $\|{\bf u}_t(0)\|_{L^2}^2$. Multiplying $(\ref{delta NSAC_1})_1$ by ${\bf u}_t$, using integration by parts and Cauchy-Schwartz's inequality, we obtain
\begin{align}
\|{\bf u}_t(t)\|_{L^2}^2&=\int_\Omega{\rm div}{\mathbb S}({\bf u})\cdot{\bf u}_t {\rm d}x-\int_\Omega\nabla p\cdot{\bf u}_t {\rm d}x-\int_\Omega{\bf u}\cdot\nabla{\bf u}\cdot{\bf u}_t {\rm d}x-\int_\Omega{\rm div}(\nabla\phi\otimes\nabla\phi)\cdot{\bf u}_t {\rm d}x                      \nonumber
\\
&\le\frac{1}{2}\|{\bf u}_t(t)\|_{L^2}^2+C\left(\|\Delta{\bf u}(t)\|_{L^2}^2+\|{\bf u}(t)\cdot\nabla{\bf u}(t)\|_{L^2}^2+\|{\rm div}(\nabla\phi\otimes\nabla\phi)(t)\|_{L^2}^2
\right).                              \nonumber
\end{align}
Taking $t\rightarrow 0$, we get
\begin{align}
\label{D30}
\|{\bf u}_t(0)\|_{L^2}^2 &\le C\left(\|\Delta{\bf u}_0\|_{L^2}^2+\|{\bf u}_0\|_{L^\infty}^2\|\nabla{\bf u}_0\|_{L^2}^2+\|\nabla\phi_0\|_{L^6}^2\|\Delta\phi_0\|_{L^3}^2
\right) \nonumber
\\
&\le C\left(\|{\bf u}_0\|_{H^2}^4+\|\phi_0\|_{H^3}^4+1
\right).
\end{align}
{\it\bfseries Step 2. Estimates of $ \|\nabla {\bf u}\|_{L^2(0,T;H^2)} $.} Next, exploiting \cite{1} (see Theorem 1.2) and (\ref{L5.1}), we deduce that
\begin{align}
\label{D9}
\|\nabla{\bf u}\|_{H^2}^2&\le C\left(\|\nabla({\bf u}\cdot\nabla{\bf u})\|_{L^2}^2+\|\nabla(\Delta\phi\nabla\phi)\|_{L^2}^2+\|\nabla{\bf u}_t\|_{L^2}^2
\right)                                                 \nonumber
\\
& \quad +C \| \left(\partial_{\bf n}\phi+\gamma_{fs}^\prime(\psi)
\right)\nabla_{\boldsymbol\tau}\psi
\|_{H^{\frac{3}{2}}(\Gamma)}^2    + C\| {\bf u}_{\boldsymbol\tau} \|_{H^{\frac{3}{2}}(\Gamma)}^2
\nonumber \\
&\le C\left(\|\nabla{\bf u}\|_{L^4}^4+\|{\bf u}\|_{L^\infty}^2\|\nabla^2{\bf u}\|_{L^2}^2+\|\nabla{\bf u}_t\|_{L^2}^2+\|\nabla\Delta\phi\|_{L^2}^2\|\nabla\phi\|_{L^\infty}^2
\right)    \nonumber
\\
&\quad+C\|\Delta\phi\|_{L^4}^2\|\nabla^2\phi\|_{L^4}^2+C\left(
\|\partial_{\bf n}\phi\|_{H^{\frac{3}{2}}(\Gamma)}^2+1
\right) \|\nabla_{\boldsymbol\tau}\psi\|_{H^{\frac{3}{2}}(\Gamma)}^2   + C\| {\bf u} \|_{H^{2}}^2  \nonumber
\\
&\le C\left(\|\nabla{\bf u}\|_{L^4}^4+\|{\bf u}\|_{L^\infty}^2\|\nabla^2{\bf u}\|_{L^2}^2+\|\nabla{\bf u}_t\|_{L^2}^2
+\|\nabla\Delta\phi\|_{L^2}^2\|\nabla\phi\|_{L^\infty}^2
\right)                                            \nonumber
\\
&\quad+C\left(\|\Delta\phi\|_{L^4}^2\|\nabla^2\phi\|_{L^4}^2
+\|\partial_{\bf n}\phi\|_{H^2}^2\|\nabla\phi\|_{H^2}^2+\|\nabla\phi\|_{H^2}^2
\right) + C\| {\bf u} \|_{H^{2}}^2          \nonumber
\\
&\le C\left(\|\nabla{\bf u}\|_{H^1}^4+\|{\bf u}\|_{H^2}^2\|\nabla^2{\bf u}\|_{L^2}^2+\|\nabla{\bf u}_t\|_{L^2}^2
+\|\nabla\Delta\phi\|_{L^2}^2\|\nabla\phi\|_{H^2}^2
\right)                                            \nonumber
\\
&\quad+C\left(\|\Delta\phi\|_{H^1}^2\|\nabla^2\phi\|_{H^1}^2
+\|\nabla\phi\|_{H^2}^4+\|\nabla\phi\|_{H^2}^2
\right) + C\| {\bf u} \|_{H^{2}}^2         \nonumber
\\
&\le C{\mathcal E}(t)^2+\tilde{C}_{2}\|\nabla{\bf u}_t\|_{L^2}^2.
\end{align}
{\it\bfseries Step 3. Estimates of $ \| \nabla \phi_{t}\|_{L^\infty(0,T;L^2)} $.} Differentiating $(\ref{delta NSAC_1})_{3,4}$ and $(\ref{bb1})$ with respect to $t$ leads to
\begin{equation}
\label{l-pt}
\begin{cases}
(\bar\mu-\mu)_t=\phi_{tt}-\delta\Delta_{\boldsymbol\tau}\phi_{t}+{\bf u}_t\cdot\nabla\phi+{\bf u}\cdot\nabla\phi_t,&\rm in\quad\Omega\times(0,T),
\\
\Delta\phi_t=-\mu_t+f_t,&\rm in\quad\Omega\times(0,T),
\\
(\bar\mu-\mu)_t=-\partial_{\bf n}\phi_t- \gamma_{fs}^{(2)}(\psi)\psi_t   ,&\rm on\quad\Gamma\times(0,T).
\end{cases}
\end{equation}
We next consider the following inner product
\begin{align}
\label{l-r7}
-\langle(\bar\mu-\mu)_t,\Delta\phi_t\rangle=\langle\nabla(\bar\mu-\mu)_t,\nabla\phi_t\rangle
-\int_\Gamma(\bar\mu-\mu)_t\partial_{\bf n}\phi_t {\rm d}S= L_1^\delta+L_2^\delta.
\end{align}
In order to compute $L_1^\delta$, exploiting the expression of $(\ref{l-pt})_1$, we have
\begin{align}
L_1^\delta&=\int_\Omega\nabla\left(\phi_{tt}-\delta\Delta_{\boldsymbol\tau}\phi_{t}+{\bf u}_t\cdot\nabla\phi+{\bf u}\cdot\nabla\phi_t
\right)\cdot\nabla\phi_t {\rm d}x  \nonumber
\\
&=\int_\Omega\left(\nabla\phi_{tt}-\delta \nabla \Delta_{\boldsymbol\tau}\phi_{t}+\nabla{\bf u}_t\cdot\nabla\phi+{\bf u}_t\cdot\nabla^2\phi+\nabla{\bf u}\cdot\nabla\phi_t+{\bf u}\cdot\nabla^2\phi_t
\right)\cdot\nabla\phi_t {\rm d}x  \nonumber
\\
&=\frac{1}{2}\frac{\rm d}{{\rm d}t} \|\nabla\phi_t\|_{L^2}^2 + \delta \| \nabla \nabla_{\boldsymbol\tau}\phi_t\|_{L^2}^2 +\int_\Omega\left(\nabla{\bf u}_t\cdot\nabla\phi+{\bf u}_t\cdot\nabla^2\phi+\nabla{\bf u}\cdot\nabla\phi_t
\right)\cdot\nabla\phi_t {\rm d}x,  \nonumber
\end{align}
where we have used integration by parts to get
\begin{align}
\int_\Omega{\bf u}\cdot\nabla^2\phi_t\cdot\nabla\phi_t{\rm d}x&=\frac{1}{2}\int_\Omega{\bf u}\cdot\nabla\left(|\nabla\phi_t|^2\right){\rm d}x  \nonumber
\\
&=-\frac{1}{2}\int_\Omega{\rm div}{\bf u}\left(|\nabla\phi_t|^2\right){\rm d}x+\frac{1}{2}\int_\Gamma{\bf u}\cdot{\bf n}\left(|\nabla\phi_t|^2\right){\rm d}S=0.  \nonumber
\end{align}
Thanks to $(\ref{l-pt})_3$, $L_2^\delta$ gives
\begin{align}
\label{L2L2}
L_2^\delta&=\int_\Gamma\left(\partial_{\bf n}\phi_t+\gamma_{fs}^{(2)}(\psi)\psi_t
\right)\partial_{\bf n}\phi_t {\rm d}S
=\|\partial_{\bf n}\phi_t\|_{L^2(\Gamma)}^2+\int_\Gamma\gamma_{fs}^{(2)}(\psi)\psi_t\partial_{\bf n}\phi_t {\rm d}S,
\end{align}
whereas the left-hand side of $(\ref{r7})$ could be written by $(\ref{l-pt})_2$ as
\begin{align}
\label{l-r8}
-\langle(\bar\mu-\mu)_t,\Delta\phi_t\rangle&=-\int_\Omega(\bar\mu-\mu)_t\left(
-\mu_t+f_t\right){\rm d}x    \nonumber
\\
&=-\|(\mu-\bar\mu)_t\|_{L^2}^2+\bar\mu_t\int_\Omega(\bar\mu-\mu)_t {\rm d}x-\int_\Omega(\bar\mu-\mu)_tf_t {\rm d}x \nonumber
\\
&=-\|(\mu-\bar\mu)_t\|_{L^2}^2-\int_\Omega(\bar\mu-\mu)_tf_t {\rm d}x,
\end{align}
where we have used the following fact
\[
\int_\Omega(\bar\mu-\mu)_t{\rm d}x=\frac{\rm d}{{\rm d}t}\int_\Omega(\bar\mu-\mu){\rm d}x=0.
\]
Substituting $L_1^\delta$, $L_2^\delta$ and $(\ref{l-r8})$ into $(\ref{l-r7})$, one has
\begin{align}
\label{D5}
&\frac{1}{2}\frac{\rm d}{{\rm d}t}\|\nabla\phi_t\|_{L^2}^2
+\delta\|\nabla\nabla_{\boldsymbol\tau}\phi_t\|_{L^2}^2
+\|(\mu-\bar\mu)_t\|_{L^2}^2 +\|\partial_{\bf n}\phi_t\|_{L^2(\Gamma)}^2
\nonumber
\\
&=-\int_\Omega\left(\nabla{\bf u}_t\cdot\nabla\phi+{\bf u}_t\cdot\nabla^2\phi+\nabla{\bf u}\cdot\nabla\phi_t
\right)\cdot\nabla\phi_t {\rm d}x \nonumber
\\
&\quad -\int_\Omega(\bar\mu-\mu)_tf_t {\rm d}x - \int_\Gamma\gamma_{fs}^{(2)}(\psi)\psi_t\partial_{\bf n}\phi_t {\rm d}S     \nonumber
\\
&\le C\|\nabla{\bf u}_t\|_{L^2}\|\nabla\phi\|_{L^\infty}\|\nabla\phi_t\|_{L^2}+C\|{\bf u}_t\|_{L^4}\|\nabla^2\phi\|_{L^4}\|\nabla\phi_t\|_{L^2}+C\|\nabla{\bf u}\|_{L^\infty}\|\nabla\phi_t\|_{L^2}^2    \nonumber
\\
&\quad
+C\|f_t\|_{L^2} \|(\mu-\bar\mu)_t\|_{L^2}
+C  \|\psi_t\|_{L^2(\Gamma)}      \| \partial_{\bf n}\phi_t\|_{L^2(\Gamma)}
\nonumber  \\
&\le\frac{1}{4} \|(\mu-\bar\mu)_t\|_{L^2}^{2} + \frac{1}{2} \| \partial_{\bf n}\phi_t\|_{L^2(\Gamma)}^{2}
+C\|\nabla{\bf u}_t\|_{L^2}\|\nabla\phi\|_{H^2}\|\nabla\phi_t\|_{L^2}                       \nonumber
\\
&\quad +C\|{\bf u}_t\|_{H^1}\|\nabla^2\phi\|_{H^1}\|\nabla\phi_t\|_{L^2} +C\|\nabla{\bf u}\|_{H^2}\|\nabla\phi_t\|_{L^2}^2
\nonumber \\
&\quad+ C \|3\phi^{2}-1\|_{L^\infty}^{2} \|\phi_t\|_{L^2}^{2} + C \|\phi_t\|_{H^1}^{2}
\nonumber \\
&\le \frac{1}{4} \|(\mu-\bar\mu)_t\|_{L^2}^{2} + \frac{1}{2} \| \partial_{\bf n}\phi_t\|_{L^2(\Gamma)}^{2}+C\|\phi_t\|_{H^1}^2
+ \varepsilon_{1}\left(\|{\bf u}_t\|_{H^1}^2+\|\nabla{\bf u}\|_{H^2}^2
\right)
\nonumber
\\
&\quad +C\|\nabla\phi\|_{H^2}^2\|\nabla\phi_t\|_{L^2}^2 +C\|\nabla^2\phi\|_{H^1}^2\|\nabla\phi_t\|_{L^2}^2+\|\nabla\phi_t\|_{L^2}^4
+  C (\|\phi\|_{H^2}^{4} + 1 ) \|\phi_t\|_{L^2}^{2} \nonumber
\\
&\le\frac{1}{4} \|(\mu-\bar\mu)_t\|_{L^2}^{2} + \frac{1}{2} \| \partial_{\bf n}\phi_t\|_{L^2(\Gamma)}^{2} +\varepsilon_{1}
\left(\|\nabla{\bf u}_t\|_{L^2}^2+\|\nabla{\bf u}\|_{H^2}^2
\right)
+C{\mathcal E}(t)^3 .
\end{align}
It follows from $(\ref{delta NSAC})_1$ that
\begin{align}
\label{D29}
\|\nabla\phi_t(0)\|_{L^2}^2
&\le\delta^2\|\nabla\Delta_{\boldsymbol\tau}\phi_0\|_{L^2}^2
+\|\nabla{\bf u}_0\cdot\nabla\phi_0\|_{L^2}^2+\|{\bf u}_0\cdot\nabla^2\phi_0\|_{L^2}^2+\|\nabla\Delta\phi_0\|_{L^2}^2+\|\nabla f_0\|_{L^2}^2                                      \nonumber
\\
&\le C\left(\|\phi_0\|_{H^3}^2+\|\nabla{\bf u}_0\|_{L^4}^2\|\nabla\phi_0\|_{L^4}^2+\|{\bf u}_0\|_{L^\infty}^2\|\nabla^2\phi_0\|_{L^2}^2+\|3\phi_0^2-1\|_{L^\infty}^2
\|\nabla\phi_0\|_{L^2}^2
\right)                                               \nonumber
\\
&\le C\left(\|\phi_0\|_{H^3}^2+\|{\bf u}_0\|_{H^2}^2\|\phi_0\|_{H^2}^2+\|\phi_0\|_{H^2}^4\|\nabla\phi_0\|_{L^2}^2
+\|\nabla\phi_0\|_{L^2}^2
\right)                                               \nonumber
\\
&\le C\left(\|\phi_0\|_{H^2}^6+\|\phi_0\|_{H^3}^2+\|{\bf u}_0\|_{H^2}^2\|\phi_0\|_{H^2}^2
\right ).
\end{align}
{\it\bfseries Step 4. Closure of the estimates.} Combining \eqref{D7}-\eqref{D9},  \eqref{D5} with \eqref{D29}, choosing $\varepsilon_{1}$ enough small such that $\varepsilon_{1} (\tilde{C}_{1} + \tilde{C}_{2}) < 1/4$, there is
\begin{align}
\label{D44}
&\frac{\rm d}{{\rm d}t}\left(\|\nabla\phi_t\|_{L^2}^2
+\|{\bf u}_t\|_{L^2}^2  \right) +\|\nabla {\bf u}_t\|_{L^2}^2+\|{\bf u}_{t{\boldsymbol\tau}}\|_{L^2(\Gamma)}^2
\nonumber \\
& \quad  +\delta\|\nabla\nabla_{\boldsymbol\tau}\phi_t\|_{L^2}^2
+\|(\mu-\bar\mu)_t\|_{L^2}^2 +\|\partial_{\bf n}\phi_t\|_{L^2(\Gamma)}^2
\le C{\mathcal E}(t)^3.
\end{align}
Due to ({\ref{E7}}) and Poincar$\rm\acute{e}$ inequality, one has
\begin{align}
\|\phi_t\|_{L^2}^2=\|\phi_t-\bar\phi_t\|_{L^2}^2\le C\|\nabla\phi_t\|_{L^2}^2,        \nonumber
\end{align}
which together with (\ref{D44}) and Gronwall inequality leads to \eqref{4.1}. Thus, we complete the proof of Lemma \ref{L-33}.
\end{proof}

\begin{Lemma}
\label{L-44}
Let $({\bf u},\phi, \psi )$ be the smooth solution to $(\ref{delta NSAC_1})$, $(\ref{delta NBC_1})$, $(\ref{initial1})$ on $[0,T]$, then it holds that
\begin{align}
\label{L-4.2}
&\sup\limits_{0\le t \le T}\left( \| \nabla {\bf u}\|_{L^2}^2+ \beta\|{\bf u}_{\boldsymbol\tau}\|_{L^2(\Gamma)}^2
\right)+\int_{0}^{T} \|{\bf u}_t\|_{L^2}^2  {\rm d}t \le C\int_{0}^{T}{\mathcal E}(t)^2{\rm d}t+C,
\end{align}
where $C$ is constant that depends on $\Omega$, $\beta$, $\hat{\mathcal E}(0)$ and the initial value, but independent of $\delta$.
\end{Lemma}

\begin{proof}[\bf Proof.]

For the later calculation, $\|\nabla{\bf u}\|_{L^\infty(0,T;L^2)}$ needs to be estimated. Multiplying $(\ref{delta NSAC_1})_1$ by ${\bf u}_t$, and integrating the result over $\Omega$, one has
\begin{align}
\|{\bf u}_t\|_{L^2}^2+\int_\Omega{\bf u}\cdot\nabla{\bf u}\cdot{\bf u}_t{\rm d}x+\int_\Omega\nabla p\cdot{\bf u}_t{\rm d}x=\int_\Omega{\rm div}\mathbb S({\bf u})\cdot{\bf u}_t{\rm d}x-\int_\Omega{\rm div}(\nabla\phi\otimes\nabla\phi)\cdot{\bf u}_t{\rm d}x.
\nonumber
\end{align}
By the same arguments as in Lemma \ref{final}, we deduce
\begin{align}
\label{D36}
&\frac{1}{2}\frac{\rm d}{{\rm d}t}\left(\frac{1}{2} \|\mathbb S({\bf u})\|_{L^2}^2+\beta\|{\bf u}_{\boldsymbol\tau}\|_{L^2(\Gamma)}^2
\right)+\|{\bf u}_t\|_{L^2}^2               \nonumber
\\
&=-\int_\Omega{\bf u}\cdot\nabla{\bf u}\cdot{\bf u}_t {\rm d}x+ \int_\Gamma\left(\partial_{\bf n}\phi+\gamma_{fs}^\prime(\psi)
\right) \nabla_{\boldsymbol\tau}\psi\cdot{\bf u}_{t{\boldsymbol\tau}} {\rm d}S -\int_\Omega\Delta\phi\nabla\phi\cdot{\bf u}_t{\rm d}x
\nonumber
\\
&\le C\|{\bf u}\|_{L^\infty}\|\nabla{\bf u}\|_{L^2}\|{\bf u}_t\|_{L^2}+C \|\left(
\partial_{\bf n}\phi+\gamma_{fs}^\prime(\psi)
\right)\nabla_{\boldsymbol\tau}\psi\|_{H^\frac{1}{2}(\Gamma)}    \|{\bf u}_{t{\boldsymbol\tau}}\|_{H^{-\frac{1}{2}}(\Gamma)}  \nonumber
\\
&\quad+C\|\Delta\phi\|_{L^2}\|\nabla\phi\|_{L^\infty}\|{\bf u}_t\|_{L^2}   \nonumber
\\
&\le C\|{\bf u}\|_{H^2}\|\nabla{\bf u}\|_{L^2}\|{\bf u}_t\|_{L^2} +C\|\Delta\phi\|_{L^2}\|\nabla\phi\|_{H^2}\|{\bf u}_t\|_{L^2}
\nonumber
\\
&\quad +C \left( \|\partial_{\bf n}\phi+\gamma_{fs}^\prime(\phi)\|_{L^\infty}\|\nabla\phi\|_{H^1}+ \|\nabla[\partial_{\bf n}\phi+\gamma_{fs}^\prime(\phi)]\|_{L^4}\|\nabla\phi\|_{L^4} \right) \|{\bf u}_t\|_{L^2}
\nonumber
\\
&\le C\|{\bf u}_t\|_{L^2}^2+C\|{\bf u}\|_{H^2}^2\|\nabla{\bf u}\|_{L^2}^2+C\left(\|\nabla\phi\|_{H^2}^2+1\right)\|\nabla\phi\|_{H^1}^2
+C\|\nabla\phi\|_{H^2}^2\|\Delta\phi\|_{L^2}^2    \nonumber
\\
&\le C{\mathcal E}(t)^2.
\end{align}
Integrating the result over $(0,T)$ and using Korn's inequality, we conclude
\begin{align}
&\sup\limits_{0\le t \le T}\left( \|\nabla{\bf u}\|_{L^2}^2+\beta\|{\bf u}_{\boldsymbol\tau}\|_{L^2(\Gamma)}^2
\right) +\int_{0}^{T}
\|{\bf u}_t\|_{L^2}^2 {\rm d}t\le C\int_{0}^{T}{\mathcal E}(t)^2{\rm d}t+C,  \nonumber
\end{align}
which  completes the proof of Lemma \ref{L-44}.
\end{proof}
\begin{Lemma}
\label{Lemma4.3}
Let $({\bf u},\phi,\psi)$ be the smooth solution to $(\ref{delta NSAC_1})$, $(\ref{delta NBC_1})$, $(\ref{initial1})$ on $[0,T]$, then it holds that
\begin{align}
\label{4.3}
&\sup\limits_{0\le t \le T} {\mathcal E}_2(t) +\int_{0}^{T}{\mathcal D}_2(t) {\rm d}t  \le C\int_{0}^{T}{\mathcal E}(t)^{16}{\rm d}t+C,
\end{align}
where
\begin{align}
{\mathcal E}_2(t)&= \|\phi\|_{H^3}^2+\|{\bf u}\|_{H^2}^2
+\delta\|\nabla\nabla_{\boldsymbol\tau}^2\phi\|_{L^2}^2
+ \|\nabla_{\boldsymbol\tau} L(\psi)\|_{L^2(\Gamma)}^2
+\|\nabla_{\boldsymbol\tau} (\mu-\bar\mu)\|_{L^2}^2,
\nonumber \\
{\mathcal D}_2(t)&=\|\phi_t\|_{H^2}^2
+ \delta\|\nabla\nabla_{\boldsymbol\tau}^3\psi\|_{L^2(\Gamma)}^2  +\|\nabla_{\boldsymbol\tau}^2(\mu-\bar\mu)\|_{L^2}^2
+\|\nabla_{\boldsymbol\tau}\partial_{\bf n}\phi\|_{H^1(\Gamma)}^2  ,       \nonumber
\end{align}
and $C$ is constant that depends on $\Omega$, $\beta$, $\hat{\mathcal E}(0)$ and the initial value, but independent of $\delta$.
\end{Lemma}
\begin{proof}[\bf Proof.] The proof consists of the following steps.
\\
{\it\bfseries Step 1. Estimates of $ \nabla \nabla_{\boldsymbol\tau}^k \phi$, $k=1, 2$.} For $k=1,2$, taking the operator $\nabla\nabla_{\boldsymbol\tau}^k$, $\nabla_{\boldsymbol\tau}^k$ and $\nabla_{\boldsymbol\tau}^k$ to $(\ref{delta NSAC_1})_{3,4}$ and $(\ref{bb1})$ respectively leads to
\begin{equation}
\label{ll-r18}
\begin{cases}
\nabla\nabla_{\boldsymbol\tau}^k(\bar\mu-\mu)
=\nabla\nabla_{\boldsymbol\tau}^k\phi_t- \delta \nabla\nabla_{\boldsymbol\tau}^k\Delta_{\boldsymbol\tau}\phi
+[\nabla\nabla_{\boldsymbol\tau}^k,{\bf u}\cdot\nabla]\phi
+{\bf u}\cdot\nabla^2\nabla_{\boldsymbol\tau}^k\phi,&\rm in\quad\Omega\times(0,T),
\\
\nabla_{\boldsymbol\tau}^k\Delta\phi=-\nabla_{\boldsymbol\tau}^k\mu
+\nabla_{\boldsymbol\tau}^k(\phi^3-\phi),&\rm in\quad\Omega\times(0,T),
\\
\nabla_{\boldsymbol\tau}^k(\bar\mu-\mu)= -\nabla_{\boldsymbol\tau}^k\partial_{\bf n}\phi
-\nabla_{\boldsymbol\tau}^k\gamma_{fs}^\prime(\psi) ,&\rm on\quad\Gamma\times(0,T),
\end{cases}
\end{equation}
where $[A,B]=AB-BA$ is commutators.  We next consider
\begin{align}
\label{ll-r19}
-\langle\nabla_{\boldsymbol\tau}^k(\bar\mu-\mu),\nabla_{\boldsymbol\tau}^k\Delta\phi\rangle
&=\langle\nabla\nabla_{\boldsymbol\tau}^k(\bar\mu-\mu),\nabla_{\boldsymbol\tau}^k\nabla\phi\rangle
-\int_\Gamma\nabla_{\boldsymbol\tau}^k(\bar\mu-\mu)\nabla_{\boldsymbol\tau}^k\partial_{\bf n}\phi {\rm d}S \nonumber
\\
&=P_1^\delta+P_2^\delta.
\end{align}
Replacing $(\ref{ll-r18})_1$ into $P_1^\delta$, we have
\begin{align}
P_1^\delta&=\int_\Omega\left(\nabla\nabla_{\boldsymbol\tau}^k\phi_t -\delta\nabla\nabla_{\boldsymbol\tau}^k \Delta_{\boldsymbol\tau}\phi+[\nabla\nabla_{\boldsymbol\tau}^k,{\bf u}\cdot\nabla]\phi
+{\bf u}\cdot\nabla^2\nabla_{\boldsymbol\tau}^k\phi
\right)\cdot\nabla\nabla_{\boldsymbol\tau}^k\phi {\rm d}x  \nonumber
\\
&=\frac{1}{2}\frac{\rm d}{{\rm d}t}\|\nabla\nabla_{\boldsymbol\tau}^k\phi\|_{L^2}^2 + \delta \|\nabla\nabla_{\boldsymbol\tau}^{k+1}\phi\|_{L^2}^2 +\int_\Omega
[\nabla\nabla_{\boldsymbol\tau}^k,{\bf u}\cdot\nabla]\phi
\cdot\nabla\nabla_{\boldsymbol\tau}^k\phi {\rm d}x,   \nonumber
\end{align}
where we have used the following fact that
\begin{align}
\int_\Omega{\bf u}\cdot\nabla^2\nabla_{\boldsymbol\tau}^k\phi
\cdot\nabla\nabla_{\boldsymbol\tau}^k\phi {\rm d}x&=\frac{1}{2}\int_\Omega{\bf u}\cdot\nabla|\nabla\nabla_{\boldsymbol\tau}^k\phi|^2{\rm d}x  \nonumber
\\
&=-\frac{1}{2}\int_\Omega{\rm div}{\bf u}|\nabla\nabla_{\boldsymbol\tau}^k\phi|^2{\rm d}x+\frac{1}{2}\int_\Gamma{\bf u}\cdot{\bf n}|\nabla\nabla_{\boldsymbol\tau}^k\phi|^2{\rm dS}=0. \nonumber
\end{align}
Then, using $(\ref{ll-r18})_3$, $P_2^\delta$ can be rewritten as
\begin{align}
\label{P2P2}
P_2^\delta&=\int_\Gamma\left(\nabla_{\boldsymbol\tau}^k\partial_{\bf n}\phi+ \nabla_{\boldsymbol\tau}\gamma_{fs}^\prime(\psi) \right)\nabla_{\boldsymbol\tau}^k\partial_{\bf n}\phi {\rm d}S   \nonumber
\\
&=\|\nabla_{\boldsymbol\tau}^k\partial_{\bf n}\phi\|_{L^2(\Gamma)}^2
+\int_\Gamma\nabla_{\boldsymbol\tau}^k\gamma_{fs}^\prime(\psi)\nabla_{\boldsymbol\tau}^k\partial_{\bf n}\phi {\rm d}S   ,
\end{align}
whereas the term on the left-hand side of (\ref{ll-r19}) can be rewritten from below by means of $(\ref{ll-r18})_2$
\begin{align}
\label{ll-r20}
-\langle\nabla_{\boldsymbol\tau}^k(\mu-\bar\mu),\nabla_{\boldsymbol\tau}^k\Delta\phi\rangle
&=-\int_\Omega\nabla_{\boldsymbol\tau}^k(\mu-\bar\mu)\left[-\nabla_{\boldsymbol\tau}^k\mu
+\nabla_{\boldsymbol\tau}^k(\phi^3-\phi)
\right]{\rm d}x  \nonumber
\\
&=-\|\nabla_{\boldsymbol\tau}^k(\mu-\bar\mu)\|_{L^2}^2
-\int_\Omega\nabla_{\boldsymbol\tau}^k(\mu-\bar\mu)\nabla_{\boldsymbol\tau}^k
(\phi^3-\phi){\rm d}x.
\end{align}
Substituting $P_1^\delta$, $P_2^\delta$, (\ref{ll-r20}) into (\ref{ll-r18}), we obtain
\begin{align}
\label{ll-r21}
&\frac{1}{2}\frac{\rm d}{{\rm d}t}\|\nabla\nabla_{\boldsymbol\tau}^k\phi\|_{L^2}^2
+ \delta \|\nabla\nabla_{\boldsymbol\tau}^{k+1}\phi\|_{L^2}^2 +\|\nabla_{\boldsymbol\tau}^k(\mu-\bar\mu)\|_{L^2}^2
+\|\nabla_{\boldsymbol\tau}^k\partial_{\bf n}\phi\|_{L^2(\Gamma)}^2 \nonumber
\\
&=-\int_\Omega
[\nabla\nabla_{\boldsymbol\tau}^k,{\bf u}\cdot\nabla]\phi
\cdot\nabla\nabla_{\boldsymbol\tau}^k\phi {\rm d}x-\int_\Omega\nabla_{\boldsymbol\tau}^k(\mu-\bar\mu)\nabla_{\boldsymbol\tau}^k
(\phi^3-\phi){\rm d}x- \int_\Gamma\nabla_{\boldsymbol\tau}^k\gamma_{fs}^\prime(\psi)\nabla_{\boldsymbol\tau}^k\partial_{\bf n}\phi {\rm d}S    \nonumber
\\
&=\sum_{i=1}^3 H_i^k.
\end{align}

Next, for the case $k=1$, the right-hand side of (\ref{ll-r21}) can be estimated as follows:
\begin{align}
H_1^1&=-\int_\Omega\left(\nabla\nabla_{\boldsymbol\tau}{\bf u}\cdot\nabla\phi+\nabla_{\boldsymbol\tau}{\bf u}\cdot\nabla^2\phi+\nabla{\bf u}\cdot\nabla\nabla_{\boldsymbol\tau}\phi
\right)\cdot\nabla\nabla_{\boldsymbol\tau}\phi {\rm d}x  \nonumber
\\
&\le C\left(\|\nabla^2{\bf u}\|_{L^2}\|\nabla\phi\|_{L^\infty}+\|\nabla{\bf u}\|_{L^4}\|\nabla^2\phi\|_{L^4}
\right)\|\nabla\nabla_{\boldsymbol\tau}\phi\|_{L^2} \nonumber
\\
&\le C\left(\|\nabla^2{\bf u}\|_{L^2}\|\nabla\phi\|_{H^2}+\|\nabla{\bf u}\|_{H^1}\|\nabla^2\phi\|_{H^1}
\right)\|\nabla\nabla_{\boldsymbol\tau}\phi\|_{L^2} \nonumber
\\
&\le C\mathcal{E} (t)^{\frac{3}{2}}.   \nonumber
\\[1em]
H_2^1&=-\int_\Omega \nabla_{\boldsymbol\tau}(\mu-\bar\mu)(3\phi^2-1)\nabla_{\boldsymbol\tau}\phi {\rm d}x \nonumber \\
& \le \frac{1}{2}\| \nabla_{\boldsymbol\tau}(\mu-\bar\mu) \|_{L^{2}}^{2} + C \|3\phi^2-1\|_{L^{\infty}}^{2} \|\nabla_{\boldsymbol\tau}\phi \|_{L^{2}}^{2}
\nonumber \\
& \le \frac{1}{2}\| \nabla_{\boldsymbol\tau}(\mu-\bar\mu) \|_{L^{2}}^{2} + C (\|\phi\|_{H^{2}}^{4} +1 ) \|\nabla\phi \|_{L^{2}}^{2}
 \nonumber \\
&\le \frac{1}{2}\| \nabla_{\boldsymbol\tau}(\mu-\bar\mu) \|_{L^{2}}^{2} + C\mathcal{E} (t)^{3}.
\nonumber
\\[1em]
H_3^1&=- \int_\Gamma\gamma_{fs}^{(2)}(\psi)\nabla_{\boldsymbol\tau}\psi\nabla_{\boldsymbol\tau}\partial_{\bf n}\phi {\rm d}S\le \frac{1}{2}\|\nabla_{\boldsymbol\tau}\partial_{\bf n}\phi\|_{L^2(\Gamma)}^{2} + C \|\nabla_{\boldsymbol\tau}\psi\|_{L^2(\Gamma)}^{2}
\nonumber \\
& \le \frac{1}{2}\|\nabla_{\boldsymbol\tau}\partial_{\bf n}\phi\|_{L^2(\Gamma)}^{2} + C \|\nabla_{\boldsymbol\tau}\phi\|_{H^1}^{2}  \le \frac{1}{2}\|\nabla_{\boldsymbol\tau}\partial_{\bf n}\phi\|_{L^2(\Gamma)}^{2} + C \mathcal{E} (t).
\nonumber
\end{align}
For the case $k=2$, similarly, we obtain
\begin{align}
H_1^2&=-\int_\Omega\left[
\nabla\nabla_{\boldsymbol\tau}^2({\bf u}\cdot\nabla\phi)
-{\bf u}\cdot\nabla^2\nabla_{\boldsymbol\tau}^2\phi
\right]\cdot\nabla\nabla_{\boldsymbol\tau}^2\phi {\rm d}x  \nonumber
\\
&=-\int_\Omega\left(\nabla\nabla_{\boldsymbol\tau}^2{\bf u}\cdot\nabla\phi+2\nabla\nabla_{\boldsymbol\tau}{\bf u}\cdot\nabla\nabla_{\boldsymbol\tau}\phi+\nabla_{\boldsymbol\tau}^2{\bf u}\cdot\nabla^2\phi
\right)\cdot\nabla\nabla_{\boldsymbol\tau}^2\phi {\rm d}x \nonumber
\\
&\quad-\int_\Omega\left(2\nabla_{\boldsymbol\tau}{\bf u}\cdot\nabla^2\nabla_{\boldsymbol\tau}\phi+\nabla{\bf u}\cdot\nabla\nabla_{\boldsymbol\tau}^2\phi
\right)\cdot\nabla\nabla_{\boldsymbol\tau}^2\phi {\rm d}x \nonumber
\\
&\le C\left(\|\nabla^3{\bf u}\|_{L^2}\|\nabla\phi\|_{L^\infty}+\|\nabla^2{\bf u}\|_{L^4}\|\nabla^2\phi\|_{L^4}+\|\nabla{\bf u}\|_{L^\infty}\|\nabla^3\phi\|_{L^2}
\right)\|\nabla\nabla_{\boldsymbol\tau}^2\phi\|_{L^2} \nonumber
\\
&\le C\|\nabla{\bf u}\|_{H^2}^{2} + C \|\nabla\phi\|_{H^2}^{2}\|\nabla\nabla_{\boldsymbol\tau}^2\phi\|_{L^2}^{2}  \le C\|\nabla{\bf u}\|_{H^2}^{2} + C \mathcal{E} (t)^{2} .
\nonumber
\\[1em]
H_2^2&=-\int_\Omega\nabla_{\boldsymbol\tau}^2(\mu-\bar\mu)\left[
(3\phi^2-1)\nabla_{\boldsymbol\tau}^2\phi+6\phi|\nabla_{\boldsymbol\tau}\phi|^2
\right]{\rm d}x
\nonumber \\
& \le \| \nabla_{\boldsymbol\tau}^2(\mu-\bar\mu) \|_{L^{2}} \left( \| 3\phi^2-1 \|_{L^{\infty}} \| \nabla_{\boldsymbol\tau}^2\phi \|_{L^{2}}  +  \| \phi\|_{L^{\infty}} \| \nabla_{\boldsymbol\tau} \phi \|_{L^{4}}^{2} \right)
\nonumber \\
& \le \| \nabla_{\boldsymbol\tau}^2(\mu-\bar\mu) \|_{L^{2}} \left[ (\| \phi\|_{H^{2}}^{2} + 1 ) \| \nabla_{\boldsymbol\tau}^2\phi \|_{L^{2}}  +  \| \phi\|_{H^{2}} \| \nabla_{\boldsymbol\tau} \phi \|_{H^{1}}^{2} \right]
\nonumber \\
& \le \frac{1}{2}\| \nabla_{\boldsymbol\tau}^2(\mu-\bar\mu) \|_{L^{2}}^{2} +  C \left[ (\| \phi\|_{H^{2}}^{4} + 1 ) \| \nabla_{\boldsymbol\tau}^2\phi \|_{L^{2}}^{2}  +  \| \phi\|_{H^{2}}^{2} \| \nabla_{\boldsymbol\tau} \phi \|_{H^{1}}^{4} \right]
\nonumber \\
& \le \frac{1}{2}\| \nabla_{\boldsymbol\tau}^2(\mu-\bar\mu) \|_{L^{2}}^{2} + C\mathcal{E} (t)^{3}.
\nonumber
\\[1em]
H_3^2&=-\int_\Gamma\left(
\gamma_{fs}^{(3)}(\psi) |\nabla_{\boldsymbol\tau}\psi|^2 +\gamma_{fs}^{(2)}(\psi)
\nabla_{\boldsymbol\tau}^2\psi
\right)\nabla_{\boldsymbol\tau}^2\partial_{\bf n}\phi {\rm d}S    \nonumber
\\
&\le C \left(\|\nabla_{\boldsymbol\tau}\psi\|_{L^4(\Gamma)}^2+\|\nabla_{\boldsymbol\tau}^2\psi\|_{L^2(\Gamma)}
\right)  \|\nabla_{\boldsymbol\tau}^2\partial_{\bf n}\phi\|_{L^2(\Gamma)}  \nonumber
\\
&\le \frac{1}{2} \|\nabla_{\boldsymbol\tau}^2\partial_{\bf n}\phi\|_{L^2(\Gamma)}^{2} +  C\left(\|\nabla_{\boldsymbol\tau}\phi\|_{H^2}^4+\|\nabla_{\boldsymbol\tau}^2\phi\|_{H^1}^{2} \right)
\nonumber
\\
& \le \frac{1}{2} \|\nabla_{\boldsymbol\tau}^2\partial_{\bf n}\phi\|_{L^2(\Gamma)}^{2} + C\mathcal{E} (t)^{2} .     \nonumber
\end{align}
Then, substituting $H_1^k$ -- $H_3^k$ $(k=1,2)$ into (\ref{ll-r21}), we obtain, for $k=1,2$,
\begin{align*}
&\frac{1}{2}\frac{\rm d}{{\rm d}t}\|\nabla\nabla_{\boldsymbol\tau}^k\phi\|_{L^2}^2+ \delta \|\nabla\nabla_{\boldsymbol\tau}^{k+1}\phi\|_{L^2}^2
+\|\nabla_{\boldsymbol\tau}^k(\mu-\bar\mu)\|_{L^2}^2
+\|\nabla_{\boldsymbol\tau}^k\partial_{\bf n}\phi\|_{L^2(\Gamma)}^2 \le C\|\nabla{\bf u}\|_{H^2}^{2} + C \mathcal{E} (t)^{3}.
\end{align*}
Integrating over (0,T) and using \eqref{4.1}, one has
\begin{align}\label{ll-r22}
&\sup\limits_{0\le t\le T}\|\nabla\nabla_{\boldsymbol\tau}^k\phi\|_{L^2}^2+ \int_{0}^{T}  \delta \|\nabla\nabla_{\boldsymbol\tau}^{k+1}\phi\|_{L^2}^2  {\rm d}t
\nonumber \\
& +  \int_{0}^{T} \left( \|\nabla_{\boldsymbol\tau}^k(\mu-\bar\mu)\|_{L^2}^2 +\|\nabla_{\boldsymbol\tau}^k\partial_{\bf n}\phi\|_{L^2(\Gamma)}^2 \right) {\rm d}t  \le  C \int_{0}^{T}  \mathcal{E} (t)^{3} {\rm d}t + C,  \hspace{0.5cm} \text{for}~k=1,2.
\end{align}

{\bf Step 2. Estimates of $ \partial_{\bf n} \partial_{\bf n} \phi$ and $ \nabla\partial_{\bf n}\partial_{\bf n} \phi$.}

Next, we consider the estimates of the normal direction of $\phi$.
Direct calculation shows that
\begin{align}
\frac{1}{2}\frac{\rm d}{{\rm d}t}\|\partial_{\bf n}\phi\|_{L^2(\Gamma)}^2
&=\int_\Gamma\partial_{\bf n}\phi\partial_{\bf n}\phi_t {\rm d}S
\le C\|\partial_{\bf n}\phi\|_{H^\frac{1}{2}(\Gamma)}^2
+C\|\partial_{\bf n}\phi_t\|_{H^{-\frac{1}{2}}(\Gamma)}^2\nonumber
\\
&\le C\|\nabla\phi\|_{H^1}^2+C\|\nabla\phi_t\|_{L^2}^2 \le C{\mathcal E}(t).
\nonumber
\end{align}
After integrating over  $(0,T)$, we can deduce
\begin{align}
\|\partial_{\bf n}\phi\|_{L^2(\Gamma)}^2\le C\int_{0}^{T}{\mathcal E}(t) {\rm d}t+C.\nonumber
\end{align}
It follows from $(\ref{delta NSAC_1})_{3,4}$ that
\begin{align}
\partial_{\bf n}\partial_{\bf n}\phi=\phi_t-\Delta_{\boldsymbol\tau}\phi
-\delta\Delta_{\boldsymbol\tau}\phi+{\bf u}\cdot\nabla\phi+f-\bar f+\overline{\Delta\phi},   \nonumber
\end{align}
which together with $0<\delta\le 1$ implies
\begin{align}\label{ll-nnphi}
\|\partial_{\bf n}\partial_{\bf n}\phi\|_{L^2}^2
&\le C\left(\|\Delta_{\boldsymbol\tau}\phi\|_{L^2}^2+\|\phi_t\|_{L^2}^2
+\|{\bf u}\cdot\nabla\phi\|_{L^2}^2+\|f\|_{L^2}+\|\bar f\|_{L^2}^2+\|\overline{\Delta\phi}\|_{L^2}^2
\right)                                                \nonumber
\\
&\le C\left(\|\nabla\nabla_{\boldsymbol\tau}\phi\|_{L^2}^2+\|\phi_t\|_{L^2}^2
+\|{\bf u}\|_{L^6}^2\|\nabla\phi\|_{L^3}^2+\|f\|_{L^2}^2
+\|\partial_{\bf n}\phi\|_{L^2(\Gamma)}^2
\right)                                                \nonumber
\\
&\le C\left(\|\nabla\nabla_{\boldsymbol\tau}\phi\|_{L^2}^2+\|\phi_t\|_{L^2}^2
+\|{\bf u}\|_{H^1}^2\|\nabla\phi\|_{L^2}\|\nabla\phi\|_{H^1}\right) \nonumber
\\
&\quad+C\left(\|\phi\|_{L^6}^6+\|\phi\|_{L^2}^2+\|\partial_{\bf n}\phi\|_{L^2(\Gamma)}^2
\right)
\nonumber
\\
&\le\frac{1}{2}\|\nabla\phi\|_{H^1}^2+ C\left(\|\nabla\nabla_{\boldsymbol\tau}\phi\|_{L^2}^2+\|\phi_t\|_{L^2}^2
+\|{\bf u}\|_{H^1}^4\|\nabla\phi\|_{L^2}^2\right) \nonumber
\\
&\quad+C\left(\|\phi\|_{H^1}^6+\|\phi\|_{L^2}^2
+\|\partial_{\bf n}\phi\|_{L^2(\Gamma)}^2\right)            \nonumber
\\
&\le\frac{1}{2} \| \partial_{\bf n}\partial_{\bf n} \phi\|_{L^2}^2+ C\left(\|\nabla\nabla_{\boldsymbol\tau}\phi\|_{L^2}^2+\|\phi_t\|_{L^2}^2
+\|{\bf u}\|_{H^1}^4\|\nabla\phi\|_{L^2}^2\right) \nonumber
\\
&\quad+C\left(\|\phi\|_{H^1}^6+\|\phi\|_{L^2}^2
+\|\partial_{\bf n}\phi\|_{L^2(\Gamma)}^2\right)   \nonumber
\\
&\le\frac{1}{2} \| \partial_{\bf n}\partial_{\bf n} \phi\|_{L^2}^2 +C\int_{0}^{T}{\mathcal E}(t)^4{\rm d}t+C,
\end{align}
where we have used \eqref{4.1}, \eqref{L-4.2} and  the fact $ \overline{\Delta\phi} = \int_{\Gamma} -\partial_{\bf n} \phi {\rm d}S \le C\|\partial_{\bf n}\phi\|_{L^2(\Gamma)}$. Thus, \eqref{ll-nnphi} together with (\ref{E6}) and ({\ref{ll-r22}}) yields
\begin{align}
\label{ll-phi2}
\|\phi\|_{H^2}^2\le C\int_{0}^{T} \mathcal{E}(t)^4{\rm d}t+C.
\end{align}

Similarly, from $(\ref{delta NSAC_1})_{3,4}$, we obtain
\begin{align}
\label{ll-D22}
\|\nabla\partial_{\bf n}\partial_{\bf n}\phi\|_{L^2}^2
&\le C\left(\|\nabla\Delta_{\boldsymbol\tau}\phi\|_{L^2}^2
+\|\nabla\phi_t\|_{L^2}^2
+\|\nabla{\bf u}\cdot\nabla\phi\|_{L^2}^2
+\|{\bf u}\cdot\nabla^2\phi\|_{L^2}^2+\|\nabla f\|_{L^2}^2
\right)                                                \nonumber
\\
&\le C\left(\|\nabla\nabla_{\boldsymbol\tau}^2\phi\|_{L^2}^2
+\|\nabla\phi_t\|_{L^2}^2
+\|\nabla{\bf u}\|_{L^2}^2\|\nabla\phi\|_{L^\infty}^2 \right) \nonumber
\\
&\quad+C\left(\|{\bf u}\|_{L^6}^2\|\nabla^2\phi\|_{L^3}^2+\|3\phi^2-1\|_{L^\infty}^2
\|\nabla\phi\|_{L^2}^2
\right)                                                \nonumber
\\
&\le C\left(\|\nabla\nabla_{\boldsymbol\tau}^2\phi\|_{L^2}^2
+\|\nabla\phi_t\|_{L^2}^2
+\|\nabla{\bf u}\|_{L^2}^2\|\nabla\phi\|_{L^2}^\frac{1}{2}\|\nabla\phi\|_{H^2}^\frac{3}{2}
\right)                                                \nonumber
\\
&\quad+C\left(\|{\bf u}\|_{H^1}^2\|\nabla^2\phi\|_{L^2}\|\nabla^2\phi\|_{H^1}+\|\phi\|_{H^2}^4
\|\nabla\phi\|_{L^2}^2+\|\nabla\phi\|_{L^2}^2\right)  \nonumber
\\
&\le\frac{1}{2}\left(\|\nabla\phi\|_{H^2}^2+\|\nabla^2\phi\|_{H^1}^2
\right)+C\left(\|\nabla\nabla_{\boldsymbol\tau}^2\phi\|_{L^2}^2
+\|\nabla\phi_t\|_{L^2}^2
\right)                                               \nonumber
\\
&\quad+C\left(\|\nabla{\bf u}\|_{L^2}^8\|\nabla\phi\|_{L^2}^2+\|{\bf u}\|_{H^1}^4\|\nabla^2\phi\|_{L^2}^2+\|\phi\|_{H^2}^4\|\nabla\phi\|_{L^2}^2
+\|\nabla^2\phi\|_{L^2}^2
\right)                                               \nonumber
\\
&\le \frac{1}{2} \|\nabla \partial_{\bf n} \partial_{\bf n} \phi\|_{L^2}^2+C\left(\|\nabla\nabla_{\boldsymbol\tau}^2\phi\|_{L^2}^2
+\|\nabla\phi_t\|_{L^2}^2+\|\nabla{\bf u}\|_{L^2}^8\right) \nonumber
\\
&\quad+C\left(\|{\bf u}\|_{H^1}^4\|\nabla^2\phi\|_{L^2}^2+\|\phi\|_{H^2}^4
+\|\nabla^2\phi\|_{L^2}^2
\right)                                               \nonumber
\\
&\le\frac{1}{2} \|\nabla \partial_{\bf n} \partial_{\bf n} \phi\|_{L^2}^2 +C\int_{0}^{T}{\mathcal E}(t)^{8}{\rm d}t+C.
\end{align}
Combining \eqref{ll-r22} and \eqref{ll-D22}, we have
\begin{align}\label{ll-phi33}
\|\nabla^3\phi\|_{L^2}\le C\|\nabla\nabla_{\boldsymbol\tau}{^2}\phi\|_{L^2}+C\|\nabla\partial_{\bf n}\partial_{\bf n}\phi\|_{L^2} \le C\int_{0}^{T}{\mathcal E}(t)^{8}{\rm d}t+C.
\end{align}
{\it\bfseries Step 3. Estimates of $\|{\bf u}\|_{L^\infty(0,T;H^2)}$}. A straight application of \cite{1} (see Theorem 1.2 therein), using Young inequality, we obtain
\begin{align}
\|{\bf u}\|_{H^2}^2 &\le C\left(\|{\bf u}\cdot\nabla{\bf u}\|_{L^2}^2+\|{\rm div}(\nabla\phi\otimes\nabla\phi)\|_{L^2}^2+\|{\bf u}_t\|_{L^2}^2 \right)
\nonumber \\
& \quad + C \|\left(\partial_{\bf n}\phi+\gamma_{fs}^\prime(\psi)
\right)\nabla_{\boldsymbol\tau}\psi \|_{H^{\frac{1}{2}}(\Gamma)}^2    + C\| {\bf u}_{\boldsymbol\tau} \|_{H^{\frac{1}{2}}(\Gamma)}^2
\nonumber \\
&\le C\left(\|{\bf u}\|_{L^\infty}^2\|\nabla{\bf u}\|_{L^2}^2+\|\Delta\phi\|_{L^2}^2\|\nabla\phi\|_{L^\infty}^2
+\|{\bf u}_t\|_{L^2}^2 \right)
\nonumber \\
& \quad +C \|\left(\partial_{\bf n}\phi+\gamma_{fs}^\prime(\phi)\right)\nabla\phi\|_{H^1}^2+ C \| {\bf u}\|_{H^{1}}^{2}
\nonumber \\
&\le C\left(\|{\bf u}\|_{L^2}^\frac{1}{2}\|{\bf u}\|_{H^2}^\frac{3}{2}\|\nabla{\bf u}\|_{L^2}^2+\|\Delta\phi\|_{L^2}^2\|\nabla\phi\|_{H^2}^2 +\|{\bf u}_t\|_{L^2}^2 \right)
\nonumber \\
& \quad +\|\partial_{\bf n}\phi+\gamma_{fs}^\prime(\phi)\|_{L^\infty}^2 \|\nabla\phi\|_{H^1}^2  + C \| \nabla\partial_{\bf n}\phi  + \nabla\phi\|_{L^{4}}^{2}\|\nabla\phi\|_{L^4}^{2}  + C \| {\bf u}\|_{H^{1}}^{2}
\nonumber \\
&\le\frac{1}{2}\|{\bf u}\|_{H^2}^2+C\|{\bf u}\|_{L^2}^2\|\nabla{\bf u}\|_{L^2}^8+C\left(\|\phi\|_{H^3}^4+\|{\bf u}_t\|_{L^2}^2+\|\nabla\phi\|_{H^2}^2 \|\nabla\phi\|_{H^1}^2 \right)
\nonumber \\
& \quad +C \|\nabla\phi\|_{H^1}^2 + C\| {\bf u}\|_{H^{1}}^{2}, \nonumber
\end{align}

which together with (\ref{E6}), (\ref{4.1}), \eqref{L-4.2}, \eqref{ll-phi2} and \eqref{ll-phi33} implies
\begin{align}
\label{ll-D8}
\|{\bf u}\|_{H^2}^2\le C\int_{0}^{T}{\mathcal E}(t)^{16}{\rm d}t+C.
\end{align}
{\it\bfseries Step 4. Estimates of $\|\phi_t\|_{L^2(0,T;H^2)}$.}
Taking the operator $\nabla\nabla_{\boldsymbol\tau} $, $\partial_t\nabla_{\boldsymbol\tau} $ and $\nabla_{\boldsymbol\tau} $ to $(\ref{delta NSAC_1})_3$, $(\ref{delta NSAC_1})_4$ and (\ref{bb1}) respectively leads to
\begin{equation}
\label{ll-r10}
\begin{cases}
\nabla\nabla_{\boldsymbol\tau}(\bar\mu-\mu)
=\nabla\nabla_{\boldsymbol\tau}\phi_t -\delta \nabla\nabla_{\boldsymbol\tau} \Delta_{\boldsymbol\tau} \phi
+[\nabla\nabla_{\boldsymbol\tau},{\bf u}\cdot\nabla]\phi
+{\bf u}\cdot\nabla^2\nabla_{\boldsymbol\tau}\phi,&\rm in\quad\Omega\times(0,T),
\\
\nabla_{\boldsymbol\tau}\Delta\phi_t=-\nabla_{\boldsymbol\tau}\mu_t
+\partial_t\nabla_{\boldsymbol\tau}(\phi^3-\phi),&\rm in\quad\Omega\times(0,T),
\\
\nabla_{\boldsymbol\tau}(\bar\mu-\mu) = -\nabla_{\boldsymbol\tau} L(\psi) =-\nabla_{\boldsymbol\tau}\partial_{\bf n}\phi
-\nabla_{\boldsymbol\tau}\gamma_{fs}^\prime(\psi),&\rm on\quad\Gamma\times(0,T),
\end{cases}
\end{equation}
where $[A,B]=AB-BA$ is commutators. We consider
\begin{align}
\label{ll-r11}
-\langle\nabla_{\boldsymbol\tau}(\bar\mu-\mu),\nabla_{\boldsymbol\tau}\Delta\phi_t\rangle
&=\langle\nabla\nabla_{\boldsymbol\tau}(\bar\mu-\mu),\nabla_{\boldsymbol\tau}\nabla\phi_t\rangle
-\int_\Gamma\nabla_{\boldsymbol\tau} (\bar\mu-\mu)\nabla_{\boldsymbol\tau}\partial_{\bf n}\phi_t {\rm d}S \nonumber
\\
&=Q_1+Q_2.
\end{align}
Putting $(\ref{ll-r10})_1$ into $Q_1$, we get
\begin{align}
Q_1&=\int_\Omega\left(\nabla\nabla_{\boldsymbol\tau}\phi_t  -\delta \nabla\nabla_{\boldsymbol\tau} \Delta_{\boldsymbol\tau} \phi +[\nabla\nabla_{\boldsymbol\tau},{\bf u}\cdot\nabla]\phi
+{\bf u}\cdot\nabla^2\nabla_{\boldsymbol\tau}\phi
\right)\cdot\nabla\nabla_{\boldsymbol\tau}\phi_t{\rm d}x  \nonumber
\\
&=\|\nabla\nabla_{\boldsymbol\tau}\phi_t\|_{L^2}^2 + \frac{\delta}{2} \frac{\rm d}{{\rm d}t} \|\nabla\nabla_{\boldsymbol\tau}^{2}\phi\|_{L^2}^2+\int_\Omega\left(
[\nabla\nabla_{\boldsymbol\tau},{\bf u}\cdot\nabla]\phi
+{\bf u}\cdot\nabla^2\nabla_{\boldsymbol\tau}\phi
\right)\cdot\nabla\nabla_{\boldsymbol\tau}\phi_t{\rm d}x.   \nonumber
\end{align}
Similarly, $Q_2$ could be rewritten by using $(\ref{ll-r10})_3$ as follows
\begin{align}
Q_2&=\int_\Gamma\nabla_{\boldsymbol\tau} L(\psi)\nabla_{\boldsymbol\tau} \partial_{\bf n}\phi_t {\rm d}S        \nonumber
\\
&=\int_\Gamma\nabla_{\boldsymbol\tau} L(\psi)\left(
\partial_t\nabla_{\boldsymbol\tau} L(\psi)-\partial_t\nabla_{\boldsymbol\tau} \gamma_{fs}^\prime(\psi)
\right){\rm d}S   \nonumber
\\
&=\frac{1}{2}\frac{\rm d}{{\rm d}t}\|\nabla_{\boldsymbol\tau} L(\psi)\|_{L^2(\Gamma)}^2
-\int_\Gamma\nabla_{\boldsymbol\tau} L(\psi)\partial_t\nabla_{\boldsymbol\tau} \gamma_{fs}^\prime(\psi) {\rm d}S.   \nonumber
\end{align}
Thanks to $(\ref{ll-r10})_2$, the left-hand side of (\ref{ll-r11}) can be calculated as follows
\begin{align}
\label{ll-r12}
-\langle\nabla_{\boldsymbol\tau} (\bar\mu-\mu),\nabla_{\boldsymbol\tau} \Delta\phi_t\rangle
&=\int_\Omega\nabla_{\boldsymbol\tau} (\bar\mu-\mu)\left[\nabla_{\boldsymbol\tau} \mu_t
-\partial_t\nabla_{\boldsymbol\tau} (\phi^3-\phi)
\right]{\rm d}x  \nonumber
\\
&=-\frac{1}{2}\frac{\rm d}{{\rm d}t}\|\nabla_{\boldsymbol\tau}(\mu-\bar\mu)\|_{L^2}^2
-\int_\Omega\nabla_{\boldsymbol\tau}(\bar\mu-\mu)\partial_t\nabla_{\boldsymbol\tau}
(\phi^3-\phi){\rm d}x.
\end{align}
Substituting $Q_1$, $Q_2$ and (\ref{ll-r12}) into (\ref{ll-r10}), we have
\begin{align}
\label{ll-r13}
&\frac{1}{2}\frac{\rm d}{{\rm d}t}\left( \delta \|\nabla\nabla_{\boldsymbol\tau}^{2}\phi\|_{L^2}^2 + \|\nabla_{\boldsymbol\tau}L(\psi)\|_{L^2(\Gamma)}^2
+\|\nabla_{\boldsymbol\tau}(\bar\mu-\mu)\|_{L^2}^2
\right)+\|\nabla\nabla_{\boldsymbol\tau}\phi_t\|_{L^2}^2 \nonumber
\\
&=-\int_\Omega\left(
[\nabla\nabla_{\boldsymbol\tau},{\bf u}\cdot\nabla]\phi
+{\bf u}\cdot\nabla^2\nabla_{\boldsymbol\tau}\phi
\right)\cdot\nabla\nabla_{\boldsymbol\tau}\phi_t{\rm d}x  \nonumber
\\
&\quad-\int_\Omega\nabla_{\boldsymbol\tau}(\bar\mu-\mu)\partial_t\nabla_{\boldsymbol\tau}
(\phi^3-\phi){\rm d}x+ \int_\Gamma\nabla_{\boldsymbol\tau}L(\psi)\partial_t\nabla_{\boldsymbol\tau}\gamma_{fs}^\prime(\psi) {\rm d}S     \nonumber
\\
&=\sum_{i=1}^3 R_i.
\end{align}
We next need to control the terms on the right-hand side of the above equality.
\begin{align}
R_1&=-\int_\Omega\left(
[\nabla\nabla_{\boldsymbol\tau},{\bf u}\cdot\nabla]\phi
+{\bf u}\cdot\nabla^2\nabla_{\boldsymbol\tau}\phi
\right)\cdot\nabla\nabla_{\boldsymbol\tau}\phi_t{\rm d}x  \nonumber
\\
&=-\int_\Omega\left(
\nabla\nabla_{\boldsymbol\tau}{\bf u}\cdot\nabla\phi+\nabla{\bf u}\cdot\nabla\nabla_{\boldsymbol\tau}\phi+\nabla_{\boldsymbol\tau}{\bf u}\cdot\nabla^2\phi
+{\bf u}\cdot\nabla^2\nabla_{\boldsymbol\tau}\phi
\right)\cdot\nabla\nabla_{\boldsymbol\tau}\phi_t{\rm d}x  \nonumber
\\
&\le C\left(\|\nabla\nabla_{\boldsymbol\tau}{\bf u}\|_{L^2}\|\nabla\phi\|_{L^\infty}+\|\nabla{\bf u}\|_{L^4}\|\nabla^2\phi\|_{L^4}+\|{\bf u}\|_{L^\infty}\|\nabla^2\nabla_{\boldsymbol\tau}\phi\|_{L^2}
\right)\|\nabla\nabla_{\boldsymbol\tau}\phi_t\|_{L^2} \nonumber
\\
&\le C\|{\bf u}\|_{H^2}\|\nabla\phi\|_{H^2} \|\nabla\nabla_{\boldsymbol\tau}\phi_t\|_{L^2} \le \frac{1}{4} \|\nabla\nabla_{\boldsymbol\tau}\phi_t\|_{L^2}^{2} + C \mathcal{E}(t)^{2} .     \nonumber
\\[1em]
R_2&=-\int_\Omega\nabla_{\boldsymbol\tau}(\bar\mu-\mu)\left[
6\phi\phi_t\nabla_{\boldsymbol\tau}\phi+(3\phi^2-1)\nabla_{\boldsymbol\tau}\phi_t
\right]{\rm d}x
\nonumber \\
& \le  \|  \nabla_{\boldsymbol\tau}(\bar\mu-\mu)  \|_{L^{2}} \left( \| \phi \|_{L^{\infty}} \| \phi_{t}\|_{L^{4}} \| \nabla_{\boldsymbol\tau}\phi \|_{L^{4}} + \| 3\phi^{2} -1 \|_{L^{\infty}} \| \nabla_{\boldsymbol\tau} \phi_{t}\|_{L^{2}} \right)
\nonumber \\
& \le  C \|  \nabla_{\boldsymbol\tau}(\bar\mu-\mu)  \|_{L^{2}} \left[ \| \phi \|_{H^{2}} \| \phi_{t}\|_{H^{1}} \| \nabla_{\boldsymbol\tau}\phi \|_{H^{1}} + (\| \phi \|_{H^{2}}^{2} + 1) \| \nabla_{\boldsymbol\tau} \phi_{t}\|_{L^{2}} \right]
\nonumber \\
& \le  C \mathcal{E}(t)^{2}.
\nonumber
\\[1em]
R_3&= \int_\Gamma\nabla_{\boldsymbol\tau}L(\psi)\left(
\gamma_{fs}^{(3)}(\psi)\psi_t\nabla_{\boldsymbol\tau}\psi+\gamma_{fs}^{(2)}(\psi)\nabla_{\boldsymbol\tau}\psi_t
\right){\rm d}S     \nonumber
\\
&\le C\|\nabla_{\boldsymbol\tau}L(\psi)\|_{L^2(\Gamma)}
\|\psi_t\|_{L^4(\Gamma)}
\|\nabla_{\boldsymbol\tau}\psi\|_{L^4(\Gamma)}+C\|\nabla_{\boldsymbol\tau}L(\psi)\|_{L^2(\Gamma)}
\|\nabla_{\boldsymbol\tau}\psi_t\|_{L^2(\Gamma)}     \nonumber
\\
&\le C \|\nabla_{\boldsymbol\tau}L(\psi)\|_{L^2(\Gamma)}
\|\psi_t\|_{H^1(\Gamma)} \|\nabla_{\boldsymbol\tau}\psi\|_{H^1(\Gamma)}+C \|\nabla_{\boldsymbol\tau}L(\psi)\|_{L^2(\Gamma)}\|\nabla_{\boldsymbol\tau}\psi_t\|_{L^2(\Gamma)}
\nonumber
\\
&\le C \|\nabla_{\boldsymbol\tau}L(\psi)\|_{L^2(\Gamma)}  (\|\phi_t\|_{H^1} + \| \nabla_{\boldsymbol\tau} \phi_t\|_{H^1})\|\nabla_{\boldsymbol\tau}\phi\|_{H^2}+C \|\nabla_{\boldsymbol\tau}L(\psi)\|_{L^2(\Gamma)}    \|\nabla_{\boldsymbol\tau}\phi_t\|_{H^1}
\nonumber
\\
&\le \frac{1}{4} \|\nabla\nabla_{\boldsymbol\tau}\phi_t\|_{L^2}^{2} + C\mathcal{E}(t)^{2} .     \nonumber
\end{align}
It is worth noting that
\begin{align}
\|\nabla_{\boldsymbol\tau}L(\psi_0)\|_{L^2(\Gamma)}^2\le C\left(\|\nabla_{\boldsymbol\tau}\partial_{\bf n}\phi_0\|_{L^2(\Gamma)}^2
+\|\nabla_{\boldsymbol\tau}\psi_0\|_{L^2(\Gamma)}^2
\right)\le C\|\phi_0\|_{H^3}^2. \nonumber
\end{align}
Then, substituting $R_1$ -- $R_3 $ into (\ref{ll-r13}), integrating the result over (0,T) to get
\begin{align}
\label{ll-r14}
&\sup\limits_{0\le t \le T}\left( \delta \|\nabla\nabla_{\boldsymbol\tau}^{2}\phi\|_{L^2}^2 + \|\nabla_{\boldsymbol\tau} L(\psi)\|_{L^2(\Gamma)}^2
+\|\nabla_{\boldsymbol\tau} (\bar\mu-\mu)\|_{L^2}^2 \right)
\nonumber \\
& \quad +\int_{0}^{T}\|\nabla\nabla_{\boldsymbol\tau} \phi_t\|_{L^2}^2 {\rm d}t \le C\int_{0}^{T} \mathcal {E}(t)^{2}  {\rm d}t + C.
\end{align}
Finally, it remains for us the normal direction estimates for $\nabla^2\phi_t$.
It follows from $(\ref{delta NSAC_1})_4$ that
\begin{align}
\bar\mu_t=\frac{1}{|\Omega|}\int_\Omega\mu_t{\rm d}x=\frac{1}{|\Omega|}\int_\Omega
\left(-\Delta\phi_t+f_t\right){\rm d}x=\frac{1}{|\Omega|}\int_\Gamma\partial_{\bf n}\phi_t {\rm d}S+\frac{1}{|\Omega|}\int_\Omega f_t{\rm d}x,
\nonumber
\end{align}
which together with \eqref{4.1} and \eqref{ll-r14} yields
\begin{align}
\label{ll-r15}
&\int_{0}^{T}\|\partial_{\bf n}\partial_{\bf n}\phi_t\|_{L^2}^2{\rm d}t \nonumber
\\
&\le C\int_{0}^{T}\left(
\|\Delta_{\boldsymbol\tau}\phi_t\|_{L^2}^2+\|(\mu-\bar\mu)_t\|_{L^2}^2+\|\bar\mu_t\|_{L^2}^2
+\|(3\phi^2-1)\phi_t\|_{L^2}^2
\right) {\rm d}t   \nonumber
\\
&\le\int_{0}^{T}\left(
\|\nabla\nabla_{\boldsymbol\tau}\phi_t\|_{L^2}^2+\|(\mu-\bar\mu)_t\|_{L^2}^2
+\|\partial_{\bf n}\phi_t\|_{L^2(\Gamma)}^2+\|3\phi^2-1\|_{L^\infty}^2\|\phi_t\|_{L^2}^2
\right) {\rm d}t   \nonumber
\\
&\le\int_{0}^{T}\left(
\|\nabla\nabla_{\boldsymbol\tau}\phi_t\|_{L^2}^2+\|(\mu-\bar\mu)_t\|_{L^2}^2
+\|\partial_{\bf n}\phi_t\|_{L^2(\Gamma)}^2+\|\nabla\phi_t\|_{L^2}^2
\right) {\rm d}t   \nonumber
\\
&\le C\int_{0}^{T} \mathcal {E}(t)^{2}  {\rm d}t + C.
\end{align}
Combining with \eqref{ll-phi2}, \eqref{ll-phi33}, \eqref{ll-D8}, \eqref{ll-r14} and \eqref{ll-r15} leads \eqref{4.3}, hence, the proof of Lemma \ref{Lemma4.3} is completed.
\end{proof}

\begin{proof}[\bf Proof of Proposition $\ref{TH im}$.]
To prove Proposition \ref{TH im}, we putting Lemma $\ref{Lemma4.1}-\ref{Lemma4.3}$ together and obtain that
\begin{align}
\sup\limits_{0\le t \le T}{\mathcal E}(t)\le C_0+C_1T  \left(\sup\limits_{0\le t \le T}  {\mathcal E}(t)
\right)^{16},                               \nonumber
\end{align}
where $C_0$ and $C_1$ are constants that depend on $\Omega$, $\beta$, $\hat{\mathcal E}(0)$ and initial value, but not on $\delta$. Just as in Section 9 of \cite{8}, which explain how our estimates allow the construction of solution on a time interval independent of $\delta$, thus we conclude
\begin{align}
\label{yizhi}
\sup\limits_{0\le t \le T_0}{\mathcal E}(t)\le 2C_0.
\end{align}
The proof of Proposition \ref{TH im} is completed.
\end{proof}
\subsubsection{\bf Proof of Theorem \ref{relaxation local}}\quad
It is easy to verify that the conditions of Theorem \ref{delta problem} are satisfied. Now, to prove Theorem \ref{relaxation local}, we first construct a sequence of $\delta$-approximate
solutions $({\bf u}^\delta,\phi^\delta, \psi^{\delta})$. Due to Theorem \ref{delta problem}, the initial boundary value problem (\ref{delta NSAC_1}), (\ref{delta NBC_1}), (\ref{initial1}) with initial data $({\bf u}_0^\delta,\phi_0^\delta, \psi_{0}^{\delta} )$ has a classical solution $({\bf u}^\delta,\phi^\delta, \psi^{\delta} )$ on $\Omega\times(0,T^*)$. Moreover, Proposition \ref{TH im} show that the solution $({\bf u}^\delta,\phi^\delta,  \psi^{\delta} )$ satisfies (\ref{yizhi}) for any $0<t<T_0$, and $C_0$ is independent of $\delta$. With all the estimates at hand, we can  easily extracts subsequences converging to some limit $({\bf u},\phi, \psi )$ in the weak sense. Then letting $\delta\rightarrow 0$, we deduce that $({\bf u},\phi,  \psi )$ is a strong solution to (\ref{delta NSAC_1}), (\ref{delta NBC_1}), (\ref{initial1}).
\vskip2mm
It remains to prove the uniqueness of solutions. To this aim, let $({\bf u}_1,\phi_1,  \psi_{1} )$ and $({\bf u}_2,\phi_2,  \psi_{2} )$ be two solutions of the original problem (\ref{NSAC1})--(\ref{initial1}). Let ${\bf u}={\bf u}_1-{\bf u}_2$, $\phi=\phi_1-\phi_2$, $\psi=\psi_1-\psi_2$, $\mu=\bar\mu_1-\bar\mu_2$ and $p=p_1-p_2$. Then, we deduce that $({\bf u},\phi,  \psi )$ and the corresponding $\mu$, $\nabla p$ satisfying
\begin{equation}
\label{U1}
\begin{cases}
{\bf u}_t-{\rm div}{\mathbb S}({\bf u})+\nabla p=-({\bf u}\cdot\nabla){\bf u}_1-({\bf u}_2\cdot\nabla){\bf u}-\Delta\phi\cdot\nabla\phi_1-\Delta\phi_2\cdot\nabla\phi ,&\rm in\quad\Omega\times(0,T),
\\
{\rm div}{\bf u}=0,&\rm in\quad\Omega\times(0,T),
\\
\phi_t=\bar\mu-\mu-{\bf u}\cdot\nabla\phi_1-{\bf u}_2\cdot\nabla\phi,&\rm in\quad\Omega\times(0,T),
\\
\mu=-\Delta\phi+\phi(\phi_1^2+\phi_1\phi_2+\phi_2^2-1),&\rm in\quad\Omega\times(0,T),
\end{cases}
\end{equation}
with the boundary conditions
\begin{equation}
\label{U2}
\begin{cases}
{\bf u}\cdot{\bf n}=0,&\rm on\quad\Gamma\times(0,T),
\\
\beta{\bf u}_{\boldsymbol {\tau}}+({\mathbb S}({\bf u})\cdot{\bf n})_{\boldsymbol{\tau}}=g,&\rm on\quad\Gamma\times(0,T),
\\
\psi_t=-{\bf u}_{\boldsymbol{\tau}}\cdot\nabla_{\boldsymbol{\tau}}\psi_1-{\bf u}_{2{\boldsymbol{\tau}}}\cdot\nabla_{\boldsymbol{\tau}}\psi
-\partial_{\bf n}\phi-\gamma_{fs}^{(2)}(\theta\psi_1+(1-\theta)\psi_2)\psi
,&\rm on\quad\Gamma\times(0,T),
\\
\phi\big|_\Gamma=\psi,&\rm on\quad(0,T),
\end{cases}
\end{equation}
where
\begin{align}
g=\partial_{\bf n}\phi\nabla_{\boldsymbol\tau}\psi_1
+\partial_{\bf n}\phi_2\nabla_{\boldsymbol\tau}\psi
+\gamma_{fs}^{(2)}(\theta\psi_1+(1-\theta)\psi_2)\psi\nabla_{\boldsymbol\tau}\psi_1
+\gamma_{fs}^\prime(\psi_2)\nabla_{\boldsymbol\tau}\psi,\;
\theta\in(0,1)                                  \nonumber
\end{align}
and the initial conditions
\begin{equation}
\label{U3}
({\bf u},\phi)\big|_{t=0}=(0,0),                                                      \;\rm in\;\Omega.
\end{equation}
\vskip2mm
First, multiplying $(\ref{U1})_1$ by ${\bf u}$ and integrating the result over $\Omega$ by parts, using (\ref{poincare u}) we
obtain
\begin{align}
\label{U7}
&\frac{1}{2}\frac{\rm d}{{\rm d}t}\|{\bf u}\|_{L^2}^2+\frac{1}{2}\|\mathbb S({\bf u})\|_{L^2}^2+\beta\|{\bf u}_{\boldsymbol\tau}\|_{L^2(\Gamma)}^2    \nonumber
\\
&=\int_\Omega\left(-{\bf u}\cdot\nabla{\bf u}_1-{\bf u}_2\cdot\nabla{\bf u}-\Delta\phi\nabla\phi_1-\Delta\phi_2\nabla\phi
\right)\cdot{\bf u}{\rm d}x                            \nonumber
\\
&\quad+\int_\Gamma\left[\partial_{\bf n}\phi\nabla_{\boldsymbol\tau}\psi_1
+\partial_{\bf n}\phi_2\nabla_{\boldsymbol\tau}\psi
+\gamma_{fs}^{(2)}(\theta\psi_1+(1-\theta)\psi_2)\psi\nabla_{\boldsymbol\tau}\psi_1
+\gamma_{fs}^\prime(\psi_2)\nabla_{\boldsymbol\tau}\psi
\right]\cdot{\bf u}_{\boldsymbol\tau}{\rm d}S         \nonumber
\\
&=\int_\Omega\left(-{\bf u}\cdot\nabla{\bf u}_1-{\bf u}_2\cdot\nabla{\bf u}+\nabla\phi\cdot\nabla^2\phi_1-\Delta\phi_2\nabla\phi
\right)\cdot{\bf u}{\rm d}x +\int_\Omega\nabla\phi\nabla\phi_1:\nabla{\bf u}{\rm d}x \nonumber
\\
&\quad+ \int_\Gamma\left[
\partial_{\bf n}  \phi_2 \nabla_{\boldsymbol\tau} \psi +\gamma_{fs}^{(2)}
(\theta\psi_1+(1-\theta)\psi_2)\psi\nabla_{\boldsymbol\tau}\psi_1
+\gamma_{fs}^{(2)}(\psi_2)\nabla_{\boldsymbol\tau}\psi
\right]\cdot{\bf u}_{\boldsymbol\tau}{\rm d}S
\nonumber \\
&\le\|{\bf u}\|_{L^2}^2\|\nabla{\bf u}_1\|_{L^\infty}+\|{\bf u}_2\|_{L^\infty}\|\nabla{\bf u}\|_{L^2}\|{\bf u}\|_{L^2}+\left(\|\nabla^2\phi_1\|_{L^4}+\|\Delta\phi_2\|_{L^4}\right)\|\nabla\phi\|_{L^2}\|{\bf u}\|_{L^4} \nonumber
\\
&\quad+\|\nabla\phi\|_{L^2}\|\nabla\phi_1\|_{L^\infty}\|\nabla{\bf u}\|_{L^2}+\|\partial_{\bf n}\phi\|_{H^{-\frac{1}{2}}(\Gamma)} \|\nabla_{\boldsymbol\tau}\psi_1\cdot{\bf u}_{\boldsymbol\tau}\|_{H^\frac{1}{2}(\Gamma)}   \nonumber
\\
&\quad +  \left( \| \partial_{\bf n}\phi_2 {\bf u}_{\boldsymbol\tau} \|_{H^\frac{1}{2}(\Gamma)} + \| \gamma_{fs}^{(2)}(\psi_2)  {\bf u}_{\boldsymbol\tau}   \|_{H^\frac{1}{2}(\Gamma)}  \right) \| \nabla_{\boldsymbol\tau} \psi \|_{H^{-\frac{1}{2}}(\Gamma)}
\nonumber \\
& \quad +\| \gamma_{fs}^{(2)} (\theta\psi_1+(1-\theta)\psi_2)\psi \|_{H^{-\frac{1}{2}}(\Gamma)}  \|  \nabla_{\boldsymbol\tau}\psi_1 {\bf u}_{\boldsymbol\tau} \|_{H^{\frac{1}{2}}(\Gamma)}
\nonumber \\
&\le C\|\nabla{\bf u}_1\|_{H^2}\|{\bf u}\|_{L^2}^2+C\|{\bf u}_2\|_{H^2}\|\nabla{\bf u}\|_{L^2}\|{\bf u}\|_{L^2}+C\left(\|\nabla^2\phi_1\|_{H^1}+\|\Delta\phi_2\|_{H^1}\right)\|\nabla\phi\|_{L^2}\|{\bf u}\|_{H^1} \nonumber
\\
&\quad+C\|\nabla\phi\|_{L^2}\|\nabla\phi_1\|_{H^2}\|\nabla{\bf u}\|_{L^2} +C\|\phi\|_{L^2} \left(\|\nabla\phi_1\|_{L^\infty}\|{\bf u}\|_{H^1}+\|\nabla^2\phi_1\|_{L^4}\|{\bf u}\|_{L^4}\right) \nonumber
\\
&\quad+C\left[ \left(\| \partial_{\bf n} \phi_{2} \|_{L^{\infty}} + 1 \right) \| {\bf u} \|_{H^{1}}  + \left( \| \nabla \partial_{\bf n} \phi_{2}  \|_{L^{4}} + \| \nabla \phi_{2}  \|_{L^{4}}  \right) \| {\bf u} \|_{L^{4}} \right] \| \nabla \phi\|_{L^2}
\nonumber \\
&\le\epsilon\|\nabla{\bf u}\|_{L^2}^2+C\left(\|\nabla{\bf u}_1\|_{H^2}+\|{\bf u}_2\|_{H^2}^2+1 \right)\|{\bf u}\|_{L^2}^2+C\left(\|\nabla\phi_1\|_{H^2}^2+\|\nabla\phi_2\|_{H^2}^2 + 1 \right)\|\phi\|_{H^1}^2 \nonumber
\\
&\le\epsilon\|\nabla{\bf u}\|_{L^2}^2+C\left(\|\nabla{\bf u}_1\|_{H^2}+1 \right)\|{\bf u}\|_{L^2}^2 +C\|\phi\|_{H^1}^2.
\end{align}
\vskip2mm
Next, we consider
\begin{align}
\label{U4}
\langle\mu-\bar\mu,\Delta\phi\rangle=-\langle\nabla\mu,\nabla\phi \rangle+\int_\Gamma(\mu-\bar\mu)\partial_{\bf n}\phi {\rm d}S=H_1+H_2.
\end{align}
Substituting $(\ref{U1})_3$ into $H_1$, then integration by parts gives
\begin{align}
\label{U5}
H_1&=\int_\Omega\left(\nabla\phi_t+\nabla{\bf u}\cdot\nabla\phi_1+{\bf u}\cdot\nabla^2\phi_1+\nabla{\bf u}_2\cdot\nabla\phi+{\bf u}_2\cdot\nabla^2\phi
\right)\cdot\nabla\phi {\rm d}x                          \nonumber
\\
&=\frac{1}{2}\frac{\rm d}{{\rm d}t}\|\nabla\phi\|_{L^2}^2+\int_\Omega\left(
\nabla{\bf u}\cdot\nabla\phi_1+{\bf u}\cdot\nabla^2\phi_1+\nabla{\bf u}_2\cdot\nabla\phi
\right)\cdot\nabla\phi {\rm d}x,
\end{align}
since
\begin{align}
\int_\Omega{\bf u}_2\cdot\nabla^2\phi\cdot\nabla\phi {\rm d}x&=\frac{1}{2}\int_\Omega{\bf u}_2\cdot\nabla(|\nabla\phi|^2){\rm d}x                  \nonumber
\\
&=-\frac{1}{2}\int_\Omega{\rm div}{\bf u}_2(|\nabla\phi|^2){\rm d}x
+\frac{1}{2}\int_\Gamma{\bf u}_2\cdot{\bf n}(|\nabla\phi|^2){\rm d}S=0.                            \nonumber
\end{align}
We notice that $\mu-\bar\mu=\partial_{\bf n}\phi+ \gamma_{fs}^{(2)}(\theta\psi_1
+(1-\theta)\psi_2)\psi$ on $\Gamma$, thus we deduce
\begin{align}
H_2&=\int_\Gamma\left[\partial_{\bf n}\phi+\gamma_{fs}^{(2)}
(\theta\psi_1+(1-\theta)\psi_2)\psi
\right]\partial_{\bf n}\phi {\rm d}S                          \nonumber
\\
&=\|\partial_{\bf n}\phi\|_{L^2(\Gamma)}^2+\int_{\Gamma}
\gamma_{fs}^{(2)}(\theta\psi_1+(1-\theta)\psi_2)\psi\partial_{\bf n}\phi {\rm d}S.
\end{align}
On the other hand, it follows from $(\ref{U1})_4$ that
\begin{align}
\label{U6}
\int_\Omega(\mu-\bar\mu)\Delta\phi {\rm d}x&=\int_\Omega(\mu-\bar\mu)\left[\phi(\phi_1^2+\phi_1\phi_2+\phi_2^2-1)-\mu
\right]{\rm d}x                                          \nonumber
\\
&=-\|\mu-\bar\mu\|_{L^2}^2-\bar\mu\int_\Omega(\mu-\bar\mu){\rm d}x
+\int_\Omega(\mu-\bar\mu)(\phi_1^2+\phi_1\phi_2+\phi_2^2-1)\phi {\rm d}x                                                 \nonumber
\\
&=-\|\mu-\bar\mu\|_{L^2}^2+\int_\Omega(\mu-\bar\mu)
(\phi_1^2+\phi_1\phi_2+\phi_2^2-1)\phi {\rm d}x.
\end{align}
Replacing (\ref{U5})-(\ref{U6}) into (\ref{U4}) we then deduce
\begin{align}
\label{U10}
&\frac{1}{2}\frac{\rm d}{{\rm d}t}\|\nabla\phi\|_{L^2}^2+\|\mu-\bar\mu\|_{L^2}^2
+\|\partial_{\bf n}\phi\|_{L^2(\Gamma)}^2        \nonumber
\\
&=-\int_\Omega\left(
\nabla{\bf u}\cdot\nabla\phi_1+{\bf u}\cdot\nabla^2\phi_1+\nabla{\bf u}_2\cdot\nabla\phi
\right)\cdot\nabla\phi {\rm d}x-\int_\Gamma
\gamma_{fs}^{(2)}(\theta\psi_1+(1-\theta)\psi_2)\psi\partial_{\bf n}\phi {\rm d}S \nonumber
\\
&\quad+\int_\Omega(\mu-\bar\mu)
(\phi_1^2+\phi_1\phi_2+\phi_2^2-1)\phi {\rm d}x          \nonumber
\\
&\le \|\nabla{\bf u}\|_{L^2}\|\nabla\phi_1\|_{L^\infty}\|\nabla\phi\|_{L^2}+\|{\bf u}\|_{L^4}\|\nabla^2\phi_1\|_{L^4}\|\nabla\phi\|_{L^2}
+\|\nabla{\bf u}_2\|_{L^\infty}\|\nabla\phi\|_{L^2}^2
\nonumber
\\
&\quad+C\|\partial_{\bf n}\phi\|_{L^2(\Gamma)} \|\psi\|_{L^2(\Gamma)}
+\|\mu-\bar\mu\|_{L^2}
\|\phi_1^2+\phi_1\phi_2+\phi_2^2-1\|_{L^\infty}\|\phi\|_{L^2}
\nonumber
\\
&\le\epsilon\|\nabla{\bf u}\|_{L^2}^2+\frac{1}{2}\|\mu-\bar\mu\|_{L^2}^2
+\frac{1}{2}\|\partial_{\bf n}\phi\|_{L^2(\Gamma)}^2+C\|{\bf u}\|_{L^2}^2+C\|\psi\|_{L^2(\Gamma)}^2    \nonumber
\\
&\quad+\left(\|\nabla\phi_1\|_{L^\infty}^2+\|\nabla^2\phi_1\|_{L^4}^2
+\|\nabla{\bf u}_2\|_{L^\infty}
\right)\|\nabla\phi\|_{L^2}^2
+C\|\phi_1^2+\phi_1\phi_2+\phi_2^2-1\|_{L^\infty}^2\|\phi\|_{L^2}^2
\nonumber
\\
&\le\epsilon\|\nabla{\bf u}\|_{L^2}^2+\frac{1}{2}\|\mu-\bar\mu\|_{L^2}^2
+\frac{1}{2}\|\partial_{\bf n}\phi\|_{L^2(\Gamma)}^2+C\|{\bf u}\|_{L^2}^2+C\|\phi\|_{H^1}^2                      \nonumber
\\
&\quad+\left(\|\nabla\phi_1\|_{H^2}^2+\|\nabla^2\phi_1\|_{H^1}^2
+\|\nabla{\bf u}_2\|_{H^2}+\|\phi_1^2+\phi_1\phi_2+\phi_2^2-1\|_{H^2}^2
\right)\|\phi\|_{H^1}^2                             \nonumber
\\
&\le\epsilon\|\nabla{\bf u}\|_{L^2}^2+\frac{1}{2}\|\mu-\bar\mu\|_{L^2}^2
+\frac{1}{2}\|\partial_{\bf n}\phi\|_{L^2(\Gamma)}^2+C\|{\bf u}\|_{L^2}^2+C\left(\|\nabla{\bf u}_2\|_{H^2}+1
\right)\|\phi\|_{H^1}^2.
\end{align}
Putting (\ref{U7}) and (\ref{U10}) together, due to Korn's inequality, we obtain
\begin{align}
\label{U8}
&\frac{\rm d}{{\rm d}t}\left(\|{\bf u}\|_{L^2}^2+\|\nabla\phi\|_{L^2}^2
\right)+\|\mathbb S({\bf u})\|_{L^2}^2+\beta\|{\bf u}_{\boldsymbol\tau}\|_{L^2(\Gamma)}^2+\|\mu-\bar\mu\|_{L^2}^2
+\|\partial_{\bf n}\phi\|_{L^2(\Gamma)}^2         \nonumber
\\
&\le C\epsilon\|{\mathbb S}({\bf u})\|_{L^2}^2+C\left(\|\nabla{\bf u}_1\|_{H^2}+\|\nabla{\bf u}_2\|_{H^2}+1\right)\left(\|{\bf u}\|_{L^2}^2+\|\phi\|_{H^1}^2
\right).
\end{align}
It is worth noting that, from $(\ref{U1})_3$ and integration by parts, we have
\begin{align}
\frac{\rm d}{{\rm d}t}\int_\Omega\phi {\rm d}x=\int_\Omega\phi_t{\rm d}x=\int_\Omega\left(\bar\mu-\mu-{\bf u}\cdot\nabla\phi_1-{\bf u}_2\cdot\nabla\phi\right){\rm d}x=0,
\nonumber
\end{align}
which implies $\int_\Omega\phi {\rm d}x=\int_\Omega\phi_0 {\rm d}x=0$.
Then, using Poincar$\rm\acute{e}$ inequality, we obtain
\begin{align}
\label{U9}
\|\phi\|_{L^2}^2=\|\phi-\bar\phi\|_{L^2}^2\le C\|\nabla\phi\|_{L^2}^2.
\end{align}
Combining (\ref{U8}) with (\ref{U9}), choosing $C\epsilon<1/2$, then Gronwall inequality give
\begin{align}
{\bf u}=\phi=0,\;\rm a.e.\; in\;\Omega.           \nonumber
\end{align}
This completes the proof of Theorem \ref{relaxation local}.
$\hfill\square$
\subsection{Global well-posedness}
In this section, basing on the local well-posedness of the problem $(\ref{NSAC1})$--$(\ref{initial1})$,
we present the proof of global existence by using continuation argument.

\subsubsection{\bf Time-independent estimates}\quad
We first introduce a key proposition of this subsection.
\begin{Proposition}
\label{proposition global}
Under the assumption of Theorem $\ref{Th rg}$, there exists a positive constant $\sigma$ and a small positive constant $\varepsilon_0$ such that if $({\bf u},\phi, \psi )$ is the solution of $(\ref{NSAC1})$--$(\ref{initial1})$ on $\Omega\times(0,T)$ satisfying
\begin{align}
\label{small energy}
\sup\limits_{0\le t \le T}\tilde{\mathcal E}(t)\le\sigma,
\end{align}
then there holds
\begin{align}
\label{555}
\sup\limits_{0\le t \le T}\tilde{\mathcal E}(t)+\int_{0}^{T}\tilde{\mathcal D}(t) {\rm d}t\le C_2\left(\varepsilon_0+\tilde{\mathcal E}(0)
\right).
\end{align}
Moreover, there exists a positive constant $\alpha$ such that
\begin{align}
\|{\bf u}(\cdot,t)\|_{L^2}^2
+\|L(\psi)(\cdot,t)\|_{L^2(\Gamma)}^2  +\|(\mu-\bar\mu)(\cdot,t)\|_{L^2}^2\le C\tilde{\mathcal E}(0)e^{-\alpha t},\quad for\;all\;t\ge0.
\end{align}
\end{Proposition}
Finally, we will explain here a conclusion that is very important for subsequent estimates, which can be easily got due to the priori assumption (\ref{small energy}) and smallness of $\sigma$ that
\begin{align}
\label{bounded phi}
\frac{2}{3}\le\phi^2\le 4,
\end{align}
which plays a key role to guarantee
\begin{align}
3\phi^2-1\ge 1>0.
\end{align}
\vskip2mm
The remainder of this subsection is devoted to the proof of Proposition \ref{proposition global}.
Our analysis commences with the derivation of the basic energy estimate.
\begin{Lemma}
\label{rgl1}
Under the assumptions of Proposition $\ref{proposition global}$, it holds that
\begin{align}
\label{r1}
\sup\limits_{0\le t \le T}\tilde{\mathcal E}_0(t)+\int_{0}^{T}\tilde{\mathcal D}_0(t) {\rm d}t
\le C\varepsilon_0,
\end{align}
where
\begin{align}
&\tilde{\mathcal E}_0(t)=\|{\bf u}\|_{L^2}^2+\|\nabla\phi\|_{L^2}^2+\frac{1}{2}\|\phi^2-1\|_{L^2}^2,               \nonumber
\\
&\tilde{\mathcal D}_0(t)=\|{\bf u}\|_{H^1}^2+\beta\|{\bf u}_{\boldsymbol\tau}\|_{L^2(\Gamma)}^2+\|\mu-\bar\mu\|_{L^2}^2
+\|L(\psi)\|_{L^2(\Gamma)}^2.      \nonumber
\end{align}
\end{Lemma}
\begin{proof}[\bf Proof.]
Testing $(\ref{NSAC1})_3$ by $(\mu-\bar\mu)$, one has
\begin{align}
\label{r2}
\int_\Omega\phi_t(\mu-\bar\mu){\rm d}x+\int_\Omega{\bf u}\cdot\nabla\phi(\mu-\bar\mu){\rm d}x=-\|\mu-\bar\mu\|_{L^2}^2.
\end{align}
Due to the incompressibility of $\bf u$ and boundary condition ${\bf u}\cdot{\bf n}=0$ on $\Gamma$, we have
\begin{align}
\label{111}
\int_\Omega{\bf u}\cdot\nabla\phi {\rm d}x=-\int_\Omega{\rm div}{\bf u}\phi {\rm d}x+\int_\Gamma{\bf u}\cdot{\bf n}{\rm d}S=0,
\end{align}
which implies
\begin{align}
\label{r3}
\int_\Omega\phi_t {\rm d}x=\int_\Omega(\bar\mu-\mu) {\rm d}x-\int_\Omega{\bf u}\cdot\nabla\phi {\rm d}x=0.
\end{align}
Thus, we have
\begin{align}
\int_\Omega\phi_t(\mu-\bar\mu){\rm d}x&=\int_\Omega\phi_t\mu {\rm d}x-\bar\mu\int_\Omega\phi_t {\rm d}x  =-\int_\Omega\Delta\phi\phi_t {\rm d}x+\int_\Omega f(\phi)\phi_t {\rm d}x
\nonumber
\\
&=\int_\Omega\nabla\phi\cdot\nabla\phi_t {\rm d}x- \int_\Gamma\psi_t\partial_{\bf n}\phi {\rm d}S+\frac{1}{4}\frac{\rm d}{{\rm d}t}\int_\Omega(\phi^2-1)^2{\rm d}x    \nonumber
\\
&=\frac{1}{2}\frac{\rm d}{{\rm d}t}\left(
\|\nabla\phi\|_{L^2}^2+\frac{1}{2}\|\phi^2-1\|_{L^2}^2
\right)- \int_\Gamma\psi_t\partial_{\bf n}\phi {\rm d}S.   \nonumber
\end{align}
It follows from $(\ref{NSAC1})_4$ and (\ref{111}) that
\begin{align}\label{6666}
\int_\Omega{\bf u}\cdot\nabla\phi(\mu-\bar\mu){\rm d}x&=\int_\Omega\mu{\bf u}\cdot\nabla\phi {\rm d}x+\bar\mu\int_\Omega{\bf u}\cdot\nabla\phi {\rm d}x                                    \nonumber
\\
&=-\int_\Omega\Delta\phi\nabla\phi\cdot{\bf u}{\rm d}x+\int_\Omega(\phi^3-\phi)\nabla\phi\cdot{\bf u} {\rm d}x  \nonumber
\\
&=-\int_\Omega{\rm div}(\nabla\phi\otimes\nabla\phi)\cdot{\bf u}{\rm d}x,
\end{align}
where we have used the following fact
\begin{align}\label{666}
\int_\Omega(\phi^3-\phi)\nabla\phi\cdot{\bf u} {\rm d}x  &=\frac{1}{2}\int_\Omega\nabla(\phi^2-1)^2\cdot{\bf u} {\rm d}x
\nonumber
\\
&=-\frac{1}{2}\int_\Omega(\phi^2-1)^2{\rm div}{\bf u}{\rm d}x+\int_\Gamma(\phi^2-1)^2{\bf u}\cdot{\bf n}{\rm d}S
=0.
\end{align}
Substituting \eqref{6666} into (\ref{r2}), we get
\begin{align}
\label{r4}
\frac{1}{2}\frac{\rm d}{{\rm d}t}\left(\|\nabla\phi\|_{L^2}^2
+\frac{1}{2}\|\phi^2-1\|_{L^2}^2
\right)
+\|\mu-\bar\mu\|_{L^2}^2= \int_\Gamma\psi_t\partial_{\bf n}\phi {\rm d}S -\int_\Omega{\rm div}(\nabla\phi\otimes\nabla\phi)\cdot{\bf u}{\rm d}x.
\end{align}
Noting that
\begin{align}
\int_\Gamma\psi_t\partial_{\bf n}\phi {\rm d}S&=\int_\Gamma\psi_t
\left(L(\psi)-\gamma_{fs}^\prime(\psi)\right) {\rm d}S             =-\frac{\rm d}{{\rm d}t}\int_\Gamma\gamma_{fs}(\psi){\rm d}S
+\int_\Gamma\psi_tL(\psi) {\rm d}S                                                     \nonumber
\\
&=-\frac{\rm d}{{\rm d}t}\int_\Gamma\gamma_{fs}(\psi){\rm d}S
+\int_\Gamma(-L(\psi)-{\bf u}_{\boldsymbol\tau}\cdot\nabla_{\boldsymbol\tau}\psi)L(\psi) {\rm d}S
\nonumber
\\
&=-\frac{\rm d}{{\rm d}t}\int_\Gamma\gamma_{fs}(\psi){\rm d}S
-\|L(\psi)\|_{L^2(\Gamma)}^2
-\int_\Gamma L(\psi){\bf u}_{\boldsymbol\tau}\cdot\nabla_{\boldsymbol\tau}\psi {\rm d}S.                    \nonumber
\end{align}
Hence, (\ref{r4}) becomes
\begin{align}
\label{r5}
&\frac{1}{2}\frac{\rm d}{{\rm d}t}\left(
\|\nabla\phi\|_{L^2}^2+\frac{1}{2}\|\phi^2-1\|_{L^2}^2
+ 2\int_\Gamma\gamma_{fs}(\psi){\rm d}S
\right)
+\|\mu-\bar\mu\|_{L^2}^2
+\|L(\psi)\|_{L^2(\Gamma)}^2            \nonumber
\\
&=\int_\Omega{\rm div}(\nabla\phi\otimes\nabla\phi)\cdot{\bf u}{\rm d}x-\int_\Gamma L(\psi){\bf u}_{\boldsymbol\tau}\cdot\nabla_{\boldsymbol\tau}\psi {\rm d}S  .
\end{align}
Then, testing $(\ref{NSAC1})_1$ by $\bf u$ and rearranging, we have
\begin{align}
\label{r6}
\frac{1}{2}\frac{\rm d}{{\rm d}t}\|{\bf u}\|_{L^2}^2+\frac{1}{2}\|{\mathbb S}({\bf u})\|_{L^2}^2+\beta\|{\bf u}_{\boldsymbol\tau}\|_{L^2(\Gamma)}^2
= \int_\Gamma L(\psi){\bf u}_{\boldsymbol\tau}\cdot\nabla_{\boldsymbol\tau}\psi {\rm d}S  -\int_\Omega {\rm div}(\nabla\phi\otimes\nabla\phi)\cdot{\bf u} {\rm d} x.
\end{align}
Combining the (\ref{r5}) with (\ref{r6}), together with Gronwall inequality shows
\begin{align}
&\sup\limits_{0\le t \le T}\tilde{\mathcal E}_1(t)+\int_{0}^{T}\tilde{\mathcal D}_1(t) {\rm d}t=\tilde{\mathcal E}_1(0),              \nonumber
\end{align}
where
\begin{align}
\tilde{\mathcal E}_1(t)&=\|{\bf u}\|_{L^2}^2+\|\nabla\phi\|_{L^2}^2+\frac{1}{2}\|\phi^2-1\|_{L^2}^2
+2\int_\Gamma\gamma_{fs}(\psi){\rm d}S ,    \nonumber
\\
\tilde{\mathcal D}_1(t)&=\|{\mathbb S}({\bf u})\|_{L^2}^2+\beta\|{\bf u}_{\boldsymbol\tau}\|_{L^2}^2+\|\mu-\bar\mu\|_{L^2}^2+ \|L(\psi)\|_{L^2(\Gamma)}^2  .
\nonumber
\end{align}
Note that
\begin{align}
\Big|\int_\Gamma\gamma_{fs}(\psi){\rm d}S \Big|\le C|\nu \cos\theta_s|\le C \varepsilon_0,
\end{align}
$\tilde{\mathcal E}_0(0)\le\tilde{\mathcal E}_1(0)+C\varepsilon_0\le C\varepsilon_0$.
Furthermore, it follows Lemma \ref{poincare U} and Korn's inequality that $\tilde{\mathcal D}_1(t)\ge C\tilde{\mathcal D}_0(t)$. Therefore, we complete the proof of Lemma \ref{rgl1}.
\end{proof}
\begin{Lemma}
\label{rgl2}
Under the assumptions of Proposition $\ref{proposition global}$, it holds that
\begin{align}
\label{r25}
&\sup\limits_{0\le t \le T}\|\phi_t\|_{H^1}^2+\int_{0}^{T}\left(\|(\mu-\bar\mu)_t\|_{L^2}^2
+\|\partial_{\bf n}\phi_t\|_{L^2(\Gamma)}^2\right) {\rm d}t  \nonumber
\\
&\le C\tilde{\mathcal E}(0)+C\int_{0}^{T}\left(\varepsilon_0+\tilde{\mathcal E}(t)^{\frac{1}{2}}
\right)\tilde{\mathcal D}(t) {\rm d}t.
\end{align}
\end{Lemma}
\begin{proof}[\bf Proof.]
Differentiating $(\ref{NSAC1})_{3,4}$ and $(\ref{bb1})$ with respect to $t$ leads to
\begin{equation}
\label{pt}
\begin{cases}
(\bar\mu-\mu)_t=\phi_{tt}+{\bf u}_t\cdot\nabla\phi+{\bf u}\cdot\nabla\phi_t,&\rm in\quad\Omega\times(0,T),
\\
\Delta\phi_t=-\mu_t+f_t,&\rm in\quad\Omega\times(0,T),
\\
(\bar\mu-\mu)_t=-\partial_{\bf n}\phi_t- \gamma_{fs}^{(2)}(\psi)\psi_t ,&\rm on\quad\Gamma\times(0,T).
\end{cases}
\end{equation}
We start estimating from
\begin{align}
\label{r7}
-\langle(\bar\mu-\mu)_t,\Delta\phi_t\rangle=\langle\nabla(\bar\mu-\mu)_t,\nabla\phi_t\rangle
-\int_\Gamma(\bar\mu-\mu)_t\partial_{\bf n}\phi_t {\rm d}S=L_1+L_2.
\end{align}
Recalling that $L_2=L_2^\delta$, where $L_2^\delta$ is defined as in (\ref{L2L2}), and
\begin{align}
L_1&=L_1^\delta+\int_\Omega\delta\Delta_{\boldsymbol\tau}\phi_t\cdot\nabla\phi_t{\rm d}x \nonumber
\\
&=\frac{1}{2}\frac{\rm d}{{\rm d}t}\|\nabla\phi_t\|_{L^2}^2+\int_\Omega\left(\nabla{\bf u}_t\cdot\nabla\phi+{\bf u}_t\cdot\nabla^2\phi+\nabla{\bf u}\cdot\nabla\phi_t
\right)\cdot\nabla\phi_t {\rm d}x, \nonumber
\end{align}
whereas the left-hand side of $(\ref{r7})$ could be written as in (\ref{l-r8}). Thus, we obtain
\begin{align}
\label{r9}
&\frac{1}{2}\frac{\rm d}{{\rm d}t}\|\nabla\phi_t\|_{L^2}^2+\|(\mu-\bar\mu)_t\|_{L^2}^2
+\|\partial_{\bf n}\phi_t\|_{L^2(\Gamma)}^2 \nonumber
\\
&=-\int_\Omega\left(\nabla{\bf u}_t\cdot\nabla\phi+{\bf u}_t\cdot\nabla^2\phi+\nabla{\bf u}\cdot\nabla\phi_t
\right)\cdot\nabla\phi_t {\rm d}x
\nonumber \\
&\quad-\int_\Omega(\bar\mu-\mu)_tf_t {\rm d}x- \int_\Gamma\gamma_{fs}^{(2)}(\psi)\psi_t\partial_{\bf n}\phi_t {\rm d}S
\nonumber \\
& =\sum_{i=1}^3A_i.
\end{align}
Now, we estimate each terms $A_i$ on the right-hand side of (\ref{r9}). We actually have
\begin{align}
A_1&\le\|\nabla{\bf u}_t\|_{L^2}\|\nabla\phi\|_{L^\infty}\|\nabla\phi_t\|_{L^2}
+\|{\bf u}_t\|_{L^4}\|\nabla^2\phi\|_{L^4}\|\nabla\phi_t\|_{L^2}
+\|\nabla{\bf u}\|_{L^\infty}\|\nabla\phi_t\|_{L^2}^2 \nonumber
\\
&\le C\|{\bf u}_t\|_{H^1}\|\nabla\phi\||_{H^2}\|\nabla\phi_t\|_{L^2}
+C\|{\bf u}\|_{H^3}\|\nabla\phi_t\|_{L^2}^2 \nonumber
\\
&\le C\tilde{\mathcal D}(t)^{\frac{1}{2}}\tilde{\mathcal E}(t)^{\frac{1}{2}}\tilde{\mathcal D}(t)^{\frac{1}{2}},   \nonumber
\\[1em]
A_2&=-\int_\Omega(\bar\mu-\mu)_t(3\phi^2-1)\phi_t {\rm d}x \nonumber
\\
&=-\int_\Omega\left(\phi_{tt}+{\bf u}_t\cdot\nabla\phi+{\bf u}\cdot\nabla\phi_t
\right)(3\phi^2-1)\phi_t {\rm d}x \nonumber
\\
&=-\frac{1}{2}\frac{\rm d}{{\rm d}t}\int_\Omega(3\phi^2-1)|\phi_t|^2 {\rm d}x+\int_\Omega3\phi\phi_t|\phi_t|^2{\rm d}x \nonumber
\\
&\quad-\int_\Omega\left({\bf u}_t\cdot\nabla\phi+{\bf u}\cdot\nabla\phi_t
\right)(3\phi^2-1)\phi_t {\rm d}x \nonumber
\\
&\le -\frac{1}{2}\frac{\rm d}{{\rm d}t}\int_\Omega(3\phi^2-1)|\phi_t|^2 {\rm d}x
+C\|\phi\|_{L^\infty}\|\phi_t\|_{L^2}\|\phi_t\|_{L^4}^2 \nonumber
\\
&\quad+C\|3\phi^2-1\|_{L^\infty}\|\phi_t\|_{L^2}\left(\|{\bf u}_t\|_{L^4}\|\nabla\phi\|_{L^4}+\|{\bf u}\|_{L^\infty}\|\nabla\phi_t\|_{L^2}
\right)  \nonumber
\\
&\le-\frac{1}{2}\frac{\rm d}{{\rm d}t}\int_\Omega(3\phi^2-1)|\phi_t|^2 {\rm d}x+C\|\phi_t\|_{L^2}\left(\|\phi_t\|_{H^1}^2+\|{\bf u}_t\|_{H^1}\|\nabla\phi\|_{H^1}+\|{\bf u}\|_{H^2}\|\nabla\phi_t\|_{L^2}
\right)  \nonumber
\\
&\le-\frac{1}{2}\frac{\rm d}{{\rm d}t}\int_\Omega(3\phi^2-1)|\phi_t|^2 {\rm d}x+C\tilde{\mathcal D}(t)^{\frac{1}{2}}\tilde{\mathcal E}(t)^{\frac{1}{2}}\tilde{\mathcal D}(t)^{\frac{1}{2}},   \nonumber
\\[1em]
A_3&\le C|\nu \cos \theta_s|  \|\psi_t\|_{L^2(\Gamma)} \|\partial_{\bf n}\phi_t\|_{L^2(\Gamma)}
\le C\varepsilon_0\|\phi_t\|_{H^1}\|\nabla\phi_t\|_{H^1}\le C\varepsilon_0\tilde{\mathcal D}(t). \nonumber
\end{align}
Then, substituting $A_1$ -- $A_3$ into (\ref{r9}), we have
\begin{align}
&\frac{1}{2}\frac{\rm d}{{\rm d}t}\left(\|\nabla\phi_t\|_{L^2}^2+\int_\Omega(3\phi^2-1)|\phi_t|^2 {\rm d}x\right)
+\|(\mu-\bar\mu)_t\|_{L^2}^2
+\|\partial_{\bf n}\phi_t\|_{L^2(\Gamma)}^2 \nonumber
\\
&\le C\left(\varepsilon_0+\tilde{\mathcal E}(t)^{\frac{1}{2}}
\right)\tilde{\mathcal D}(t). \nonumber
\end{align}
Integrating it over (0,T), noting that $3\phi^2-1\ge 1$, we obtain
\begin{align}
&\sup\limits_{0\le t \le T}\|\phi_t\|_{H^1}^2+\int_{0}^{T}\left(\|(\mu-\bar\mu)_t\|_{L^2}^2
+\|\partial_{\bf n}\phi_t\|_{L^2(\Gamma)}^2\right){\rm d}t \nonumber
\\
&\le C\tilde{\mathcal E}(0)+C\int_{0}^{T}\left(\varepsilon_0+\tilde{\mathcal E}(t)^{\frac{1}{2}}
\right)\tilde{\mathcal D}(t) {\rm d}t. \nonumber
\end{align}
Hence, we complete the proof of Lemma \ref{rgl2}.
\end{proof}
\begin{Lemma}\label{L477}
Under the assumptions of Proposition $\ref{proposition global}$, it holds that
\begin{align}
\label{r16}
&\sup\limits_{0\le t \le T}\left( \|\nabla_{\boldsymbol\tau}^k L(\psi)\|_{L^2(\Gamma)}^2 +\|\nabla_{\boldsymbol\tau}^k(\mu-\bar\mu)\|_{L^2}^2
\right)+\int_{0}^{T}\|\nabla\phi_t\|_{H^1}^2 {\rm d}t  \nonumber
\\
&\le C\left(\varepsilon_0+\tilde{\mathcal E}(0)\right)+C\int_{0}^{T}\left(\varepsilon_0+\tilde{\mathcal E}(t)^{\frac{1}{2}}
\right)\tilde{\mathcal D}(t) {\rm d}t,\quad k=0,1.
\end{align}
\end{Lemma}
\begin{proof}[\bf Proof.]
For $k=0,1$, taking the operator $\nabla\nabla_{\boldsymbol\tau}^k$, $\partial_t\nabla_{\boldsymbol\tau}^k$ and $\nabla_{\boldsymbol\tau}^k$ to $(\ref{NSAC1})_3$, $(\ref{NSAC1})_4$ and (\ref{bb1}) respectively leads to
\begin{equation}
\label{r10}
\begin{cases}
\nabla\nabla_{\boldsymbol\tau}^k(\bar\mu-\mu)
=\nabla\nabla_{\boldsymbol\tau}^k\phi_t
+[\nabla\nabla_{\boldsymbol\tau}^k,{\bf u}\cdot\nabla]\phi
+{\bf u}\cdot\nabla^2\nabla_{\boldsymbol\tau}^k\phi,&\rm in\quad\Omega\times(0,T),
\\
\nabla_{\boldsymbol\tau}^k\Delta\phi_t=-\nabla_{\boldsymbol\tau}^k\mu_t
+\partial_t\nabla_{\boldsymbol\tau}^k(\phi^3-\phi),&\rm in\quad\Omega\times(0,T),
\\
\nabla_{\boldsymbol\tau}^k(\bar\mu-\mu)=
-\nabla_{\boldsymbol\tau}^kL(\psi)=-\nabla_{\boldsymbol\tau}^k\partial_{\bf n}\phi
-\nabla_{\boldsymbol\tau}^k\gamma_{fs}^\prime(\psi)  ,&\rm on\quad\Gamma\times(0,T),
\end{cases}
\end{equation}
where $[A,B]=AB-BA$ is commutators. We consider
\begin{align}
\label{r11}
-\langle\nabla_{\boldsymbol\tau}^k(\bar\mu-\mu),\nabla_{\boldsymbol\tau}^k\Delta\phi_t\rangle
&=\langle\nabla\nabla_{\boldsymbol\tau}^k(\bar\mu-\mu),\nabla_{\boldsymbol\tau}^k\nabla\phi_t\rangle
-\int_\Gamma\nabla_{\boldsymbol\tau}^k(\bar\mu-\mu)\nabla_{\boldsymbol\tau}^k\partial_{\bf n}\phi_t {\rm d}S \nonumber
\\
&=W_1+W_2.
\end{align}
Putting $(\ref{r10})_1$ into $W_1$, we get
\begin{align}
W_1&=\int_\Omega\left(\nabla\nabla_{\boldsymbol\tau}^k\phi_t+[\nabla\nabla_{\boldsymbol\tau}^k,{\bf u}\cdot\nabla]\phi
+{\bf u}\cdot\nabla^2\nabla_{\boldsymbol\tau}^k\phi
\right)\cdot\nabla\nabla_{\boldsymbol\tau}^k\phi_t{\rm d}x  \nonumber
\\
&=\|\nabla\nabla_{\boldsymbol\tau}^k\phi_t\|_{L^2}^2+\int_\Omega\left(
[\nabla\nabla_{\boldsymbol\tau}^k,{\bf u}\cdot\nabla]\phi
+{\bf u}\cdot\nabla^2\nabla_{\boldsymbol\tau}^k\phi
\right)\cdot\nabla\nabla_{\boldsymbol\tau}^k\phi_t{\rm d}x.   \nonumber
\end{align}
Similarly, $W_2$ could be rewritten by using $(\ref{r10})_3$ as follows
\begin{align}
W_2&=\int_\Gamma\nabla_{\boldsymbol\tau}^kL(\psi)\nabla_{\boldsymbol\tau}^k\partial_{\bf n}\phi_t {\rm d}S        \nonumber
\\
&=\int_\Gamma\nabla_{\boldsymbol\tau}^kL(\psi)\left(
\partial_t\nabla_{\boldsymbol\tau}^kL(\psi)-\partial_t\nabla_{\boldsymbol\tau}^k\gamma_{fs}^\prime(\psi)
\right){\rm d}S   \nonumber
\\
&=\frac{1}{2}\frac{\rm d}{{\rm d}t}\|\nabla_{\boldsymbol\tau}^kL(\psi)\|_{L^2(\Gamma)}^2
-\int_\Gamma\nabla_{\boldsymbol\tau}^kL(\psi)\partial_t\nabla_{\boldsymbol\tau}^k\gamma_{fs}^\prime(\psi) {\rm d}S.   \nonumber
\end{align}
Thanks to $(\ref{r10})_2$, the left-hand side of (\ref{r11}) can be calculated as follows
\begin{align}
\label{r12}
-\langle\nabla_{\boldsymbol\tau}^k(\bar\mu-\mu),\nabla_{\boldsymbol\tau}^k\Delta\phi_t\rangle
&=\int_\Omega\nabla_{\boldsymbol\tau}^k(\bar\mu-\mu)\left[\nabla_{\boldsymbol\tau}^k\mu_t
-\partial_t\nabla_{\boldsymbol\tau}^k(\phi^3-\phi)
\right]{\rm d}x  \nonumber
\\
&=-\frac{1}{2}\frac{\rm d}{{\rm d}t}\|\nabla_{\boldsymbol\tau}^k(\mu-\bar\mu)\|_{L^2}^2
-\int_\Omega\nabla_{\boldsymbol\tau}^k(\bar\mu-\mu)\partial_t\nabla_{\boldsymbol\tau}^k
(\phi^3-\phi){\rm d}x.
\end{align}
Substituting $W_1$, $W_2$ and (\ref{r12}) into (\ref{r10}), we have
\begin{align}
\label{r13}
&\frac{1}{2}\frac{\rm d}{{\rm d}t}\left(\|\nabla_{\boldsymbol\tau}^kL(\psi)\|_{L^2(\Gamma)}^2
+\|\nabla_{\boldsymbol\tau}^k(\bar\mu-\mu)\|_{L^2}^2
\right)+\|\nabla\nabla_{\boldsymbol\tau}^k\phi_t\|_{L^2}^2 \nonumber
\\
&=-\int_\Omega\left(
[\nabla\nabla_{\boldsymbol\tau}^k,{\bf u}\cdot\nabla]\phi
+{\bf u}\cdot\nabla^2\nabla_{\boldsymbol\tau}^k\phi
\right)\cdot\nabla\nabla_{\boldsymbol\tau}^k\phi_t{\rm d}x  \nonumber
\\
&\quad-\int_\Omega\nabla_{\boldsymbol\tau}^k(\bar\mu-\mu)\partial_t\nabla_{\boldsymbol\tau}^k
(\phi^3-\phi){\rm d}x+ \int_\Gamma\nabla_{\boldsymbol\tau}^kL(\psi)\partial_t\nabla_{\boldsymbol\tau}^k\gamma_{fs}^\prime(\psi) {\rm d}S  \nonumber
\\
&=\sum_{i=1}^3K_i^k,
\end{align}
and we need to control the terms on the right-hand side. First of all, for the case $k=0$, we have following estimates
\begin{align}
K_1^0&=-\int_\Omega\left(\nabla{\bf u}\cdot\nabla\phi+{\bf u}\cdot\nabla^2\phi
\right)\cdot\nabla\phi_t{\rm d}x  \nonumber
\\
&\le\left(\|\nabla{\bf u}\|_{L^2}\|\nabla\phi\|_{L^\infty}+\|{\bf u}\|_{L^4}\|\nabla^2\phi\|_{L^4}
\right)\|\nabla\phi_t\|_{L^2} \nonumber
\\
&\le C\|{\bf u}\|_{H^1}\|\nabla\phi\|_{H^2}\|\nabla\phi_t\|_{L^2}
\le C\tilde{\mathcal D}(t)^{\frac{1}{2}}\tilde{\mathcal E}(t)^{\frac{1}{2}}\tilde{\mathcal D}(t)^{\frac{1}{2}}.   \nonumber
\\[1em]
K_2^0&=-\int_\Omega(\bar\mu-\mu)(3\phi^2-1)\phi_t{\rm d}x \nonumber
\\
&=-\int_\Omega(\bar\mu-\mu)(3\phi^2-1)\left[(\bar\mu-\mu)-{\bf u}\cdot\nabla\phi
\right]{\rm d}x  \nonumber
\\
&=-\int_\Omega(3\phi^2-1)|\mu-\bar\mu|^2{\rm d}x+\int_\Omega(\bar\mu-\mu)(3\phi^2-1){\bf u}\cdot\nabla\phi {\rm d}x    \nonumber
\\
&\le-\int_\Omega(3\phi^2-1)|\mu-\bar\mu|^2{\rm d}x+\|3\phi^2-1\|_{L^\infty}\|\mu-\bar\mu\|_{L^2}\|{\bf u}\|_{L^4}\|\nabla\phi\|_{L^4}  \nonumber
\\
&\le-\int_\Omega(3\phi^2-1)|\mu-\bar\mu|^2{\rm d}x+C\|\mu-\bar\mu\|_{L^2}\|{\bf u}\|_{H^1}\|\nabla\phi\|_{H^1}  \nonumber
\\
&\le-\int_\Omega(3\phi^2-1)|\mu-\bar\mu|^2{\rm d}x+C\tilde{\mathcal E}(t)^{\frac{1}{2}}\tilde{\mathcal D}(t)^{\frac{1}{2}}\tilde{\mathcal D}(t)^{\frac{1}{2}}.     \nonumber
\\[1em]
K_3^0&=\int_\Gamma L(\psi)\gamma_{fs}^{(2)}(\psi)\psi_t {\rm d}S\le C|\nu \cos\theta_s| \|L(\psi)\|_{L^2(\Gamma)}\|\psi_t\|_{L^2(\Gamma)} \nonumber
\\
&\le C\varepsilon_0\|L(\psi)\|_{L^2(\Gamma)}\|\psi_t\|_{L^2(\Gamma)} \le C\varepsilon_0\tilde{\mathcal D}(t),  \nonumber
\end{align}
where we have used the boundedness of $\phi$. For the case $k=1$, in the same way, we get
\begin{align}
K_1^1&=-\int_\Omega\left(
[\nabla\nabla_{\boldsymbol\tau},{\bf u}\cdot\nabla]\phi
+{\bf u}\cdot\nabla^2\nabla_{\boldsymbol\tau}\phi
\right)\cdot\nabla\nabla_{\boldsymbol\tau}\phi_t{\rm d}x  \nonumber
\\
&=-\int_\Omega\left(
\nabla\nabla_{\boldsymbol\tau}{\bf u}\cdot\nabla\phi+\nabla{\bf u}\cdot\nabla\nabla_{\boldsymbol\tau}\phi+\nabla_{\boldsymbol\tau}{\bf u}\cdot\nabla^2\phi
+{\bf u}\cdot\nabla^2\nabla_{\boldsymbol\tau}\phi
\right)\cdot\nabla\nabla_{\boldsymbol\tau}\phi_t{\rm d}x  \nonumber
\\
&\le C\left(\|\nabla\nabla_{\boldsymbol\tau}{\bf u}\|_{L^2}\|\nabla\phi\|_{L^\infty}+\|\nabla{\bf u}\|_{L^\infty}\|\nabla^2\phi\|_{L^2}+\|{\bf u}\|_{L^\infty}\|\nabla^2\nabla_{\boldsymbol\tau}\phi\|_{L^2}
\right)\|\nabla\nabla_{\boldsymbol\tau}\phi_t\|_{L^2} \nonumber
\\
&\le C\|{\bf u}\|_{H^3}\|\nabla\phi\|_{H^2}\|\phi_t\|_{H^2}\le C\tilde{\mathcal E}(t)^{\frac{1}{2}}\tilde{\mathcal D}(t)^{\frac{1}{2}}\tilde{\mathcal D}(t)^{\frac{1}{2}}.     \nonumber
\\[1em]
K_2^1&=-\int_\Omega\nabla_{\boldsymbol\tau}(\bar\mu-\mu)\left[
6\phi\phi_t\nabla_{\boldsymbol\tau}\phi+(3\phi^2-1)\nabla_{\boldsymbol\tau}\phi_t
\right]{\rm d}x  \nonumber
\\
&=-\int_\Omega\nabla_{\boldsymbol\tau}(\bar\mu-\mu)(3\phi^2-1)\left[
\nabla_{\boldsymbol\tau}(\bar\mu-\mu)-\nabla_{\boldsymbol\tau}{\bf u}\cdot\nabla\phi-{\bf u}\cdot\nabla\nabla_{\boldsymbol\tau}\phi
\right]{\rm d}x  \nonumber
\\
&\quad-\int_\Omega6\phi\phi_t\nabla_{\boldsymbol\tau}\phi
\nabla_{\boldsymbol\tau}(\bar\mu-\mu){\rm d}x  \nonumber
\\
&=-\int_\Omega(3\phi^2-1)|\nabla_{\boldsymbol\tau}(\mu-\bar\mu)|^2{\rm d}x
-\int_\Omega6\phi\phi_t\nabla_{\boldsymbol\tau}\phi
\nabla_{\boldsymbol\tau}(\bar\mu-\mu){\rm d}x  \nonumber
\\
&\quad+\int_\Omega(3\phi^2-1)\nabla_{\boldsymbol\tau}(\bar\mu-\mu)\left(
\nabla_{\boldsymbol\tau}{\bf u}\cdot\nabla\phi+{\bf u}\cdot\nabla\nabla_{\boldsymbol\tau}\phi
\right){\rm d}x  \nonumber
\\
&\le-\int_\Omega(3\phi^2-1)|\nabla_{\boldsymbol\tau}(\mu-\bar\mu)|^2{\rm d}x
+C\|\phi\|_{L^\infty}\|\phi_t\|_{L^4}\|\nabla_{\boldsymbol\tau}\phi\|_{L^4}\|\nabla_{\boldsymbol\tau}(\mu-\bar\mu)\|_{L^2}
\nonumber
\\
&\quad+C\|3\phi^2-1\|_{L^\infty}\|\nabla_{\boldsymbol\tau}(\mu-\bar\mu)\|_{L^2}\left(
\|\nabla_{\boldsymbol\tau}{\bf u}\|_{L^4}\|\nabla\phi\|_{L^4}+\|{\bf u}\|_{L^\infty}\|\nabla\nabla_{\boldsymbol\tau}\phi\|_{L^2}
\right)   \nonumber
\\
&\le-\int_\Omega(3\phi^2-1)|\nabla_{\boldsymbol\tau}(\mu-\bar\mu)|^2{\rm d}x
+C\|\nabla\phi\|_{H^1} \|\phi_t\|_{H^1}\|\nabla_{\boldsymbol\tau}(\mu-\bar\mu)\|_{L^2}
\nonumber
\\
&\quad+C \|\nabla\phi\|_{H^1}\|\nabla_{\boldsymbol\tau}(\mu-\bar\mu)\|_{L^2}
\|{\bf u}\|_{H^2}   \nonumber
\\
&\le-\int_\Omega(3\phi^2-1)|\nabla_{\boldsymbol\tau}(\mu-\bar\mu)|^2{\rm d}x+C\tilde{\mathcal E}(t)^{\frac{1}{2}}\tilde{\mathcal D}(t)^{\frac{1}{2}}\tilde{\mathcal D}(t)^{\frac{1}{2}}.     \nonumber
\\[1em]
K_3^1&= \int_\Gamma\nabla_{\boldsymbol\tau}L(\psi)\left(
\gamma_{fs}^{(3)}(\psi)\psi_t\nabla_{\boldsymbol\tau}\psi+\gamma_{fs}^{(2)}(\psi)\nabla_{\boldsymbol\tau}\psi_t
\right){\rm d}S    \nonumber
\\
&\le \|\nabla_{\boldsymbol\tau}L(\psi)\|_{L^2(\Gamma)}
\|\gamma_{fs}^{(3)}(\psi)\|_{L^\infty(\Gamma)}\|\psi_t\|_{L^4(\Gamma)}
\|\nabla_{\boldsymbol\tau}\psi\|_{L^4(\Gamma)}   \nonumber
\\
&\quad+C \|\nabla_{\boldsymbol\tau}L(\psi)\|_{L^2(\Gamma)}
\|\gamma_{fs}^{(2)}(\psi)\|_{L^\infty(\Gamma)}
\|\nabla_{\boldsymbol\tau}\psi_t\|_{L^2(\Gamma)}    \nonumber
\\
&\le C|\nu{\rm cos}\theta_s| \|\nabla_{\boldsymbol\tau}L(\psi)\|_{L^2(\Gamma)}\left(
\|\psi_t\|_{H^1(\Gamma)}\|\nabla_{\boldsymbol\tau}\psi\|_{H^1(\Gamma)}+\|\nabla_{\boldsymbol\tau}\psi_t\|_{L^2(\Gamma)}
\right)   \nonumber
\\
&\le C\varepsilon_0 \|\nabla_{\boldsymbol\tau}L(\psi)\|_{L^2(\Gamma)}   \|\phi_t\|_{H^2}\|\nabla_{\boldsymbol\tau}\phi\|_{H^2}
+C\varepsilon_0 \|\nabla_{\boldsymbol\tau}L(\psi)\|_{L^2(\Gamma)} \|\nabla_{\boldsymbol\tau}\phi_t\|_{H^1}
\nonumber
\\
&\le C\varepsilon_0\tilde{\mathcal D}(t)^{\frac{1}{2}}\tilde{\mathcal D}(t)^{\frac{1}{2}}\tilde{\mathcal E}(t)^{\frac{1}{2}}+C\varepsilon_0\tilde{\mathcal D}(t).     \nonumber
\end{align}
It is worth noting $3\phi^2-1\ge 1$ and
\begin{align}
 \|\nabla_{\boldsymbol\tau}L(\psi_0)\|_{L^2(\Gamma)}^2 \le C\left(\|\nabla_{\boldsymbol\tau}\partial_{\bf n}\phi_0\|_{L^2(\Gamma)}^2
+\|\nabla_{\boldsymbol\tau}\psi_0\|_{L^2(\Gamma)}^2
\right)\le C\|\nabla\phi_0\|_{H^2}^2\le C\varepsilon_0. \nonumber
\end{align}
Then, substituting $K_1^k-K_3^k(k=0,1)$ into (\ref{r13}), integrating the result over (0,T) to get
\begin{align}
\label{r14}
&\sup\limits_{0\le t \le T}\left( \|\nabla_{\boldsymbol\tau}^kL(\psi)\|_{L^2(\Gamma)}^2
+\|\nabla_{\boldsymbol\tau}^k(\bar\mu-\mu)\|_{L^2}^2
\right)+\int_{0}^{T}\|\nabla\nabla_{\boldsymbol\tau}^k\phi_t\|_{L^2}^2 {\rm d}t \nonumber
\\
&\le C\left(\varepsilon_0+\tilde{\mathcal E}(0)\right)+C\int_{0}^{T}\left(\varepsilon_0+\tilde{\mathcal E}(t)^{\frac{1}{2}}
\right)\tilde{\mathcal D}(t) {\rm d}t.
\end{align}
\vskip2mm
Finally, it remains for us the normal direction estimates for $\nabla^2\phi_t$.
It follows from $(\ref{NSAC1})_4$ that
\begin{align}
\bar\mu_t=\frac{1}{|\Omega|}\int_\Omega\mu_t{\rm d}x=\frac{1}{|\Omega|}\int_\Omega
\left(-\Delta\phi_t+f_t\right){\rm d}x=\frac{1}{|\Omega|}\int_\Gamma\partial_{\bf n}\phi_t {\rm d}S+\frac{1}{|\Omega|}\int_\Omega f_t{\rm d}x,
\nonumber
\end{align}
which yields
\begin{align}
\label{r15}
&\int_{0}^{T}\|\partial_{\bf n}\partial_{\bf n}\phi_t\|_{L^2}^2{\rm d}t \nonumber
\\
&\le C\int_{0}^{T}\left(
\|\Delta_{\boldsymbol\tau}\phi_t\|_{L^2}^2+\|(\mu-\bar\mu)_t\|_{L^2}^2+\|\bar\mu_t\|_{L^2}^2
+\|(3\phi^2-1)\phi_t\|_{L^2}^2
\right) {\rm d}t   \nonumber
\\
&\le\int_{0}^{T}\left(
\|\nabla\nabla_{\boldsymbol\tau}\phi_t\|_{L^2}^2+\|(\mu-\bar\mu)_t\|_{L^2}^2
+\|\partial_{\bf n}\phi_t\|_{L^2(\Gamma)}^2+\|3\phi^2-1\|_{L^\infty}^2\|\phi_t\|_{L^2}^2
\right) {\rm d}t   \nonumber
\\
&\le\int_{0}^{T}\left(
\|\nabla\nabla_{\boldsymbol\tau}\phi_t\|_{L^2}^2+\|(\mu-\bar\mu)_t\|_{L^2}^2
+\|\partial_{\bf n}\phi_t\|_{L^2(\Gamma)}^2+\|\nabla\phi_t\|_{L^2}^2
\right) {\rm d}t   \nonumber
\\
&\le C\tilde{\mathcal E}(0)+C\int_{0}^{T}\left(\varepsilon_0+\tilde{\mathcal E}(t)^{\frac{1}{2}}
\right)\tilde{\mathcal D}(t) {\rm d}t,
\end{align}
where we used $(\ref{111})$ and the following result obtained from Poincar$\rm \acute{e}$ inequality
\begin{align}
\label{0}
\|\phi_t\|_{L^2}^2=\|\phi_t-\bar\phi_t\|_{L^2}^2\le C\|\nabla\phi_t\|_{L^2}^2.
\end{align}
Combining (\ref{r14}) with (\ref{r15}) leads to (\ref{r16}). This completes the proof of Lemma \ref{L477}.
\end{proof}
\begin{Lemma}
\label{444}
Under the assumptions of Proposition $\ref{proposition global}$, it holds that
\begin{align}
\label{r17}
&\sup\limits_{0\le t \le T}\left(\|\nabla\phi\|_{H^2}^2+\|\mu-\bar\mu\|_{H^1}^2\right) +\int_{0}^{T}\left(\|\nabla_{\boldsymbol\tau}^k(\mu-\bar\mu)\|_{L^2}^2
+\|\nabla_{\boldsymbol\tau}^k\partial_{\bf n}\phi\|_{L^2(\Gamma)}^2
\right) {\rm d}t  \nonumber
\\
&\le C\left(\varepsilon_0+\tilde{\mathcal E}(0)\right)+C\int_{0}^{T}\left(\varepsilon_0+\tilde{\mathcal E}(t)^{\frac{1}{2}}
\right)\tilde{\mathcal D}(t) {\rm d}t,\quad k=1,2.
\end{align}
\end{Lemma}
\begin{proof}[\bf Proof.]
First, for $k=1,2$, taking the operator $\nabla\nabla_{\boldsymbol\tau}^k$, $\nabla_{\boldsymbol\tau}^k$ and $\nabla_{\boldsymbol\tau}^k$ to $(\ref{NSAC1})_1$, $(\ref{NSAC1})_4$ and (\ref{bb1}) respectively, one has
\begin{equation}
\label{r18}
\begin{cases}
\nabla\nabla_{\boldsymbol\tau}^k(\bar\mu-\mu)
=\nabla\nabla_{\boldsymbol\tau}^k\phi_t
+[\nabla\nabla_{\boldsymbol\tau}^k,{\bf u}\cdot\nabla]\phi
+{\bf u}\cdot\nabla^2\nabla_{\boldsymbol\tau}^k\phi,&\rm in\quad\Omega\times(0,T),
\\
\nabla_{\boldsymbol\tau}^k\Delta\phi=-\nabla_{\boldsymbol\tau}^k\mu
+\nabla_{\boldsymbol\tau}^k(\phi^3-\phi),&\rm in\quad\Omega\times(0,T),
\\
\nabla_{\boldsymbol\tau}^k(\bar\mu-\mu)=
-\nabla_{\boldsymbol\tau}^kL(\psi)=-\nabla_{\boldsymbol\tau}^k\partial_{\bf n}\phi
-\nabla_{\boldsymbol\tau}^k\gamma_{fs}^\prime(\psi) ,&\rm on\quad\Gamma\times(0,T),
\end{cases}
\end{equation}
where $[A,B]=AB-BA$ is commutators.  We now consider
\begin{align}
\label{r19}
-\langle\nabla_{\boldsymbol\tau}^k(\bar\mu-\mu),\nabla_{\boldsymbol\tau}^k\Delta\phi\rangle
=\langle\nabla\nabla_{\boldsymbol\tau}^k(\bar\mu-\mu),\nabla_{\boldsymbol\tau}^k\nabla\phi\rangle
-\int_\Gamma\nabla_{\boldsymbol\tau}^k(\bar\mu-\mu)\nabla_{\boldsymbol\tau}^k\partial_{\bf n}\phi {\rm d}S
=P_1+P_2.
\end{align}
Recall $P_2=P_2^\delta$, where $P_2^\delta$ is defined as in (\ref{P2P2}), and we find
\begin{align}
P_1&=P_1^\delta+\int_\Omega\delta\nabla\nabla_{\boldsymbol\tau}^k\Delta_{\boldsymbol\tau} \phi\cdot\nabla\nabla_{\boldsymbol\tau}^k\phi{\rm d}x \nonumber
\\
&=\frac{1}{2}\frac{\rm d}{{\rm d}t}\|\nabla\nabla_{\boldsymbol\tau}^k\phi\|_{L^2}^2+\int_\Omega
[\nabla\nabla_{\boldsymbol\tau}^k,{\bf u}\cdot\nabla]\phi
\cdot\nabla\nabla_{\boldsymbol\tau}^k\phi {\rm d}x,   \nonumber
\end{align}
whereas the term on the left-hand side of (\ref{r19}) can be rewritten as in (\ref{ll-r20}). Then, we have
\begin{align}
\label{r21}
&\frac{1}{2}\frac{\rm d}{{\rm d}t}\|\nabla\nabla_{\boldsymbol\tau}^k\phi\|_{L^2}^2
+\|\nabla_{\boldsymbol\tau}^k(\bar\mu-\mu)\|_{L^2}^2
+\|\nabla_{\boldsymbol\tau}^k\partial_{\bf n}\phi\|_{L^2(\Gamma)}^2 \nonumber
\\
&=-\int_\Omega
[\nabla\nabla_{\boldsymbol\tau}^k,{\bf u}\cdot\nabla]\phi
\cdot\nabla\nabla_{\boldsymbol\tau}^k\phi {\rm d}x-\int_\Omega\nabla_{\boldsymbol\tau}^k(\bar\mu-\mu)\nabla_{\boldsymbol\tau}^k
(\phi^3-\phi){\rm d}x- \int_\Gamma\nabla_{\boldsymbol\tau}^k\gamma_{fs}^\prime(\psi)\nabla_{\boldsymbol\tau}^k\partial_{\bf n}\phi {\rm d}S    \nonumber
\\
&=\sum_{i=1}^3M_i^k,
\end{align}
\vskip2mm
Next, for the case $k=1$, the right-hand side of (\ref{r21}) can be estimated as follows:
\begin{align}
M_1^1&=-\int_\Omega\left(\nabla\nabla_{\boldsymbol\tau}{\bf u}\cdot\nabla\phi+\nabla_{\boldsymbol\tau}{\bf u}\cdot\nabla^2\phi+\nabla{\bf u}\cdot\nabla\nabla_{\boldsymbol\tau}\phi
\right)\cdot\nabla\nabla_{\boldsymbol\tau}\phi {\rm d}x  \nonumber
\\
&\le C\left(\|\nabla^2{\bf u}\|_{L^2}\|\nabla\phi\|_{L^\infty}+\|\nabla{\bf u}\|_{L^4}\|\nabla^2\phi\|_{L^4}
\right)\|\nabla\nabla_{\boldsymbol\tau}\phi\|_{L^2} \nonumber
\\
&\le C\left(\|\nabla^2{\bf u}\|_{L^2}\|\nabla\phi\|_{H^2}+\|\nabla{\bf u}\|_{H^1}\|\nabla^2\phi\|_{H^1}
\right)\|\nabla\nabla_{\boldsymbol\tau}\phi\|_{L^2} \nonumber
\\
&\le C\tilde{\mathcal D}(t)^{\frac{1}{2}}\tilde{\mathcal E}(t)^{\frac{1}{2}}\tilde{\mathcal D}(t)^{\frac{1}{2}}.   \nonumber
\\[1em]
M_2^1&=-\int_\Omega\nabla_{\boldsymbol\tau}(\bar\mu-\mu)(3\phi^2-1)\nabla_{\boldsymbol\tau}\phi {\rm d}x \nonumber
\\
&=-\int_\Omega\left(\nabla_{\boldsymbol\tau}\phi_t+\nabla_{\boldsymbol\tau}{\bf u}\cdot\nabla\phi+{\bf u}\cdot\nabla\nabla_{\boldsymbol\tau}\phi
\right)(3\phi^2-1)\nabla_{\boldsymbol\tau}\phi {\rm d}x  \nonumber
\\
&=-\frac{1}{2}\frac{\rm d}{{\rm d}t}\int_\Omega(3\phi^2-1)|\nabla_{\boldsymbol\tau}\phi|^2{\rm d}x
+\int_\Omega3\phi\phi_t|\nabla_{\boldsymbol\tau}\phi|^2{\rm d}x \nonumber
\\
&\quad-\int_\Omega\left(
\nabla_{\boldsymbol\tau}{\bf u}\cdot\nabla\phi+{\bf u}\cdot\nabla\nabla_{\boldsymbol\tau}\phi
\right)(3\phi^2-1)\nabla_{\boldsymbol\tau}\phi {\rm d}x    \nonumber
\\
&\le-\frac{1}{2}\frac{\rm d}{{\rm d}t}\int_\Omega(3\phi^2-1)|\nabla_{\boldsymbol\tau}\phi|^2{\rm d}x
+C\|\phi\|_{L^\infty}\|\phi_t\|_{L^2}\|\nabla_{\boldsymbol\tau}\phi\|_{L^4}^2 \nonumber
\\
&\quad+C\|3\phi^2-1\|_{L^\infty}\|\nabla_{\boldsymbol\tau}\phi\|_{L^4}\left(
\|\nabla_{\boldsymbol\tau}{\bf u}\|_{L^2}\|\nabla\phi\|_{L^4}+\|{\bf u}\|_{L^4}\|\nabla\nabla_{\boldsymbol\tau}\phi\|_{L^2}
\right)   \nonumber
\\
&\le-\frac{1}{2}\frac{\rm d}{{\rm d}t}\int_\Omega(3\phi^2-1)|\nabla_{\boldsymbol\tau}\phi|^2{\rm d}x
+C\|\phi_t\|_{L^2}\|\nabla_{\boldsymbol\tau}\phi\|_{H^1}^2 \nonumber
\\
&\quad+C\|\nabla_{\boldsymbol\tau}\phi\|_{H^1}\left( \|\nabla_{\boldsymbol\tau}{\bf u}\|_{L^2} \|\nabla\phi\|_{H^1}+\|{\bf u}\|_{H^1}\|\nabla\nabla_{\boldsymbol\tau}\phi\|_{L^2}
\right)    \nonumber
\\
&\le-\frac{1}{2}\frac{\rm d}{{\rm d}t}\int_\Omega(3\phi^2-1)|\nabla_{\boldsymbol\tau}\phi|^2{\rm d}x
+C\tilde{\mathcal D}(t)^{\frac{1}{2}}\tilde{\mathcal D}(t)^{\frac{1}{2}}\tilde{\mathcal E}(t)^{\frac{1}{2}}.     \nonumber
\\[1em]
M_3^1&= -\int_\Gamma\gamma_{fs}^{(2)}(\psi)\nabla_{\boldsymbol\tau}\psi\nabla_{\boldsymbol\tau}\partial_{\bf n}\phi {\rm d}S\le C\|\gamma_{fs}^{(2)}(\psi)\|_{L^\infty}\|\nabla_{\boldsymbol\tau}\psi\|_{L^2(\Gamma)}\|\nabla_{\boldsymbol\tau}\partial_{\bf n}\phi\|_{L^2(\Gamma)}   \nonumber
\\
&\le C|\nu \cos\theta_s| \|\nabla_{\boldsymbol\tau}\phi\|_{H^1}\|\nabla_{\boldsymbol\tau}\partial_{\bf n}\phi\|_{L^2(\Gamma)}\le C\varepsilon_0\tilde{\mathcal D}(t).  \nonumber
\end{align}
For the case $k=2$, similarly, we obtain
\begin{align}
M_1^2&=-\int_\Omega\left[
\nabla\nabla_{\boldsymbol\tau}^2({\bf u}\cdot\nabla\phi)
-{\bf u}\cdot\nabla^2\nabla_{\boldsymbol\tau}^2\phi
\right]\cdot\nabla\nabla_{\boldsymbol\tau}^2\phi {\rm d}x  \nonumber
\\
&=-\int_\Omega\left(\nabla\nabla_{\boldsymbol\tau}^2{\bf u}\cdot\nabla\phi+2\nabla\nabla_{\boldsymbol\tau}{\bf u}\cdot\nabla\nabla_{\boldsymbol\tau}\phi+\nabla_{\boldsymbol\tau}^2{\bf u}\cdot\nabla^2\phi
\right)\cdot\nabla\nabla_{\boldsymbol\tau}^2\phi {\rm d}x \nonumber
\\
&\quad-\int_\Omega\left(2\nabla_{\boldsymbol\tau}{\bf u}\cdot\nabla^2\nabla_{\boldsymbol\tau}\phi+\nabla{\bf u}\cdot\nabla\nabla_{\boldsymbol\tau}^2\phi
\right)\cdot\nabla\nabla_{\boldsymbol\tau}^2\phi {\rm d}x \nonumber
\\
&\le C\left(\|\nabla^3{\bf u}\|_{L^2}\|\nabla\phi\|_{L^\infty}+\|\nabla^2{\bf u}\|_{L^4}\|\nabla^2\phi\|_{L^4}+\|\nabla{\bf u}\|_{L^\infty}\|\nabla^3\phi\|_{L^2}
\right)\|\nabla\nabla_{\boldsymbol\tau}^2\phi\|_{L^2} \nonumber
\\
&\le C\|\nabla{\bf u}\|_{H^2}\|\nabla\phi\|_{H^2}\|\nabla\nabla_{\boldsymbol\tau}^2\phi\|_{L^2} \nonumber
\\
&\le C\tilde{\mathcal D}(t)^{\frac{1}{2}}\tilde{\mathcal E}(t)^{\frac{1}{2}}\tilde{\mathcal D}(t)^{\frac{1}{2}}.     \nonumber
\\[1em]
M_2^2&=-\int_\Omega\nabla_{\boldsymbol\tau}^2(\bar\mu-\mu)\left[
(3\phi^2-1)\nabla_{\boldsymbol\tau}^2\phi+6\phi|\nabla_{\boldsymbol\tau}\phi|^2
\right]{\rm d}x  \nonumber
\\
&=-\int_\Omega\left(\nabla_{\boldsymbol\tau}^2\phi_t+\nabla_{\boldsymbol\tau}^2{\bf u}\cdot\nabla\phi+2\nabla_{\boldsymbol\tau}{\bf u}\cdot\nabla\nabla_{\boldsymbol\tau}\phi+{\bf u}\cdot\nabla\nabla_{\boldsymbol\tau}^2\phi
\right)(3\phi^2-1)\nabla_{\boldsymbol\tau}^2\phi {\rm d}x \nonumber
\\
&\quad-\int_\Omega6\phi|\nabla_{\boldsymbol\tau}\phi|^2\nabla_{\boldsymbol\tau}^2(\bar\mu-\mu){\rm d}x
\nonumber
\\
&=-\frac{1}{2}\frac{\rm d}{{\rm d}t}\int_\Omega(3\phi^2-1)|\nabla_{\boldsymbol\tau}^2\phi|^2{\rm d}x
+\int_\Omega3\phi\phi_t|\nabla_{\boldsymbol\tau}^2\phi|^2{\rm d}x
-\int_\Omega6\phi|\nabla_{\boldsymbol\tau}\phi|^2\nabla_{\boldsymbol\tau}^2(\bar\mu-\mu){\rm d}x
\nonumber
\\
&\quad-\int_\Omega\left(\nabla_{\boldsymbol\tau}^2{\bf u}\cdot\nabla\phi+2\nabla_{\boldsymbol\tau}{\bf u}\cdot\nabla\nabla_{\boldsymbol\tau}\phi+{\bf u}\cdot\nabla\nabla_{\boldsymbol\tau}^2\phi
\right)(3\phi^2-1)\nabla_{\boldsymbol\tau}^2\phi {\rm d}x \nonumber
\\
&\le-\frac{1}{2}\frac{\rm d}{{\rm d}t}\int_\Omega(3\phi^2-1)|\nabla_{\boldsymbol\tau}^2\phi|^2{\rm d}x
+C\|\phi\|_{L^\infty}\|\nabla_{\boldsymbol\tau}^2\phi\|_{L^4}^2\left(
\|\phi_t\|_{L^2}+\|\nabla_{\boldsymbol\tau}^2(\mu-\bar\mu)\|_{L^2}
\right)  \nonumber
\\
&\quad+C\|3\phi^2-1\|_{L^\infty}\|\nabla_{\boldsymbol\tau}^2\phi\|_{L^2}\left(
\|\nabla_{\boldsymbol\tau}^2{\bf u}\|_{L^2}\|\nabla\phi\|_{L^\infty}+\|\nabla_{\boldsymbol\tau}{\bf u}\|_{L^4}\|\nabla\nabla_{\boldsymbol\tau}\phi\|_{L^4}
\right) \nonumber
\\
&\quad+C\|3\phi^2-1\|_{L^\infty}\|\nabla_{\boldsymbol\tau}^2\phi\|_{L^2}\|{\bf u}\|_{L^\infty}\|\nabla\nabla_{\boldsymbol\tau}^2\phi\|_{L^2} \nonumber
\\
&\le-\frac{1}{2}\frac{\rm d}{{\rm d}t}\int_\Omega(3\phi^2-1)|\nabla_{\boldsymbol\tau}^2\phi|^2{\rm d}x
+C\|\nabla_{\boldsymbol\tau}^2\phi\|_{H^1}^2\left(
\|\phi_t\|_{L^2}+\|\nabla_{\boldsymbol\tau}^2(\mu-\bar\mu)\|_{L^2}
\right)  \nonumber
\\
&\quad+C\|\nabla\phi\|_{H^2}\|\nabla_{\boldsymbol\tau}^2\phi\|_{L^2}\|{\bf u}\|_{H^2}
\nonumber
\\
&\le-\frac{1}{2}\frac{\rm d}{{\rm d}t}\int_\Omega(3\phi^2-1)|\nabla_{\boldsymbol\tau}^2\phi|^2{\rm d}x
+C\tilde{\mathcal E}(t)^{\frac{1}{2}}\tilde{\mathcal D}(t)^{\frac{1}{2}}\tilde{\mathcal D}(t)^{\frac{1}{2}}.     \nonumber
\\[1em]
M_3^2&=-\int_\Gamma\left(
\gamma_{fs}^{(3)}(\psi)|\nabla_{\boldsymbol\tau}\psi|^2+\gamma_{fs}^{(2)}(\psi)
\nabla_{\boldsymbol\tau}^2\psi
\right)\nabla_{\boldsymbol\tau}^2\partial_{\bf n}\phi {\rm d}S   \nonumber
\\
&\le C|\nu \cos\theta_s| \left(\|\nabla_{\boldsymbol\tau}\psi\|_{L^4(\Gamma)}^2+\|\nabla_{\boldsymbol\tau}^2\psi\|_{L^2(\Gamma)}
\right)\|\nabla_{\boldsymbol\tau}^2\partial_{\bf n}\phi\|_{L^2(\Gamma)}  \nonumber
\\
&\le C\varepsilon_0\left(\|\nabla_{\boldsymbol\tau}\phi\|_{H^2}^2+\|\nabla_{\boldsymbol\tau}^2\phi\|_{H^1}
\right)\|\nabla_{\boldsymbol\tau}^2\partial_{\bf n}\phi\|_{L^2(\Gamma)}
\nonumber
\\
&\le C\varepsilon_0\tilde{\mathcal D}(t)+C\tilde{\mathcal E}(t)^{\frac{1}{2}}\tilde{\mathcal D}(t)^\frac{1}{2}\tilde{\mathcal D}(t)^\frac{1}{2}.     \nonumber
\end{align}
Then, substituting $M_1^k$ - $M_3^k$ $(k=1,2)$ into (\ref{r21}), we obtain
\begin{align}
&\frac{1}{2}\frac{\rm d}{{\rm d}t}\left(\|\nabla\nabla_{\boldsymbol\tau}^k\phi\|_{L^2}^2
+\int_\Omega(3\phi^2-1)|\nabla_{\boldsymbol\tau}^2\phi|^k{\rm d}x
\right)+\|\nabla_{\boldsymbol\tau}^k(\bar\mu-\mu)\|_{L^2}^2
+\|\nabla_{\boldsymbol\tau}^k\partial_{\bf n}\phi\|_{L^2(\Gamma)}^2 \nonumber
\\
&\le C\left(\varepsilon_0+\tilde{\mathcal E}(t)^{\frac{1}{2}}
\right)\tilde{\mathcal D}(t). \nonumber
\end{align}
Integrating it over $(0,T)$, and noting $3\phi^2-1\ge 1$, one has
\begin{align}
\label{r22}
&\sup\limits_{0\le t \le T}\|\nabla\nabla_{\boldsymbol\tau}^k\phi\|_{L^2}^2 +\int_{0}^{T}\left(\|\nabla_{\boldsymbol\tau}^k(\mu-\bar\mu)\|_{L^2}^2
+\|\nabla_{\boldsymbol\tau}^k\partial_{\bf n}\phi\|_{L^2(\Gamma)}^2
\right) {\rm d}t  \nonumber
\\
&\le C\tilde{\mathcal E}(0)+C\int_{0}^{T}\left(\varepsilon_0+\tilde{\mathcal E}(t)^{\frac{1}{2}}
\right)\tilde{\mathcal D}(t) {\rm d}t,\quad k=1,2.
\end{align}
\vskip2mm
Finally, the differentiation in the normal direction can be estimated by using the differential equation $(\ref{NSAC1})_4$ as following
\begin{align}
\label{r23}
\|\partial_{\bf n}\partial_{\bf n}\phi\|_{L^2}^2&\le C\left(\|\Delta_{\boldsymbol\tau}\phi\|_{L^2}^2+\|\mu-\bar\mu\|_{L^2}^2
+\|\bar\mu\|_{L^2}^2+\|\phi^3-\phi\|_{L^2}^2
\right)   \nonumber
\\
&\le C\left(\|\nabla_{\boldsymbol\tau}^2\phi\|_{L^2}^2+\|\mu-\bar\mu\|_{L^2}^2+\|\partial_{\bf n}\phi\|_{L^2(\Gamma)}^2+\|\phi\|_{L^\infty}\|\phi^2-1\|_{L^2}^2
\right)  \nonumber
\\
&\le C\left(\|\nabla\nabla_{\boldsymbol\tau}\phi\|_{L^2}^2+\|\mu-\bar\mu\|_{L^2}^2
+\|L(\psi)\|_{L^2(\Gamma)}^2+|\nu \cos\theta_s|+\|\phi^2-1\|_{L^2}^2
\right)  \nonumber
\\
&\le C\left(\|\nabla\nabla_{\boldsymbol\tau}\phi\|_{L^2}^2+\|\mu-\bar\mu\|_{L^2}^2
+\|L(\psi)\|_{L^2(\Gamma)}^2+\varepsilon_0
\right),
\end{align}
in which we have used (\ref{r1}) and the following fact:
\begin{align}
\bar\mu=\frac{1}{|\Omega|}\int_\Omega\mu {\rm d}x=\frac{1}{|\Omega|}\int_\Omega
\left(-\Delta\phi+f\right){\rm d}x=\frac{1}{|\Omega|}\int_\Gamma\partial_{\bf n}\phi {\rm d}S+\frac{1}{|\Omega|}\int_\Omega f{\rm d}x.
\nonumber
\end{align}
Hence, (\ref{r23}) together with (\ref{r16}) and (\ref{r22}) yields
\begin{align}
\label{r24}
\|\nabla^2\phi\|_{L^2}^2\le C\left(\varepsilon_0+\tilde{\mathcal E}(0)\right)+C\int_{0}^{T}\left(\varepsilon_0+\tilde{\mathcal E}(t)^{\frac{1}{2}}
\right)\tilde{\mathcal D}(t) {\rm d}t.
\end{align}
Due to $(\ref{0})_1$, (\ref{r25}) and (\ref{r24}), we deduce that
\begin{align}
\label{r27}
&\|\mu-\bar\mu\|_{L^2}^2+\|\nabla(\mu-\bar\mu)\|_{L^2}^2 \nonumber
\\
&\le C\left(\|\phi_t\|_{H^1}^2+\|{\bf u}\cdot\nabla\phi\|_{H^1}^2
\right)  \nonumber
\\
&\le C\left(\|\phi_t\|_{H^1}^2+\|{\bf u}\|_{L^\infty}^2\|\nabla\phi\|_{L^2}^2+\|\nabla{\bf u}\|_{L^4}^2\|\nabla\phi\|_{L^4}^2
+\|{\bf u}\|_{L^\infty}^2\|\nabla^2\phi\|_{L^2}^2
\right)  \nonumber
\\
&\le C\left(\|\phi_t\|_{H^1}^2+\|{\bf u}\|_{H^2}^2\|\nabla\phi\|_{H^1}^2
\right)  \nonumber
\\
&\le C\left(\|\phi_t\|_{H^1}^2+\tilde{\mathcal E}(t)\|\nabla\phi\|_{H^1}^2
\right)  \nonumber
\\
&\le C\left(\|\phi_t\|_{H^1}^2+\|\nabla\phi\|_{H^1}^2
\right),
\end{align}
which together with (\ref{r1}), (\ref{r25}), (\ref{r22}) and (\ref{r24}) implies
\begin{align}
\label{r26}
\|\nabla\partial_{\bf n}\partial_{\bf n}\phi\|_{L^2}^2&\le C\left(
\|\nabla\Delta_{\boldsymbol\tau}\phi\|_{L^2}^2+\|\nabla(\mu-\bar\mu)\|_{L^2}^2
+\|(3\phi^2-1)\nabla\phi\|_{L^2}^2
\right)  \nonumber
\\
&\le C\left(\|\nabla\nabla_{\boldsymbol\tau}^2\phi\|_{L^2}^2+\|\phi_t\|_{H^1}^2
+\|\nabla\phi\|_{H^1}^2+\|3\phi^2-1\|_{L^\infty}^2\|\nabla\phi\|_{L^2}^2
\right)  \nonumber
\\
&\le C\left(\varepsilon_0+\tilde{\mathcal E}(0)\right)+C\int_{0}^{T}\left(\varepsilon_0+\tilde{\mathcal E}(t)^{\frac{1}{2}}
\right)\tilde{\mathcal D}(t) {\rm d}t.
\end{align}
Combining (\ref{r26}) with (\ref{r27}), (\ref{r22}) and (\ref{r24}) shows (\ref{r17}). This completes the proof of Lemma \ref{444}.
\end{proof}
\begin{Lemma}
\label{rgl5}
Under the assumptions of Proposition $\ref{proposition global}$, it holds that
\begin{align}
\label{r31}
&\sup\limits_{0\le t \le T}\|\phi^2-1\|_{H^2}^2
+\int_{0}^{T}\left(\|\nabla_{\boldsymbol\tau}\phi\|_{H^2}^2
+\|\nabla_{\boldsymbol\tau}^2\Delta\phi\|_{L^2}^2
\right) {\rm d}t  \nonumber
\\
&\le C\left(\varepsilon_0+\tilde{\mathcal E}(0)\right)+C\int_{0}^{T}\left(\varepsilon_0+\tilde{\mathcal E}(t)^{\frac{1}{2}}
\right)\tilde{\mathcal D}(t) {\rm d}t.
\end{align}
\end{Lemma}
\begin{proof}[\bf Proof.]
First, taking the tangential derivative of $(\ref{NSAC1})_4$ with respect to $x$ leads to
\begin{align}
\label{222}
\nabla_{\boldsymbol\tau}(\mu-\bar\mu)=-\nabla_{\boldsymbol\tau}\Delta\phi
+(3\phi^2-1)\nabla_{\boldsymbol\tau}\phi.
\end{align}
Multiplying (\ref{222}) by $\nabla_{\boldsymbol\tau}\phi$, integrating the result over $\Omega$ by parts, one has
\begin{align}
\int_\Omega\nabla_{\boldsymbol\tau}\mu\nabla_{\boldsymbol\tau}\phi {\rm d}x&=-\int_\Omega\nabla_{\boldsymbol\tau}\Delta\phi\nabla_{\boldsymbol\tau}\phi {\rm d}x+\int_\Omega(3\phi^2-1)|\nabla_{\boldsymbol\tau}\phi|^2 {\rm d}x \nonumber
\\
&=\|\nabla_{\boldsymbol\tau}\nabla\phi\|_{L^2}^2- \int_\Gamma\nabla_{\boldsymbol\tau}\partial_{\bf n}\phi\nabla_{\boldsymbol\tau}\psi {\rm d}S +\int_\Omega(3\phi^2-1)|\nabla_{\boldsymbol\tau}\phi|^2 {\rm d}x. \nonumber
\end{align}
Then, noting $3\phi^2-1\ge 1$, we infer
\begin{align}
\|\nabla_{\boldsymbol\tau}\phi\|_{L^2}^2+\|\nabla_{\boldsymbol\tau}\nabla\phi\|_{L^2}^2 &\le \int_\Gamma\nabla_{\boldsymbol\tau}\partial_{\bf n}\phi\nabla_{\boldsymbol\tau}\psi {\rm d}S  +\int_\Omega\nabla_{\boldsymbol\tau}(\mu-\bar\mu)\nabla_{\boldsymbol\tau}\phi {\rm d}x \nonumber
\\
&\le\frac{1}{4}\|\nabla_{\boldsymbol\tau}\phi\|_{L^2}^2+C\|\nabla_{\boldsymbol\tau}(\mu-\bar\mu)\|_{L^2}^2
+\|\nabla_{\boldsymbol\tau}\partial_{\bf n}\phi\|_{L^2(\Gamma)} \|\nabla_{\boldsymbol\tau}\psi\|_{L^2(\Gamma)}     \nonumber
\\
&\le\frac{1}{4}\|\nabla_{\boldsymbol\tau}\phi\|_{L^2}^2+C\|\nabla_{\boldsymbol\tau}(\mu-\bar\mu)\|_{L^2}^2
+C\|\nabla_{\boldsymbol\tau}\partial_{\bf n}\phi\|_{L^2(\Gamma)}\|\nabla_{\boldsymbol\tau}\phi\|_{H^1} \nonumber
\\
&\le\frac{1}{2}\|\nabla_{\boldsymbol\tau}\phi\|_{H^1}^2+C\|\nabla_{\boldsymbol\tau}(\mu-\bar\mu)\|_{L^2}^2
+C\|\nabla_{\boldsymbol\tau}\partial_{\bf n}\phi\|_{L^2(\Gamma)}^2, \nonumber
\end{align}
which together with (\ref{r17}) implies
\begin{align}
\label{r44}
\int_{0}^{T}\|\nabla_{\boldsymbol\tau}\phi\|_{H^1}^2{\rm d}t\le C\left(\varepsilon_0+\tilde{\mathcal E}(0)\right)+C\int_{0}^{T}\left(\varepsilon_0+\tilde{\mathcal E}(t)^{\frac{1}{2}}
\right)\tilde{\mathcal D}(t) {\rm d}t.
\end{align}
\vskip2mm
Then, taking the tangential derivative of $(\ref{222})$ with respect to $x$, we have
\begin{align}
\label{333}
\nabla_{\boldsymbol\tau}^2(\mu-\bar\mu)=-\nabla_{\boldsymbol\tau}^2\Delta\phi
+6\phi|\nabla_{\boldsymbol\tau}\phi|^2+(3\phi^2-1)\nabla_{\boldsymbol\tau}^2\phi,
\end{align}
which together with (\ref{r17}) and (\ref{r44}) yields
\begin{align}
\label{r28}
&\int_{0}^{T}\|\nabla_{\boldsymbol\tau}^2\Delta\phi\|_{L^2}^2{\rm d}t \nonumber
\\
&\le C\int_{0}^{T}\left(
\|\nabla_{\boldsymbol\tau}^2(\mu-\bar\mu)\|_{L^2}^2+\|3\phi^2-1\|_{L^\infty}^2
\|\nabla_{\boldsymbol\tau}^2\phi\|_{L^2}^2+\|\phi\|_{L^\infty}^2
\|\nabla_{\boldsymbol\tau}\phi\|_{L^\infty}^2\|\nabla_{\boldsymbol\tau}\phi\|_{L^2}^2
\right) {\rm d}t  \nonumber
\\
&\le C\int_{0}^{T}\left(\|\nabla_{\boldsymbol\tau}^2(\mu-\bar\mu)\|_{L^2}^2
+\|\nabla_{\boldsymbol\tau}^2\phi\|_{L^2}^2+\tilde{\mathcal E}(t)\|\nabla_{\boldsymbol\tau}\phi\|_{L^2}^2
\right) {\rm d}t  \nonumber
\\
&\le C\int_{0}^{T}\left(\|\nabla_{\boldsymbol\tau}^2(\mu-\bar\mu)\|_{L^2}^2
+\|\nabla_{\boldsymbol\tau}^2\phi\|_{L^2}^2+\|\nabla_{\boldsymbol\tau}\phi\|_{L^2}^2
\right) {\rm d}t  \nonumber
\\
&\le C\left(\varepsilon_0+\tilde{\mathcal E}(0)\right)+C\int_{0}^{T}\left(\varepsilon_0+\tilde{\mathcal E}(t)^{\frac{1}{2}}
\right)\tilde{\mathcal D}(t) {\rm d}t.
\end{align}
\vskip2mm
Next, multiplying (\ref{333}) by $\nabla_{\boldsymbol\tau}^2\phi$, integrating the result over $\Omega$, it follows from integration by parts that
\begin{align}
\int_\Omega\nabla_{\boldsymbol\tau}^2\mu\nabla_{\boldsymbol\tau}^2\phi {\rm d}x&=-\int_\Omega\nabla_{\boldsymbol\tau}^2\Delta\phi\nabla_{\boldsymbol\tau}^2\phi {\rm d}x+\int_\Omega6\phi|\nabla_{\boldsymbol\tau}\phi|^2 \nabla_{\boldsymbol\tau}^{2}\phi {\rm d}x+\int_\Omega(3\phi^2-1)|\nabla_{\boldsymbol\tau}^2\phi|^2 {\rm d}x \nonumber
\\
&=\|\nabla_{\boldsymbol\tau}^2\nabla\phi\|_{L^2}^2- \int_\Gamma\nabla_{\boldsymbol\tau}^2\partial_{\bf n}\phi\nabla_{\boldsymbol\tau}^2\psi {\rm d}S  +\int_\Omega6\phi|\nabla_{\boldsymbol\tau}\phi|^2\nabla_{\boldsymbol\tau}^2\phi {\rm d}x+\int_\Omega(3\phi^2-1)|\nabla_{\boldsymbol\tau}^2\phi|^2 {\rm d}x, \nonumber
\end{align}
which together with $3\phi^2-1\ge1$ implies
\begin{align}
\label{r45}
&\|\nabla_{\boldsymbol\tau}^2\phi\|_{L^2}^2+\|\nabla_{\boldsymbol\tau}^2\nabla\phi\|_{L^2}^2
\nonumber
\\
&\le \int_\Gamma\nabla_{\boldsymbol\tau}^2\partial_{\bf n}\phi\nabla_{\boldsymbol\tau}^2\psi {\rm d}S    +\int_\Omega\nabla_{\boldsymbol\tau}^2\phi\left(
\nabla_{\boldsymbol\tau}^2(\mu-\bar\mu)-6\phi|\nabla_{\boldsymbol\tau}\phi|^2
\right) {\rm d}x \nonumber
\\
&\le\frac{1}{4}\|\nabla_{\boldsymbol\tau}^2\phi\|_{L^2}^2+C\|\nabla_{\boldsymbol\tau}^2(\mu-\bar\mu)\|_{L^2}^2
+C\|\phi\|_{L^\infty}^2\|\nabla_{\boldsymbol\tau}\phi\|_{L^4}^4+\|\nabla_{\boldsymbol\tau}^2\partial_{\bf n}\phi\|_{L^2(\Gamma)} \|\nabla_{\boldsymbol\tau}^2\psi\|_{L^2(\Gamma)}    \nonumber
\\
&\le\frac{1}{4}\|\nabla_{\boldsymbol\tau}^2\phi\|_{L^2}^2+C\|\nabla_{\boldsymbol\tau}^2(\mu-\bar\mu)\|_{L^2}^2
+C\|\nabla_{\boldsymbol\tau}\phi\|_{H^1}^4+C\|\nabla_{\boldsymbol\tau}^2\partial_{\bf n}\phi\|_{L^2(\Gamma)}\|\nabla_{\boldsymbol\tau}^2\phi\|_{H^1} \nonumber
\\
&\le\frac{1}{4}\|\nabla_{\boldsymbol\tau}^2\phi\|_{L^2}^2+C\|\nabla_{\boldsymbol\tau}^2(\mu-\bar\mu)\|_{L^2}^2
+C\tilde{\mathcal E}(t)\|\nabla_{\boldsymbol\tau}\phi\|_{H^1}^2+C\|\nabla_{\boldsymbol\tau}^2\partial_{\bf n}\phi\|_{L^2(\Gamma)}\|\nabla_{\boldsymbol\tau}^2\phi\|_{H^1} \nonumber
\\
&\le\frac{1}{2}\|\nabla_{\boldsymbol\tau}^2\phi\|_{H^1}^2+C\|\nabla_{\boldsymbol\tau}^2(\mu-\bar\mu)\|_{L^2}^2
+C\|\nabla_{\boldsymbol\tau}\phi\|_{H^1}^2+C\|\nabla_{\boldsymbol\tau}^2\partial_{\bf n}\phi\|_{L^2(\Gamma)}^2.
\end{align}
Then, integrating (\ref{r45}) over $(0,T)$, combining the result with (\ref{r17}) and (\ref{r44}), we obtain
\begin{align}
\label{r46}
\int_{0}^{T}\|\nabla_{\boldsymbol\tau}^2\phi\|_{H^1}^2{\rm d}t\le C\left(\varepsilon_0+\tilde{\mathcal E}(0)\right)+C\int_{0}^{T}\left(\varepsilon_0+\tilde{\mathcal E}(t)^{\frac{1}{2}}
\right)\tilde{\mathcal D}(t) {\rm d}t.
\end{align}
As for the normal estimate, we have the following equation
\begin{align}
\label{r47}
\|\nabla^2\nabla_{\boldsymbol\tau}\phi\|_{L^2}^2&\le C\left(\|\nabla\nabla_{\boldsymbol\tau}^2\phi\|_{L^2}^2+\|\nabla\nabla_{\boldsymbol\tau}\partial_{\bf n}\phi\|_{L^2}^2\right)  \nonumber
\\
&\le C\left(\|\nabla\nabla_{\boldsymbol\tau}^2\phi\|_{L^2}^2+\|\nabla_{\boldsymbol\tau}^2\partial_{\bf n}\phi\|_{L^2}^2+\|\nabla_{\boldsymbol\tau}\partial_{\bf n}\partial_{\bf n}\phi\|_{L^2}^2\right)  \nonumber
\\
&\le C\left(\|\nabla\nabla_{\boldsymbol\tau}^2\phi\|_{L^2}^2+\|\nabla_{\boldsymbol\tau}\partial_{\bf n}\partial_{\bf n}\phi\|_{L^2}^2\right).
\end{align}
It follows from (\ref{222}) that
\begin{align}
\label{r48}
\int_{0}^{T}\|\nabla_{\boldsymbol\tau}\partial_{\bf n}\partial_{\bf n}\phi\|_{L^2}^2 {\rm d}t&\le C\int_{0}^{T}\left(\|\nabla_{\boldsymbol\tau}\Delta_{\boldsymbol\tau}\phi\|_{L^2}^2
+\|\nabla_{\boldsymbol\tau}(\mu-\bar\mu)\|_{L^2}^2
+\|(3\phi^2-1)\nabla_{\boldsymbol\tau}\phi\|_{L^2}^2\right) {\rm d}t \nonumber
\\
&\le C\int_{0}^{T}\left(\|\nabla_{\boldsymbol\tau}^3\phi\|_{L^2}^2
+\|\nabla_{\boldsymbol\tau}(\mu-\bar\mu)\|_{L^2}^2
+\|3\phi^2-1\|_{L^\infty}^2\|\nabla_{\boldsymbol\tau}\phi\|_{L^2}^2\right) {\rm d}t \nonumber
\\
&\le C\int_{0}^{T}\left(\|\nabla_{\boldsymbol\tau}^3\phi\|_{L^2}^2
+\|\nabla_{\boldsymbol\tau}(\mu-\bar\mu)\|_{L^2}^2
+\|\nabla_{\boldsymbol\tau}\phi\|_{L^2}^2\right) {\rm d}t  \nonumber
\\
&\le C\left(\varepsilon_0+\tilde{\mathcal E}(0)\right)+C\int_{0}^{T}\left(\varepsilon_0+\tilde{\mathcal E}(t)^{\frac{1}{2}}
\right)\tilde{\mathcal D}(t) {\rm d}t,
\end{align}
where we have used (\ref{r17}), (\ref{r44}) and (\ref{r46}).
\vskip2mm
Finally, it is worthy noting that $\tilde{\mathcal E}(t)\le 1$, due to (\ref{r1}) and (\ref{r17}), so we have
\begin{align}
\|\phi^2-1\|_{H^2}^2&\le C\|\phi^2-1\|_{L^2}^2+C\|\phi\nabla\phi\|_{L^2}^2+\|\phi\nabla^2\phi+|\nabla\phi|^2\|_{L^2}^2
\nonumber
\\
&\le C\|\phi^2-1\|_{L^2}^2+C\|\phi\|_{L^\infty}^2\|\nabla\phi\|_{L^2}^2
+C\|\phi\|_{L^\infty}\|\nabla^2\phi\|_{L^2}^2+C\|\nabla\phi\|_{L^4}^4 \nonumber
\\
&\le C\|\phi^2-1\|_{L^2}^2+C\|\nabla\phi\|_{H^1}^2+C\tilde{\mathcal E}(t)\|\nabla\phi\|_{H^1}^2 \nonumber
\\
&\le C\|\phi^2-1\|_{L^2}^2+C\|\nabla\phi\|_{H^1}^2  \nonumber
\\
&\le C\left(\varepsilon_0+\tilde{\mathcal E}(0)\right)+C\int_{0}^{T}\left(\varepsilon_0+\tilde{\mathcal E}(t)^{\frac{1}{2}}
\right)\tilde{\mathcal D}(t) {\rm d}t,  \nonumber
\end{align}
which together with (\ref{r44}), (\ref{r28}), (\ref{r46})-(\ref{r48}) shows (\ref{r31}). Hence, we complete the proof of Lemma \ref{rgl5}.
\end{proof}
\begin{Lemma}
\label{rgl6}
Under the assumptions of Proposition $\ref{proposition global}$, it holds that
\begin{align}
\label{r37}
&\sup\limits_{0\le t \le T}\left(\|{\bf u}\|_{H^2}^2+\|{\bf u}_t\|_{L^2}^2
\right)+\int_{0}^{T}\left(\|{\bf u}\|_{H^3}^2
+\|{\bf u}_t\|_{H^1}^2+\|{\bf u}_{t{\boldsymbol\tau}}\|_{L^2(\Gamma)}^2
\right) {\rm d}t  \nonumber
\\
&\le C\left(\varepsilon_0+\tilde{\mathcal E}(0)\right)+C\int_{0}^{T}\left(\varepsilon_0+\tilde{\mathcal E}(t)^{\frac{1}{2}}
\right)\tilde{\mathcal D}(t) {\rm d}t.
\end{align}
\end{Lemma}
\begin{proof}[\bf Proof.]
We complete the proof by several steps.
\\
{\it\bfseries Step 1. Estimates of ${\bf u}_t$.}
Differentiating $(\ref{NSAC1})_2$ with respect to $t$ leads to
\begin{align}
{\bf u}_{tt}+{\bf u}_t\cdot\nabla{\bf u}+{\bf u}\cdot\nabla{\bf u}_t+\nabla p_t={\rm div}{\mathbb S}({\bf u}_t)-\Delta\phi_t\cdot\nabla\phi-\Delta\phi\cdot\nabla\phi_t.
\nonumber
\end{align}
Multiplying it by ${\bf u}_t$, and integrating the result over $\Omega$, due to incompressibility of velocity ${\bf u}$ and boundary condition ${\bf u}\cdot{\bf n}=0$ on $\Gamma$, we derive
\begin{align}
&\frac{1}{2}\frac{\rm d}{{\rm d}t}\|{\bf u}_t\|_{L^2}^2+\frac{1}{2}\|{\mathbb S}({\bf u}_t)\|_{L^2}^2+\beta\|{\bf u}_{t{\boldsymbol\tau}}\|_{L^2(\Gamma)}^2        \nonumber
\\
&=-\int_\Omega{\bf u}_t\cdot\nabla{\bf u}\cdot{\bf u}_t {\rm d}x-\int_\Omega\Delta\phi_t\nabla\phi\cdot{\bf u}_t{\rm d}x-\int_\Omega\Delta\phi\nabla\phi_t\cdot{\bf u}_t{\rm d}x \nonumber
\\
&\quad+\int_\Gamma
\left(\partial_{\bf n}\phi_t+\gamma_{fs}^{(2)}(\psi)\psi_t\right
)\nabla_{\boldsymbol\tau}\psi\cdot{\bf u}_{t{\boldsymbol\tau}}{\rm d}S+\int_\Gamma
\left(\partial_{\bf n}\phi+\gamma_{fs}^\prime(\psi)\right)\nabla_{\boldsymbol\tau}\psi_t \cdot{\bf u}_{t{\boldsymbol\tau}}{\rm d}S                \nonumber
\\
&\le\|{\bf u}_t\|_{L^2}\left(\|\nabla{\bf u}\|_{L^4}\|{\bf u}_t\|_{L^4}+\|\Delta\phi_t\|_{L^2}\|\nabla\phi\|_{L^\infty}+\|\Delta\phi\|_{L^4}\|\nabla\phi_t\|_{L^4}
\right)   \nonumber
\\
&\quad+\left( \|\partial_{\bf n}\phi_t\nabla_{\boldsymbol\tau}\psi\|_{H^{-\frac{1}{2}}(\Gamma)}+\|\partial_{\bf n}\phi\nabla_{\boldsymbol\tau}\psi_t\|_{H^{-\frac{1}{2}}(\Gamma)}
\right)\|{\bf u}_{t{\boldsymbol\tau}}\|_{H^{\frac{1}{2}}(\Gamma)}  \nonumber
\\
&\quad+C\|\psi_t\|_{L^4(\Gamma)}
\|\nabla_{\boldsymbol\tau}\psi\|_{L^4(\Gamma)} \|{\bf u}_{t{\boldsymbol\tau}}\|_{L^2(\Gamma)}+ \|\gamma_{fs}^\prime(\psi)\|_{L^\infty(\Gamma)}
\|\nabla_{\boldsymbol\tau}\psi_t\|_{L^2(\Gamma)}    \|{\bf u}_{t{\boldsymbol\tau}}\|_{L^2(\Gamma)}  \nonumber
\\
&\le C\|{\bf u}_t\|_{L^2}\left(\|\nabla{\bf u}\|_{H^1}\|{\bf u}_t\|_{H^1}+\|\Delta\phi_t\|_{L^2}\|\nabla\phi\|_{H^2}+\|\Delta\phi\|_{H^1}\|\nabla\phi_t\|_{H^1}
\right)   \nonumber
\\
&\quad+C\left(\|\partial_{\bf n}\phi_t\|_{L^2}\|\nabla_{\boldsymbol\tau}\phi\|_{L^\infty}
+\|\partial_{\bf n}\phi\|_{L^\infty}\|\nabla_{\boldsymbol\tau}\phi_t\|_{L^2}
\right)\|{\bf u}_t\|_{H^1} \nonumber
\\
&\quad+C\|\phi_t\|_{H^2}\|\nabla_{\boldsymbol\tau}\phi\|_{H^2}\|{\bf u}_t\|_{H^1}
+C|\nu \cos\theta_s| \|\nabla_{\boldsymbol\tau}\phi_t\|_{H^1}\|{\bf u}_t\|_{H^1}  \nonumber
\\
&\le C\|{\bf u}_t\|_{H^1}\left(\|\nabla{\bf u}\|_{H^1}\|{\bf u}_t\|_{H^1}+\|\nabla\phi\|_{H^2}\|\nabla\phi_t\|_{H^1}
\right)+C\varepsilon_0\|\nabla\phi_t\|_{H^1}\|{\bf u}_t\|_{H^1} \nonumber
\\
&\le C\tilde{\mathcal D}(t)^{\frac{1}{2}}\tilde{\mathcal E}(t)^{\frac{1}{2}}\tilde{\mathcal D}(t)^{\frac{1}{2}}+C\varepsilon_0\tilde{\mathcal D}(t),    \nonumber
\end{align}
where we have used Lemma \ref{Trace} and (\ref{L21}). Then, due to Korn's inequality and (\ref{poincare u}), we obtain 
\begin{align}
\label{r34}
&\sup\limits_{0\le t \le T}\|{\bf u}_t\|_{L^2}^2
+\int_{0}^{T}\left(\|{\bf u}_t\|_{H^1}^2+\|{\bf u}_{t{\boldsymbol\tau}}\|_{L^2(\Gamma)}^2
\right) {\rm d}t \nonumber
\\
&\le C\tilde{\mathcal E}(0)+C\int_{0}^{T}\left(\varepsilon_0+\tilde{\mathcal E}(t)^{\frac{1}{2}}
\right)\tilde{\mathcal D}(t) {\rm d}t.
\end{align}
{\it\bfseries Step 2. Estimates of $\nabla{\bf u}$.}
Testing $(\ref{NSAC1})_2$ by ${\bf u}_t$, by using Lemma \ref{Trace}, one has
\begin{align}
&\frac{1}{2}\frac{\rm d}{{\rm d}t}\left(\frac{1}{2}\|\mathbb S({\bf u})\|_{L^2}^2+\beta\|{\bf u}_{\boldsymbol\tau}\|_{L^2(\Gamma)}^2
\right)+\|{\bf u}_t\|_{L^2}^2               \nonumber
\\
&=-\int_\Omega{\bf u}\cdot\nabla{\bf u}\cdot{\bf u}_t {\rm d}x-\int_\Omega\Delta\phi\nabla\phi\cdot{\bf u}_t{\rm d}x+ \int_\Gamma L(\psi)\nabla_{\boldsymbol\tau}\psi\cdot{\bf u}_{t{\boldsymbol\tau}}{\rm d}S     \nonumber
\\
&\le C\|{\bf u}\|_{L^\infty}\|\nabla{\bf u}\|_{L^2}\|{\bf u}_t\|_{L^2}+C\|\Delta\phi\|_{L^2}\|\nabla\phi\|_{L^\infty}\|{\bf u}_t\|_{L^2} \nonumber
\\
&\quad+ C \|L(\psi)\|_{L^4(\Gamma)}\|\nabla_{\boldsymbol\tau}\psi\|_{L^4(\Gamma)} \|{\bf u}_{t{\boldsymbol\tau}}\|_{L^2(\Gamma)}   \nonumber
\\
&\le C\|{\bf u}\|_{H^2}\|\nabla{\bf u}\|_{L^2}\|{\bf u}_t\|_{L^2}+C\|\Delta\phi\|_{L^2}\|\nabla\phi\|_{H^2}\|{\bf u}_t\|_{L^2} \nonumber
\\
&\quad+C\|L(\psi)\|_{H^1(\Gamma)}\|\nabla_{\boldsymbol\tau}\phi\|_{H^2}\|{\bf u}_t\|_{H^1}  \nonumber
\\
&\le C\tilde{\mathcal D}(t)^{\frac{1}{2}}\tilde{\mathcal E}(t)^{\frac{1}{2}}\tilde{\mathcal D}(t)^{\frac{1}{2}}.   \nonumber
\end{align}
Integrating it over (0,T), then combining with Korn's inequality and (\ref{r1}) shows
\begin{align}
\label{r33}
&\sup\limits_{0\le t \le T}\left(\|{\bf u}\|_{H^1}^2+\|{\bf u}_{\boldsymbol\tau}\|_{L^2(\Gamma)}^2
\right)+\int_{0}^{T}\|{\bf u}_t\|_{L^2}^2 {\rm d}t  \nonumber
\\
&\le C\left(\varepsilon_0+\tilde{\mathcal E}(0)\right)+C\int_{0}^{T}\left(\tilde{\mathcal E}(t)^{\frac{1}{2}}
\right)\tilde{\mathcal D}(t) {\rm d}t.
\end{align}
{\it\bfseries Step 3. Estimates of $\nabla^k{\bf u}, k=2,3$.}
Directly applying \cite{1} (see Theorem 1.2 therein) and noting that $\tilde{\mathcal E}(t)\le 1$, together with Lemma $\ref{Trace}$ to obtain
\begin{align}
\label{r35}
\|{\bf u}\|_{H^2}^2&\le C\left(\|{\bf u}\cdot\nabla{\bf u}\|_{L^2}^2+\|{\rm div}(\nabla\phi\otimes\nabla\phi)\|_{L^2}^2+\|{\bf u}_t\|_{L^2}^2+ \|\left(\partial_{\bf n}\phi+\gamma_{fs}^\prime(\psi)
\right)\nabla_{\boldsymbol\tau}\psi
\|_{H^{\frac{1}{2}}(\Gamma)}^2
\right)                                                 \nonumber
\\
&\le C\left(\|{\bf u}\|_{L^\infty}^2\|\nabla{\bf u}\|_{L^2}^2+\|\Delta\phi\|_{L^2}^2\|\nabla\phi\|_{L^\infty}^2
+\|{\bf u}_t\|_{L^2}^2+\left\Vert\left(\partial_{\bf n}\phi+\gamma_{fs}^\prime(\phi)\right)\nabla\phi\right\Vert_{H^1}^2
\right)                                                  \nonumber
\\
&\le C\left(\|{\bf u}\|_{H^2}^2\|\nabla{\bf u}\|_{L^2}^2+\|\Delta\phi\|_{L^2}^2\|\nabla\phi\|_{H^2}^2
+\|{\bf u}_t\|_{L^2}^2+\left\Vert\partial_{\bf n}\phi+\gamma_{fs}^\prime(\phi)\right\Vert_{L^\infty}^2
\|\nabla\phi\|_{H^1}^2
\right)                                                \nonumber
\\
&\le C\left(\|{\bf u}\|_{H^2}^2\|\nabla{\bf u}\|_{L^2}^2+\|\Delta\phi\|_{L^2}^2\|\nabla\phi\|_{H^2}^2
+\|{\bf u}_t\|_{L^2}^2+\|\nabla\phi\|_{H^2}^2\|\nabla\phi\|_{H^1}^2+\|\nabla\phi\|_{H^1}^2
\right)                                              \nonumber
\\
&\le C\left(\tilde{\mathcal E}(t)+1\right)\left(\|\nabla{\bf u}\|_{L^2}^2+\|\Delta\phi\|_{L^2}^2+\|{\bf u}_t\|_{L^2}^2+\|\nabla\phi\|_{H^1}^2
\right)  \nonumber
\\
&\le C\left(\varepsilon_0+\tilde{\mathcal E}(0)\right)+C\int_{0}^{T}\left(\varepsilon_0+\tilde{\mathcal E}(t)^{\frac{1}{2}}
\right)\tilde{\mathcal D}(t) {\rm d}t,
\end{align}
where we have used $(\ref{r17})$, $(\ref{r34})$ and $(\ref{r33})$.
 Similarly, using (\ref{L5.1}) for $r=s_1=s_2=\frac{3}{2}>\frac{2}{2}$, it follows from (\ref{poincare u}) and Korn's inequality that
\begin{align}
\|\nabla{\bf u}\|_{H^2}^2&\le C\left(\|\nabla\left({\bf u}\cdot\nabla{\bf u}\right)\|_{L^2}^2+\|\nabla(\nabla\phi\Delta\phi)\|_{L^2}^2+\|\nabla{\bf u}_t\|_{L^2}^2+ \|\left(\partial_{\bf n}\phi+\gamma_{fs}^\prime(\psi)
\right)\nabla_{\boldsymbol\tau}\psi
\|_{H^{\frac{3}{2}}(\Gamma)}^2
\right)                                                 \nonumber
\\
&\le C\left(\|\nabla{\bf u}\cdot\nabla{\bf u}+{\bf u}\cdot\nabla^2{\bf u}\|_{L^2}^2
+\|\nabla^2\phi\Delta\phi+\nabla\phi\nabla\Delta\phi\|_{L^2}^2
\right)                          \nonumber
\\
&\quad+C\|\nabla{\bf u}_t\|_{L^2}^2+C\left(\|\partial_{\bf n}\phi\|_{H^{\frac{3}{2}}(\Gamma)}^2
+ \|\gamma_{fs}^\prime(\psi)\|_{H^{\frac{3}{2}}(\Gamma)}^2
\right) \|\nabla_{\boldsymbol\tau}\psi\|_{H^{\frac{3}{2}}(\Gamma)}^2  \nonumber
\\
&\le C\left(\|\nabla{\bf u}\|_{L^4}^4+\|{\bf u}\|_{L^\infty}^2\|\nabla^2{\bf u}\|_{L^2}^2+\|\nabla^2\phi\|_{L^4}^2\|\Delta\phi\|_{L^2}^2
+\|\nabla\phi\|_{L^\infty}^2\|\nabla\Delta\phi\|_{L^2}^2
\right)                                            \nonumber
\\
&\quad+C\|\nabla{\bf u}_t\|_{L^2}^2+C\left(\|\nabla\phi\|_{H^2}^2+\|\phi^2-1\|_{H^2}^2
\right)\|\nabla_{\boldsymbol\tau}\phi\|_{H^2}^2
\nonumber
\\
&\le C\left(\|\nabla{\bf u}\|_{H^1}^4+\|{\bf u}\|_{H^2}^2\|\nabla^2{\bf u}\|_{L^2}^2
+\|\nabla^2\phi\|_{H^1}^2\|\Delta\phi\|_{L^2}^2+\|\nabla\phi\|_{H^2}^2\|\nabla\Delta\phi\|_{L^2}^2
\right)                                            \nonumber
\\
&\quad+ C \|\nabla{\bf u}_t\|_{L^2}^2 +C\left(\|\nabla\phi\|_{H^2}^2+\|\phi^2-1\|_{H^2}^2
\right)\|\nabla_{\boldsymbol\tau}\phi\|_{H^2}^2
\nonumber
\\
&\le C\|{\bf u}_t\|_{H^1}^2+C\tilde{\mathcal E}(t)\tilde{\mathcal D}(t),  \nonumber
\end{align}
which together with (\ref{r1}) implies
\begin{align}
\label{r36}
\int_{0}^{T}\|{\bf u}\|_{H^3}^2{\rm d}t\le C\left(\varepsilon_0+\tilde{\mathcal E}(0)\right)+C\int_{0}^{T}\left(\varepsilon_0+\tilde{\mathcal E}(t)^{\frac{1}{2}}
\right)\tilde{\mathcal D}(t) {\rm d}t,
\end{align}
Combining (\ref{r36}) with $(\ref{r34})-(\ref{r35})$ yields (\ref{r37}). Hence, we complete the proof.
\end{proof}
\vskip2mm
With the above estimates at hand, we are in position to prove Proposition \ref{proposition global}.
\\
\noindent{\rm\bfseries Proof of Proposition $\ref{proposition global}$.}\quad
Combining Lemma $\ref{rgl1}-\ref{rgl6}$ yields that
\begin{align}
\sup\limits_{0\le t \le T}\tilde{\mathcal E}(t)+\int_{0}^{T}\tilde{\mathcal D}(t) {\rm d}t&\le C_2\left(\varepsilon_0+\tilde{\mathcal E}(0)\right)+C_2\int_{0}^{T}\left(
\varepsilon_0+\tilde{\mathcal E}(t)^\frac{1}{2}\right)\tilde{\mathcal D}(t) {\rm d}t
\nonumber
\\
&\le C_2\left(\varepsilon_0+\tilde{\mathcal E}(0)\right)+\left(C_2\varepsilon_0+C_2\sigma^\frac{1}{2}\right)\int_{0}^{T}\tilde{\mathcal D}(t) {\rm d}t, \nonumber
\end{align}
where $C_2\sigma^\frac{1}{2}<1/4$ and $C_2\varepsilon_0<1/4$, then (\ref{555}) holds.
\vskip2mm
Next, multiplying $(\ref{NSAC1})_2$ by $\bf u$, integrating the result over $\Omega$ by parts, together with (\ref{111}) and (\ref{666}) yields
\begin{align}
&\frac{1}{2}\frac{\rm d}{{\rm d}t}\|{\bf u}\|_{L^2}^2+\frac{1}{2}\|{\mathbb S}({\bf u})\|_{L^2}^2+\beta\|{\bf u}_{\boldsymbol\tau}\|_{L^2(\Gamma)}^2
\nonumber
\\
&=\int_\Gamma L(\psi){\bf u}_{\boldsymbol\tau}\cdot\nabla_{\boldsymbol\tau}\psi {\rm d}S -\int_\Omega\Delta\phi\nabla\phi\cdot{\bf u} {\rm d}x \nonumber
\\
&=\int_\Gamma L(\psi){\bf u}_{\boldsymbol\tau}\cdot\nabla_{\boldsymbol\tau}\psi {\rm d}S +\int_\Omega\left[(\mu-\bar\mu)-f\right]\nabla\phi\cdot{\bf u}{\rm d}x+\bar\mu\int_\Omega\nabla\phi\cdot{\bf u}{\rm d}x \nonumber
\\
&=-\|L(\psi)\|_{L^2(\Gamma)}^2-\int_\Gamma L(\psi)\psi_t {\rm d}S +\int_\Omega(\mu-\bar\mu)\nabla\phi\cdot{\bf u}{\rm d}x
\nonumber
\\
&\le- \|L(\psi)\|_{L^2(\Gamma)}^2+\|L(\psi)\|_{L^2(\Gamma)}\|\psi_t\|_{L^2(\Gamma)}
+\|\mu-\bar\mu\|_{L^2}\|\nabla\phi\|_{L^\infty}\|{\bf u}\|_{L^2}         \nonumber
\\
&\le-\frac{1}{2}\|L(\psi)\|_{L^2(\Gamma)}^2
+C\|\psi_t\|_{L^2(\Gamma)}^2
+C\|\mu-\bar\mu\|_{L^2}^2+C\|\nabla\phi\|_{H^2}^2\|{\bf u}\|_{L^2}^2  \nonumber
\\
&\le-\frac{1}{2}\|L(\psi)\|_{L^2(\Gamma)}^2
+C\|\phi_t\|_{H^1}^2+C\|\mu-\bar\mu\|_{L^2}^2+C\tilde{\mathcal E}(t)\|\nabla{\bf u}\|_{L^2}^2, \nonumber
\end{align}
which together with Korn's inequality and (\ref{small energy}) implies 
\begin{align}
\label{r38}
&\frac{\rm d}{{\rm d}t}\|{\bf u}\|_{L^2}^2+\|\nabla{\bf u}\|_{L^2}^2+\|L(\psi)\|_{L^2(\Gamma)}^2 +\|{\bf u}_{\boldsymbol\tau}\|_{L^2(\Gamma)}^2 \nonumber
\\
&\le C_3\left(\|\mu-\bar\mu\|_{L^2}^2+\|\nabla\phi_t\|_{L^2}^2+\varepsilon_0\|\nabla{\bf u}\|_{L^2}^2\right).
\end{align}
Then, substituting $K_1^0-K_3^0$ into $(\ref{r13})$ for $k=0$, we obtain
\begin{align}
&\frac{1}{2}\frac{\rm d}{{\rm d}t}\left( \|L(\psi)\|_{L^2(\Gamma)}^2 +\|\mu-\bar\mu\|_{L^2}^2
\right)+\|\nabla\phi_t\|_{L^2}^2+\int_\Omega(3\phi^2-1)|\mu-\bar\mu|^2{\rm d}x \nonumber
\\
&\le C\left(\|\nabla{\bf u}\|_{L^2}\|\nabla\phi\|_{H^2}\|\nabla\phi_t\|_{L^2}+\|\mu-\bar\mu\|_{L^2}\|{\bf u}\|_{H^1}\|\nabla\phi\|_{H^1}+\varepsilon_0 \|L(\psi)\|_{L^2(\Gamma)}\|\psi_t\|_{L^2(\Gamma)}   \right)  \nonumber
\\
&\le C\left(\tilde{\mathcal E}(t)^\frac{1}{2}\|\nabla{\bf u}\|_{L^2}\|\nabla\phi_t\|_{L^2}+\tilde{\mathcal E}(t)^\frac{1}{2}\|\mu-\bar\mu\|_{L^2}\|\nabla{\bf u}\|_{L^2}+\varepsilon_0  \|L(\psi)\|_{L^2(\Gamma)} \|\nabla\phi_t\|_{L^2}\right) \nonumber
\\
&\le\frac{1}{2}\|\mu-\bar\mu\|_{L^2}^2
+\frac{1}{2}\|\nabla\phi_t\|_{L^2}^2+C\left(
\tilde{\mathcal E}(t)\|\nabla{\bf u}\|_{L^2}^2+\varepsilon_0\|L(\psi)\|_{L^2(\Gamma)}^2  \right), \nonumber
\end{align}
which together with $3\phi^2-1\ge1$ and (\ref{small energy}) yields
\begin{align}
\label{r39}
&\frac{\rm d}{{\rm d}t}\left( \|L(\psi)\|_{L^2(\Gamma)}^2 +\|\mu-\bar\mu\|_{L^2}^2
\right)+\|\nabla\phi_t\|_{L^2}^2+\|\mu-\bar\mu\|_{L^2}^2 \nonumber
\\
&\le C_4\left(\sigma\|\nabla{\bf u}\|_{L^2}^2+\varepsilon_0 \|L(\psi)\|_{L^2(\Gamma)}^2
\right).
\end{align}
Multiplying (\ref{r39}) by $C_3+1$, combining the result with (\ref{r38}) and choosing $\sigma$ and $\varepsilon_0$ small enough such that $C_3\varepsilon_0\le\frac{1}{8}$, $(C_3+1)C_4\sigma\le\frac{1}{8}$ and $(C_3+1)C_4\varepsilon_0\le\frac{1}{4}$, then we have
\begin{align}
\label{r40}
\frac{\rm d}{{\rm d}t}{\mathcal A}_3(t)+{\mathcal B}_3(t)\le 0,
\end{align}
where
\begin{align}
&{\mathcal A}_3(t)=\|{\bf u}\|_{L^2}^2+(C_3+1) \|L(\psi)\|_{L^2(\Gamma)}^2 +(C_3+1)\|\mu-\bar\mu\|_{L^2}^2,  \nonumber
\\
&{\mathcal B}_3(t)=\|\nabla\phi_t\|_{L^2}^2+\|\mu-\bar\mu\|_{L^2}^2+\frac{1}{2}\|\nabla{\bf u}\|_{L^2}^2+\|{\bf u}_{\boldsymbol\tau}\|_{L^2(\Gamma)}^2+\frac{1}{2} \|L(\psi)\|_{L^2(\Gamma)}^2 . \nonumber
\end{align}
It is easy to get ${\mathcal A}_3(t)\le C{\mathcal B}_3(t)$ from the Korn's inequality. Therefore, it can be obtained from (\ref{r40}) that there exists a positive constant $\alpha$ such that
\begin{align}
\|{\bf u}(\cdot,t)\|_{L^2}^2+\| L(\psi)(\cdot,t)\|_{L^2(\Gamma)}^2 +\|(\mu-\bar\mu)(\cdot,t)\|_{L^2}^2\le C\tilde{\mathcal E}(0)e^{-\alpha t},
\end{align}
for all $t>0$.
Thus, the proof of Proposition \ref{proposition global} is completed.
\hfill$\Box$
\subsubsection{\bf Proof of Theorem \ref{Th rg}}\quad
With the a priori estimates at hand, we are prepared for the proof of Theorem \ref{Th rg}. First, according to Proposition \ref{TH im}, there exists a $T_0>0$ such that the problem (\ref{NSAC1})--(\ref{initial1}) has a unique local-in-time strong solution $({\bf u},\phi, \psi)$. We will extend the local existence time $T_0$ to be infinity and therefore prove the global existence result. Recall the constant $\sigma_0$, $C_1$ in Proposition \ref{TH im} and $C_2$ in Proposition \ref{proposition global}. Without loss of generality, assume $C_1,C_2\ge1$. It follows from the definition of $\tilde{\mathcal E}(0)$ and (\ref{small initial}) that there exists a positive constant $C_0\ge1$ such that $\tilde{\mathcal E}(0)\le C_0\varepsilon_0$ for
\begin{align}
\label{r41}
\varepsilon_0={\rm min}\left\{\frac{1}{8C_3},\frac{1}{4C_2},\frac{1}{4(C_3+1)C_4},
\frac{\sigma_0}{2C_0C_1C_2},\frac{\sigma}{2C_0C_1C_2}
\right\}.
\end{align}
Let $T^*$ be the maximal time of existence for the strong solution of (\ref{NSAC1})-(\ref{initial1}), satisfying
$$\sup\limits_{0\le t \le T}\tilde{\mathcal E}(t)\le 2C_0C_1C_2\varepsilon_0,$$
and therefore $T^*\ge T_0$. The claim is that $T^*=\infty$,  otherwise, we assume  $T^*<\infty$ for a contradiction. First, for any $0<T\le T^*$, it follows from (\ref{r41}) that
\begin{align}
\sup\limits_{0\le t \le T}\tilde{\mathcal E}(t)\le 2C_0C_1C_2\varepsilon_0\le\sigma.
\end{align}
Then, from Proposition \ref{proposition global}, we obtain
\begin{align}
\label{r43}
\sup\limits_{0\le t \le T}\tilde{\mathcal E}(t)\le C_2(\varepsilon_0+\tilde{\mathcal E}(0))\le C_2(1+C_0)\varepsilon_0,
\end{align}
which implies
\begin{align}
\sup\limits_{0\le t \le T}\tilde{\mathcal E}(t)\le\sigma_0.
\end{align}
Now we take $\tilde{\mathcal E}(T^*)$ as the initial data at $T^*$ and then use Proposition \ref{TH im} and (\ref{r43}) to find that there exists a unique solution to (\ref{NSAC1})--(\ref{initial1}) on $[T^*,T^*+T_0]$ satisfying
\begin{align}
\sup\limits_{T^*\le t \le T^*+T_0}\tilde{\mathcal E}(t)\le C_1\tilde{\mathcal E}(T^*)\le C_1C_2(1+C_0)\varepsilon_0\le 2C_0C_1C_2\varepsilon_0.
\end{align}
This contradicts the assumption on $T^*$, so $T^*=\infty$. We finally show that $T^*$ could be infinity and complete the proof of the global existence of the strong solution and the estimates (\ref{big time}) follows from Proposition \ref{proposition global}.

\section*{Acknowledgments}
Li's work is supported by the National Natural Science Foundation of China (No.12371205)
and the Natural Science Foundation of Guangdong Province (No.2025A1515012026).


\end{document}